\documentclass[11pt,leqno]{article}
\usepackage{a4wide}
\usepackage{amssymb,upgreek}
\usepackage{amsmath}
\usepackage{amsthm}
\usepackage{mathrsfs,fancybox,mdframed}
\usepackage{bm,euscript,csquotes,comment,appendix,multirow}
\usepackage{upgreek}
\usepackage[curve]{xypic}
\usepackage[usenames,dvipsnames]{xcolor}
\usepackage{enotez}
\usepackage{calligra,yfonts}
\DeclareMathOperator{\cHom}{\mathscr{H}\text{\kern -3pt {\calligra\large om}}\,}
\newenvironment{psmallmat}
  {\left(\begin{smallmatrix}}
  {\end{smallmatrix}\right)}

\usepackage{hyperref}

\hypersetup{colorlinks,  citecolor=ForestGreen,  linkcolor=blue, urlcolor=black}

\theoremstyle{plain}

\newcommand{\id}{\operatorname{id}}
\newcommand{\im}{\operatorname{im}}

\newcommand{\pr}{\operatorname{pr}}
\newcommand{\ev}{\operatorname{ev}}
\newcommand{\Sh}{\operatorname{Sh}}
\newcommand{\res}{\operatorname{res}}

\newcommand{\Hom}{\operatorname{Hom}}
\newcommand{\End}{\operatorname{End}}
\newcommand{\Ext}{\operatorname{Ext}}

\newcommand{\Mod}{\operatorname{Mod}}

\newcommand{\Inj}{\operatorname{Inj}}
\newcommand{\ind}{\operatorname{ind}}

\newcommand{\coker}{\operatorname{coker}}

\newcommand{\chara}{\operatorname{char}}
\newcommand{\cores}{\operatorname{cores}}
\newcommand{\val}{\operatorname{val}}

\def\c{\mathbf{ c}}
\def\a{\mathbf{ a}}

\def\fhom{\mathfrak h}
\def\ften{\mathfrak t}
\def\X{{\mathbf X}}

\def\anti{\EuScript J}
\def\trace{\EuScript S}
\def\chara{{\rm char}}
\def\ima{\check\alpha([\mathbb F_q^\times])}
\def\CZ{{\mathcal C_{E^*}(Z)}}
\newcommand{\cz}[1]{{\mathcal C_{E^#1}(Z)} }

\def\ima{\check\alpha([\mathbb F_q^\times])}

\def\go{{\gamma_0}}

\def\pw{{{}^{\varpi} \!w}}

\newtheorem{theorem}{Theorem}[section]
\newtheorem{corollary}[theorem]{Corollary}
\newtheorem{lemma}[theorem]{Lemma}
\newtheorem{proposition}[theorem]{Proposition}

\newtheorem*{theorem*}{Theorem}
\newtheorem*{proposition*}{Proposition}
\newtheorem{fact}[theorem]{Fact}

\theoremstyle{remark}

\newtheorem{example}[theorem]{Example}

\newtheorem{notn}[theorem]{Notation}

\theoremstyle{definition}

\newtheorem{remark}[theorem]{Remark}

\def\colorred#1{\textcolor{red}{#1}}
\def\colorblue#1{\textcolor{blue}{#1}}
\def\colororange#1{\textcolor{Orange}{#1}}

\def\colormag#1{\textcolor{magenta}{#1}}

\title{On the pro-$p$ Iwahori Hecke Ext-algebra of ${\rm SL}_2(\mathbb Q_p)$}
\author{Rachel Ollivier, Peter Schneider}
\date{\today}

\begin{document}

\maketitle

\begin{abstract}
Let
 $G={\rm  SL}_2(\mathfrak F) $ where $\mathfrak F$ is a finite extension of $\mathbb Q_p$. We suppose that
the pro-$p$ Iwahori subgroup  $I$ of $G$ is a Poincar\'e group of dimension $d$.
Let $k$ be a field  containing the residue field of $\mathfrak F$.


In this article, we study the
 graded Ext-algebra  $E^*=\Ext_{\Mod(G)}^*(k[G/I], k[G/I])$.
Its degree zero piece $E^0$ is the usual pro-$p$ Iwahori-Hecke algebra $H$.

We study $E^d$ as an $H$-bimodule and deduce  that for an irreducible admissible smooth representation of $G$, we have $H^d(I,V)=0$ unless $V$ is the trivial representation.

When $\mathfrak F=\mathbb Q_p$ with $p\geq 5$, we have $d=3$. In that case we describe $E^*$ as an $H$-bimodule and   give the structure  as an algebra of the centralizer in $E^*$ of the center of $H$.
We deduce results on the values of the functor $H^*(I, {}_-)$ which attaches to a (finite length) smooth $k$-representation $V$ of $G$ its cohomology with respect to $I$. We prove that $H^*(I,V)$ is always  finite dimensional.  Furthermore, if  $V$ is  irreducible, then $V$ is supersingular if and only if
$H^*(I,V)$ is a supersingular $H$-module.
\end{abstract}

\setcounter{tocdepth}{4}

\tableofcontents

\section{Introduction}

Let $\mathfrak F$ be a locally compact nonarchimedean field with residue
characteristic $p$,   and  let    $G$ be the group of $\mathfrak{F}$-rational points
of a connected reductive group  $\mathbf G$ over $\mathfrak{F}$. We
suppose that $\mathbf G$ is $\mathfrak F$-split.

Let $k$ be a field  of characteristic $p$ and let  $\Mod(G)$ denote the category of all smooth
representations of $G$ in $k$-vector spaces. For a general  $\mathbf G$  and  $\mathfrak F$  this category is still poorly understood.
One way of approaching it consists in  considering the Hecke algebra $H$ of the
pro-$p$ Iwahori subgroup $I \subset G$. In this case the natural left exact functor
\begin{align*}
 \fhom: \Mod(G) & \longrightarrow \Mod(H) \\
              V & \longmapsto V^I = \Hom_{k[G]}(\mathbf{X},V) \
\end{align*}
sends a nonzero representation onto a nonzero module. Its left adjoint is
\begin{align*}
\ften:\Mod(H) & \longrightarrow  \Mod^I(G)\subseteq \Mod(G) \\
                M & \longmapsto \mathbf{X} \otimes_H M \ .
\end{align*}
Here $\mathbf X$ denotes the space of $k$-valued functions with compact
support on $G/I$ with the natural left action of $G$.
The functor $\ften$ has values in the category  $\Mod^I(G)$ of all smooth
$k$-representations of $G$ generated by their $I$-fixed vectors.  This category,
which a priori has no reason to be an abelian subcategory of $\Mod(G)$, contains all
irreducible representations. But in general $\ften$ is not an equivalences of categories
and little is known about $\Mod^I(G)$ and $\Mod(G)$ unless $G={\rm GL}_2(\mathbb Q_p)$ or $G={\rm SL}_2(\mathbb Q_p)$ (\cite{Koz1}, \cite{Ollequiv}, \cite{embed}, \cite{Pas}).

The functor $\fhom$, although left exact, is not right exact since $p$ divides the pro-order of $I$. It is therefore natural to consider the derived functor. In \cite{SDGA} the following result is shown: When
$\mathfrak F$ is a finite extension of $\mathbb Q_p$ and $I$ is a torsionfree
pro-$p$ group, there exists a derived version of the functors $\fhom$ and $\ften $ providing an equivalence between the derived category $D(G)$ of smooth representations of $G$ in $k$-vector spaces and the derived category of differential graded modules over a certain differential graded pro-$p$ Iwahori-Hecke algebra $H^\bullet$.

The article \cite{Ext} opened up the study of the  Hecke differential graded algebra $H^\bullet$ by giving the first results on its
cohomology algebra $E^*:=\Ext_{\Mod(G)}^*(\mathbf X, \mathbf X)$. This is the pro-$p$ Iwahori Hecke Ext-algebra we refer to in the title of the current article. We  suppose in this introduction  that $I$ is torsion free which forces $\mathfrak F$ to  be a finite extension of $\mathbb Q_p$. We denote by $d$ the dimension of $I$ as a Poincar\'e group. The $\Ext$ algebra $E^*$ is supported in degrees $0$ to  $d$.

When  $\mathbf G$ is almost simple and simply connected,  the central ideal $\mathfrak J H$ which controls the supersingularity (see \S\ref{subsec:basicshecke}) has finite codimension in $H$. We show that we have an isomorphism of $H$-bimodules
$$
\Ext_{\Mod(G)}^d(\mathbf X, \mathbf X)\cong \chi_{triv}\oplus \Inj((H/\mathfrak{J} H)^\vee)
$$
where $\chi_{triv}$ is the trivial character of $H$ and
$\Inj(H/\mathfrak{J} H)^\vee)$ is an injective envelope of the dual module $(H/\mathfrak{J} H)^\vee$. We exploit this result when $\mathbf G={\rm SL}_2$ as follows.

\begin{proposition*}(Proposition \ref{prop:HdVzero}).
  We have $H^d(I, V)=0$ for any irreducible admissible representation of ${\rm SL}_2(\mathfrak F)$ except when  $V$ is the trivial representation in which case $H^d(I,k_{triv})\cong\chi_{triv} $ as an $H$-bimodule.
\end{proposition*}

We then focus on the case $G={\rm SL}_2(\mathbb Q_p)$.  In that case the center of $H$ contains a polynomial algebra $k[\zeta]$ and  $\mathfrak J H=\zeta H$. We give explicit formulas for the action of $H$  on $E^*$ (see in particular \S\ref{subsec:HonE1} and \S\ref{subsec:HonE2}). This will allow us
\begin{itemize}
  \item[--] to study the left and right $k[\zeta]$-torsion in $E^*$ in \S\ref{sec:zetatorsion},
  \item[--] and then in \S\ref{sec:structure} to give a complete description of $E^*$ as an $H$-bimodule.
\end{itemize}
As a major application we prove the following result in \S\ref{sec:valuesH*}.

\begin{theorem*}(Theorem \ref{cohomology}).
Let $G={\rm SL}_2(\mathbb Q_p)$ with $p\neq 2,3$. For any representation of finite length in $\Mod(G)$ we have:
\begin{itemize}
  \item[i.] The $k$-vector space $H^*(I,V)$ is finite dimensional;
  \item[ii.] if $V$  is generated by its subspace of $I$-fixed vectors $V^I$  and $Q(\zeta) V^I = 0$ for some nonzero polynomial $Q \in k[X]$, then the left $H$-module $H^*(I,V)$ is $P(\zeta)$-torsion for the polynomial $P(X) := Q(X)Q(\frac{1}{X})X^{\deg(Q)}$.
\end{itemize}
\end{theorem*}

The most interesting consequence of this theorem is that, under the same hypotheses, an irreducible representation $V$ in $\Mod(G)$ is supersingular if and only if the left $H$-module $H^*(I,V)$ is supersingular (Corollary \ref{coro:supersingular}). This strongly indicates that the notion of supersingularity can be extended to objects in the derived category $D(G)$ by introducing a theory of supports via the dg algebra $H^\bullet$. We hope to return to this in another paper.

The center $Z := Z(H)$ of $H = E^0$ is no longer central in $E^*$. This makes it interesting to determine its centralizer $\CZ$ in $E^*$. In \S\ref{sec:structure} we show that $\CZ$ actually coincides with the centralizer in $E^*$ of the element $\zeta$. This makes it possible to give a complete description of the algebra $\CZ$ in \S\ref{sec:commutator}.

In \cite{embed} \S 3.5 we studied the representation theoretic meaning of the localization $H_\zeta$ of the Hecke algebra in the central element $\zeta$. Despite the fact that $\zeta$ is no longer central in $E^*$ it turns out (Remark \ref{Ore}) that $\zeta^{\mathbf{N}_0}$ is a left and right Ore set in $E^*$, so that the localization $E^*_\zeta$ does exist. We will show elsewhere that $E^*_\zeta$ again is a Yoneda Ext-algebra and will investigate its meaning for the nonsupersingular ${\rm SL}_2(\mathbb Q_p)$-representations.

This collaboration was partially funded by the NSERC Discovery Grant of the first author and by the Deutsche Forschungsgemeinschaft (DFG, German Research Foundation) under Germany's Excellence Strategy EXC 2044 –390685587, Mathematics Münster: Dynamics–Geometry–Structure. Both authors also acknowledge the support of PIMS at UBC Vancouver.

\section{Notations and preliminaries\label{sec:not}}

Throughout  the paper we fix a locally compact nonarchimedean field $\mathfrak{F}$ (for now of any characteristic) with ring of integers $\mathfrak{O}$, its maximal ideal $\mathfrak{M}$,  and a prime element $\pi$. The residue field $\mathfrak{O}/\pi \mathfrak{O}$ of $\mathfrak{F}$ is $\mathbb{F}_q$ for some power $q = p^f$ of the residue characteristic $p$. We choose the valuation $\val_{\mathfrak{F}}$ on   $\mathfrak{F}$  normalized by $\val_{\mathfrak{F}}(\pi)=1$ We let $G := \mathbf{G}(\mathfrak{F})$ be the group of $\mathfrak{F}$-rational points of a connected reductive group $\mathbf{G}$  over $\mathfrak{F}$ which we always assume to be $\mathfrak{F}$-split. We will very soon specialize to the case when
$\mathbf G$ is almost simple and simply connected (starting Section \ref{sec:top}) and in fact the core of this article (starting Section \ref{sec:E1}) will  focus on the case when $\mathbf{G}={\rm SL}_2$ and
 $\mathfrak{F}=\mathbb Q_p$ with $p\neq 2,3$.

We fix an $\mathfrak{F}$-split maximal torus $\mathbf{T}$ in $\mathbf{G}$, put $T := \mathbf{T}(\mathfrak{F})$, and let $T^0$ denote the maximal compact subgroup of $T$ and $T^1$ the pro-$p$ Sylow subgroup of $T^0$. We also fix a chamber $C$ in the apartment of the semisimple Bruhat-Tits building  $\mathscr X$ of $G$ which corresponds to $\mathbf{T}$. The stabilizer $\EuScript P_C^\dagger$ of $C$ contains an Iwahori subgroup $J$. Its pro-$p$ Sylow subgroup $I$ is called the pro-$p$ Iwahori subgroup. We have $T\cap J= T^0$ and $T\cap I= T^1$. If $N(T)$ is the normalizer of $T$ in $G$, then we define the group $\widetilde{W} := N(T)/T^1$. In particular, it contains $\Omega:=T^0/T^1$. The quotient $W:=N(T)/T^0\cong\widetilde{W}/\Omega$ is the extended affine Weyl group. The finite Weyl group is $W_0:= N(T)/T$. The length on $W$ pulls back to a length function  $\ell$ on $\widetilde W$ (see \cite{Ext} \S 2.1.4).

For any compact open subset $A \subseteq G$ we let $\mathrm{char}_A$ denote the characteristic function of $A$.

The coefficient field for all representations in this paper is an arbitrary field $k$ of characteristic $p > 0$.
For any open subgroup $U \subseteq G$ we let $\Mod(U)$ denote the abelian category of smooth representations of $U$ in $k$-vector spaces.

\subsection{The pro-$p$-Iwahori Hecke algebra\label{subsec:basicshecke}}

We consider the compact induction $\mathbf{X} := \ind_I^G(1)$ of the trivial $I$-representation.  It can be seen as the space of  compactly supported functions $G\rightarrow k$ which are constant on the left cosets mod $I$. It lies in $\Mod(G)$.
For $Y$ a compact subset of $G$ which is right invariant under $I$,  the characteristic function $\chara_Y$ is an element of $\X$. Equivalently one may view $\X = k[G/I]$ as the $k$-vector space with basis the cosets $gI \in G/I$. The pro-$p$ Iwahori-Hecke algebra is defined to be the $k$-algebra $$H := \End_{k[G ]}(\mathbf{X})^{\mathrm{op}}\ .$$ We often will identify $H$, as a right $H$-module, via the map $
   H  \rightarrow{}  \mathbf{X} ^I,
    h  \mapsto (\mathrm{char}_I ) h
$
 with the submodule $\X^I $ of $I $-fixed vectors in $\X$.
The Bruhat-Tits decomposition of $G$ says that $G$ is the disjoint union of the double cosets $IwI$ for $w \in \widetilde{W}$. Hence we have the $I$-equivariant decomposition
\begin{equation}\label{f:Xbruhat}
  \mathbf{X} = \oplus_{w \in \widetilde{W}} \mathbf{X}(w)  \quad\text{with}\quad  \mathbf{X}(w) := \ind_I^{IwI}(1) \ ,
\end{equation}
where the latter denotes the subspace of those functions in $\mathbf{X}$ which are supported on the double coset $IwI$. In particular, we have $\mathbf{X}(w)^I = k \tau_w$ where $\tau_w := \mathrm{char}_{IwI}$ and hence
$
  H = \oplus_{w \in \widetilde{W}} k \tau_w \
$ as a  $k$-vector space.

The defining (braid and quadratic)  relations of $H$ are recalled in \cite{Ext} \S2.2. They ensure in particular that we have a well defined trivial character of $H$ denoted by $\chi_{triv}$ and defined by (\cite{Ext} \S 2.2.2):
\begin{equation}\label{defitriv}
    \chi_{triv}: \tau_{w}\longmapsto 0, \: \tau_{\omega} \longmapsto 1, \textrm{ for any $w\in \widetilde W$ with $\ell(w)\geq 1$ and $\omega \in \widetilde W$ with $\ell(\omega)=0$.}
\end{equation}


To define the notion of supersingularity for $H$-modules, we refer to \cite{Ext} \S 2.3.
Recall that there is a  a central subalgebra  $\mathcal Z^0(H)$ of $H$  which is isomorphic to the affine semigroup algebra $k[X_*^{dom}(T)]$, where
$X_*^{dom}(T)$ denotes the semigroup of all dominant cocharacters of $T$.
The cocharacters $\lambda\in X_*^{dom}(T)\setminus (-X_*^{dom}(T))$ generate a proper ideal of $k[X_*^{dom}(T)]$, the image of which in $\mathcal Z^0(H)$  is denoted by $\mathfrak J$.
We call an $H$-module $M$ supersingular  if any element in $M$ is annihilated by a power of $\mathfrak J$.

\subsection{The $\Ext$-algebra}\label{subsec:Ext}
We refer to \cite{Ext} \S3. We form the graded $\Ext$-algebra
\begin{equation*}
  E^* := \Ext_{\Mod(G)}^*(\mathbf{X},\mathbf{X})^{\mathrm{op}}
\end{equation*}
over $k$ with the multiplication being the (opposite of the) Yoneda product. Obviously
\begin{equation*}
  H := E^0 = \End_{\Mod(G)}^*(\mathbf{X},\mathbf{X})^{\mathrm{op}}
\end{equation*}
is the usual pro-$p$ Iwahori-Hecke algebra over $k$. By using Frobenius reciprocity for compact induction and the fact that the restriction functor from $\Mod(G)$ to $\Mod(I)$ preserves injective objects we obtain the identification
\begin{equation}\label{f:frob}
  E^* = \Ext_{\Mod(G)}^*(\mathbf{X},\mathbf{X})^{\mathrm{op}} = H^*(I,\mathbf{X}) \ .
\end{equation}
The only part of the multiplicative structure on $E^*$ which is still directly visible on the cohomology $H^*(I,\mathbf{X})$ is the right multiplication by elements in $E^0 = H$, which is functorially induced by the right action of $H$ on $\mathbf{X}$. In \cite{Ext}, we  made the full multiplicative structure visible on $H^*(I,\mathbf{X})$. We recall that for $* = 0$ the above identification is given by
$
  H  \xrightarrow{\; \cong \;} \mathbf{X}^I,
  \tau \longmapsto (\mathrm{char}_I) \tau
$.

Noting that the cohomology of profinite groups commutes with arbitrary sums, we obtain from the $I$-equivariant decomposition \eqref{f:Xbruhat} a decomposition of vector spaces
\begin{equation*}
  H^*(I,\mathbf{X}) = \oplus_{w \in \widetilde{W}} H^*(I,\mathbf{X}(w)) \ .
\end{equation*} For $w\in \widetilde W$, we let $I_w:= I\cap w I w^{-1}$ (see \cite{Ext} \S 2.1.5). We call Shapiro isomorphism
and denote  by $\Sh_w$ the composite map
\begin{equation}\label{f:Shapiro1}
  \Sh_w : H^*(I,\mathbf{X}(w)) \xrightarrow{\; \res \;} H^*(I_w,\mathbf{X}(w)) \xrightarrow{\; H^*(I_w,\ev_w) \;} H^*(I_w,k)
\end{equation}
where $\ev_w : \mathbf{X}(w)  \longrightarrow k, \: f \longrightarrow f(w)$ (see also \cite{Ext} \S 3.2).

\subsubsection{The cup product}

Recall from \cite{Ext} \S 3.3 that there is a naive product structure on the cohomology $H^*(I,\mathbf{X})$. By multiplying maps we obtain the $G$-equivariant map $  \mathbf{X} \otimes_k \mathbf{X}  \longrightarrow \mathbf{X}$,
  $f \otimes f'  \longmapsto f f'$.
It gives rise to the cup product
\begin{equation}\label{f:cup}
  H^i(I, \mathbf{X}) \otimes_k H^j(I,\mathbf{X}) \xrightarrow{\; \cup \;} H^{i+j}(I,\mathbf{X}) \
\end{equation} which has the property that
  $H^i(I, \mathbf{X}(v)) \cup H^j(I,\mathbf{X}(w)) = 0$ whenever $v \neq w$.
On the other hand, since $\ev_{{w}}(f f') = \ev_{{w}}(f) \ev_{{w}}(f')$ and since the cup product is functorial and commutes with cohomological restriction maps, we have the commutative diagrams
\begin{equation}\label{f:cup+Sh}
  \xymatrix{
     H^i(I, \mathbf{X}(w)) \otimes_k H^j(I,\mathbf{X}(w)) \ar[d]_{\Sh_w \otimes \Sh_w} \ar[r]^-{\cup} & H^{i+j}(I,\mathbf{X}(w)) \ar[d]^{\Sh_w} \\
     H^i(I_w,k) \otimes_k H^j(I_w,k) \ar[r]^-{\cup} & H^{i+j}(I_w,k)   }
\end{equation}
for any $w \in \widetilde{W}$, where the bottom row is the usual cup product on the cohomology algebra $H^*(I_w,k)$.  In particular,  the cup product \eqref{f:cup} is anticommutative.

\subsubsection{The Yoneda product}

The Yoneda product in $E^*$ (\cite{Ext} \S 4.2) satisfies the following property:
\begin{equation}\label{f:support}
  H^i(I,\mathbf{X}(v)) \cdot H^j(I,\mathbf{X}(w)) \subseteq H^{i+j}(I, \ind_I^{IvI \cdot IwI}(1)) \text{ for  $v,w\in \widetilde W$. }
\end{equation}
The product of $\alpha\in   H^i(I,\mathbf{X}(v))$ by
$\beta\in H^j(I,\mathbf{X}(w))$ is explicitly described in  \cite{Ext} Prop. 5.3.  We  record here the following  results.

\begin{proposition}\label{prop:yo}
Let $v,w\in\widetilde W$ and  $\alpha\in   H^i(I,\mathbf{X}(v))$,
$\beta\in H^j(I,\mathbf{X}(w))$.
\begin{itemize}
\item[--] if $\ell(vw)=\ell(v)+\ell(w)$, then
\begin{equation}\label{f:prod-goodlength}
 \alpha\cdot \beta = (\alpha\cdot \tau_w) \cup  (\tau_v\cdot \beta)  \in  H^{i+j}(I,\mathbf{X}(vw))\ ;
\end{equation}
\item[--] if $\ell(v)=1$ and   $\ell(v w)=\ell(w)-1$, then
$\alpha\cdot \beta$ lies in  $H^{i+j}(I,\mathbf{X}(v w))\oplus \bigoplus_{\omega\in T^0/T^1} H^{i+j}(I,\mathbf{X}( \omega w))$. If furthermore $\mathbf G$ is semisimple and simply connected, then
\begin{equation}\label{f:prod-badlength}
   \alpha\cdot \beta- (\alpha\cdot \tau_w)\cup (\tau_{v}\cdot \beta) \in H^{i+j}(I, \X(v w))\ .
\end{equation}
\end{itemize}
\end{proposition}
\begin{proof}
The first point is \cite{Ext} Cor. 5.5. We prove the second point in \S\ref{proofbadlength} of the Appendix.
\end{proof}

\subsubsection{Anti-involution\label{subsubsec:anti}}

We refer to  \cite{Ext} \S 6.  The graded algebra $E^*$ is equipped with an
an involutive anti-automorphism. It is defined the following way.
For $w\in \widetilde{W}$, we have $I_{w^{-1}}=w^{-1} I_w w$ and a linear isomorphism
$
  (w^{-1})_*:  H^i(I_w, k)\overset{\cong}\rightarrow  H^i(I_{w^{-1}}, k) \ ,
$
for all $i\geq 0$.
 Via the Shapiro isomorphism \eqref{f:Shapiro1}, this induces the linear isomorphism  $\anti_w$:
\begin{equation}\label{f:invodefi}
  \xymatrix{
H^i(I, \X(w))\ar[d]_{\Sh_w} \ar[rrr]^{\anti_w}_{\cong} && & H^i(I, \X(w^{-1})) \ar[d]^{\Sh_{w^{-1}}}_\cong\\
H^i(I_w, k)\ar[rrr]^{(w^{-1})_*} _\cong& && H^i(I_{w^{-1}}, k)  }
\end{equation}
Summing over all $w\in \widetilde W$, the maps $(\anti_w)_{w\in \widetilde{W}}$ induce a linear isomorphism
\begin{equation*}
  \anti:  H^i(I,\X) \overset{\cong}\longrightarrow H^i(I,\X) \ .
\end{equation*}  and it is proved in \cite{Ext}  Prop. 6.1 that $\anti$ is an
anti-automorphism of the graded algebra $E^*$.
Restricted to $E^0=H$, the map $\anti$ coincides with the anti-involution   $\tau_g\mapsto \tau_{g^{-1}}\textrm{ for any $g\in G$}$ of the algebra $H$.

We may twist the action of $H$ on a left, resp. right, module $Y$ by $\anti$ and thus obtain the right, resp. left module $Y^\anti$, resp. ${}^\anti Y$, with the twisted action of $H$ given by  $(y, h)\mapsto  \anti(h)y$, resp. $(h,y)\mapsto y \anti(h)$. If $Y$ is an $H$-bimodule, then we may define the twisted $H$-bimodule
${}^\anti Y^\anti$ the obvious way and we recall that  $({}^\anti Y^\anti)^\vee\cong {}^\anti(Y^\vee)^\anti$ (\cite{Ext} Rmk. 7.1), where $(-)^\vee = \Hom_k(-,k)$.

\subsubsection{Filtrations\label{subsubsec:fil}}

Let $i\geq 0$. We define on $E^i$ two filtrations:

\begin{itemize}
\item  a decreasing filtration $(F^n E^i)_{n\geq 0}$ where $F^n E^i:= \oplus_{w\in \widetilde W, \,\ell(w)\geq n} H^i (I, \X(w))$;
\item   an increasing filtration $(F_n E^i)_{n\geq 0}$ where $F_n E^i:= \oplus_{w\in \widetilde W, \,\ell(w)\leq n} H^i (I, \X(w))$.
\end{itemize}

When $i=0$, we will often write $F^n H$ (resp. $F_n H$) instead of  $F^n E^0$ (resp. $F_n E^0$). Recall that
$(F^n H)_{n\geq 0}$  is a filtration of $H$ as  an $H$-bimodule.

Moreover, $F_n E^*$ is an algebra filtration, which means that $F_n E^i \cdot F_m E^j \subseteq F_{n+m} E^{i+j}$. This follows from \eqref{f:support} together
with the fact (\cite{Ext} Cor.\ 2.5-ii and Remark 2.10) that
\begin{equation}\label{f:prod-fil}
  IvI \cdot IwI
  \begin{cases}
  = IvwI   &   \text{if $\ell(vw) = \ell(v) + \ell(w)$},  \\
  \subseteq \bigcup_{\ell(v') < \ell(v) + \ell(w)} Iv'I  & \text{if $\ell(vw) < \ell(v) + \ell(w)$}.
  \end{cases}
\end{equation}

\subsubsection{Duality\label{subsubsec:duality}}

Recall that, given a vector space $Y$, we denote by $Y^\vee$ the dual  space $Y^\vee:=\Hom_k(Y, k)$ of $Y$.
For $Y$ a vector space which decomposes into a direct sum $Y = \oplus_{w\in\widetilde W} Y_w$, we denote by $Y^{\vee, f}$ the so-called finite dual of $Y$ which is defined to be the image in $Y^\vee = \prod_{w\in \widetilde W } Y_w^\vee$ of $\oplus_{w\in \widetilde W } Y_w^\vee$.

In this  paragraph, we always \textbf{assume} that the pro-$p$ Iwahori group $I$ is torsion free. This forces the field $\mathfrak{F}$ to be a finite extension of $\mathbb{Q}_p$ with $p \geq 5$. Then $I$, as well as every subgroup $I_w$ for $w\in \widetilde W$, is a Poincar\'e group of dimension $d$ where $d$ is the dimension of $G$ as a $p$-adic Lie group. It implies that $H^d(I, k)$ is one-dimensional. Let $\eta: H^d(I, k)\cong k$ a fixed isomorphism (we will make a specific choice for $\eta$ when $G:={\rm SL}_2(\mathbb Q_p)$, $p\neq 2,3$ in \S\ref{subsubsec:normalize}). Furthermore the $\Ext$-algebra is
supported in degrees $0$ to  $d$.
We refer to \cite{Ext} \S 7.2.
There is a duality between its
$i^{\rm th}$ and $d-i^{\rm th}$ pieces (\cite{Ext} \S  7.2.4) which we recall here. Let $\trace\in \mathbf X^\vee$ be the linear map given by
$\trace:= \sum_{g\in G/I} \ev_g$.
It is easy to check that $\trace:\X\rightarrow k$ is $G$-equivariant  when $k$ is endowed with the trivial action of $G$. We denote by $\trace^i := H^i(I, \trace)$ the maps induced on cohomology.
By \cite{Ext} Prop. 7.18,  the map
\begin{align}\label{f:dual}
\Delta^i : E^i = H^i(I,\mathbf X) &\longrightarrow H^{d-i}(I, \X)^\vee = (E^{d-i})^\vee\nonumber \\
\alpha & \longmapsto  l_\alpha(\beta):=\eta\circ \trace^d (\alpha\cup \beta)   \nonumber
\end{align}
 induces an  injective homomorphism of $H$-bimodules
$
  E^i \longrightarrow ({}^\anti (E^{d-i})^\anti)^\vee\
$
with image $({}^\anti (E^{d-i})^\anti)^{\vee,f}$.
Here we consider (as in \S\ref{subsubsec:anti}) the  twisted $H$-bimodule ${}^\anti (E^{d-i})^\anti$  namely the space $E^{d-i}$
with the action of $H$ on $\beta \in E^{d-i}$ given by
$
  (\tau, \beta, \tau')\mapsto \anti(\tau')\cdot \beta  \cdot \anti(\tau)$  for $\tau, \tau'\in H$.
 The anti-involution $\anti$ was introduced in \S\ref{subsubsec:anti}.
We still denote by $\Delta^i$ the isomorphism of
\begin{equation}\label{f:dual}
\Delta^i : E^i \longrightarrow  ({}^\anti (E^{d-i})^\anti)^{\vee,f}\ .
\end{equation}

Recall that the choice of $\eta$ defines naturally a basis for $E^d$ namely, as in \cite{Ext} \S 8, we single out the unique element $\phi_w\in H^d(I, \X(w))$
 such that  (see also Rmk. 7.4 \emph{loc. cit.})\begin{equation}\label{defiphi}\eta\circ \trace^d(\phi_w)=\eta\circ \cores^{I_w}_I \circ \Sh_w(\phi_w)=1.\end{equation}

\subsubsection{Automorphisms of the pair $(G,\mathbf{X})$}\label{subsubsec:auto-pair}

For $U$ a locally compact and totally disconnected group let $\Mod(U)$ be the abelian category of smooth $U$-representations in $k$-vector spaces. It has enough injective objects.

We consider now a continuous group homomorphism $\xi : U' \rightarrow U$ between two such groups. Any object $M$ in $\Mod(U)$ can be viewed via $\xi$ as an object $\xi^* M$ in $\Mod(U')$. An equivariant map $f : M \rightarrow M'$ between an object $M$ in $\Mod(U)$ and an object $M'$ in $\Mod(U')$ is, by definition, a morphism $f : \xi^* M \rightarrow M'$ in $\Mod(U')$. In other words $f : M \rightarrow M'$ is a $k$-linear map such that $f(\xi(g')m) = g' f(m)$ for any $m \in M$ and $g' \in U'$. We observe the following: Let $M \xrightarrow{\sim} \mathcal{I}_M^\bullet$ and $M' \xrightarrow{\sim} \mathcal{I}_{M'}^\bullet$ be injective resolutions in $\Mod(U)$ and $\Mod(U')$, respectively. Then $\xi^* M \xrightarrow{\sim} \xi^* \mathcal{I}_M^\bullet$ is a resolution in $\Mod(U')$ and $f$ extends to a unique homotopy class of maps of resolutions $\xi^* \mathcal{I}_M^\bullet \xrightarrow{\tilde{f}} \mathcal{I}_{M'}^\bullet$ in $\Mod(U')$. This means that we may derive $f$ to a map between any appropriate cohomological functors on $\Mod(U)$ and $\Mod(U')$.

We will apply this in the following two contexts. First suppose that $U$ and $U'$ are profinite groups. Then $f$ extends to a map on cohomology
\begin{equation*}
  (\xi,f)^* : H^i(U,M) \longrightarrow H^i(U',M') \ .
\end{equation*}
Secondly, let $U$ and $U'$ be general again. For any further object $L$ in $\Mod(U)$ we obtain natural maps
\begin{align*}
  (\xi,f)^* : \Ext^i_{\Mod(U)} (L,M) & \longrightarrow \Ext^i_{\Mod(U')} (\xi^* L,M')   \\
    (\mathcal{I}_L^\bullet \rightarrow \mathcal{I}_M^\bullet[i]) & \longmapsto (\xi^* \mathcal{I}_L^\bullet \rightarrow \xi^* \mathcal{I}_M^\bullet[i] \xrightarrow{\tilde{f}[i]} \mathcal{I}_{M'}^\bullet[i])
\end{align*}
and, in particular,
\begin{equation*}
  \xi^* := (\xi,\id_M)^* : \Ext^i_{\Mod(U)} (L,M)  \longrightarrow \Ext^i_{\Mod(U')} (\xi^* L, \xi^* M) \ .
\end{equation*}
The latter map is evidently compatible with the Yoneda product, since in the derived category it is simply the composition product. Now suppose that $\xi$ and $f$ are isomorphisms. Then we have the ``conjugation'' homomorphism
\begin{align*}
  \Ext^i_{\Mod(U)} (M,M) & \longrightarrow \Ext^i_{\Mod(U')} (M',M')   \\
    (\mathcal{I}_M^\bullet \xrightarrow{\tau} \mathcal{I}_M^\bullet[i]) & \longmapsto (\mathcal{I}_{M'}^\bullet \xrightarrow{\widetilde{f^{-1}}} \xi^* \mathcal{I}_M^\bullet[i]  \xrightarrow{\xi^* \tau} \xi^* \mathcal{I}_M^\bullet[i] \xrightarrow{\tilde{f}[i]} \mathcal{I}_{M'}^\bullet[i]) \ ,
\end{align*}
which again is compatible with the Yoneda product.

We now return to our group $G$ and suppose given an automorphism $\xi : G \xrightarrow{\cong} G$ with the property that $\xi(I) = I$. It induces the $G$-equivariant bijection $\mathcal{X} : \xi^* \mathbf{X} \xrightarrow{\cong} \mathbf{X}$ which sends $gI$ to $\xi^{-1}(g)I$. We therefore obtain the $k$-linear graded bijections
\begin{equation*}
  \Gamma_\xi : E^* \xrightarrow{\cong} E^*   \qquad\text{and}\qquad   \Gamma_{\xi} : H^*(I,\mathbf{X}) \xrightarrow{\cong} H^*(I,\mathbf{X})
\end{equation*}
which correspond to each other under the identification \eqref{f:frob}. The left hand one is an algebra automorphism. Both are involutions provided we have $\xi^2 = \id_G$. In terms of elements of $\mathbf{X}$ as functions we have $\mathcal{X}(f) = f \circ \xi$. This immediately implies that $\Gamma_\xi$ is compatible with the cup product \eqref{f:cup} on $H^*(I,\mathbf{X})$. In the following we list further properties, but for which we assume in addition that $\xi(T) = T$. Then $\xi(N(T)) = N(T)$, so that $\xi$ induces an automorphism $\xi$ of $\widetilde{W}$.
\begin{enumerate}
\item  For all $w\in \widetilde W$, $\Gamma_\xi$ induces a map
\begin{align}\label{gammaxiw}
      H^*(I, \X(w)) \longrightarrow H^*(I, \X(\xi^{-1}(w))) \ .
\end{align}
Since $\xi(I_w) = I_{\xi(w)}$ we correspondingly have a map
\begin{align} \label{resgammaxiw}
     H^*(I_w, \X(w)) \longrightarrow H^*(I_{\xi^{-1}(w)}, \X(\xi^{-1}(w))) \ .
\end{align}

\item Because $\ev_{\xi^{-1}(w)} \circ \mathcal{X}_{|\mathbf{X}(w)} = \ev_w$, the above maps are compatible with the Shapiro isomorphism in the sense that the following diagram
\begin{equation}
  \xymatrix{
H^*(I,\mathbf{X}(w)) \ar@/^2pc/[rrrrr]^{\Sh_w}\ar[d]_{\eqref{gammaxiw}} \ar[rr]^{\; \res^I_{I_w} \;} && H^*(I_w,\mathbf{X}(w))\ar[d] _{\eqref{resgammaxiw}}\ar[rrr]^{\; H^*(I_w,\ev_w) \;}  &&&  H^*(I_w,k)\ar[d] & \\
H^*(I,\mathbf{X}(\xi^{-1}(w)))\ar@/_2pc/[rrrrr]_-{\Sh_{\xi^{-1}(w)}} \ar[rr]^-{\; \res^I_{I_{\xi^{-1}(w)}} \;} && H^*(I_{\xi^{-1}(w)},\mathbf{X}(\xi^{-1}(w)))\ar[rrr]^-{\; H^*(I_{\xi^{-1}(w)},\ev_{\xi^{-1}(w)}) \;}  &&&  H^*(I_{\xi^{-1}(w)},k)&   }
\label{f:xi-shapi}
\end{equation}
commutes. Its horizontal arrows are the Shapiro isomorphisms \eqref{f:Shapiro1} and
the right hand side vertical arrow is induced by the isomorphism $I_{\xi^{-1}(w)} \xrightarrow[\cong]{\xi} I_w$.

\item $\Gamma_\xi$ commutes with $\anti$ defined in \eqref{f:invodefi}; more precisely, each diagram
\begin{equation}
  \xymatrix{
     H^*(I, \mathbf{X}(w))  \ar[d]_{ \eqref{gammaxiw}} \ar[r]^-{\anti} & H^*(I,\mathbf{X}(w^{-1})) \ar[d]^{\eqref{gammaxiw}} \\
  H^*(I, \mathbf{X}(\xi^{-1}(w)))  \ar[r]^-{\anti} & H^*(I,\mathbf{X}(\xi^{-1}(w)^{-1}))   }\label{f:anti+Gamma}
\end{equation}
is commutative.

\item We have noted already the compatibility of $\Gamma_\xi$ with the cup product on $H^*(I,\X)$. It now holds in the more precise form of the commutativity of the diagrams
\begin{equation}
  \xymatrix{
     H^i(I, \mathbf{X}(w)) \otimes_k H^j(I,\mathbf{X}(w)) \ar[d]_{\eqref{gammaxiw}  \otimes \eqref{gammaxiw} } \ar[r]^-{\cup} & H^{i+j}(I,\mathbf{X}(w)) \ar[d]^{\eqref{gammaxiw}} \\
  H^i(I, \mathbf{X}(\xi^{-1}(w))) \otimes_k H^j(I,\mathbf{X}(\xi^{-1}(w))) \ar[r]^-{\cup} & H^{i+j}(I,\mathbf{X}(\xi^{-1}(w))) .  }\label{f:cup+Gamma}
\end{equation}
\end{enumerate}

\subsection{The pro-$p$-Iwahori Hecke algebra of ${\rm SL}_2$}

For \S\ref{rootdatum}--\ref{subsubsec:idempo} we refer to \cite{embed} \S 3.

\subsubsection{\label{rootdatum}Root datum}

To fix ideas we consider
 $I = \left(
\begin{smallmatrix}
1+\mathfrak{M} & \mathfrak{O} \\ \mathfrak{M} & 1+\mathfrak{M}
\end{smallmatrix}
\right)$ (by abuse of notation, here and later in this paragraph, all matrices are understood to have determinant one). We let $T \subseteq G$ be the torus of diagonal matrices, $T^0$ its maximal compact subgroup, $T^1$ its maximal pro-$p$ subgroup, and $N(T)$ the normalizer of $T$ in $G$. We choose the positive root with respect to $T$ to be $\alpha( \left(
\begin{smallmatrix}
t & 0 \\ 0 & t^{-1}
\end{smallmatrix}
\right) ) := t^2$, which corresponds to the Borel subgroup of upper triangular matrices. The affine Weyl group $W$ sits in the short exact sequence
\begin{equation*}
  0 \longrightarrow \Omega = T^0 / T^1 \longrightarrow \widetilde{W} := N(T)/T^1 \longrightarrow W := N(T)/T^0 \longrightarrow 0 \ .
\end{equation*}
Let $s_0 := s_\alpha := \left(
\begin{smallmatrix}
0 & 1 \\ -1 & 0
\end{smallmatrix}
\right)$, $s_1 := \left(
\begin{smallmatrix}
0 & -\pi^{-1} \\ \pi & 0
\end{smallmatrix}
\right)$, and $\theta := \left(
\begin{smallmatrix}
\pi & 0 \\ 0 & \pi^{-1}
\end{smallmatrix}
\right)$, such that $s_0 s_1 = \theta$. The images of $s_0$ and $s_1$ in $W$ are the two reflections corresponding to the two vertices of the standard edge fixed by $I$ in the tree of $G$. They generate $W$, i.e., we have $W = \langle s_0,s_1 \rangle = \theta^{\mathbb{Z}} \dot\cup s_0 \theta^{\mathbb{Z}}$ (by abuse of notation we do not distinguish in the notation between a matrix and its image in $W$ or $\widetilde{W}$). We let $\ell$ denote the length function on $W$ corresponding to these generators as well as its pull-back to $\widetilde{W}$. One has
\begin{equation*}
  \ell(\theta^i) = |2i| \qquad\text{and}\qquad \ell(s_0\theta^i) = |1 - 2i|.
\end{equation*}

\begin{remark}\label{remark:N}
Consider $\mathbf{SL_2}(\mathfrak F)$ as a subgroup of $\mathbf{GL_2}(\mathfrak F)$. Then
the matrix $\varpi:=
\left(\begin{smallmatrix}
       0 & 1 \\
       \pi & 0
\end{smallmatrix}\right)$ normalizes $I$; furthermore, $s_1=\varpi  s_0  \varpi^{-1}.$
\end{remark}

\subsubsection{\label{gene-rel}Generators and relations}
The characteristic functions $\tau_w := \mathrm{char}_{IwI}$ of the double cosets $IwI$ form a $k$-basis of $H$ when $w$ ranges over $\widetilde W$.
Let $e_1 := - \sum_{\omega \in \Omega} \tau_\omega$. The relations in $H$ are
\begin{equation}\label{f:braid}
\tau_v\tau _w=\tau_{vw} \qquad\textrm{whenever}\ \ell(w)+\ell(v)=\ell(wv)\quad\textrm{ and}
\end{equation}
\begin{equation}
\label{f:quad} \tau_{s_i}^2= -e_1\tau_{s_i} \qquad\textrm{for $i=0,1$.}
\end{equation}The elements $\tau_\omega, \tau_{s_i}$, for $\omega \in \Omega$ and $i = 0,1$, generate $H$ as a $k$-algebra.  Note that the $k$-algebra $k[\Omega]$ identifies naturally with a subalgebra of  $H$ via $\omega\mapsto \tau_\omega$.

The trivial character of $H$ (see \eqref{defitriv}) may be defined by
\begin{equation}\label{trivSL2}
    \chi_{triv}: \tau_{s}\longmapsto 0, \: \tau_{\omega} \longmapsto 1, \textrm{ for $s\in\{s_0, s_1\}$  and $\omega \in \Omega$.}
\end{equation}
The sign character $\chi_{sign}$ of $H$, which can be introduced in general as in \cite{Ext} \S 2.2.2, is easy to describe in the current context when $\mathbf G={\rm SL}_2$:
\begin{equation}\label{signSL2}
    \chi_{sign}: \tau_{s}\longmapsto -1, \: \tau_{\omega} \longmapsto 1, \textrm{ for $s\in\{s_0, s_1\}$  and $\omega \in \Omega$.}
\end{equation}

\subsubsection{The involution $\upiota$\label{subsubsec:involution}}

There is an  involutive automorphism $\upiota$  of $H$ satisfying
\begin{equation}\label{f:upiotadef}
\upiota(\tau_s)= -e_1 -\tau_{s}\textrm{ for $s\in\{s_0, s_1\}$ and $\upiota(\tau_\omega)=\tau_\omega$ for $\omega\in \Omega$.}
\end{equation}
(see  \cite{OS1} \S 4.8). It fixes $\zeta$. In particular, it induces an involutive automorphism  of $H_\zeta$.
For $\epsilon=0,1$, the following  sequence  of left $H$-modules is exact:
\begin{equation}\label{f:basic-es}
0 \longrightarrow H \tau_{s_\epsilon}\longrightarrow H\longrightarrow H \upiota(\tau_{s_\epsilon})\longrightarrow 0\end{equation}(see the remark after the proof of \cite{embed} Prop. 3.54).\\\\
For a left (resp. right) $H$-module $M$, we denote by $\upiota M$ (resp. $M\upiota$) the $H$-module on the space $M$ with the action of $H$ twisted by $\upiota$.

\subsubsection{The central element $\zeta$}\label{sec:zeta}

We refer to \cite{embed} \S3.2.2. Consider the element
\begin{equation}\label{f:zetadefi}
   \zeta := (\tau_{s_0} + e_1)(\tau_{s_1} + e_1) + \tau_{s_1}\tau_{s_0}=(\tau_{s_1} + e_1)(\tau_{s_0} + e_1) + \tau_{s_0}\tau_{s_1}.
\end{equation}
Notice that $\anti(\zeta) = \zeta$ and that $\chi_{triv}(\zeta)=\chi_{sign}(\zeta)=1$. The element $\zeta$ is central in $H$, and the subalgebra $k[\zeta]$ of $H$ generated by $\zeta$ is the algebra of polynomials in the variable $\zeta$. Furthermore,  $\zeta$ is not a zero divisor in $H$ and  the $k$-algebra $H/H\zeta$ is finite dimensional  (see for example \cite{embed} Lemma 1.3). We will denote by $H_\zeta$ the algebra obtained by localizing  $H$ in $\zeta$. The anti-involution $\anti$ extends to $H_\zeta$.

 The following result is \cite{embed} Cor. 3.4.

\begin{lemma}\label{freeness}
Let $\epsilon = 0$ or $1$; the morphism of  $(H_{x_\epsilon},k[\zeta])$-bimodules
\begin{equation*}
\begin{array}{ccccc}
H_{x_\epsilon} \otimes_k k[\zeta] & \oplus & H_{x_\epsilon} \otimes_k k[\zeta] & \longrightarrow & H \cr 1 \otimes 1 &&& \longmapsto & 1\cr & & 1\otimes 1 & \longmapsto & \tau_{s_{1-\epsilon}}
\end{array} \end{equation*}is an isomorphism.
In particular,  $H$ is a  free and finitely generated  $k[\zeta]$-module  of rank
$4(q-1)$.
\end{lemma}

\begin{fact} \label{fact:noclassif}
Suppose that  $\mathbb F_q\subseteq k$ and that $p\neq 2$ or $\mathfrak{F} = \mathbb{Q}_p$.  Then for $V$ an  irreducible quotient of $ \X e_1/ \X e_1(\zeta-1)$ we have
$V^I\cong \chi_{triv}$ or
 $V^I\cong  \chi_{sign}$ as a left $H$-module.

\end{fact}
\begin{proof}
A basis of $He_1/ He_1(\zeta-1)$ is given by
the image in the quotient of
$$
\upiota (\tau_{s_0})\tau_ {s_1}e_1, \:\tau_{s_0}\upiota (\tau_{s_1})e_1,  \: \upiota (\tau_{s_0})e_1,  \: \tau_{s_0} e_1
$$
(compare with  Lemma \ref{freeness}).
The elements $\upiota (\tau_{s_0})\tau_ {s_1}e_1$  and $\tau_{s_0}\upiota (\tau_{s_1})e_1$ support respectively the characters $\chi_{triv}$  and $\chi_{sign}$.
This follows from using repeatedly $\upiota (\tau_{s_0})\tau_{s_1}e_1+\upiota (\tau_{s_1})\tau_{s_0}e_1= (-\zeta+1) e_1\equiv 0$ in $He_1/ He_1(\zeta-1)$ and likewise
$\tau_{s_0}\upiota (\tau_{s_1})e_1+\tau_{s_1}\upiota (\tau_{s_0})e_1\equiv 0$ in $He_1/ He_1(\zeta-1)$). Then it is easy to see that in the resulting quotient, $\upiota (\tau_{s_0})e_1$ and $\tau_{s_0} e_1$ support respectively the characters $\chi_{triv}$  and $\chi_{sign}$.
So we have an exact sequence of left $H$-modules
\begin{equation}0\rightarrow \chi_{triv}\oplus\chi_{sign}\rightarrow He_1/ He_1(\zeta-1)\rightarrow  \chi_{triv}\oplus\chi_{sign}\rightarrow 0\ .\label{f:exacttrivsign}\end{equation}

All the modules in question are annihilated by $\zeta-1$ so they are $H_\zeta$-modules. Suppose furthermore that $\mathbb F_q\subseteq k$ and that $p\neq 2$. We may apply \cite{embed} Thm. 3.33  which ensures that the functor $\X\otimes_H{}_-$ is exact on \eqref{f:exacttrivsign}, provides an exact sequence of $G$ representations
\begin{equation*}
0 \rightarrow \X\otimes_H \chi_{triv} \oplus \X\otimes_H \chi_{sign} \rightarrow \X e_1/ \X e_1(\zeta-1) \rightarrow  \X\otimes_H \chi_{triv}\oplus\X\otimes_H \chi_{sign} \rightarrow 0 \end{equation*}
and that for $\chi\in\{ \chi_{sign},  \chi_{triv}\}$ we have
$(\X\otimes_H \chi)^I\cong \chi$  and therefore $\X\otimes_H \chi$ is an irreducible representation of $G$. Therefore any irreducible quotient of $\X e_1/ \X e_1(\zeta-1)$ is isomorphic to  $\X\otimes_H \chi_{triv}$ or $\X\otimes_H \chi_{sign}$.
\end{proof}

\begin{remark}
After localizing  \eqref{f:basic-es} in $\zeta$ we  get an exact sequence of left $H_\zeta$-modules
\begin{equation}\label{f:basic-esloc}
0\longrightarrow H_\zeta \tau_{s_\epsilon}\longrightarrow H_\zeta\longrightarrow H _\zeta \upiota(\tau_{s_\epsilon})\longrightarrow 0\ .\end{equation} Notice that the map
$h\mapsto \zeta^{-1}h\tau_{s_{1-\epsilon}}\tau_{s_\epsilon}$ splits the inclusion $H_\zeta \tau_{s_\epsilon}\longrightarrow H_\zeta$ because  $\zeta \tau_{s_\epsilon}=\tau_{s_\epsilon}\tau_{s_{1-\epsilon}}\tau_{s_\epsilon}$ (compare with the proof of \cite{embed}  Lemma 3.30).
So we have $H_\zeta\cong H_\zeta \tau_{s_\epsilon}\oplus
H _\zeta \upiota(\tau_{s_\epsilon})$  as left $H_\zeta$-modules.

\end{remark}
\begin{remark} The element $\zeta$ depends on the choice of the uniformizer $\pi$. Let $\alpha\in\mathfrak O^\times$. We verify  that if we pick $\alpha\pi$ as a  uniformizer, the  new corresponding central element $\zeta_\alpha$ is
\begin{equation}\label{f:zetaalpha}
   \zeta_\alpha :=\tau_{\omega_{\alpha^{-1}}} (\tau_{s_0} + e_1)(\tau_{s_1} + e_1) +\tau_{\omega_{\alpha}}  \tau_{s_1}\tau_{s_0}=\tau_{\omega_{\alpha}}  (\tau_{s_1} + e_1)(\tau_{s_0} + e_1) + \tau_{\omega_{\alpha^{-1}}}\tau_{s_0}\tau_{s_1}.
\end{equation} where
$\omega_{\alpha}$  is the element $\left(
\begin{smallmatrix}
  \alpha^{-1} & 0 \\
  0 & \alpha
\end{smallmatrix}
\right) T^1\in\Omega$.
Of course we have $\zeta=\zeta_1$.
A system of generators of the center $Z$ of $H$ as a $k$-vector space  is given by the set of all $\zeta_\alpha$ for $\alpha$ ranging over a system if representatives of $(\mathfrak O/\mathfrak M )^\times$ (to which one has to add $\tau_1$ if $p=2$) (see \cite{embed} (24) in Remark 3.5).

 We have the formula:
$
\zeta_\alpha\,\zeta_\beta=\zeta_{\alpha\beta}\,\zeta
$ for any $\alpha,\beta\in\mathfrak O^\times$.  In particular
\begin{equation}\label{f:zetaa}
\zeta_\alpha\,\zeta_{\alpha^{-1}}=\zeta^2\text{ and }\quad \zeta_{\alpha^2}\,\zeta=\zeta_{\alpha}^2 \ .
\end{equation} These identities ensure that the localized algebra $H_\zeta$ does not depend on the choice of the uniformizer.
\end{remark}

\subsubsection{Supersingularity\label{subsubsec:supersing}}

In the current context where $\mathbf G={\rm SL}_2$, the  ideal $\mathfrak J$ introduced in \S\ref{subsec:basicshecke} is the central ideal $\zeta k[\zeta]$.  Following the definition introduced in that paragraph, an $H$-module $M$ is called supersingular  if any element in $M$ is annihilated by a power of $\zeta$.

\begin{remark} \label{rema:indeppi}  Let $\alpha\in\mathfrak O^\times$. From \eqref{f:zetaa}, one easily deduces that an element in $M$ is annihilated by a power of $\zeta$ if and only if it is annihilated by a power of $\zeta_\alpha$. Therefore,  even if $\zeta$ does depend on the choice of a uniformizer, the notion of supersingularity does not.

\end{remark}

\subsubsection{Idempotents}\label{subsubsec:idempo}

The element  $e_1$ is a central idempotent in $H$. More generally, to any $k$-character $\lambda : \Omega \rightarrow k^\times$ of $\Omega$, we  associate the following idempotent in $H$:
\begin{equation}\label{defel}
\ e_\lambda := - \sum_{\omega \in \Omega} \lambda(\omega^{-1}) \tau_\omega \ .
\end{equation}
Note that $\anti(e_\lambda) = e_{\lambda^{-1}}$ and $e_\lambda\tau_\omega = \tau_\omega e_\lambda=\lambda(\omega)e_\lambda$ for any $\omega\in \Omega$.
We parameterize $\Omega$ by the isomorphism
\begin{align}
  (\mathfrak{O}/\mathfrak{M})^\times & \xrightarrow{\;\cong\;} \Omega \cr
  u & \longmapsto \omega_u := \left(
\begin{smallmatrix}
  [u]^{-1} & 0 \\
  0 & [u]
\end{smallmatrix}
\right) T^1 \ ,\label{omegaz}\end{align}
where $[u]$ is a lift   in $\mathfrak O$ for $u\in  (\mathfrak{O}/\mathfrak{M})^\times$, and we pick the multiplicative Teichm\"uller lift.

\begin{remark}\label{tll}
Given  an homomorphism of groups $\Lambda:  (\mathfrak{O}/\mathfrak{M})^\times\rightarrow k^\times$, we may consider the character $\lambda:{\Omega}\rightarrow k^\times$ obtained via composition with the inverse if \eqref{omegaz} and the corresponding
idempotent as in \eqref{defel}. We will then use the shortcut $e_\Lambda$ to denote the latter. This will be used in the following context:

If $q = p$ we have the homomorphism $\id : (\mathfrak O/\mathfrak M)^\times = \mathbb{F}_p^\times \xrightarrow{\subseteq} k^\times$, which will play an important role later on. For $m\in \mathbb Z$, we will consider the idempotent element
\begin{equation}\label{f:idm}
e_{\id^m}\in k[\Omega]
\end{equation}
with the above convention. When $m=0$ this is  consistent with the notation $e_1$ in \S\ref{gene-rel}.
\end{remark}

Suppose  for a moment that  $\mathbb{F}_q \subseteq k$.
Then all simple modules of $k[\Omega]$ are one dimensional.  The set $\widehat{\Omega}$ of all $k$-characters of $\Omega$ has cardinality $q-1$ which is prime to $p$. This implies that the family  $\{e_\lambda\}_\lambda\in \widehat\Omega$ is a family of orthogonal idempotents with sum equal to $1$. It  gives the  ring decomposition $  k[\Omega] = \prod_{\lambda \in \widehat{\Omega}} k e_\lambda .$
Let $\Gamma := \{ \{\lambda,\lambda^{-1}\} : \lambda \in \widehat{\Omega} \}$ denote the set of $s_0$-orbits in $\widehat{\Omega}$. To
$\gamma\in \Gamma$ we attach the element $e_\gamma:=e_\lambda+e_{\lambda^{-1}}$ (resp.  $e_\gamma:=e_\lambda$) if $\gamma=\{\lambda,\lambda^{-1}\}$ with $\lambda\neq\lambda^{-1}$ (resp. $\gamma=\{\lambda\}$). Using the braid relations, one sees that $e_\gamma$ is a central idempotent in $H$ and we have the ring  decomposition
$H = \prod_{\gamma \in \Gamma} H e_\gamma$. If $q=p$ then the idempotent
\begin{equation}\label{f:gamma0}
e_{\gamma_0} := e_{\id}+e_{\id^{-1}}
\end{equation}
will be of particular importance (see \eqref{f:idm}).

\subsubsection{Certain $H$-modules\label{subsubsec:certain}}

For later purposes we construct in this section certain families of $H$-modules. The reader may skip this at first reading coming back to it only when needed.
We fix a homomorphism of $k$-algebras $\kappa : H \rightarrow R$ as well as an element $z \in Z(R)$ in the center of $R$. Let $M_2(R)$ denote, as usual, the algebra of $2$ by $2$ matrices over $R$. We also fix a character $\mu : \Omega \rightarrow k^\times$. With these choices we define the matrices
\begin{equation*}
  M_0 :=
  \left(
  \begin{smallmatrix}
  - \kappa(e_\mu) & 0 \\
  z \kappa(\tau_{s_1}) & 0
  \end{smallmatrix}\right) \ , \
  M_1 :=
  \left(
  \begin{smallmatrix}
  0 & z \kappa(\tau_{s_0}) \\
  0 & - \kappa(e_{\mu^{-1}})
  \end{smallmatrix}\right) \ , \ \text{and} \
  M_\omega :=
  \left(
  \begin{smallmatrix}
  \mu^{-1}(\omega) \kappa(\tau_\omega) & 0 \\
  0 & \mu(\omega) \kappa(\tau_\omega)
  \end{smallmatrix}\right) \
  \text{for $\omega \in \Omega$}.
\end{equation*}
It is straightforward to check that these matrices satisfy the relations
\begin{equation*}
  M_i^2 = \sum_{\omega \in \Omega} M_\omega M_i \ , \ M_\omega M_i = M_i M_{\omega^{-1}} \ , \ \text{and} \ M_\omega M_{\omega'} = M_{\omega \omega'} \ .
\end{equation*}
Hence we obtain a $k$-algebra homomorphism $\kappa_2 : H \rightarrow M_2(R)$ by sending $\tau_{s_i}$ to $M_i$ and $\tau_\omega$ to $M_\omega$. By using this homomorphism to equip the left $R$-module $R \oplus R$ with a right $H$-module structure we obtain an $(R,H)$-bimodule denoted by $(R \oplus R)[\kappa, z, \mu]$.

\subsubsection{Frobenius extensions}

The space $\Hom_{k[\zeta]}(H, k[\zeta])$ is naturally an $H$-bimodule via $(h,\Lambda, h')\mapsto \Lambda(h'_{-} h)$.
\begin{proposition}
We have an isomorphism of $H$-bimodules
$$
\upiota H\cong H \upiota\cong  \Hom_{k[\zeta]}(H, k[\zeta])\ .
$$
\end{proposition}
\begin{proof}The first isomorphism is given by the map $\upiota:H\rightarrow H$.
From Lemma \ref{freeness} we know that $H$ is a free $k[\zeta]$-module with basis the set of all $\tau_w$ for $w$ ranging over the set
\begin{equation}\label{f:Cbasis}
\omega, \,\omega s_0, \,  \omega s_1, \,  \omega s_0s_1 \textrm{ when $\omega\in \Omega$}.
\end{equation}
We define  in $ \Hom_{k[\zeta]}(H, k[\zeta])$ the  dual basis  namely for
each  $x\in \eqref{f:Cbasis}$, we define  the map  $\Lambda_x \in   \Hom_{k[\zeta]}(H, k[\zeta])$  which sends
each $\tau_y$ with $y\in \eqref{f:Cbasis}$ to $0$  except $\Lambda_x(\tau_x)=1\in k[\zeta]$.
We check that
\begin{equation}\label{f:checkLambda}
\Lambda_{s_0s_1}(\tau\tau')=\Lambda_{s_0s_1}(\upiota(\tau')\tau)
\end{equation}
which ensures that
\begin{align}\label{f:frobi}
f : \upiota H  & \longrightarrow   \Hom_{k[\zeta]}(H, k[\zeta])  \\
        \tau & \longmapsto f(\tau)(\tau') := \Lambda_{s_0s_1} (\tau \tau')
\end{align}
defines a homomorphism of $H$-bimodules.


Let $w,w'\in \widetilde W$ and $\tau:=\tau_{w}$, $\tau':=\tau_{w'}$. Since $\Lambda_{s_0s_1}$ is $k[\zeta]$-linear it is enough to verify \eqref{f:checkLambda} when $w,w'\in \eqref{f:Cbasis}$. And in fact it is easy to see that both sides of \eqref{f:checkLambda}  are then zero except possibly in the following cases. Let $\omega, \omega'\in \Omega$. The verifications below rely on the quadratic formulas \eqref{f:quad} and the expression
$\zeta = (\tau_{s_0} + e_1)(\tau_{s_1} + e_1) + \tau_{s_1}\tau_{s_0}=(\tau_{s_1} + e_1)(\tau_{s_0} + e_1) + \tau_{s_0}\tau_{s_1}.$ We spell out a few of them.
\begin{itemize}
\item If  $w=\omega s_0$ and $w'=\omega' s_1$, we have  $\Lambda_{s_0s_1}(\tau\tau')=\Lambda_{s_0s_1}(\tau_{\omega\omega'^{-1}}\tau_{s_0s_1})$ which is equal to $1$ if $\omega=\omega'$ and to $0$ otherwise.\\
We have
$\Lambda_{s_0s_1}(\upiota(\tau')\tau)=-\Lambda_{s_0s_1}(\tau_{\omega'\omega^{-1}}(\tau_{s_1}+e_1)\tau_{s_0})=
-\Lambda_{s_0s_1}(\tau_{\omega'\omega^{-1}}(\zeta-(\tau_{s_1}e_1+e_1)-\tau_{s_0s_1}))=\Lambda_{s_0s_1}(\tau_{\omega'\omega^{-1}s_0s_1})$ which is also equal to $1$ if $\omega=\omega'$ and to $0$ otherwise.

\item If  $w=\omega s_1$ and $w'=\omega' s_0$, we easily check that  both  $\Lambda_{s_0s_1}(\tau\tau')$ and $\Lambda_{s_0s_1}(\upiota(\tau')\tau)$
are equal to $-1$ if $\omega=\omega'$ and to $0$ otherwise.

 \item If  $w=\omega s_0$ and $w'=\omega' s_0s_1$, we have  $\Lambda_{s_0s_1}(\tau\tau')=-\Lambda_{s_0s_1}(\tau_{\omega\omega'^{-1}}e_1\tau_{s_0s_1})=-\Lambda_{s_0s_1}(e_1\tau_{s_0s_1})$ which is equal to $1$.

 We have
$\Lambda_{s_0s_1}(\upiota(\tau')\tau)=\Lambda_{s_0s_1}(\tau_{\omega'\omega}(\tau_{s_0}+e_1)(\tau_{s_1}+e_1)\tau_{s_0})=\Lambda_{s_0s_1}(\tau_{\omega'\omega}(\tau_{s_0}+e_1)(\tau_{s_1}+e_1)\tau_{s_0})=\Lambda_{s_0s_1}(\tau_{\omega'\omega}(\zeta-\tau_{s_1s_0})\tau_{s_0})=\Lambda_{s_0s_1}(\tau_{\omega'\omega}e_1\tau_{s_1s_0})=\Lambda_{s_0s_1}(e_1\tau_{s_1s_0})=-\Lambda_{s_0s_1}(\sum_{u\in \Omega}\tau_u\tau_{s_1s_0})$ which is equal to $1$ (see the previous case).

 \item If  $w=\omega s_0s_1$ and $w'=\omega' s_1$, we check that  $\Lambda_{s_0s_1}(\tau\tau')=\Lambda_{s_0s_1}(\upiota(\tau')\tau)=1$.

%

\item If
$w=\omega s_1$ and $w'=\omega' s_0s_1$, we have  $\Lambda_{s_0s_1}(\tau\tau')=\Lambda_{s_0s_1}(\tau_{\omega\omega'^{-1}}\tau_{s_1s_0s_1})=\Lambda_{s_0s_1}(\tau_{\omega\omega'^{-1}}\zeta\tau_{s_1})=0$.
We have
$\Lambda_{s_0s_1}(\upiota(\tau')\tau)=\Lambda_{s_0s_1}(\tau_{\omega'\omega}(\tau_{s_0}+e_1)(\tau_{s_1}+e_1)\tau_{s_1})=0$.
\item If  $w=\omega s_0s_1$ and $w'=\omega' s_0$, we have likewise   $\Lambda_{s_0s_1}(\tau\tau')=\Lambda_{s_0s_1}(\upiota(\tau')\tau)=0$.
\item If
$w=\omega s_0s_1$ and $w'=\omega' s_0s_1$, we have  $\Lambda_{s_0s_1}(\tau\tau')=
\Lambda_{s_0s_1}(\tau_{\omega\omega'}\tau_{s_0s_1s_0s_1})=
\Lambda_{s_0s_1}(\tau_{\omega\omega'}\zeta \tau_{s_0s_1})$  which is equal to $\zeta$
if $\omega'=\omega^{-1}$ and to $0$ otherwise. We have\\
$\Lambda_{s_0s_1}(\upiota(\tau')\tau)=\Lambda_{s_0s_1}(\tau_{\omega'\omega}(\tau_{s_0}+e_1)(\tau_{s_1}+e_1)\tau_{s_0s_1})=
\Lambda_{s_0s_1}(\tau_{\omega'\omega}(\zeta-\tau_{s_1s_0})\tau_{s_0s_1})=
\Lambda_{s_0s_1}(\tau_{\omega'\omega}(\zeta\tau_{s_0s_1}-\zeta e_1 \tau_{s_1})$ which is also equal to $\zeta$
if $\omega'=\omega^{-1}$ and to $0$ otherwise.
\end{itemize}
To prove that \eqref{f:frobi} is surjective, we verify the following.
We have
\begin{itemize}\item[a)]
$-\tau_{s_0}\cdot \Lambda_{s_0s_1}= \Lambda_{{s_1}}$.
\item[b)] $(\tau_{s_1}+e_1)\cdot \Lambda_{s_0s_1}= \Lambda_{s_0}$.
\item[c)] $-(\tau_{s_1}+e_1)\tau_{s_0}\cdot \Lambda_{s_0s_1}= \Lambda_{1}$.
\item[d)] for all $w\in \eqref{f:Cbasis}$ and $\omega\in \Omega$,  we have $\Lambda_{w}\cdot \tau_{\omega^{-1}}= \Lambda_{\omega w}$.
\end{itemize}
 Property d) is immediate.  The other properties are easily verified by evaluating explicitly the left hand side at all elements of the form $\tau_w$ for $w\in \eqref{f:Cbasis}$. For example $-(\tau_{s_1}+e_1)\tau_{s_0}\cdot \Lambda_{s_0s_1}(\tau_\omega)=-\Lambda_{s_0s_1}(\tau_\omega(\tau_{s_1}+e_1)\tau_{s_0})
$ which we already computed above is equal to $1$ if $\omega=1$ and to $0$ otherwise.

Once it is proved that  \eqref{f:frobi} is surjective, the injectivity is immediate since both spaces are free $k[\zeta]$-modules of the same rank.
\end{proof}

Using a free resolution of any arbitrary left (resp. right) $k[\zeta]$-module, and since $H$ is finitely generated free hence projective over $k[\zeta]$, it follows immediately:

\begin{corollary}
Let $M$ be a left, resp. right, $k[\zeta]$-module. We have an isomorphism of left, resp. right, $H$-modules
$$
H \otimes_{k[\zeta] } M \cong  \upiota H \otimes_{k[\zeta] } M \cong \Hom_{k[\zeta]}(H, M) \quad \textrm{ resp. }
 \quad M \otimes_{k[\zeta] }  H \cong M \otimes_{k[\zeta] }  H \upiota \cong \Hom_{k[\zeta]}(H, M) \ .
$$
\end{corollary}
\begin{proof}
For the left hand isomorphisms note that $\upiota H$ (resp. $H\upiota$) is naturally isomorphic to $H$ as an $(H, k[\zeta])$-bimodule  (resp as a $(k[\zeta], H)$-bimodule) since $\upiota$ fixes $\zeta$.
\end{proof}

\begin{corollary}
For $a\in k$, the finite dimensional $k$-algebra $H/(\zeta-a) H$ is Frobenius.  \end{corollary}
\begin{proof}  The isomorphism of $H$-bimodules \eqref{f:frobi} clearly factors through an isomomorphism  of $H/(\zeta-a) H$-bimodules \begin{equation}\upiota (H/(\zeta-a)H)\cong \Hom_{k[\zeta]}(H/(\zeta-a) H, k[\zeta]/(\zeta-a))\cong  \Hom_{k}(H/(\zeta-a) H, k)
\ .\label{f:frobeniuszeta}\end{equation}

\end{proof}

\subsubsection{Finite duals\label{subsubsec:fduals}}

We consider the finite dual $H^{\vee, f}$ of $H$ (see \S\ref{subsubsec:duality}) with basis $(\tau^\vee_w)_{w\in\widetilde W }$  defined to be the dual of $(\tau_w)_{w\in\widetilde W }$. When $I$ is a Poincar\'e group of dimension $d$, we have an isomorphism between $E^d$ and the twisted $H$-bimodule ${}^\anti(H^{\vee, f})^\anti$ given by  \eqref{f:dual}. In  \S\ref{subsubsec:duality}, just like in\cite{Ext} \S 8, we denoted by $\phi_w$ the element of $E^d$ corresponding to $\tau_w^\vee$  and we computed in Prop. 8.2 \emph{loc. cit} that the structure of $H$-bimodule of ${}^\anti(H^{\vee,f})^\anti$ is given by the following formulas. Let  $w\in {\widetilde W}$, $\omega\in \widetilde \Omega$ and $s\in\{s_0,s_1\}$. \begin{equation}\label{f:Hdomega}
 \tau^\vee_w\cdot \tau_\omega= \tau^\vee_{w\omega} \ ,\ \tau_\omega \cdot \tau^\vee_w= \tau^\vee_{\omega w},
\end{equation}
\begin{equation}\label{f:leftrightHd}
 \tau^\vee_w\cdot \tau_{s}= \begin{cases}  \tau^\vee_{w s}-\tau^\vee_w\cdot e_1 & \text{ if $\ell(w\tilde s)=\ell(w)-1$,}\cr 0& \text{ if $\ell(w s)=\ell(w)+1$,}\end{cases}
\quad \quad \tau_{s}\cdot \tau^\vee_w= \begin{cases}\tau^\vee_{s w}- e_1\cdot \tau^\vee_w& \text{ if $\ell( sw)=\ell(w)-1$,}\cr 0& \text{ if $\ell(\tilde sw)=\ell(w)+1$.}\end{cases}
\end{equation}

\begin{remark}\label{rema:psiw}
For all $w\in \widetilde W$ with length $\geq 1$, there is a unique $\epsilon\in\{0, 1\}$ such that $\ell(s_\epsilon w)=\ell(w)-1$. We let $\psi_w:= \tau_{s_\epsilon}\cdot \phi_w=\phi_{s_\epsilon w}-e_1\cdot \phi_w$.
From the formulas above we get  $\zeta \cdot \psi_w=\psi_{s_{1-\epsilon} s_\epsilon w}$ if $\ell(w)\geq 3$ and $\zeta\cdot \psi_w=0$ if $\ell(w)=1,2$.
So the subspace generated by the $\psi_w$ is of $\zeta$-torsion and contained in $\ker(\trace^d)$. We show that this subspace is in fact equal to $\ker(\trace^d)$. First of all we recall from the proof of \cite{Ext} Prop.\ 8.6 that $E^d = \ker(\trace^d) \oplus k e_1\cdot \phi_1$. Then we notice that $\Psi$ is stable under the left action of $\tau_\omega$ for $\omega\in \Omega$. So $\Psi = e_1 \cdot \Psi \oplus (1-e_1) \cdot \Psi$, and it is enough to show that
$(1-e_1)\cdot \Psi= (1-e_1)\cdot E^d$ and  $e_1\cdot \Psi \oplus k e_1\cdot \phi_1 = e_1\cdot E^d$.
The first identity is true because, for $w\in \widetilde W$, there exists $\eta\in\{0,1\}$ such that $\ell(s_\eta w)
=\ell(w)+1$ and  $(1-e_1)\cdot \phi_w= (1-e_1)\cdot \psi_{s_\eta^{-1} w}$. To prove the second identity, we  let $w\in \widetilde W$. If $\ell(w)=0$, then
$e_1\cdot \phi_w= e_1\cdot \phi_1$.  If $\ell(w)>1$, let $\epsilon\in\{0, 1\}$ such that $\ell(s_\epsilon w)=\ell(w)-1$. Then
$e_1\cdot \phi_w= e_1\cdot \phi_{s_\epsilon w}-e_1\cdot \psi_w$ lies in $e_1\cdot \Psi \oplus k e_1\cdot \phi_1$ by induction on $\ell(w)$.
\end{remark}

Let  $m\geq 1$.  The restriction map $H^{\vee,f}\rightarrow (F^m H)^{\vee, f}$ is  an homomorphism of $H$-bimodules and makes  the finite dual
  $(F^m H)^{\vee, f}$ of $F^m H$  a quotient of the $H$-bimodule
$H^{\vee, f}$. Furthermore, $(F^mH/F^{m+1}H)^\vee$ identifies with the sub-$H$-bimodule of $(F^m H)^{\vee, f}$ of the linear forms which are trivial on $F^{m+1}H$.
We consider the linear map defined by
\begin{align}\label{f:mm+1}
   F^mH/F^{m+1}H & \longrightarrow  {}^{\anti}((F^mH/F^{m+1}H)^\vee)^\anti \cr
    \tau_w  & \longmapsto  \tau_w^\vee\vert_{F^mH} \text{  for $w\in \widetilde W$ such that $\ell(w)=m$}
\end{align}
 By the above formulas, it is an isomorphism of $H$-bimodules.

\subsubsection{The equivalence of categories\label{subsubsec:eqcat}} When $G={\rm SL}_2(\mathbb Q_p)$, the functors $H^0(I, {}_-)$ and $\X\otimes_H{}_-$ are quasi-inverse equivalences  between the category $\Mod^I(G)$ of all smooth representations generated by their $I$-fixed vectors and the category of left $H$-modules. In particular,  $H^0(I, {}_-)$ is exact in $\Mod^I(G)$. (See \cite{embed} Prop. 3.25).

\section{\label{sec:top}Top cohomology functor $H^d(I, {}_-)$ when $\mathbf G$ is almost simple simply connected}

\subsection{${\rm Ker} (\trace^d)$ is an injective hull of $(H/\mathfrak{J}H)^\vee$}

Without extra conditions on $\mathbf G$ or on $\mathfrak F$, we have the following.
The ideal $\mathfrak{J}$ (\S\ref{subsec:basicshecke}) generates a two-sided ideal $\mathfrak{J}H$ in $H$.
Recall that we denote by $V^\vee$ the $k$-linear dual of a $k$-vector space $V$. We consider the obvious inclusion of $H$-bimodules
\begin{equation*}
  (H/\mathfrak{J}H)^\vee \longrightarrow I((H/\mathfrak{J}H)^\vee) := \bigcup_m (H/\mathfrak{J}^m H)^\vee \ .
\end{equation*}

\begin{lemma}\label{lemma:extcont3.1}
\begin{itemize}
  \item $I((H/\mathfrak{J}H)^\vee)$ is an injective  $H$-module on the left and on the right.
  \item If furthermore $\mathbf G$ is semisimple, then  $I((H/\mathfrak{J}H)^\vee)$ is an injective
   hull of $(H/\mathfrak{J}H)^\vee$ as a left as well as a right $H$-module.
\end{itemize}
\end{lemma}
\begin{proof}
The following argument arose from a discussion with K.\ Ardakov. The other case being entirely analogous we prove the statement as left $H$-modules.

\textit{Step1:} We show that the left $H$-module $E((H/\mathfrak{J}H)^\vee)$ is injective. By Baer's criterion it suffices to consider test diagrams of the form
\begin{equation*}
  \xymatrix{
    L \ar[d]_{\alpha} \ar[r]^{\subseteq} & H  \\
    I((H/\mathfrak{J}H)^\vee)  &    }
\end{equation*}
where $L \subseteq H$ is a left ideal. The ring $H$ being noetherian the left ideal $L$ is finitely generated. Hence the image of $\alpha$ is contained in $(H/\mathfrak{J}^a H)^\vee$ for any sufficiently large $a$. The homomorphism then must factorize through a homomorphism $\bar{\alpha} : L/\mathfrak{J}^a L \rightarrow (H/\mathfrak{J}^a H)^\vee$. Furthermore, since the ideal $\mathfrak{J}H$ in the noetherian ring $H$ is centrally generated it has the Artin-Rees property (cf.\ \cite{MCR} Prop.\ 4.2.6). This implies that we find an integer $b \geq a$ such that $\mathfrak{J}^bH \cap L \subseteq \mathfrak{J}^a L$. This reduces us to finding the broken arrow in the diagram
\begin{equation*}
  \xymatrix{
    L/\mathfrak{J}^b H \cap L \ar[d]_{\pr} \ar[r]^{} & H/\mathfrak{J}^b H \ar@{-->}[dddl]^{} \\
    L/\mathfrak{J}^a L \ar[d]_{\bar{\alpha}}  &  \\
    (H/\mathfrak{J}^aH)^\vee \ar[d]_{\subseteq}  &  \\
    (H/\mathfrak{J}^bH)^\vee .  &    }
\end{equation*}
We note that the horizontal arrow is injective and that this is a diagram of $H/\mathfrak{J}^b H$-modules. So it suffices to show that that the $H/\mathfrak{J}^b H$-module $(H/\mathfrak{J}^bH)^\vee$ is injective. But the computation
\begin{equation*}
  \Hom_{H/\mathfrak{J}^bH}(M,(H/\mathfrak{J}^bH)^\vee) = \Hom_k(H/\mathfrak{J}^bH \otimes_{H/\mathfrak{J}^bH} M,k) = \Hom_k(M,k)
\end{equation*}
shows that these functors are exact in the $H/\mathfrak{J}^bH$-module $M$.

\textit{Step 2:} Assume that the group $\mathbf{G}$ is semisimple. Then $H/\mathfrak{J}^m H$ is finite dimensional over $k$ for any $m \geq 1$. We show that the inclusion $(H/\mathfrak{J}H)^\vee \subseteq I((H/\mathfrak{J}H)^\vee)$ is essential, i.e., that any nonzero $H$-submodule $Y$ of $E((H/\mathfrak{J}H)^\vee)$ has nonzero intersection with $(H/\mathfrak{J}H)^\vee$. It, of course, suffices to consider the case when $Y$ is a cyclic module. We then have
$Y \subseteq (H/\mathfrak{J}^mH)^\vee$ for some large $m$. Let $Y^\perp \subseteq H/\mathfrak{J}^mH$ be the orthogonal complement of $Y$. Suppose that $Y \cap (H/\mathfrak{J}H)^\vee = 0$. This means that $Y^\perp + \mathfrak{J}H/\mathfrak{J}^m H = H/\mathfrak{J}^mH$
But $\mathfrak{J}H/\mathfrak{J}^m H$ is contained in the Jacobson radical of $H/\mathfrak{J}^mH$. Hence the Nakayama lemma implies that $Y^\perp = H/\mathfrak{J}^mH$, which gives rise to the contradiction that $Y = 0$.
\end{proof}

\begin{remark}\label{rema:omittwists}
The anti-involution $\anti: H\rightarrow H$ yields an isomorphism of $H$-bimodules $H\cong {}^{\!\anti} H^\anti$. By  \cite{Ext} Remark 6.3, it preserves the central ideal $\mathfrak J$, as well as  the central ideal $\mathfrak J^m$ for any $m\geq 1$. Therefore, we have an isomorphism of $H$-bimodules $H/\mathfrak J^m H\cong {}^{\!\anti} (H/\mathfrak J^m H)^\anti $.
By  \cite{Ext} Remark 7.1, we also have $(H/\mathfrak J^m H)^\vee \cong{}^{\!\anti} ( (H/\mathfrak J^m H))^\vee)^\anti $.
\end{remark}

Until the end of  this  paragraph, we assume as in  \S\ref{subsubsec:duality} that the pro-$p$ Iwahori group $I$ is torsion free. Therefore it is a Poincar\'e group of dimension $d$. The map $\trace^d: H^d(I, \X)\rightarrow k$ was introduced in \S\ref{subsubsec:duality}.
Assume also that   $\mathbf{G}$ is almost simple and simply connected. Then
in  \cite{Ext} \S 8, we studied $E^d$ using the isomorphism \begin{equation}\label{f:88}E^d\xrightarrow{\cong}\: ({}^\anti E^0\,^\anti)^{\vee, f}\end{equation}
recalled in \eqref{f:dual}.
(Notice that some of the results there  are true under weaker hypotheses than the ones of the current context). By Prop. 8.6 \emph{loc. cit.}, we have an isomorphism  of $H$-bimodules
\begin{equation}\label{f:keydecompEd}
   E^d\cong\ker(\EuScript S^d)\oplus \chi_{triv}.
\end{equation}

\begin{proposition} \label{prop:kerS-injhull}
Suppose that  $\mathbf{G}$ is almost simple and simply connected. Then we have an isomorphism of $H$-bimodules
$$
\ker(\trace^d)\cong
  \bigcup_m (H/\mathfrak{J}^m H)^\vee.
$$
In particular, $\ker(\trace^d)$ is an injective hull of the left (resp. right) $H$-module $(H/\mathfrak{J}H)^\vee$ and is supersingular as a left (resp.\ right) $H$-module.
\end{proposition}
\begin{proof}
In fact, via  \eqref{f:88}, we have the isomorphism
     $ \ker(\trace^d)\cong (\,{}^{\!\anti} \ker(\chi_{triv})^\anti)^{\vee,f}$  where
    $ (\ker(\chi_{triv}))^{\vee,f}$ is the image of
    $ ( E^0)^{\vee,f}$ in the natural restriction map $(E^0)^\vee\rightarrow  (\ker(\chi_{triv}))^{\vee}$.
This gives the alternate description
of $  \ker(\trace^d) $ as an $H$-bimodule:
\begin{equation}\label{f:injhull1}
  {}^{\!\anti}( \ker(\trace^d))^{\anti} \cong \bigcup_m (\ker(\chi_{triv})/F^m H \cap \ker(\chi_{triv}))^\vee.
\end{equation} Recall indeed that $\mathbf{G}$ being semisimple, $H/ F^m H$ is a finite dimensional vector space.
On the other hand, the character $\chi_{triv}$ is not supersingular (\cite{Ext} Remark 2.12.iv and Lemma 2.13)  and therefore we have $\mathfrak{J}^m H + \ker(\chi_{triv}) = H$ for any $m \geq 1$. Hence
\begin{equation}\label{f:injhull2}
  \bigcup_m (H/\mathfrak{J}^m H)^\vee = \bigcup_m (\ker(\chi_{triv})/\mathfrak{J}^m H \cap \ker(\chi_{triv}))^\vee \ .
\end{equation}
But, since $\mathbf G$ is almost simple simply connected, \cite{Ext} Lemma 2.14 says that
\begin{equation*}
  \mathfrak{J}^m H \cap \ker(\chi_{triv}) = \mathfrak{J}^m \cdot \ker(\chi_{triv}) \subseteq F^m H  \subseteq \ker(\chi_{triv}) \qquad\text{for any $m \geq 1$}
\end{equation*}
(the left equality coming from $\mathfrak{J}^m H + \ker(\chi_{triv}) = H$). Furthermore, the braid relations imply that $F^{jm} H \subseteq (F^j H)^m$.
\begin{fact}
  There is a $j \geq 1$ such that $F^j H \subseteq \mathfrak{J}H$.
\end{fact}
\begin{proof}
By a finite base extension of $k$ we may assume that $\mathbb{F}_q \subseteq k$. Then any simple supersingular $H$-module is a character (\cite{embed} Lemma 3.8).
But any supersingular character of $H$ must vanish on $\tau_s$ for at least one simple affine reflection $s$. This implies that $F^{r+1} H$, where $r$ denotes the rank of $\mathbf{G}$, is contained in the intersection $\mathfrak{R}$ of all the supersingular characters. But $\mathfrak{R}/\mathfrak{J}H$ is the Jacobson radical of the artinian ring $H/\mathfrak{J}H$. In any artinian ring the Jacobson radical is nilpotent. Hence we find an $n \geq 1$ such that $\mathfrak{R}^n \subseteq \mathfrak{J}H$. Now take $j := n(r+1)$.
\end{proof}

The fact implies that $F^{jm} H \subseteq \mathfrak{J}^m H$ for any $m \geq 1$. It follows that the two filtrations $\mathfrak{J}^m H \cap \ker(\chi_{triv})$ and $F^m H \cap \ker(\chi_{triv})$ of $\ker(\chi_{triv})$ are cofinal. Hence, the right hand sides of \eqref{f:injhull1} and of \eqref{f:injhull2} are isomorphic and we have
 $ {}^{\!\anti}( \ker(\trace^d))^{\anti} \cong
  \bigcup_m (H/\mathfrak{J}^m H)^\vee$ as $H$-bimodules.  Now using Remark \ref{rema:omittwists}:
\begin{equation*}
 \ker(\trace^d)\cong
  \bigcup_m (H/\mathfrak{J}^m H)^\vee\end{equation*} as $H$-bimodules
and   by Lemma \ref{lemma:extcont3.1} we have proved that $\ker(\trace^d)$ is an injective hull of the left (resp. right) $H$-module $(H/\mathfrak{J}H)^\vee$.

\end{proof}

\subsection{On  some values of the functor $H^d(I, _{\,-})$ when ${\mathbf G}={\rm SL}_2$}

We assume that    $ \mathbf G= {\rm SL}_2$ and that $I$ is torsionfree and therefore a Poincar\'e group of dimension $d$. It follows, in particular, that $p \geq 5$.
By \eqref{f:keydecompEd} and Proposition \ref{prop:kerS-injhull} we have
$$E^d\cong \ker(\trace^d) \oplus \chi_{triv}$$   as $H$-bimodules where $\ker(\trace^d)\cong \bigcup_{n\geq 1}(H/\zeta^n H)^\vee$. As a left or right $H$-module,  $\ker(\trace^d)$ is   an  injective envelope of
$(H/\zeta H)^\vee$. Being injective, this is a $\xi$-divisible  module on the left, resp. right, for any
$\xi\in H$ which is a non-zero-divisor.
For example, we know  that $H$ is free over $k[\zeta]$ (Lemma \ref{freeness}) so
 $Q(\zeta)$ is a non-zero-divisor for any nonzero polynomial $Q(X) \in k[X]$.
If furthermore $\chi_{triv}(\xi)\neq 0$, then  the whole space $E^d$ is $\xi$-divisible.
Recall that $\chi_{triv}(\zeta)=1$.

\begin{remark}\label{rema:fdimquotient} $\chi_{triv}$ is the only nontrivial finite dimensional quotient of $E^d$ as a left or right $H$-module.\end{remark}

\begin{proof} Since $\ker(\trace^d)$ is left and right $\zeta$-torsion, a finite dimensional quotient of $\ker(\trace^d)$  as a left, resp.\ right, module is annihilated by  a power $\zeta^m$ of $\zeta$ from the left, resp.\ right. But $\ker(\trace^d) \cdot \zeta^m = \zeta^m \cdot  \ker(\trace^d)= \ker(\trace^d)$ since $\ker(\trace^d)$ is $\zeta$-divisible. Therefore any  finite dimensional module quotient of $\ker(\trace^d)$ is trivial.
\end{proof}

Recall that $H^d(I,-)$ is  a right exact functor which  commutes with arbitrary
direct sums. By choosing a free presentation of an arbitrary left $H$-module $M$ this easily implies the formula
\begin{equation*}
  H^d(I, \X \otimes_H M) \cong E^d \otimes_H M .
\end{equation*}
This is an isomorphism of left $H$-modules.

\begin{proposition}\label{prop:xinonzerodiv}
  Let $G={\rm SL}_2(\mathfrak F)$. For any non-zero-divisor $\xi\in H$ such that $\xi$ is central in $H$ and $\chi_{triv}(\xi) \neq 0$, we have $H^d(I, \X/{\X \xi})=0$.
\end{proposition}
\begin{proof}
 Using the equality $\mathbf{X}/\mathbf{X}\xi = \mathbf{X} \otimes_H H/H\xi$ we compute
\begin{align*}
  H^d(I, \X/{\X \xi}) & = E^d \otimes_H H/H\xi = \chi_{triv} \otimes_H H/H\xi \oplus \ker(\trace^d) \otimes_H H/H\xi \\
  & = k/\chi_{triv}(\xi)k \oplus \ker(\trace^d)/\ker(\trace^d)\xi = 0 \ .
\end{align*}
\end{proof}
%

\begin{corollary}\label{cor:xinonzerodiv}
Let $Q(X) \in k[X]$ be a nonzero polynomial. Then $H^d(I, \X/\X Q(\zeta))=0$, resp.\ $\cong \chi_{triv}$ as an $H$-bimodule, if $Q(1) \neq 0$, resp.\ $Q(1)= 0$.
\end{corollary}
\begin{proof}
For the second part of the result, we simply notice that $\chi_{triv}\otimes_ H H/HQ(\zeta) \cong \chi_{triv}$ as a left $H$-module. Therefore,  proceeding as above, we obtain an isomorphism of left $H$-modules $H^d(I, \X/\X Q(\zeta))\cong \chi_{triv}$. This is an isomorphism of $H$-bimodules because   $H^d(I, \X/\X Q(\zeta))$ is a one-dimensional quotient of $E^d$ and using Remark \ref{rema:fdimquotient}.
\end{proof}

\begin{proposition}\label{prop:HdVzero}
   We have $H^d(I, V)=0$ for any irreducible admissible representation of $G := {\rm SL}_2(\mathfrak F)$ except when  $V= k_{triv}$ is the trivial representation in which case:
\begin{equation*}
  H^d(I,k_{triv})\cong\chi_{triv}  \qquad\text{as an $H$-bimodule}.
\end{equation*}
\end{proposition}
\begin{proof}
 The case when $V=k_{triv}$ is the trivial representation of $G$  is a particular case of \cite{Ext} Prop. 8.4.i. For the rest of the proof we therefore assume that $V \ncong k_{triv}$.  We first make we the following observations. Let $\bar{k}/k$ denote an algebraic closure of $k$. Then the scalar extension $V_{\bar{k}} := \bar{k} \otimes_k V$ is a smooth $G$-representation over $\bar{k}$.
\begin{itemize}
  \item Since $H^d(I,-)$ commutes with arbitrary direct sums we have $H^d(I,V_{\bar{k}}) = H^d(I, V) \otimes_k \bar{k}$.
  \item Since $V$ is admissible $\End_{\Mod(G)}(V)$ is finite dimensional over $k$.
  \item The $G$-representation $V_{\bar{k}}$ is of finite length with each irreducible constituent being admissible and not isomorphic to $k_{triv}$ (\cite{HV} Thm.\ III.4.1)-2), which needs the previous point as input).
\end{itemize}
By an argument with the exact cohomology sequence these observations reduce us to proving our assertion over $\bar{k}$. In fact, all we need in the following is that $\mathbb F_q\subseteq k$.

Given an irreducible admissible representation $V$ of $G$, the space $V^I$ is finite dimensional.  Let $Q\in k[X]$ denote the minimum polynomial of $\zeta$ on $V^I$, so that $Q(\zeta) V^I = 0$.  We claim that $V$ is a quotient representation of $\X/{\X Q(\zeta)}$. For this we choose a nonzero vector $v_0 \in V^I$, which gives rise to the surjective $G$-equivariant map $\mathbf{X} \twoheadrightarrow V$ sending $gI$ to $gv_0$. It restricts to the map $H \rightarrow V^I$ sending $\mathrm{char}_I$ to $v_0$. But $(gI)Q(\zeta) = gQ(\zeta) \mapsto gQ(\zeta)v_0 = 0$. It follows that the initial map factors over $\mathbf{X}/\mathbf{X}Q(\zeta)$.

If $Q(1) \neq 0$, then we have $H^d(I, \X/\X Q(\zeta))=0$, but $H^d(I, V)$ being a quotient of that space it is also zero. It remains to treat the case $Q(1) = 0$. Then we can choose the above vector $v_0$ so that $(\zeta - 1)v_0 = 0$ and $M := H v_0$ is a simple $H$-submodule of $V^I$. Since $\zeta$ is the identity on $M$, it follows from \cite{embed} Thm.\ 3.33 that $\mathbf{X} \otimes_H M$ is an irreducible $G$-representation with $(\mathbf{X} \otimes_H M)^I = M$. The inclusion $M \subseteq V^I$ induces a nonzero map $\mathbf{X} \otimes_H M \rightarrow V$ which by irreducibility must be an isomorphism. It follows that $V = \mathbf{X} \otimes_H V^I$ and that $V^I$ is a simple $H$-module. Hence there is a unique $\gamma\in \Gamma$ such that $V^I = e_\gamma V^I$ (notation in \S\ref{subsubsec:idempo}). It further follows that $H^d(I,V) = H^d(I, \mathbf{X} \otimes_H V^I) = E^d \otimes_H V^I \cong \chi_{triv} \otimes_H V^I \oplus  \ker(\trace^d) \otimes_H V^I = \chi_{triv} \otimes_H V^I$, the latter equality since $\ker(\trace^d)$ is divisible by $\zeta- 1$.
\begin{itemize}
\item If $\gamma\neq\{1\}$, then the idempotent  $e_\gamma$ satisfies
        $\chi_{triv}(e_\gamma)=0$, so  $H^d(I,V)= \{0\}$.
 \item If $\gamma=\{1\}$, then  we use Fact \ref{fact:noclassif}  to deduce that
     $V^I\cong \chi_{triv}$ or  $V^I\cong \chi_{sign}$.
 If  $V^I\cong \chi_{sign}$, then $ \chi_{triv}\otimes_H V^I=\{0\}$ because $\chi_{triv}(\tau_{s_0})=0$ and  $\chi_{sign}(\tau_{s_0})=-1$. If $V^I=\chi_{triv}$, then by \cite{OV} Lemma 2.25 we know that $\X\otimes _H  V^I\cong k_{triv}$ so $V= k_{triv}$.
\end{itemize}

\end{proof}

{\begin{remark}  \label{rema:centralizeEd}Let $z\in H$ be a central element $H$. Then $\anti(z)$ is also a central element and from the isomorphism
$\Delta^d:\quad E^d\xrightarrow{\cong}\: ({}^\anti E^0\,^\anti)^{\vee, f}$ (see \eqref{f:dual}) we  deduce that $z$ centralizes the elements of the $H$-bimodule $E^d$, namely $z\cdot \phi=\phi\cdot z$ for any $\phi\in E^d$. In particular the  left and the right actions  on $E^d$ of the central element $\zeta\in H $  coincide.
\end{remark}}

 \begin{lemma}\phantomsection\label{lemma:kerzetaE3}
 The kernel of the (left or right) action of $\zeta$  on $E^d$   is    isomorphic to $ \upiota ( H/\zeta H)$ as an $H$-bimodule.

%
%
%
%
%
\end{lemma}

\begin{proof}
By \eqref{f:keydecompEd} and Proposition \ref{prop:kerS-injhull},  we have $E^d\cong \bigcup_{n\geq 1}{(H/\zeta^n H)^\vee}\oplus \chi_{triv}$   as $H$-bimodules. Recall that $\chi_{triv}(\zeta)=1$.
The kernel of the action of $\zeta$ on $\bigcup_{n\geq 1}{(H/\zeta^n H)^\vee}\oplus \chi_{triv}$   is isomorphic to the $H$-bimodule  $(H/\zeta H)^\vee $ which, by \eqref{f:frobeniuszeta}, is isomorphic to  $\upiota ( H/\zeta H)$.
 \end{proof}

\section{\label{sec:E1}Formulas for the left action of $H$ on $E^1$ when ${\mathbf G}={\rm SL}_2(\mathbb Q_p)$, $p\neq 2,3$}

There is no hypothesis on $\mathfrak F$  and $G={\rm SL}_2(\mathfrak F)$ in \S\ref{subsec:conjugation} --\S\ref{subsec:commuel} with the exception that we assume $p \neq 2$ from \S\ref{triples} on.

\subsection{Conjugation by $\varpi$\label{subsec:conjugation}}

Recall the matrix  $\varpi:=
\left(\begin{smallmatrix}
       0 & 1 \\
       \pi & 0
\end{smallmatrix}\right)$ (Remark \ref{remark:N}) which normalizes the Iwahori subgroup $J$ and its pro-$p$ Sylow $I$ as well as the torus $T$.
We apply Section \ref{subsubsec:auto-pair} to the following automorphism of the pair $(g,\mathbf{X})$:
\begin{equation}\label{f:xiX}
   \xi: G\longrightarrow  G, \quad g\longmapsto \varpi^{-1} g\varpi\ \quad\text{and}\quad    \mathcal X: \X\longrightarrow  \X, \quad f\longmapsto f\circ \xi \ (\text{resp.\ $gI \mapsto \varpi g \varpi^{-1} I$}) .
\end{equation}
It gives rise to the involutive automorphism
\begin{align} \label{gammapi}
   \Gamma_\varpi := \Gamma_\xi : E^* = H^*(I, \X) \longrightarrow E^* = H^*(I, \X)
\end{align}
which is multiplicative for the Yoneda product as well as the cup product. It has all the properties listed in Section \ref{subsubsec:auto-pair}. In the following we sometimes abbreviate $\pw := \varpi w \varpi^{-1}$ for any $w \in \widetilde{W}$. We need the following additional fact. Recall that $\phi_w\in H^d(I,\mathbf{X}(w))$ was defined in \eqref{defiphi}.

\begin{lemma}
   Assume $I$ is a Poincar\'e group of dimension $d$. For $w\in \widetilde W$ we have
\begin{equation}\label{f:conjphi}
   \Gamma_\varpi(\phi_w)=\phi_{\varpi w\varpi^{-1}}\ .
\end{equation}
\end{lemma}
\begin{proof}
We recall from \eqref{f:xi-shapi} that we have the commutative diagram
\begin{equation*}
  \xymatrix{
H^d(I,\mathbf{X}(w)) \ar[r]^{\Sh_w} \ar[d]_{\Gamma_\varpi} &  H^d(I_w,k) \ar[r]^{\cores} \ar[d]^{\varpi_*} & H^d(I,k) \ar[d]^{\varpi_*}  \\
H^d(I,\mathbf{X}(\pw)) \ar[r]^-{\Sh_{\pw}} &  H^d(I_{\pw},k) \ar[r]^{\cores} & H^d(I,k)  }
\end{equation*}
where $\varpi_* = (\varpi^{-1})^*$ is the conjugation operator given on cocycles by sending $c$ to $c({\varpi^{-1}}_{-} \varpi)$. We will prove that the operator $\varpi_*$ on $H^d(I,k)$ is the identity. For this we follow the same idea as in  \cite{Ext} \S 7.2.3 and \cite{Koziol} Thm. 7.1.

For any $m \geq 1$ we have the open subgroup $K_{C,m} := \left(
\begin{smallmatrix}
1+\mathfrak{M} & \mathfrak{M}^m \\ \mathfrak{M}^{m+1} & 1+\mathfrak{M}
\end{smallmatrix}
\right)$ of $I$. It is normalized by $\varpi$. Since $\cores^{K_{C,m}}_I : H^d(K_{C,m},k) \xrightarrow{\cong} H^d(I,k)$ is an isomorphism (\cite{Ext} Rmk. 7.3) and commutes with corestriction we are reduced to showing that the operator $\varpi_*$ on $H^d(K_{C,m},k)$ is the identity. But for $m$ large enough the pro-$p$ group $K_{C,m}$ is uniform by \cite{Ext} Cor.\ 7.8 and Rmk.\ 7.10. So by \cite{Laz} V.2.2.6.3 and V.2.2.7.2, the one dimensional $k$-vector space $H^d(K_{C,m},k)$ is the maximal exterior power (via the cup product) of the $d$-dimensional $k$-vector space $H^1(K_{C,m},k)$. Conjugation commuting with the cup product, the action of $\varpi_*$ on $H^d(K_{C,m},k)$ is the determinant of $\varpi_*$  on $H^1(K_{C,m},k)$. The latter is the dual of the Frattini quotient $(K_{C,m})_\Phi$. This reduces us further to showing that the determinant of $\varpi_*$ on $(K_{C,m})_\Phi$ is equal to $1$. For this we consider the subgroups $ \EuScript{U}^-_{m+1}=\begin{psmallmat} 1& 0\cr \mathfrak M^{m+1}&  t \end{psmallmat}$, $ \EuScript{U}^+_m=\begin{psmallmat} 1& \mathfrak M^{m}\cr 0&  1 \end{psmallmat}$, and $T^m := \begin{psmallmat} 1+ \mathfrak M^m& 0\cr0&  1+ \mathfrak M^m\end{psmallmat}$ of $K_{C,m}$. According to \cite{Ext} Cor.\ 7.9 multiplication gives an isomorphism
\begin{equation*}
 \EuScript{U}^-_{m+1}/(\EuScript{U}^-_{m+1})^p \times T^m/(T^m)^p \times  \EuScript{U}^+_m/(\EuScript{U}^+_m)^p \xrightarrow{\;\cong\;} (K_{C,m})_\Phi \ .
\end{equation*}
One easily checks that $\varpi_*$ restricts to an isomorphism $\EuScript{U}^-_{m+1}/(\EuScript{U}^-_{m+1})^p \cong \EuScript{U}^+_m/(\EuScript{U}^+_m)^p$. These are $\mathbb{F}_p$-vector spaces of dimension equal to $[\mathfrak F:\mathbb Q_p]$. Hence the determinant of $\varpi_*$ on $\EuScript{U}^-_{m+1}/(\EuScript{U}^-_{m+1})^p \times  \EuScript{U}^+_m/(\EuScript{U}^+_m)^p$ is equal to $(-1)^{[\mathfrak F:\mathbb Q_p]}$. On the other hand, for $m$ large enough, the logarithm induces an isomorphism $T^m/(T^m)^p \cong 1+\pi^m \mathfrak{O}/(1+\pi^m\mathfrak{O})^p \cong \pi^m \mathfrak{O}/ p\pi^m\mathfrak{O} \cong \mathfrak{O}/p\mathfrak{O}$ with respect to which $\varpi_*$ corresponds to multiplication by $-1$. Hence its determinant on this factor is again equal to $(-1)^{[\mathfrak F:\mathbb Q_p]}$.
\end{proof}

\subsection{Elements of  $E^1$ as triples} \label{triples}
From now on we assume $p\neq 2$ unless it is specifically stated otherwise.

\subsubsection{Definition}

We refer to the notation introduced in \S\ref{rootdatum}.
We introduce the following subsets of $\widetilde W$:
\begin{align*}
{\widetilde W}^0&:=\{w\in {\widetilde W}, \: \ell(s_0 w)=\ell(w)+1\} \text{ and }\\
{\widetilde W}^1&:=\{w\in {\widetilde W}, \: \ell(s_1w)=\ell(w)+1\}.
\end{align*}
Note that the intersection of these  two subsets coincides with the set $\Omega=T^0/T^1$ of all elements in ${\widetilde W}$ with length $0$. Recall as in \cite[3.3]{embed}, we define for $m\geq 0$ the subgroups
\begin{equation}\label{defiIn}
  I_m^+ := \left(
\begin{smallmatrix}
1+\mathfrak{M} & \mathfrak{O} \\ \mathfrak{M}^{m+1} & 1+\mathfrak{M}
\end{smallmatrix}
\right) \qquad\text{and}\qquad I_m^- =\varpi I_m ^+ \varpi^{-1}= \varpi^{-1} I_m^+\varpi= \left(
\begin{smallmatrix}
1+\mathfrak{M} & \mathfrak{M}^m \\ \mathfrak{M} & 1+\mathfrak{M}
\end{smallmatrix}
\right)
\end{equation}
of $I$ and recall that
\begin{equation}\label{f:cap}
I_w=  I \cap wIw^{-1} =
  \begin{cases}
  I_{\ell(w)}^+ & \text{if $w\in {\widetilde W}^0$}, \\
  I_{\ell(w)}^- & \text{if $w\in {\widetilde W}^1$}.
  \end{cases}
\end{equation}

We abbreviate
$
  h^1 := H^1(I,\mathbf{X})$  and $   h^1(w) := H^1(I,\mathbf{X}(w))$ for $w \in \widetilde{W}$.
Recall the Shapiro isomorphism
$
  h^1(w) \cong H^1(I_w,k) = \Hom((I_w)_\Phi,k) $  (\S\ref{subsec:Ext}) where $(I_w)_\Phi$ denotes the Frattini quotient of $I_w$ (\cite{embed} \S 3.8).
   By \cite{embed} Prop. 3.62 we have isomorphisms
\begin{equation*}
(I_w)_\Phi \xrightarrow{\; \cong \;}  \mathfrak{O}/\mathfrak{M} \times (1+\mathfrak{M}) \big/ (1+\mathfrak{M}^{\ell(w)+1})(1+\mathfrak{M})^p \times \mathfrak{O}/\mathfrak{M}
\end{equation*}
for any $w \in \widetilde{W}$ (depending on a choice of a prime element in $\mathfrak{M}$). More precisely,    when $w\in {\widetilde W}^0$:
\begin{align}\label{f:I+ab}
 (I_{\ell(w)}^+)_\Phi&\xrightarrow{\; \cong \;}  \mathfrak{O}/\mathfrak{M} \times (1+\mathfrak{M}) \big/ (1+\mathfrak{M}^{\ell(w)+1})(1+\mathfrak{M})^p \times \mathfrak{O}/\mathfrak{M} \cr\begin{psmallmat} 1+\pi x & y \cr \pi^{\ell(w)+1} z & 1+\pi t \end{psmallmat} \bmod \Phi(I_w)&\longmapsto (z\bmod \mathfrak M, \,1+\pi x\bmod (1+\mathfrak{M}^{\ell(w)+1})(1+\mathfrak{M})^p,\, y\bmod \mathfrak M)
\end{align} and when $w\in {\widetilde W}^1$:
\begin{align}\label{f:I-ab}
 (I_{\ell(w)}^-)_\Phi&\xrightarrow{\; \cong \;}  \mathfrak{O}/\mathfrak{M} \times (1+\mathfrak{M}) \big/ (1+\mathfrak{M}^{\ell(w)+1})(1+\mathfrak{M})^p \times \mathfrak{O}/\mathfrak{M} \cr\begin{psmallmat} 1+\pi x & \pi ^{\ell(w)} y \cr \pi z & 1+\pi t \end{psmallmat} \bmod \Phi(I_w)&\longmapsto (z\bmod \mathfrak M, \,1+\pi x\bmod (1+\mathfrak{M}^{\ell(w)+1})(1+\mathfrak{M})^p,\, y\bmod \mathfrak M).\end{align}
By applying $\Hom(_-,k)$ and using the Shapiro isomorphism we deduce, for any $w \in \widetilde{W}$, a decomposition
\begin{equation*}
  h^1(w) = h^1_-(w) \oplus h^1_0(w) \oplus h^1_+(w)
\end{equation*}
such that
\begin{equation*}
  \begin{aligned}
  h^1_-(w) \\
  h^1_0(w) \\
  h^1_+(w)
  \end{aligned}
  \cong \Hom\left(
  \begin{aligned}
  \text{left factor} \\
  \text{middle factor} \\
  \text{right factor}
  \end{aligned}
  ,\: \: k\right) \ .
\end{equation*}
For any element $c \in h^1(w)$ we write this decomposition as
\begin{align}\label{tripdec}
  &{\Sh_w(c) = (c^-,c^0,c^+) }\text{ with }\cr c^\pm \in \Hom(\mathfrak{O}/\mathfrak{M},k) &\text{ and } c^0 \in \Hom((1+\mathfrak{M}) \big/ (1+\mathfrak{M}^{\ell(w)+1})(1+\mathfrak{M})^p ,k) \ .
\end{align}
We will often denote by \begin{equation}\label{f:defitrip}(c^-, c^0, c^+)_w\end{equation} the element in $h^1(w)$
which has image the triple $(c^-, c^0, c^+)\in H^1(I_w,k)$ via Shapiro isomorphism (with $c^0$ implicitly equal to $0$ when $\ell(w)=0$).

\begin{remark} \label{rema:simpliQp}
When $\mathfrak F=\mathbb Q_p$ and $p\neq 2$, we have $1+p^2\mathbb Z_p=(1+p\mathbb Z_p)^p$  since $\log : 1+p\mathbb{Z}_p \xrightarrow{\cong} p \mathbb{Z}_p$. Therefore,  when $\ell(w)\geq 1$, the identifications \eqref{f:I+ab} and \eqref{f:I-ab} become:
\begin{align}\label{f:I+abQp}
 (I_w)_\phi=(I_{\ell(w)}^+)_\Phi&\xrightarrow{\; \cong \;}  \mathbb Z_p/p\mathbb Z_p \times (1+p\mathbb Z_p) \big/ (1+p^2\mathbb Z_p) \times \mathbb Z_p/p\mathbb Z_p \cr\begin{psmallmat} 1+p x & y \cr p^{m+1} z & 1+p t \end{psmallmat} \bmod \Phi(I_w)&\longmapsto (z \bmod
 p\mathbb Z_p, 1+p x \bmod
 1+p^2\mathbb Z_p, y\bmod
 p\mathbb Z_p) \textrm{ when $w\in {\widetilde W}^0$} \ .
 \end{align}
 (in particular,
for $w\in \widetilde W^0$, $\ell(w)\geq 1$ we have $\res^{I_{s_1}}_{I_w}(\Sh_{s_1}(0, c^0,c^+)_{s_1})=\Sh_{w}((0, c^0,c^+)_w)$;
and
\begin{align}\label{f:I-abQp}
  (I_w)_\phi=(I_{\ell(w)}^-)_\Phi&\xrightarrow{\; \cong \;}  \mathbb Z_p/p\mathbb Z_p \times (1+p\mathbb Z_p) \big/ (1+p^2\mathbb Z_p) \times \mathbb Z_p/p\mathbb Z_p \cr\begin{psmallmat} 1+p x & p ^m y \cr p z & 1+p t \end{psmallmat} \bmod \Phi(I_w)&\longmapsto (z\bmod  p\mathbb Z_p , \,1+p x  \bmod 1+ p^2\mathbb Z_p, y \bmod  p\mathbb Z_p)  \textrm{ when $w\in {\widetilde W}^1$}.\end{align}
 (in particular,
for $w\in \widetilde W^1$, $\ell(w)\geq 1$ we have $\res^{I_{s_0}}_{I_w}(\Sh_{s_0}(c^-, c^0,0)_{s_0})=\Sh_{w}(c^-, c^0,0)_w)$. When $\ell(w)=0$, we have
$(I_w)_\Phi = I_\Phi \xrightarrow{\; \cong \;}  \mathbb Z_p/p\mathbb Z_p  \times \mathbb Z_p/p\mathbb Z_p$.
\end{remark}

\begin{notn}For any subset $U \subseteq \widetilde{W}$ we have the $k$-subspaces
\begin{equation*}
  h^1_-(U) := \oplus_{w \in U} h^1_-(w), \quad h^1_0(U) := \oplus_{w \in U} h^1_0(w),\quad \text{and}\ h^1_+(U) := \oplus_{w \in U} h^1_+(w)
\end{equation*}
of $h^1$.  We also let  $h^1_\pm(U):=h^1_-(U) \oplus h^1_+(U)$ and $h^1(U) := h_0^1(U) \oplus h_\pm^1(U)$.
 The subsets of most interest to us are:
\begin{align*}
 & \widetilde{W}^\epsilon  := \{w \in \widetilde{W} : \ell(s_\epsilon w) = \ell(w) +1\} \quad\text{for $\epsilon \in \{0,1\}$} \text{ as defined above, and}, \\
 & \widetilde{W}^{\epsilon, odd}  := \{w \in \widetilde{W}^\epsilon : \ell(w)\ \text{is odd}\}, \\
  &\widetilde{W}^{\epsilon, even}  :=  \{w \in \widetilde{W}^\epsilon : \ell(w)\ \text{is even}\}, \\
  &\widetilde{W}^{\epsilon, +even}  := \widetilde{W}^{\epsilon,even} \setminus \Omega.
\end{align*}
We also define, for $k\geq 0$ and $\epsilon\in \{0,1\}$:
\begin{align*}
 &\widetilde{W}^{ \ell\geq k} := \{w \in \widetilde{W} : \textrm{$\ell(w)\geq k$}\} \\
&\widetilde{W}^{\epsilon, \ell\geq k} := \{w \in \widetilde{W} : \ell(s_\epsilon w) = \ell(w) +1 \textrm{ and $\ell(w)\geq k$}\} \quad\text{for $\epsilon \in \{0,1\}$}.\end{align*}

\end{notn}
\subsubsection{Triples and conjugation by $\varpi$}

\begin{lemma} Let $w\in \widetilde W$ and $(c^-, c^0, c^+)_w\in  h^1(w)$.
Its image  by the map  $\Gamma_\varpi$  of conjugation by $\varpi$ defined in \eqref{gammapi}
is
$$(c^+, -c^0, c^-)_{\varpi w\varpi^{-1}}\in h^1(\varpi w\varpi^{-1}) \  $$ and
if $w\in \widetilde W^\epsilon$, then $\varpi w\varpi^{-1}\in \widetilde W^{1-\epsilon}$.
\label{lemma:conjtrip}
\end{lemma}

\begin{proof} See Remark \eqref{remark:N} for the second claim.
By definition of the triples and by commutativity of  diagram \eqref{f:xi-shapi}, the first claim follows directly from the observation that
the matrices
$$\begin{psmallmat} 1+\pi x & \pi^{\ell(w)} y \cr \pi  z & 1+\pi t \end{psmallmat}\in I_{\ell(w)}^- \quad \text{ and }\begin{psmallmat} 1+\pi t & z \cr \pi^{{\ell(w)}+1} y & 1+\pi x \end{psmallmat}\in I_{\ell(w)}^+$$ are conjugate to each other via $\varpi$.

\end{proof}

\subsubsection{\label{subsubsec:normalize}Triples and cup product}

Suppose $\mathfrak F=\mathbb Q_p$, $p\neq 2,3$. We introduce  the isomorphism  \begin{equation} \iota: 1+p\mathbb  Z_p/1+ p^2\mathbb Z_p\overset{\simeq}\longrightarrow   \mathbb Z_p/p\mathbb Z_p, \quad 1+px\mapsto x\bmod p\mathbb Z_p \label{f:iota}\ .\end{equation}
We choose and fix  elements with the following constraints
\begin{equation}\label{f:normalize}
\upalpha\in \mathbb Z_p/p\mathbb Z_p\setminus\{0\}, \:\quad\upalpha^0=\iota^{-1}(\upalpha), \:\quad
\c\in \Hom (\mathbb Z_p/p\mathbb Z_p,k)\text{ such that }\c(\upalpha)=1, \:\quad \c^0:=\c\iota
\end{equation}

When $\ell(w)>0$, the dimension of the Frattini quotient of $I_w$ is $3$, namely  the dimension of $I_w$ as a $p$-adic manifold. By \cite{KS} Cor. 1.8 this means that $I_w$ is uniform. Therefore,
 the algebra
$H^*(I_w, k)$  is the exterior power (via the cup product) of the $3$-dimensional $k$-vector space $H^1(I_w,k)$.
 In particular,
$(\c, 0,0)_{s_0}\cup(0, \c^0,0)_{s_0}\cup (0, 0, \c)_{s_0}$ is a nonzero element of $H^3(I, \X(s_0))$ and its image via
$$H^3(I, \X(s_0))\xrightarrow{\Sh_{s_0}}  H^3(I_{s_0},k)\xrightarrow{\cores^{I_{s_0}}_I} H^3(I,k)$$ is a nonzero element of the one dimensional vector space
$H^3(I,k)$ (see \cite{Ext} Rmk. 7.3). We choose  the isomorphism
$\eta: H^3(I,k)\overset{\simeq}\rightarrow k$ sending that element to $1$.  As in \S\ref{subsubsec:duality}, this choice of $\eta$ yields a
choice of a basis $(\phi_w)_{w\in \widetilde W}$  of $H^d(I,\X)$ which is dual to $(\tau_w)_{w\in \widetilde W}$ via \eqref{f:dual}.  By definition, we have
$$(\c, 0,0)_{s_0}\cup(0, \c^0,0)_{s_0}\cup (0, 0, \c)_{s_0}=\phi_{s_0}$$
\begin{lemma}
For any $w\in\widetilde W$ with $\ell(w)\geq 1$, we have
\begin{equation}\label{f:orientation}
(\c, 0,0)_{w}\cup(0, \c^0,0)_{w}\cup (0, 0,\c)_{w}=\phi_{w}
\end{equation}
\end{lemma}

\begin{proof}
By definition \eqref{defiphi} of $\phi_w$,  it is enough to prove that
$$
\cores_I^{I_w}\circ \Sh_w\big((\c, 0,0)_{w}\cup(0, \c^0,0)_{w}\cup (0, 0,\c)_{w}\big)=\cores_I^{I_{s_0}}\circ \Sh_{s_0}\big((\c, 0,0)_{s_0}\cup(0, \c^0,0)_{s_0}\cup (0, 0, \c)_{s_0}\big).
$$
\noindent $\bullet$  First suppose that  $w\in \widetilde W^1$.
Recall (see \cite{Ext} \S 3.3), that the Shapiro isomorphism commutes with the cup product.
We compute that $\cores_{I_{s_0}}^{I_w}\circ \Sh_w\big((\c, 0,0)_{w}\cup(0, \c^0,0)_{w}\cup (0, 0,\c)_{w}\big)$ is equal to
\begin{align*}&
\cores_{I_{s_0}}^{I_w}[\Sh_w\big((\c, 0,0)_{w}\big)\cup  \Sh_w\big((0, \c^0,0)_{w}\big)\cup\Sh_w\big((0, 0,\c)_{w}\big)]\cr=&
\cores_{I_{s_0}}^{I_w}[\res^{I_{s_0}}_{I_w}\big(\Sh_{s_0}\big((\c, 0,0)_{s_0}\big)\cup  \Sh_{s_0}\big((0, \c^0,0)_{s_0}\big)\big)\cup\Sh_w\big((0, 0,\c)_{w}\big)] \text{ by Remark \ref{rema:simpliQp}}\cr=&
\Sh_{s_0}\big((\c, 0,0)_{s_0}\big)\cup  \Sh_{s_0}\big((0, \c^0,0)_{s_0}\big)\cup \cores_{I_{s_0}}^{I_w}[\Sh_w\big((0, 0,\c)_{w}\big)] \text{ by the projection formula (\cite{Ext} \S 4.6)}
\cr=&
\Sh_{s_0}\big((\c, 0,0)_{s_0}\big)\cup  \Sh_{s_0}\big((0, \c^0,0)_{s_0}\big)\cup \Sh_{s_0} \big((0, 0,\c)_{s_0}\big)\text{ by \cite{embed} Lemma 3.68-iv.}\cr=&\Sh_{s_0}\big((\c, 0,0)_{s_0}\cup (0, \c^0,0)_{s_0}\cup (0, 0,\c)_{s_0}\big)=\Sh_{s_0}\big(\phi_{s_0}) \
\end{align*} which proves the expected statement after applying $\cores^{I_{s_0}}_I$.
\\\noindent $\bullet$
If $w\in \widetilde W^0$, we conjugate by $\varpi$ using  $\Gamma_\varpi$  (see \eqref{gammapi}):
\begin{gather*}
\Gamma_\varpi\big((\c, 0,0)_{w}\cup(0, \c^0,0)_{w}\cup (0, 0,\c)_{w}\big)  \qquad\qquad\qquad\qquad\qquad\qquad\qquad\qquad\qquad\qquad\qquad\qquad  \\
\begin{split}
& = -(0, 0,\c)_{\varpi w \varpi^{-1}}\cup(0, \c^0,0)_{\varpi w \varpi^{-1}}\cup (\c, 0,0)_{\varpi w \varpi^{-1}} \quad\text{by \eqref{f:cup+Gamma} and Lemma \ref{lemma:conjtrip} } \\
& = (\c, 0,0)_{\varpi w \varpi^{-1}}\cup(0, \c^0,0)_{\varpi w \varpi^{-1}}\cup (0, 0,\c)_{\varpi w \varpi^{-1}}  \quad\text{by anticommutativity of $\cup$} \\
& = \phi_{\varpi w \varpi^{-1}} \quad\text{ since $\varpi w \varpi^{-1}\in \widetilde W^1$}  \\
& = \Gamma_\varpi(\phi_{w})  \quad\text{ by \eqref{f:conjphi}}
\end{split}
\end{gather*}
which concludes the proof since $\Gamma_\varpi$ is bijective.
\end{proof}

\begin{example}\label{ex:H*I}The subalgebra $H^*(I, \X(1))$ of $E^*$:
\begin{itemize}
\item[-]  $H^0(I, \X(1))$ has dimension $1$,
\item[-]  $H^3(I, \X(1))$ has dimension $1$ with basis $\phi_1$ which satisfies $\eta(\phi_1)=1$.
\item[-]  $H^1(I, \X(1))$ has dimension $2$ and basis $(\c, 0,0)_1$ and  $(0,0,\c)_1$,
\item[-]  $H^2(I, \X(1))$ is dual to $H^1(I, \X(1))$  via the cup product. We denote by $(\upalpha, 0,0)_1$ and  $(0,0,\upalpha)_1$ the dual of the basis of  $H^1(I, \X(1))$ given above, it satisfies  by definition:
\begin{itemize}
\item $(\upalpha, 0,0)_1\cup (\c, 0,0)_1=(\c, 0,0)_1\cup(\upalpha, 0,0)_1=\phi_1=(0,0,\upalpha)_1\cup (0,0,\c)_1=(0, 0,\c)_1\cup(0,0,\upalpha)_1$, while
\end{itemize}
\begin{itemize}
\item $(\c, 0,0)_1\cup (0,0,\c)_1=(0,0,\c)_1\cup (\c, 0,0)_1=0$.


\end{itemize}

\end{itemize}
\end{example}

\subsection{\label{subsec:imageanti}Image of a triple   under the anti-involution $\anti$}

Let $c\in h^1(w)$ seen as a triplet $(c^-, c^0, c^+)_w$ as in \eqref{tripdec}. Its image by $\anti$  is an element in $ h^1(w^{-1})$  whose image by Shapiro isomorphism is given by (see \eqref{f:invodefi}) $$(\Sh_w c)(w _{-} w^{-1}): I_{w^{-1}}\rightarrow k \:.$$

\begin{lemma}\label{lemma:antievenodd}
Let $w\in \widetilde W$  and  $c=(c^-, c^0, c^+)_w\in h^1(w)$.\\ If $\ell(w)$ is even then
 \begin{equation}\label{f:antieven}\anti(c)=(c^-({{u^2}}_{-}), c^0, c^+({u^{-2}}_{-}))_{w^{-1}}.\end{equation}
  If $\ell(w)$ is odd then
 \begin{equation}\label{f:antiodd}\anti(c)=(-c^+({u^{-2}}_{-}), -c^0, -c^-({u^2}_{-}))_{w^{-1}}.\end{equation}
 where $u\in (\mathfrak{O}/\mathfrak{M})^\times$ is such that $\omega_u^{-1} w$ lies in the subgroup of $\widetilde W$ generated by $s_0$ and $s_1$.
\end{lemma}
\begin{proof} Notice  that the intersection of $\Omega $  and of the subgroup of $\widetilde W$ generated by $s_0$ and $s_1$ is equal to $\{\pm 1\}$, therefore ${u^2}$ is determined by $w$.

\begin{itemize}
\item If $w= \omega_u(s_0s_1)^n$, then $I_{w^{-1}}=I_{2n}^+$ and
for  $X=\begin{psmallmat} 1+\pi x &  y \cr \pi^{1+2n} z & 1+\pi t \end{psmallmat}\in I_{w^{-1}}$ we have
$w X w^{-1}=  \omega_u\begin{psmallmat} 1+\pi x &  \pi^{2n} y \cr \pi z & 1+\pi t \end{psmallmat} \omega_u^{-1}=\begin{psmallmat} 1+\pi x &   [u]  ^{-2}\pi^{2n} y \cr [u]  ^2 \pi z & 1+\pi t \end{psmallmat}$.
So \begin{equation*}\Sh_{w^{-1}}(\anti(c))=(c^-( {u  ^2}_{-}), c^0, c^+( {u  ^{-2}}_{-}))_{w^{-1}}.\end{equation*}
\item If $w= \omega_u(s_1s_0)^n$, then $I_{w^{-1}}=I_{2n}^-$ and
for  $X=\begin{psmallmat} 1+\pi x & \pi^{2n} y \cr \pi z & 1+\pi t \end{psmallmat}\in I_{w^{-1}}$ we have
$w X w^{-1}=  \omega_u\begin{psmallmat} 1+\pi x &   y \cr \pi^{1+2n} z & 1+\pi t \end{psmallmat} \omega_u^{-1}=\begin{psmallmat} 1+\pi x &   [u]  ^{-2} y \cr [u]  ^2 \pi^{1+2n} z & 1+\pi t \end{psmallmat}$
So \begin{equation*}\Sh_{w^{-1}}(\anti(c))=(c^-( {u  ^2}_{-}), c^0, c^+( {u  ^{-2}}_{-}))_{w^{-1}}.\end{equation*}
\item If $w= \omega_u (s_1s_0)^n s_1$, then $I_{w^{-1}}=I_{2n+1}^+$ and
for  $X=\begin{psmallmat} 1+\pi x &  y \cr \pi^{2+2n} z & 1+\pi t \end{psmallmat}\in I_{w^{-1}}$ we have
$w X w^{-1}=  \omega_u\begin{psmallmat} 1+\pi t &  -z \cr -\pi^{2+2n} y & 1+\pi x \end{psmallmat} \omega_u^{-1}=\begin{psmallmat} 1+\pi t &  - [u]  ^{-2}z \cr -\pi^{2+2n}  [u]  ^2 y & 1+\pi x \end{psmallmat} $
So \begin{equation*} \Sh_{w^{-1}}(\anti(c))=(-c^+( {u  ^{-2}}_{-}), -c^0, -c^-( {u  ^2}_{-}))_{w^{-1}}.\end{equation*}
\item If $w= \omega_u(s_0s_1)^ns_0$, then $I_{w^{-1}}=I_{2n+1}^-$ and
for  $X=\begin{psmallmat} 1+\pi x & \pi^{2n+1} y \cr \pi z & 1+\pi t \end{psmallmat}\in I_{w^{-1}}$ we have
$w X w^{-1}=  \omega_u\begin{psmallmat} 1+\pi t &  -\pi^{1+2n} z \cr -\pi y & 1+\pi z \end{psmallmat} \omega_u^{-1}=\begin{psmallmat} 1+\pi t &  -\pi^{1+2n} [u]  ^{-2} z \cr -\pi  [u]  ^{2} y & 1+\pi x \end{psmallmat}$
So \begin{equation*}\Sh_{w^{-1}}(\anti(c))=(-c^+( {u  ^{-2}}_{-}), -c^0, -c^-( {u  ^2}_{-}))_{w^{-1}}.\end{equation*}

\end{itemize}
\end{proof}

\subsection{\label{subsec:omega}Action of  $\tau_\omega$ on $E^1$ for $\omega\in \Omega $}

Let  $w\in \widetilde W$,  $\omega \in T^0/T^1 $ and $c\in h ^i(w)$ for some $i\geq 0$. By \cite{Ext} Prop.\ 5.6, the left action  of $\tau_{\omega}$ on $c$
corresponds to the following transformation, where again we identify $c$ with its image in $H^i(I_w,k)$ by Shapiro isomorphism:
\begin{equation}\label{f:omega}
  \xymatrix{
    h^i( w) \ar[d]^{\Sh_w} \ar[rrr]^{\tau_{\omega} }_{\cong} && & h^i(\omega w) \ar[d]_{\Sh_{\omega w}} \\
    H^i(I_w,k) \ar[rrr]^{{\omega}_*(c) = c({{\omega}^{-1}} _- \omega)} & && H^i(I_w,k)   }
\end{equation}
In other words,
 for $\omega \in \Omega $, we have $\tau_\omega	\cdot c\in h^i(\omega w)$ and
\begin{equation}\label{f:leftomega0}
  \Sh_{\omega w}(\tau_\omega	\cdot c)=\omega_*\Sh_w(c) \ .
\end{equation}
Using $ c\cdot  \tau_\omega=\anti(\tau_{\omega^{-1}}\cdot \anti(c))$,
 we also obtain $ c \cdot \tau_\omega \in h^i(w\omega )$   and \begin{equation}\label{f:rightomega0}
  \Sh_{w \omega }( c\cdot  \tau_\omega)=\Sh_w(c) \ .
\end{equation}

Now we suppose $i=1$. We identify $c\in h^1(w)$ with a triple   $(c^-,c^0,c^+)_w$ as in  \eqref{tripdec}.
  For $u\in (\mathfrak O/\mathfrak M)^\times$
and $\left(\begin{smallmatrix} x&y \cr  z&t \end{smallmatrix}\right)\in I_w$ we have
 $\omega_u^{-1}\left(\begin{smallmatrix} x&y \cr  z&t \end{smallmatrix}\right)\omega_u=\left(\begin{smallmatrix} x&[u]^2y \cr [u]^{-2} z&t \end{smallmatrix}\right)$ and therefore\begin{equation}\label{f:leftomega}
 \tau_{\omega_u}\cdot (c^-, c^0, c^+)_w=(c^-({u^{-2}}\:_-), c^0, c^+({u^{2}}\:_-))_{\omega_u w}\ \in h^1(\omega_u w) \ .
 \end{equation}
In particular,
  \begin{equation}\label{f:lefts2}
 \tau_{s^2}\cdot (c^-, c^0, c^+)_w=(c^-, c^0, c^+)_{s^2 w}  \in h^1(s^2 w)\end{equation}
 for $s\in \{s_0,s_1\}$ since $s^2=\omega_{-1}$.
%
\noindent For the right action, it follows from \eqref{f:rightomega0} that
\begin{equation}\label{f:rightomegaE1}
(c^-, c^0, c^+)_w \cdot \tau_{\omega_u}=(c^-, c^0, c^+)_{ w\omega_u}\ \in h^1( w\omega_u) \ .
 \end{equation}


\subsection{\label{subsec:commuel}Action of the idempotents $e_\lambda$}
For  $\lambda:{\Omega}\rightarrow k^\times$ and $w\in {\widetilde W}$, recall that  we defined the idempotent
$
e_\lambda\in k[\Omega]$
  (see \eqref{defel}) and that, for any $\omega\in \Omega $ we have $e_\lambda\tau_\omega=\tau_{\omega} e_\lambda=
\lambda(\omega) e_\lambda.$

%
%
%

\begin{lemma} \label{leftidem}Let $\lambda, \mu: {\Omega}\rightarrow k^\times$, $w\in \widetilde W$. We consider an element  $c\in h^i(w)$ with image $c_w\in H^i(I_w,k)$ by Shapiro isomorphism.
We have \begin{center} $e_\lambda  \cdot c   = c  \cdot e _\mu$ if and only if
$c_w= \mu(w^{-1} \omega w) \lambda(\omega^{-1}) \,\omega_*(c_w)$\quad  for any $\omega\in \Omega  $.\end{center}
\end{lemma}

\begin{proof}
The element $e_\lambda  \cdot c $ lies in $\oplus_{\omega\in \Omega} H^i (I, \X(\omega w))$ and  its component in
$H^i (I, \X(\omega w))$ is
$$-\lambda(\omega^{-1})\Sh_{\omega w}^{-1}\big( \omega_* c_w\big)$$
The element $c\cdot e_\mu $ lies in $\oplus_{t\in \Omega} H^i (I, \X( wt))$ and  its component in
$H^i (I, \X(wt))=H^i (I, \X(wtw^{-1} w))$ is
$$-\mu(t^{-1})\Sh_{wt}^{-1}\big( c_w\big)=-\mu(w^{-1} (wt^{-1}w^{-1}) w)\Sh_{wtw^{-1} w}^{-1}\big( c_w\big)$$
These two elements are equal if and only if  for any $\omega\in \Omega$ we have
$\lambda(\omega^{-1}) \omega_* c_w=\mu(w^{-1} \omega^{-1} w) c_w$.
\end{proof}

In the same context as in the lemma, we suppose that $i=1$. Then we may see the image in $H^1(I_w,k)$ by Shapiro isomorphism of
 $c\in h^1(w)$  as a  $(c^-, c^0, c^+)$ as in \eqref{tripdec}.
  For $u\in (\mathfrak O/\mathfrak M)^\times$, we know   from the calculation that gave \eqref{f:leftomega} that
$${\omega_u}_*(c^-, c^0, c^+) =(c^-({u^{-2}}\,_-)\, , c^0,\, c^+({u^{2}}\,_-)) \in H^1(I_w,k).$$
If $\ell(w)$ is even, then the conjugation of $\mu$ by $w$ is equal to $\mu$ and therefore   $e_\lambda\cdot  c  = c \cdot  e_\mu  $ if and only if
$c=\mu\lambda^{-1}(\omega_u ) \,{\omega_u}_*(c )$ for any $u\in (\mathfrak O/\mathfrak M)^\times$. So
\begin{equation}\label{f:condeven}
(\ell(w)\text{ even})):\quad  e_\lambda\cdot  c  = c \cdot e_\mu   \textrm{ if and only if }\left\lbrace\begin{array}{l}c ^-=\mu\lambda^{-1}(\omega_u ) c ^-(u^{-2}\,_-)\cr \cr c ^0=\mu\lambda^{-1}(\omega_u ) c ^0\cr\cr
c ^+=\mu\lambda^{-1}(\omega_u ) c ^+(u^{2}\,_-).
 \end{array}\right. \textrm{ for any $u\in(\mathfrak O/\mathfrak M)^\times$}.
\end{equation}
If $\ell(w)$ is odd, then the conjugation of $\mu$ by $w$ is equal to $\mu^{-1}$ and therefore   $e_\lambda\cdot c  = c \cdot e_\mu  $ if and only if
$c_{ w}=(\mu\lambda)^{-1}(\omega_u ) \,\omega_*(c )$ for any $u\in (\mathfrak O/\mathfrak M)^\times$ which is equivalent to
\begin{equation}\label{f:condodd}
(\ell(w)\text{ odd})):\quad e_\lambda\cdot c  = c \cdot e_\mu  \textrm{ if and only if } \left\lbrace\begin{array}{l}c ^-=(\mu\lambda)^{-1}(\omega_u ) c ^-(u^{-2}\,_-)\cr \cr c ^0=(\mu\lambda)^{-1}(\omega_u ) c ^0\cr\cr
c ^+=(\mu\lambda)^{-1}(\omega_u ) c ^+(u^{2}\,_-).
 \end{array}\right. \textrm{ for any $u\in(\mathfrak O/\mathfrak M)^\times$}.
\end{equation}
{ An important special case of the above is the following. Suppose that $q = p$; for any $m \in \mathbb{Z}$ and $w \in \widetilde{W}$ we then have
\begin{equation}\label{f:id-id}
  (c^-,c^0,c^+)_w \cdot e_{\id^m} = e_{\id^{m(-1)^{\ell(w)} - 2}} \cdot (c^-,0,0)_w + e_{\id^{m(-1)^{\ell(w)}}} \cdot (0,c^0,0)_w +  e_{\id^{m(-1)^{\ell(w)} + 2}} \cdot (0,0,c^+)_w \ .
\end{equation}
}

\subsection{\label{subsec:HonE1}Action of $H$ on $E^1$ when $G={\rm SL}_2(\mathbb Q_p)$, $p\neq 2,3$}

In this whole subsection,  $G={\rm SL}_2(\mathbb Q_p)$ with $p\neq 2,3$. \textbf{We also choose $\pi=p$}. This is required in the proof of Lemma \ref{lemma:conjtrans} which is used in the proof of Proposition \ref{prop:theformulas}.
The isomorphism $\iota$ was introduced in \eqref{f:iota}.
The following proposition is proved in \S\ref{subsec:proofformu}. Together with \eqref{f:leftomega}, it gives the explicit left action of $H$ on  $E^1$ when $G={\rm SL}_2(\mathbb Q_p)$ with $p\neq 2,3$.

\begin{proposition}\label{prop:theformulas}
Let $w\in \widetilde W$ and $(c^-,c^0,c^+)_w\in h^1(w)$.
\begin{multline*}
  \tau_{s_0} \cdot (c^-,c^0,c^+)_w =  \\
  \begin{cases}
  (0,{-}c^0 , -c^-)_{s_0 w}  & \text{if $w \in \widetilde{W}^0$, $\ell(w) \geq 1$,}\cr
  e_1 \cdot (-c^-, -c^0 , -c^+)_{ w} + e_{\id} \cdot (0, -2c^-\iota , 0)_{ w} + (0,0,-c^-)_{s_0 w} & \text{if  $w \in \widetilde{W}^1$ with $\ell(w) \geq 2$, }\cr
  e_1 \cdot (-c^-, -c^0 , -c^+)_{ w} + e_{\id} \cdot (0, -2c^-\iota , c^0\iota^{-1})_{ w} & \cr
  \qquad\qquad\qquad\qquad + \, e_{\id^2} \cdot (0,0 , c^-)_{ w} +(0,0,-c^-)_{s_0 w} & \text{if  $w \in \widetilde{W}^1$ with $\ell(w) =1$}.\cr
  \end{cases}\\
   \shoveleft{\tau_{s_1} \cdot (c^-,c^0,c^+)_w =}   \\
  \begin{cases}    (-c^+,  -c^0, 0)_{s_1 w}      & \text{if $w \in \widetilde{W}^1$, $\ell(w) \geq 1$},\cr     e_1 \cdot (-c^-, -c^0 , -c^+)_{ w}+
  e_{\id^{-1}} \cdot ( 0,2c^+\iota  , 0)_{ w} + (-c^+,0,0)_{s_1 w} & \text{if  $w \in \widetilde{W}^0$ with $\ell(w) \geq 2$},\cr   e_1 \cdot (-c^-, -c^0 , -c^+)_{ w}+e_{\id^{-1}} \cdot (-c^0\iota^{-1},2c^+\iota , 0)_{ w}&\cr
  \qquad\qquad\qquad\qquad + \, e_{\id^{-2}} \cdot ( c^+,0 , 0)_{ w} +(-c^+,0,0)_{s_1 w} & \text{if  $w \in \widetilde{W}^0$ with $\ell(w) = 1$.}
  \end{cases}\\
  \shoveleft{\tau_{s_0}\cdot (c^-,0,c^+)_\omega  =
 (0, 0, -c^-)_{s_0 \omega} \quad\text{for $\omega \in \Omega $}}. \\
  \shoveleft{\tau_{s_1}\cdot (c^-,0,c^+)_\omega  =
   (-c^+, 0, 0)_{s_1 \omega} \quad\text{for $\omega \in \Omega $}}.\\
\end{multline*}
\end{proposition}

In these formulas, we use the  notation $e_{\id^m}$ as introduced in \eqref{f:idm} for $m\in \mathbb Z$. Recall, using \eqref{f:leftomega}, that for $(d^-, d^0, d^+)_w\in h^1(w)$,
the component in $h^1(\omega_u w)$ of
$e_{\id^m}\cdot (d^-, d^0, d^+)_w\in \bigoplus_{{u\in \mathbb F_p^\times }} h^1(\omega_u w)$ is  given by
\begin{equation}\label{f:uplambdac}
 - \id^m(u^{-1})\,\tau_{\omega_u} \cdot (d^-, d^0, d^+)_w=-u^{-m}(d^-({u^{-2}}_-), d^0(_-), d^+({u^{2}}_-))_{\omega_uw} \ .
\end{equation}

\begin{corollary}\label{coro:form-zeta}
Let $w\in \widetilde W$, $\omega\in \Omega$, and $(c^-,c^0,c^+)_w\in h^1(w)$.
\begin{multline*}
\zeta\cdot  (c^-,0,c^+)_\omega = \\
   (c^-,0,0)_{s_1 s_0 \omega} +  (0,0,c^+)_{s_0 s_1 \omega}
                 + e_1 \cdot  (0,0,-c^-)_{s_0 \omega} + e_1 \cdot  (-c^+,0,0)_{s_1\omega} + e_1 \cdot (c^-, 0, c^+)_\omega . \qquad\ 
\end{multline*}
\begin{multline*}
\zeta \cdot(c^-,c^0,c^+)_w =  \\
  \begin{cases}
(c^-,c^0 , 0)_{s_1s_0 w} + e_{\id} \cdot ( 0,  -2c^+\iota , 0)_{ s_0w} & \cr
   \qquad\qquad + \, e_{\id}  \cdot  (0, 2c^+\iota ,0)_{ s_1w} + (0,0,c^+)_{s_0 s_1w}
 &  \text{if $w \in \widetilde{W}^0$,  $\ell(w)\geq 3$,} \cr
(c^-,c^0 , 0)_{s_1s_0 w} + e_{\id} \cdot  ( 0,  -2c^+\iota, 0)_{ s_0w} & \cr
    \qquad\qquad + \, e_{\id} \cdot  (0, 2c^+\iota ,0)_{ s_1w} + e_{\id^2} \cdot  (0, 0 , -c^+)_{ s_1w} + (0,0,c^+)_{s_0 s_1w}
  &  \text{if $w \in \widetilde{W}^0$,  $\ell(w)=2$,} \cr
 (c^-,c^0 , 0)_{s_1s_0 w} + e_{\id} \cdot  ( 0,  -2c^+\iota, c^0\iota^{-1}) _{ s_0w}  & \cr
     \qquad\qquad + \, e_{\id^2} \cdot (0,0,-c^+ )_{ s_0w} + (0,0,c^+)_{s_0s_1 w} + e_1 \cdot (-c^+,0,0)_{s_1w}  &  \text{if $w \in  s_1\Omega$.}
  \end{cases}
\end{multline*}

\begin{multline*}
 \zeta \cdot (c^-,c^0,c^+)_w = \\
 \begin{cases}
 (0,c^0 , c^+)_{s_0s_1 w} + e_{\id^{-1}}  \cdot ( 0,  2c^-\iota , 0 )_{ s_1w}  & \cr
 \qquad\quad + \, e_{\id^{-1}}  \cdot (0, -2c^-\iota , 0)_{ s_0w} + (c^-,0,0)_{s_1 s_0w}
 &  \text{if $w \in \widetilde{W}^1$,  $\ell(w)\geq 3 $,} \cr
 (0,c^0 , c^+)_{s_0s_1 w} + e_{\id^{-1}}  \cdot ( 0,  2c^-\iota , 0 )_{ s_1w}  & \cr
 \qquad\quad + \, e_{\id^{-1}}  \cdot (0, -2c^-\iota, 0)_{ s_0w}+e_{\id^{-2}} \cdot  (-c^-, 0, 0)_{ s_0w} + (c^-,0,0)_{s_1 s_0w}
  &  \text{if $w \in \widetilde{W}^1$,  $\ell(w)=2$,} \cr
(0,c^0 , c^+)_{s_0s_1 w} + e_{\id^{-1}}  \cdot (-c^0\iota^{-1},  2c^-\iota, 0 )_{ s_1w}  & \cr
 \qquad\quad + \, e_{\id^{-2}}  \cdot ( -c^-,  0 , 0 )_{ s_1w}+ (c^-,0,0)_{s_1s_0 w} + e_1 \cdot (0,0,-c^-)_{s_0w}  & \text{if $w \in  s_0\Omega$.}
 \end{cases} 
\end{multline*}
\end{corollary}

The decreasing filtration $( F^m E^1)_{m\geq 1}$ was introduced in \S\ref{subsubsec:fil}.

\begin{corollary}\label{coro:zetaE1F}
   We have $\zeta\cdot E^1\supseteq F^3 E^1$
\end{corollary}
\begin{proof}
It is easy to see that $\zeta\cdot E^1$ contains $h^1_-(\widetilde W^{0, \ell\geq 3})$ and $h^1_+(\widetilde W^{1, \ell\geq 3})$. Noticing that it also contains $h^1_0(\widetilde W^{\ell\geq 4})$, we deduce that it contains
$h^1_-(\widetilde W^{1, \ell\geq 3})$ and $h^1_+(\widetilde W^{0, \ell\geq 3})$.
But  for $c^0$ as above and $\omega\in \Omega$, we have $\zeta\cdot (0, c^0, 0)_{s_0\omega}=(0,c^0 , 0)_{s_0s_1 s_0\omega} + e_{\id^{-1}}  \cdot (-c^0\iota^{-1},  0, 0 )_{ s_1s_0\omega}=(0,c^0 , 0)_{s_0s_1 s_0\omega} + \zeta e_{\id^{-1}}  \cdot (-c^0\iota^{-1},  0, 0 )_{\omega}$ so
$(0,c^0 , 0)_{s_0s_1 s_0\omega}\in \zeta\cdot E^1$ and likewise we would obtain
$(0,c^0 , 0)_{s_1s_0 s_1\omega}\in \zeta\cdot E^1$.
\end{proof}

Using the anti-involution $\anti$, we would obtain the explicit right action of $H$ on $E^1$. For example, using
$(c^-,0,c^+)_1\cdot \zeta=\anti(\zeta\cdot \anti((c^-,0,c^+)_1))=\anti(\zeta\cdot (c^-,0,c^+)_1))$ we can compute:
\begin{align}   (c^-,0,c^+)_1\cdot \zeta
=&  (c^-,0,0)_{s_0 s_1 } +  (0,0,c^+)_{s_1 s_0 } \cr
&+ e_{\id^{-2}} (c^-,0,0)_{s_0 } + e_{\id^{2}}(0,0,c^+)_{s_1} + e_{\id^{-2}}(c^-, 0, 0)_1 +e_{\id^{2}}(0, 0, c^+)_1.
\label{f:1right}\end{align}
We give now further  partial results on the right action of $H$ on $E^1$.
\begin{lemma}\label{lemma:righteasy} Let $v,w\in \widetilde W$ such that  $\ell(w)\geq 1$ and  $(c^-,c^0,c^+)_v\in h^1(v)$.

\begin{itemize}
\item[i.] Suppose $\ell(v)+\ell(w)=\ell(vw)$. Then
$(c^-,c^0,c^+)_{v} \cdot  \tau_{w}=\begin{cases} (c^-,c^0 , 0)_{vw}&\textrm{ if  $vw\in  \widetilde  W^1$,}\cr (0,c^0 , c^+)_{vw}&\textrm{ if  $vw\in \widetilde W^0$.} \end{cases}$
\item[ii.] In the case when $v\in\{s_0, s_1\}$ and  $\ell(vw)=\ell(w)-1$ we have:
$$(0, c^0, 0)_{s_0}\cdot \tau_{w}=
-e_1\cdot (0, c^0, 0)_{w}- e_{\id^{-1}}\cdot (c^0\iota^{-1}, 0,0)_{w}$$
$$(0, c^0, 0)_{s_1}\cdot \tau_{w}= -e_1\cdot (0, c^0, 0)_{w}+e_{\id}\cdot (0,0,c^0\iota^{-1})_{w}$$

\end{itemize}


  \end{lemma}

%
%
%

\begin{proof}  i. Using \eqref{f:rightomegaE1}, we see that we may restrict the proof to the case when $v$ belongs to the set $\{(s_i s_{1-i})^n, s_1(s_i s_{1-i})^n : \: i = 0,1, \: n\geq 0\}$. We treat the case $v\in W^1$. First suppose
$v= (s_0s_1)^n$. Then, using Lemma \ref{lemma:antievenodd},  \eqref{f:lefts2} and Proposition \ref{prop:theformulas}:
\begin{align*}
  (c^-,c^0,c^+)_v \cdot  \tau_{s_0} & = \anti(\tau_{s_0^{-1}}\cdot (c^-,c^0,c^+)_{v^{-1}}) = \anti(\tau_{s_0}\cdot (c^-,c^0,c^+)_{s_0^2v^{-1}}) \\
   & =\anti((0,{-}c^0 , -c^-)_{s_0^{-1} v^{-1}}) = (c^-,c^0 , 0)_{ vs_0} \ .
\end{align*}
Next suppose that $v=s_0 (s_1s_0)^n$. Then
\begin{multline*}
  (c^-,c^0,c^+)_v  \cdot \tau_{s_1} = \anti(\tau_{s_1^{-1}} \cdot (-c^+,-c^0,-c^-)_{v^{-1}}) = \\
       \anti(\tau_{s_1} \cdot (-c^+,-c^0,-c^-)_{s_1^2v^{-1}}) = \anti((c^-,c^0 , 0)_{s_1^{-1} v^{-1}}) = (c^-, c^0, 0)_{ vs_1}.
\end{multline*}
This is enough to conclude the proof when $v \in \widetilde  W^1$ by induction on $\ell(w)$.

ii.  We treat the case $v=s_0$ and suppose first that $w={s_0}$. Then, using \eqref{f:antiodd}, Proposition \ref{prop:theformulas}, \eqref{f:lefts2} and \eqref{f:condodd}
\begin{align*}
(0, \c^0, 0)_{s_0}\cdot \tau_{s_0}& = - \anti(\tau_{s_0^{-1}}\cdot (0, \c^0, 0)_{s_0^{-1}}) = - \anti(\tau_{s_0}\cdot (0, \c^0, 0)_{s_0})   \\
                 & = - \anti((0, -c^0 , 0)_{ s_0}\cdot e_{1} + (0, 0 , c^0\iota^{-1})_{ s_0}\cdot e_{\id})  \\
                 & = - e_{1}\cdot  (0, c^0 , 0)_{ s_0^{-1}} - e_{\id^{-1}} \cdot  (-c^0\iota^{-1}, 0 ,0 )_{ s_0}^{-1} \\
                 & = - e_{1}\cdot  (0, c^0 , 0)_{ s_0} - e_{\id^{-1}} \cdot  (c^0\iota^{-1}, 0 ,0 )_{ s_0} \ .
\end{align*}
For $w= s_0\omega$ with $\omega\in\Omega$, apply $\tau_\omega$ on the right to the above formula and use \eqref{f:rightomegaE1}.
For general $w$ such that $\ell(s_0w)=\ell(w)-1$, apply $\tau_{s_0^{-1}w}$ on the right to the above formula and use Point i.


%
%
\end{proof}
The increasing filtration $(F_nE^1)_{n\geq 0}$ was defined in \S\ref{subsubsec:fil}.
\begin{lemma}
\label{lemma:calculsg}
If $\omega\in \Omega$, we have
$$\zeta\cdot   (c^-, 0, c^+)_\omega-(c^-, 0, c^+)_\omega\cdot \zeta\equiv  (0, 0, c^+)_{s_0s_1\omega}+ (c^-, 0, 0)_{s_1s_0\omega}- (0, 0, c^+)_{s_1s_0\omega} - (c^-, 0, 0)_{s_0s_1\omega}\bmod F_{1} E^1 \ .$$
 If $w\in \widetilde{W}^1$ of length $\geq 1$
$$\zeta  \cdot (c^-, c^0, c^+)_w-(c^-, c^0, c^+)_w\cdot \zeta\equiv  (0, 0, c^+)_{s_0s_1w}- (c^-, 0, 0)_{s_0s_1w}\bmod  F_{\ell(w)+1} E^1\ .$$
 If $w\in \widetilde{W}^0$ of length $\geq 1$
$$\zeta \cdot  (c^-, c^0, c^+)_w-(c^-, c^0, c^+)_w\cdot \zeta\equiv  (c^-, 0, 0)_{s_1s_0w}- (0, 0, c^+)_{s_1s_0w}\bmod  F_{\ell(w)+1} E^1\ .$$

\end{lemma}

\begin{proof}  We use Cor.\ \ref{coro:form-zeta}.
Recall from \eqref{f:rightomegaE1} that  $(c^-, c^0, c^+)_{w\omega}=  (c^-, c^0, c^+)_{w}\tau_\omega$ for $\omega\in \Omega $. So it is enough to  prove the lemma
for $\omega=1$ and  for $w$ of the form $(s_\epsilon s_{1-\epsilon})^n s_\epsilon$ or $(s_\epsilon s_{1-\epsilon})^n$ where $\epsilon\in\{0,1\}$.
By \eqref{f:1right} we have
$$
\zeta\cdot  (c^-, 0, c^+)_1 -(c^-, 0, c^+)_1 \cdot  \zeta \equiv  (0, 0, c^+)_{s_0s_1}+ (c^-, 0, 0)_{s_1s_0}- (c^-, 0, 0)_{s_0s_1}- (0, 0, c^+)_{s_1s_0} \bmod F_{1} E^1.
$$
Now  for  $w=(s_0s_1)^n$ with $n\geq 1$ we have
\begin{align*}
  \zeta \cdot (c^-, c^0, c^+)_w & \equiv  (0, c^0, c^+)_{s_0s_1w} \ \bmod F_{2n+1} E^1   \quad\textrm{ and } \\
  \zeta \cdot (c^-, c^0, c^+)_{w^{-1}} & \equiv  (c^-, c^0, 0)_{s_1s_0w^{-1}} \ \bmod F_{2n+1} E^1.
\end{align*}
Since $\anti$ preserves $F_{2n+1} E^1$ we have, using \eqref{f:antieven}:
\begin{multline*}
   (c^-, c^0, c^+)_w\cdot \zeta=\anti(\zeta \cdot(c^-, c^0, c^+)_{w^{-1}})\equiv \\
      \anti((c^-, c^0, 0)_{s_1s_0w^{-1}})\equiv  (c^-, c^0, 0)_{ws_0s_1}\equiv (c^-, c^0, 0)_{s_0s_1w}  \ \bmod  F_{2n+1} E^1
\end{multline*}
which gives the expected formula. Using $\anti$, we then obtain the expected result for $w=(s_1s_0)^n$. Likewise we treat the case $w=(s_0s_1)^ns_0$ with $n\geq 0$. We have
\begin{equation*}
  \zeta \cdot (c^-, c^0, c^+)_w\equiv  (0, c^0, c^+)_{s_0s_1w}\bmod F_{2n+2} E^1
\end{equation*}
and
\begin{multline*}
   (c^-, c^0, c^+)_w \cdot \zeta = \anti(\zeta\cdot (-c^+, -c^0, -c^-)_{w^{-1}})   \\
           \equiv \anti((0, -c^0, -c^-)_{s_0s_1w^{-1}})\equiv  (c^-, c^0, 0)_{ws_1s_0}\equiv  (c^-, c^0, 0)_{s_0s_1w}   \ \bmod  F_{2n+2} E^1
\end{multline*}
which gives the expected result for  $w=(s_0s_1)^ns_0$  and similarly we would treat the case $w=(s_1s_0)^ns_1$.
\end{proof}

\subsection{Sub-$H$-bimodules of $E^1$}

\subsubsection{The $H$-bimodule $F^1 H$}\label{subsubsec:F1-rel}

In this paragraph \ref{subsubsec:F1-rel}, there is no condition on $\mathfrak F$ (in fact we may even have $p=2$).

The elements $x_i := \tau_{s_i} \in F^1 H$ satisfy the relations:
\begin{itemize}
  \item[1)] $\tau_{s_i} x_i = - e_1 x_i = x_i \tau_{s_i}$ for $i \in \{0,1\}$;
  \item[2)] $\tau_\omega x_i = x_i \tau_{\omega^{-1}}$ for $i \in \{0,1\}$ and $\omega \in \Omega$;
  \item[3)] $\tau_{s_0} x_1 = x_0 \tau_{s_1}$ and $\tau_{s_1} x_0 = x_1 \tau_{s_0}$.
\end{itemize}

Given any $H$-bimodule $M$, a pair of elements $x_0, x_1 \in M$ which satisfy the relations 1) - 3) will be called a $F^1 H$-pairs in $M$. The $F^1 H$ pairs in $M$ form a $k$-vector subspace of $M \times M$.

\begin{example}
  For any $\ell \geq 0$ the elements $\tau_{s_0}(\tau_{s_1} \tau_{s_0})^\ell$ and $\tau_{s_1}(\tau_{s_0} \tau_{s_1})^\ell$ form an $F^1 H$-pair in $F^1 H$.
\end{example}

\begin{lemma}\phantomsection\label{lemma:defifx}
\begin{itemize}
\item[i.]  Given a $F^1H$-pair $(x_0, x_1)\in M\times M$, there is a unique  $H$-bimodule homomorphism $f_{(x_0,x_1)} : F^1 H \rightarrow M$ satisfying
$$
f_{(x_0,x_1)}(\tau_{s_0})=x_0, \quad\text{and}\quad f_{(x_0,x_1)}(\tau_{s_1})= x_1\ .
$$
\item[ii.] The map $f\mapsto (f(\tau_{s_0}),f(\tau_{s_1}))$ yields a bijection between the space of all
$H$-bimodule homomorphism $F^1 H \rightarrow M$ and the space of all $F^1 H$-pairs in $M$. The inverse map is given by $(x_0, x_1)\mapsto f_{(x_0,x_1)}$.
\end{itemize}
\end{lemma}
\begin{proof} As a right $H$-module, we have $F^1H=\tau_{s_0}H\oplus \tau_{s_1}H$ and  $\tau_{s_i}H\simeq H/(\tau_{s_i}+e_1) H$ for $i =0,1$.
Let $(x_0, x_1) \in M \times M$ satisfying $x_i(\tau_{s_i}+e_1)=0$  for $i =0,1$. There
is a unique  homomorphism of right $H$-modules
$$
f_{(x_0, x_1)}: F^1H\longrightarrow  M\textrm{ such that } f(\tau_{s_0})= x_0 \text{ and }   f(\tau_{s_1})= x_1 \ .
$$
We prove that $f$ is an homomorphism of $H$-bimodules if and only if $x_0, x_1$ is an $F^1 H$-pair in $M$. The direct implication is clear. Now suppose that $x_0, x_1 \in M$  satisfy the relations 1) - 3). Let $w \in \widetilde W$. We want to show that the maps $\tau \mapsto \tau_w \cdot f(\tau)$ and   $\tau\mapsto f(\tau_w\tau)$ are equal.  Since they are both homomorphisms of  right $H$-modules, it is enough to show that they coincide at $\tau_{s_i}$ for  $i =0,1$ namely that $ \tau_w x_i = f(\tau_w\tau_{s_i})$. We proceed by induction on $\ell(w)$. Using relations  2), it is easy to check that this equality  holds when $w$ has length $0$.
Now let $w\in \widetilde W$ with length $\geq 1$.
\begin{itemize}
\item[-] If $u := ws_{1-i}^{-1}$ has length $<\ell(w)$ we have:
\begin{align*}
\tau_w x_i & =\tau_{u} \tau_{s_{1-i}} x_i = \tau_{u}  x_{1-i} \tau_{s_i}    \quad\text{by 3)}   \\
    & = f(\tau_{u}  \tau_{s_{1-i}}) \tau_{s_i} = f(\tau_{u}  \tau_{s_{1-i} }\tau_{s_i}) =
            f(\tau_{w} \tau_{s_i})   \quad\text{by induction and then  right $H$-equivariance.}
\end{align*}
\item[-] Otherwise, $v := ws_i^{-1}$ has length $<\ell(w)$ and we have
\begin{align*}
    \tau_w x_i & = \tau_{v} \tau_{s_i} x_i = -\tau_{v}  x_i  e_1    \quad\text{by 1) and 2)}  \\
     & = -f(\tau_{v}  \tau_{s_i})  e_1     \quad\text{by induction}  \\
     & = f(-\tau_{v}  \tau_{s_i}  e_1) = f(\tau_{v}  \tau_{s_i}^2) = f(\tau_{w}  \tau_{s_i})    \quad\text{by  right $H$-equivariance.}
\end{align*}
\end{itemize}
\end{proof}

\begin{remark}\label{rema:pairzeta}
 For any $F^1 H$-pair $(x_0,x_1)$ in $M$ we have $\im(f_{(x_0,x_1)}) \subseteq \{m \in M : \zeta m = m \zeta\}$.
\end{remark}

\subsubsection{$F^1 H$-pairs in $E^1$}

In this paragraph we assume that $\mathfrak{F} = \mathbb{Q}_p$ with $p \geq 5$ and that $\pi=p$.

\begin{lemma}\label{F1-E1}
   The $F^1 H$-pairs $(x_0,x_1)$ in $E^1$ which are contained in $h^1(s_0) \oplus h^1(s_1) \oplus h^1(\Omega)$ are given by
\begin{equation*}
  x_0 := -(0,c^0, 0)_{s_0} - e_{\id^{-1}} \cdot (  c^0\iota^{-1}, 0,0)_1
   \quad\text{and}\quad
   x_1 := (0,c^0, 0)_{s_1}  -e_{\id} \cdot ( 0,  0, c^0\iota^{-1})_1
\end{equation*}
where $c^0$ runs over the $1$-dimensional $k$-vector space $\Hom(1+p\mathbb Z_p \big/ 1+p^2\mathbb Z_p ,k)$.
\end{lemma}
\begin{proof}
To check that the pairs $(x_0,x_1)$ in the assertion are indeed $F^1 H$-pairs is an explicit computation based on the formulas in Sections \ref{subsec:omega} and \ref{subsec:HonE1}.

As noted in Remark \ref{rema:pairzeta}, an element which satisfies the relations 1), 2) and 3) commutes with the action of $\zeta$. We  determine the elements in  $h^1(s_0) \oplus h^1(s_1) \oplus h^1(\Omega)$ which commute with the action of $\zeta$. Let $x$ be such an element. Since the elements in the assertion of the lemma do commute with the action of $\zeta$, we may assume that $x$ is of the form
$$
  x= (c_0^-,0,c_0^+)_{s_0} + (c_1^-,0,c_1^+)_{s_1} + \sum_{\omega \in \Omega} (c_\omega^-,0,c_\omega^+)_\omega \in h^1(s_0) \oplus h^1(s_1) \oplus h^1(\Omega) \ .
$$
By Lemma \ref{lemma:calculsg}, we know that
$$
  \zeta\cdot x- x\cdot \zeta \equiv (0, 0, c_0^+)_{s_0s_1s_0}- (c_0^-, 0, 0)_{s_0s_1s_0}+ (c_1^-, 0, 0)_{s_1s_0s_1}- (0, 0, c_1^+)_{s_1s_0s_1} \bmod  F_2 E^1 \ .
$$
Therefore we have $c_0^-=c_0^+=c_1^+= c_1^-=0$ and   $x=  \sum_{\omega \in \Omega} (c_\omega^-,0,c_\omega^+)_\omega \in  h^1(\Omega)$. By Lemma \ref{lemma:calculsg} again,
$$
  \zeta\cdot x- x\cdot \zeta \equiv \sum_{\omega \in \Omega} \Big(
   (0, 0, c_\omega ^+)_{s_0s_1\omega}+ (c_\omega^-, 0, 0)_{s_1s_0\omega}- (0, 0, c^+_\omega)_{s_1s_0\omega} - (c^-_\omega, 0, 0)_{s_0s_1\omega} \Big) \bmod  F_1E^1
$$
and therefore $x=0$. This proves that the only elements  in $E^1$ which are contained in $h^1(s_0) \oplus h^1(s_1) \oplus h^1(\Omega)$  and commute with the action of $\zeta$ are given by the formulas announced in the lemma. Therefore, these are also the only $F^1 H$-pairs $(x_0,x_1)$ in $h^1(s_0) \oplus h^1(s_1) \oplus h^1(\Omega)$.
\end{proof}

In the following we choose $\c^0 \in \Hom(1+p\mathbb Z_p \big/ 1+p^2\mathbb Z_p ,k) $ as in \S\ref{subsubsec:normalize} and let $(\mathbf{x}_0,\mathbf{x}_1)$ be the corresponding $F^1 H$-pair in $E^1$ of Lemma \ref{F1-E1}. Recall that the $H$-bimodule homomorphism $f_{(\mathbf{x}_0,\mathbf{x}_1)}$ was introduced in Lemma  \ref{lemma:defifx}.

\begin{proposition}\phantomsection\label{F1part}
\begin{itemize}
   \item[i.] For $\tau_w \in F^1 H$ we have
\begin{equation*}
    f_{(\mathbf{x}_0,\mathbf{x}_1)}(\tau_w) =
\begin{cases}
    (0,\c^0, 0)_w & \text{if $w\in\widetilde W^0$ and $\ell(w)\geq 2$},  \\
     -(0,\c^0, 0)_w & \text{if $w\in\widetilde W^1$ and $\ell(w)\geq 2$},   \\
    (0,\c^0, 0)_{s_1\omega}-e_{\id} \cdot ( 0,  0, \c^0\iota^{-1}) _{\omega} & \text{if $w=s_1\omega\in s_1\Omega$},   \\
     -(0,\c^0, 0)_{s_0\omega}-e_{\id^{-1}} \cdot (  \c^0\iota^{-1}, 0,0) _{\omega} & \text{if $w=s_0\omega\in s_0\Omega$.}   \\
\end{cases}
\end{equation*}
   \item[ii.] The $H$-bimodule homomorphism $f_{(\mathbf{x}_0,\mathbf{x}_1)} : F^1 H \longrightarrow E^1$ is injective.
   \item[iii.] The image of $f_{(\mathbf{x}_0,\mathbf{x}_1)}$ is contained in the centralizer of $\zeta$.
   \item[iv.] $\anti \circ f_{(\mathbf{x}_0,\mathbf{x}_1)} = - f_{(\mathbf{x}_0,\mathbf{x}_1)} \circ \anti$.
   \item[v.] $\Gamma_\varpi \circ f_{(\mathbf{x}_0,\mathbf{x}_1)} (\tau_w) = f_{(\mathbf{x}_0,\mathbf{x}_1)} (\tau_{\varpi w \varpi^{-1}})$ for any $\tau_w \in F^1 H$.
\end{itemize}
\end{proposition}
\begin{proof}
i. For $\omega \in \Omega$ we have by definition that $f_{(\mathbf{x}_0,\mathbf{x}_1)}(\tau_{s_i \omega}) = \mathbf{x}_i \tau_\omega$. Hence the last two equalities follow directly from \eqref{f:rightomegaE1}.

For the first two equalities we first consider the cases $w = s_0 s_1$ and $w = s_1 s_0$. By the left $H$-equivariance of $f_{(\mathbf{x}_0,\mathbf{x}_1)}$ we have
\begin{equation*}
    f_{(\mathbf{x}_0,\mathbf{x}_1)}(\tau_w) =
   \begin{cases}
   \tau_{s_0}\cdot  \mathbf{x}_1 & \text{if $w = s_0 s_1$},  \\
   \tau_{s_1}\cdot   \mathbf{x}_0 & \text{if $w = s_1 s_0$}.
   \end{cases}
\end{equation*}
Using Prop.\ \ref{prop:theformulas} one easily checks that $\tau_{s_0} \cdot   \mathbf{x}_1 = -(0,\c^0,0)_w$ and $\tau_{s_1} \cdot   \mathbf{x}_0 = (0,\c^0,0)_w$. The assertion for a general $w$ follows from this by using again the left $H$-equivariance together with the following general observation. For any $v, w \in \widetilde{W}$ such that $\ell(v) + \ell(w) = \ell(vw)$ and $\ell(w) \geq 1$ we have,  by \eqref{f:leftomega} and Prop.\ \ref{prop:theformulas}:
\begin{equation*}
  \tau_v\cdot  (0,c^0,0)_w = (0,(-1)^{\ell(v)} c^0,0)_{vw} \ .
\end{equation*}

ii. It is immediate from i. that the set $\{f_{(\mathbf{x}_0,\mathbf{x}_1)}(\tau_w)\}_{w \in F^1 H}$ is a $k$-basis of $\im(f_{(\mathbf{x}_0,\mathbf{x}_1)})$.

iii. This is obvious, as noted in Remark \ref{rema:pairzeta}.

iv. We first check that $\anti (\mathbf{x}_i) = - \tau_{s_i^2} \cdot \mathbf{x}_i$ holds true. The case $i=1$ being analogous we only compute
\begin{align*}
  \anti (\mathbf{x}_0) & = - \anti ((0,\c^0,0)_{s_0}) - \anti ((\c^0 \iota^{-1},0,0)_1) \anti (e_{\id^{-1}})   \\
      & = (0,\c^0,0)_{s_0^{-1}} - (\c^0 \iota^{-1},0,0)_1 \cdot e_{\id}   \quad\text{by Lemma \ref{lemma:antievenodd}}   \\
      & = (0,\c^0,0)_{s_0^2 s_0} - e_{\id^{-1}} \cdot (\c^0 \iota^{-1},0,0)_1  \quad\text{by \eqref{f:leftomega} and \eqref{f:rightomegaE1}}   \\
      & = \tau_{s_0^2}\cdot (0,\c^0,0)_{s_0} - e_{\id^{-1}}\cdot (\c^0 \iota^{-1},0,0)_1   \quad\text{by \eqref{f:lefts2}}   \\
      & = \tau_{s_0^2} \cdot(0,\c^0,0)_{s_0} + \tau_{s_0^2} e_{\id^{-1}}\cdot (\c^0 \iota^{-1},0,0)_1   \quad\text{by $- e_{\id^{-1}} = \tau_{s_0^2} \cdot e_{\id^{-1}}$}  \\
      & = - \tau_{s_0^2} \cdot \mathbf{x}_0 \ .
\end{align*}
For a general $w \in \widetilde{W}^{1-i,\ell \geq 1}$  we have $\tau_{w}=\tau_{s_i}\tau_{s_i^{-1}w}$ and  we  deduce that
\begin{align*}
  \anti (f_{(\mathbf{x}_0,\mathbf{x}_1)}(\tau_w)) & = \anti (\mathbf{x}_i \cdot \tau_{s_i^{-1} w}) = \anti(\tau_{s_i^{-1} w})  \cdot \anti (x_i) = - \tau_{w^{-1} s_i} \tau_{s_i^2}  \cdot \mathbf{x}_i = - \tau_{w^{-1} s_i^{-1}}  \cdot \mathbf{x}_i  \\
  & = - f_{(\mathbf{x}_0,\mathbf{x}_1)}(\tau_{w^{-1}}) = - f_{(\mathbf{x}_0,\mathbf{x}_1)}(\anti (\tau_w))
\end{align*}
using left $H$-equivariance in the fifth equality.

v. Lemma \ref{lemma:conjtrip} easily implies that $\Gamma_\varpi(\mathbf{x}_i) = \mathbf{x}_{1-i}$. For a general $w \in \widetilde{W}^{1-i,\ell \geq 1}$ we have $\varpi w \varpi^{-1} \in \widetilde{W}^{i,\ell \geq 1}$ and we deduce that
\begin{align*}
  \Gamma_\varpi (f_{(\mathbf{x}_0,\mathbf{x}_1)} (\tau_w)) & = \Gamma_\varpi(\mathbf{x}_i \cdot  \tau_{s_i^{-1} w}) = \Gamma_\varpi(\mathbf{x}_i)  \cdot  \Gamma_\varpi(\tau_{s_i^{-1} w}) = \mathbf{x}_{1-i}\cdot  \tau_{\varpi s_i^{-1} w \varpi^{-1}}  \\
  & = \mathbf{x}_{1-i} \cdot \tau_{s_{1-i}^{-1} \varpi w \varpi^{-1}} = f_{(\mathbf{x}_0,\mathbf{x}_1)} (\tau_{\varpi w \varpi^{-1}})
\end{align*}
using in the second equality that $\Gamma_\varpi$ is multiplicative (cf.\ \S\ref{subsubsec:auto-pair}).
\end{proof}

In Prop.\ \ref{prop:isokerg} we will see that the inclusion in part ii. of the above proposition, in fact, is an equality. This, in particular, shows that there are no nonzero $F^1 H$-pairs in $E^1 \setminus \im(f_{(\mathbf{x}_0,\mathbf{x}_1)})$.

\begin{remark}\label{rema:1-gamma0F1}
   Recalling that $e_{\gamma_0}$ was introduced in \eqref{f:gamma0} we have
\begin{equation*}
  (1-e_{\gamma_0})\cdot  \im(f_{(\mathbf{x}_0,\mathbf{x}_1)}) = \im(f_{(\mathbf{x}_0,\mathbf{x}_1)})\cdot (1-e_{\gamma_0}) = (1-e_{\gamma_0}) \cdot \im(f_{(\mathbf{x}_0,\mathbf{x}_1)}) \cdot (1-e_{\gamma_0}) = (1-e_{\gamma_0}) \cdot h_0^1(\widetilde{W}) \ .
\end{equation*}
\end{remark}
\begin{proof}
Since $1-e_{\gamma_0}$ is central in $H$ it also must centralize $\im(f_{(\mathbf{x}_0,\mathbf{x}_1)})$. Recall that $e_{\gamma_0} = e_{\id} + e_{\id^{-1}}$. The last equality then is immediate from Prop.\ \ref{F1part}-i.
\end{proof}

\subsubsection{An $H_\zeta$-bimodule inside $E^1$}\label{subsubsec:Hzeta-inside}

\paragraph{A left $H_\zeta$-bimodule inside $E^1$}

Let $M$ be any $H$-bimodule. To give a homomorphism of left (or right) $H$-modules $f : H \longrightarrow M$ simply means to give any element $x \in M$ as the image $x = f(1)$. We state a simple sufficient condition on $x$ such that the corresponding $f$ extends to the localization $H_\zeta$.

\begin{lemma}\label{lemma:extension}
   Let $x \in M$ be such that $\zeta \cdot x\cdot  \zeta = x$. Then
\begin{align*}
  H_\zeta & \longrightarrow M \\
  \zeta^{-i}\tau & \longmapsto f_x(\zeta^{-i}\tau) := \tau \cdot x \cdot  \zeta^i,\ \text{resp.\ ${_x f}(\zeta^{-i}\tau) := \zeta^i\cdot  x \cdot  \tau$, for $i \geq 0$ and $\tau \in H$,}
\end{align*}
is a well defined homomorphism of left, resp.\ right, $H$-modules; its image is contained in the space  $\{y \in M : \zeta \cdot y\cdot  \zeta = y\}$.
\end{lemma}
\begin{proof}
Easy exercise.
\end{proof}
Assume that $\mathfrak{F} = \mathbb{Q}_p$ with $p \geq 5$ and that $\pi=p$.
We will apply the above lemma to the bimodule $E^1$.

\begin{lemma}\label{lemma:offdiag-E1}
The elements $x \in E^1$ which satisfy $\zeta\cdot  x \cdot  \zeta = x$ and lie in $h^1(1) \oplus e_{\id}h^1(s_0) \oplus e_{\id}h^1(s_1 s_0)$ with $\tau_{s_0} \cdot  x = 0$, resp.\ in $h^1(1) \oplus e_{\id^{-1}} h^1(s_1) \oplus e_{\id^{-1}} h^1(s_0 s_1)$ with $\tau_{s_1} \cdot  x = 0$, are
\begin{align*}
   & x^+ := (0,0,c^+)_1 - e_{\id} \cdot (0,2c^+ \iota,0)_{s_0} - e_{\id} \cdot (0,0,c^+)_{s_1 s_0} , \text{resp.} \\
   & x^- := (c^-,0,0)_1 + e_{\id^{-1}} \cdot (0,2c^- \iota,0)_{s_1} - e_{\id^{-1}} \cdot (c^-,0,0)_{s_0 s_1} ,
\end{align*}
where $c^+$ and $c^-$ run over the $1$-dimensional $k$-vector space $\Hom(\mathbb Z_p \big/ p\mathbb Z_p ,k)$.
\end{lemma}
\begin{proof}
We treat the first case, the other one being analogous. Consider any
\begin{equation*}
  x = (c^-,0,c^+)_1 + e_{\id} \cdot (b^-,b^0,b^+)_{s_0} + e_{\id} \cdot  (d^-,d^0,d^+)_{s_1 s_0}
\end{equation*}
such that $\tau_{s_0}\cdot  x = 0$. Using Prop.\ \ref{prop:theformulas} we compute
\begin{align*}
  0 & = \tau_{s_0}\cdot  x = \tau_{s_0} \cdot  (c^-,0,c^+)_1 + e_{\id^{-1}} \tau_{s_0}\cdot  (b^-,b^0,b^+)_{s_0} + e_{\id^{-1}} \tau_{s_0} \cdot (d^-,d^0,d^+)_{s_1 s_0}  \\
  & = (0,0,-c^-)_{s_0} + e_{\id^{-1}} \cdot \big(- e_1\cdot  (b^-, b^0, b^+)_{s_0} + e_{\id}\cdot (0,-2b^- \iota, b^0 \iota^{-1})_{s_0} + e_{\id^2} \cdot  (0,0,b^-)_{s_0} - (0,0,b^-)_{s_0^2} \big)   \\
     & \qquad - e_{\id^{-1}} \cdot  (0,d^0, d^-)_{s_0 s_1 s_0}    \\
  & = (0,0,-c^-)_{s_0} - e_{\id^{-1}} \cdot  (0,0,b^-)_{s_0^2} - e_{\id^{-1}} \cdot  (0,d^0, d^-)_{s_0 s_1 s_0}  \ .
\end{align*}
It follows that $c^- = b^- = d^0 = d^- = 0$ and hence that
\begin{equation}\label{f:x-intermediate}
  x = (0,0,c^+)_1 + e_{\id}\cdot  (0,b^0,b^+)_{s_0} + e_{\id} \cdot (0,0,d^+)_{s_1 s_0} \ .
\end{equation}
Now we assume in addition that $\zeta\cdot  x \cdot \zeta = x$. From Cor.\ \ref{coro:form-zeta} we deduce that
\begin{align*}
  \zeta \cdot  x & = \zeta \cdot  (0,0,c^+)_1 + e_{\id} \zeta \cdot (0,b^0,b^+)_{s_0} + e_{\id} \zeta \cdot (0,0,d^+)_{s_1 s_0}  \\
     & = (0,0,c^+)_{s_0 s_1} - e_1 \cdot  (c^+,0,0)_{s_1} + e_1 \cdot  (0,0,c^+)_1 + e_{\id} \cdot  (0,b^0,b^+)_{s_0 s_1 s_0}   \\
     &  \quad - e_{\id} \cdot  (0,2d^+ \iota,0)_{s_0 s_1 s_0} + e_{\id} \cdot (0,2d^+ \iota,0)_{s_1^2 s_0} + e_{\id} \cdot (0,0,d^+)_1 \ .
\end{align*}
Using Lemma \ref{lemma:antievenodd}, Cor.\ \ref{coro:form-zeta}, Section \ref{subsec:omega}, and \eqref{f:id-id} we compute
\begin{align}\label{f:Vzeta}
  (0,0,c^+)_{s_0 s_1} \cdot \zeta & = - e_{\id} \cdot (0,2c^+ \iota,0)_{s_0 s_1 s_0} - e_{\id} \cdot (0,2c^+ \iota,0)_{s_0} + e_1 \cdot  (c^+,0,0)_{s_0} + (0,0,c^+)_1    \\
   - e_1 \cdot (c^+,0,0)_{s_1} \cdot \zeta & = - e_1 \cdot (0,0,c^+)_{s_1 s_0} - e_1\cdot  (c^+,0,0)_{s_0}   \nonumber  \\
   e_{\id}\cdot  (0,2d^+ \iota,0)_{s_1^2 s_0}\cdot  \zeta & = - e_{\id}\cdot  (0,2d^+ \iota,0)_{s_0 s_1 s_0}  \nonumber   \\
   e_1\cdot (0,0,c^+)_1 \cdot \zeta & = e_1 \cdot (0,0,c^+)_{s_1 s_0}    \nonumber   \\
   e_{\id}\cdot  (0,0,d^+)_1 \cdot \zeta & = e_{\id}\cdot  (0,0,d^+)_{s_1 s_0}     \nonumber    \\
   e_{\id} \cdot (0,b^0,b^+)_{s_0 s_1 s_0}\cdot  \zeta & = - e_{\id} \cdot (0,b^0,0)_{(s_1 s_0)^2 s_0^2} + e_{\id}\cdot  (0, 2b^+ \iota,0)_{(s_0 s_1)^2} + e_{\id} \cdot (0,0,b^+)_{s_0}  \nonumber   \\
   - e_{\id}\cdot  (0,2d^+ \iota,0)_{s_0 s_1 s_0} \cdot \zeta & = e_{\id}\cdot  (0,2d^+ \iota,0)_{(s_1 s_0)^2 s_0^2}  \ .   \nonumber
\end{align}
Comparing the sum of these equations with \eqref{f:x-intermediate} shows that $d^+ = - c^+$, $b^0 = - 2c^+ \iota$, and $b^+ = 0$. We conclude that $x = x^+$.
\end{proof}

We now choose $c^+ := c^- := \c \in \Hom(\mathbb Z_p \big/ p\mathbb Z_p ,k) $ as in \S\ref{subsubsec:normalize} and let $(\mathbf{x}^+,\mathbf{x}^-)$ be the corresponding elements of Lemma \ref{lemma:offdiag-E1}. By Lemma \ref{lemma:extension} they give rise to the left $H$-module homomorphisms
\begin{equation}\label{f:defifpm}
   f_{\mathbf{x}^\pm} : H_\zeta \longrightarrow E^1 \ .
\end{equation}

\begin{remark}\phantomsection\label{rem:offdiag-Gamma-anti}
\begin{itemize}
\item[1.] We have $\Gamma_\varpi(\zeta) = \zeta$. Hence $\Gamma_\varpi$ extends to an automorphism of $H_\zeta$. The multiplicativity of $\Gamma_\varpi$, the formula $\Gamma_\varpi(e_\lambda) = e_{\lambda^{-1}}$, and Lemma \ref{lemma:conjtrip} then imply that
\begin{equation*}
  \Gamma_\varpi \circ f_{{\bf x}^+} = f_{{\bf x}^-} \circ \Gamma_\varpi  \quad\text{and}\quad  \Gamma_\varpi \circ f_{{\bf x}^-} = f_{{\bf x}^+} \circ \Gamma_\varpi
\end{equation*}
and, in particular, $\Gamma_\varpi(\mathbf{x}^+) = \mathbf{x}^-$.
\item[2.]  Here and in the subsequent points let $x^-$ and $x^+$ be as in Lemma \ref{lemma:offdiag-E1}. We compute
\begin{align*}
  \anti (x^+) & = \anti ( (0,0,c^+)_1) - \anti ((0,2c^+ \iota,0)_{s_0}) \cdot  \anti (e_{\id}) - \anti ((0,0,c^+)_{s_1 s_0}) \cdot  \anti (e_{\id})    \\
   & = (0,0,c^+)_1 + (0,2c^+ \iota,0)_{s_0^{-1}} \cdot  e_{\id^{-1}} - (0,0,c^+)_{s_0 s_1} \cdot  e_{\id^{-1}} \quad\text{by Lemma \ref{lemma:antievenodd}}   \\
   & = (0,0,c^+)_1 + e_{\id} \cdot  (0,2c^+ \iota,0)_{s_0^{-1}} - e_{\id} \cdot (0,0,c^+)_{s_0 s_1} \quad\text{by \eqref{f:leftomega} and \eqref{f:rightomegaE1}} \\
   & = (0,0,c^+)_1 + e_{\id} \tau_{s_0^2} \cdot  (0,2c^+ \iota,0)_{s_0} - e_{\id} \cdot  (0,0,c^+)_{s_0 s_1}   \quad\text{by \eqref{f:lefts2}}   \\
      & = (0,0,c^+)_1 - e_{\id}\cdot  (0,2c^+ \iota,0)_{s_0} - e_{\id} \cdot  (0,0,c^+)_{s_0 s_1}   \quad\text{by $- e_{\id} \tau_{s_0^2} = e_{\id}$}  \\
      & = x^+ + e_{\id}\cdot  (0,0,c^+)_{s_1 s_0} - e_{\id}\cdot  (0,0,c^+)_{s_0 s_1}   \\
      & = (1 - e_{\id} - e_{\id} \tau_{s_0 s_1}) \cdot x^+   \quad\text{by Prop.\ \ref{prop:theformulas}}   \\
      & = (1 - e_{\id} - e_{\id} \zeta) \cdot  x^+  \ .
\end{align*}
and similarly
\begin{equation*}
  \anti (x^-) = (1 - e_{\id^{-1}} - e_{\id^{-1}} \zeta) \cdot  x^-  \ .
\end{equation*}
\end{itemize}
\end{remark}

\begin{lemma}
\phantomsection
\begin{itemize}
\item[1.]   For any $u \in \mathbb{F}_p^\times$ we have $  x^+ \cdot  \tau_{\omega_u} =u^{-2} \tau_{\omega_u} \cdot  x^+$ and $x^- \cdot \tau_{\omega_u} = u^2 \tau_{\omega_u}\cdot  x^-$.

\item[2.]  We have $x^+ \cdot  \tau_{s_0} = \tau_{s_0}\cdot  x^+=0$ and $x^- \cdot  \tau_{s_1} = \tau_{s_1}\cdot  x^-=0$.
\item[3.]
We have \begin{align*}
  {x}^-\cdot \upiota (\tau_{s_1}) & =- e_{\id^{-2}} \cdot x^-    \quad\text{and}  \\
  {x}^+ \cdot \upiota (\tau_{s_0}) & =  - e_{\id^{2}} \cdot x^+ \ .
\end{align*} while, for $\bf x^+$ and $\bf x^-$ as above,
\begin{align*}
  \mathbf{x}^+\cdot \upiota (\tau_{s_1}) & =  -  \tau_{\omega_{-1}} \upiota( \tau_{s_0})\cdot \mathbf{x}^-  \cdot \zeta   \quad\text{and}  \\
  \mathbf{x}^- \cdot \upiota (\tau_{s_0}) & = - \tau_{\omega_{-1}} \upiota(\tau_{s_1})\cdot  \mathbf{x}^+\cdot  \zeta \ .
\end{align*}
where we recall that the involution $\upiota$ was introduced in \eqref{f:upiotadef}.

\end{itemize}
\label{lemma:leftright}
\end{lemma}
\begin{proof}
1. For any $u \in \mathbb{F}_p^\times$ we compute using \eqref{f:leftomega} and \eqref{f:rightomegaE1}
\begin{align*}
  x^+ \cdot  \tau_{\omega_u} & = (0,0,c^+)_1 \cdot  \tau_{\omega_u} - e_{\id}\cdot  (0,2c^+ \iota,0)_{s_0} \cdot \tau_{\omega_u} - e_{\id} \cdot (0,0,c^+)_{s_1 s_0}\cdot  \tau_{\omega_u} \\
   & = (0,0,c^+)_{\omega_u} - e_{\id}\cdot  (0,2c^+ \iota,0)_{\omega_u^{-1}s_0} - e_{\id} \cdot (0,0,c^+)_{\omega_u s_1 s_0}   \\
   & = u^{-2} \tau_{\omega_u} \cdot  (0,0,c^+)_1 - e_{\id} \tau_{\omega_u^{-1}}\cdot  (0,2c^+ \iota,0)_{s_0} - u^{-2} e_{\id} \tau_{\omega_u} \cdot (0,0,c^+)_{s_1 s_0}   \\
   & = u^{-2} \tau_{\omega_u} \cdot \big( (0,0,c^+)_1 - e_{\id} \tau_{\omega_u^{-2}} \cdot (0,2c^+ \iota,0)_{s_0} - e_{\id} (0,0,c^+)_{s_1 s_0}  \big)   \\
   & = u^{-2} \tau_{\omega_u} \cdot  x^+
\end{align*}
and, by an analogous computation (or by applying Remark \ref{rem:offdiag-Gamma-anti}-1), we obtain  $x^- \cdot \tau_{\omega_u} = u^2 \tau_{\omega_u}\cdot  x^-$.\\
2. For the  identity   $  \tau_{s_0}\cdot x^+=0$, see Lemma \ref{lemma:offdiag-E1}. Now we compute:
\begin{align*}
  x^+ \cdot  \tau_{s_0} & = \anti (\anti (\tau_{s_0}) \cdot \anti (x^+)) = \anti (\tau_{s_0^2} \tau_{s_0}\cdot  (x^+ + e_{\id} (0,0,c^+)_{s_1 s_0} - e_{\id} \cdot (0,0,c^+)_{s_0 s_1})) \quad\text{by Remark \ref{rem:offdiag-Gamma-anti}-2} \\
    & = \anti (\tau_{s_0^2} ( \tau_{s_0} \cdot  x^+ + e_{\id^{-1}} \tau_{s_0} \cdot  (0,0,c^+)_{s_1 s_0} - e_{\id^{-1}} \tau_{s_0}\cdot  (0,0,c^+)_{s_0 s_1}))    \\
    & = 0  \quad\text{by Lemma \ref{lemma:offdiag-E1} and Prop.\ \ref{prop:theformulas}}
\end{align*} We obtain the analogous statements for $x^-$ using Remark \ref{rem:offdiag-Gamma-anti}-1.\\
3. The first identities easily come from Points 1 and 2. We treat the second equation of the last statement. The first one can either be established by an analogous computation or by applying Remark \ref{rem:offdiag-Gamma-anti}-1 to the second equation. Both sides of the second equation lie in the sub-$H$-bimodule $\ker(\zeta\cdot  \id_{E^1} \cdot \zeta - \id_{E^1})$ of $E^1$ on which left multiplication by $\zeta$ is injective. Hence we may instead check the equation
\begin{equation*}
 -  \zeta\cdot {\bf x}^- \cdot (\tau_{s_0}+e_1)  =  \tau_{\omega_{-1}}\cdot  ( \tau_{s_1} +e_1)\cdot  {\bf x}^+ \ .
   \end{equation*}
For the left hand side we first have, using Lemma \ref{lemma:righteasy} and Point 1:
\begin{align*}
   {x}^- \cdot (\tau_{s_0} +e_1)&= (   {c^-},0,0)_{s_0} + e_{\id^{-1}}\cdot (0,2   {c^-}\iota,0)_{s_1 s_0} - e_{\id^{-1}} \cdot  (   {c^-},0,0)_{s_0 s_1 s_0}+e_{\id^{-2}}\cdot   {x^-}\cr&= (   {c^-},0,0)_{s_0} + e_{\id^{-1}}\cdot (0,2   {c^-}\iota,0)_{s_1 s_0} - e_{\id^{-1}} \cdot  (   {c^-},0,0)_{s_0 s_1 s_0}+e_{\id^{-2}}\cdot  (   {c^-},0,0)_{1}
  \end{align*}
and then by Cor.\ \ref{coro:form-zeta}
\begin{align*}
  - \zeta \cdot   {x}^-\cdot  (\tau_{s_0}+e_1) & = e_{\id^{-2}} \cdot (  {c^-},0,0)_{s_1 s_0} - (  {c^-},0,0)_{s_1 \omega_{-1}}  \\
  & \qquad\qquad\qquad + e_{\id^{-1}} \cdot (  {c^-},0,0)_{s_0} + e_1 \cdot (0,0,  {c^-})_1  - e_{\id^{-2}}\cdot  (  {c^-},0,0)_{s_1s_0}  \\
  & = - (  {c^-},0,0)_{s_1 \omega_{-1}} + e_{\id^{-1}} \cdot (  {c^-},0,0)_{s_0} + e_1 \cdot (0,0,  {c^-})_1 \ .
\end{align*}
For the right hand side we first compute using Prop.\ \ref{prop:theformulas}
\begin{align*}
  \tau_{s_1 s_0^2} \cdot  {x}^+ & = - ( {c^+},0,0)_{s_1 \omega_{-1}} + e_{\id^{-1}} \cdot ( {c^+},0,0)_{s_0}   \\
  e_1\cdot   {x}^+ & = e_1 \cdot (0,0, {c^+})_1
\end{align*}
and then see, by adding up, that it coincides with the above computation for the left hand side when $c^+=c^-=\mathbf c$.
\end{proof}

\begin{lemma}\label{lemma:fpm}
The maps $f_{{\bf x}^+}$ and $f_{{\bf x}^-}$  defined in \eqref{f:defifpm} induce an  injective homomorphism of left $H$-modules
\begin{equation*}
 H_\zeta / H_\zeta \tau_{s_0} \oplus H_\zeta / H_\zeta \tau_{s_1} \xrightarrow{\; f^\pm:=f_{{\bf x}^+} + f_{{\bf x}^-} \;} E^1
\end{equation*}
the image of which  is contained in the kernel of the endomorphism $\zeta \cdot\id_{E^1}\cdot  \zeta - \id_{E^1}$.
\end{lemma}
\begin{proof}
By Lemma \ref{lemma:leftright}-2,  the map is well defined.  By definition of ${\bf x}^+$and ${\bf x}^-$,
the last statement of the lemma is clear.
We prove that the map is the injective. We first observe that it suffices to check the injectivity of the restriction of $f^\pm$ to $H/H\tau_{s_0} \oplus H/H\tau_{s_1}$. The elements $\tau_w$ with $w \in \widetilde{W}$ such that $\ell(w s_0) = \ell(w) + 1$ from a $k$-basis of $H/H\tau_{s_0}$; they are of the form $w = \omega(s_0 s_1)^m$ or $= \omega s_1 (s_0 s_1)^m$ with $m \geq 0$ and $\omega \in \Omega$. Using \eqref{f:leftomega} and Prop.\ \ref{prop:theformulas} we obtain
\begin{equation*}
  \tau_w \cdot  (0,0,c^+)_1 \in
  \begin{cases}
  \mathbb{F}_p^\times (0,0,c^+)_w  & \text{if $w = \omega(s_0 s_1)^m$},  \\
  \mathbb{F}_p^\times (c^+,0,0)_w  & \text{if $w = \omega s_1(s_0 s_1)^m$},
  \end{cases}
\end{equation*}
and
\begin{equation*}
  \tau_w\cdot  (0,c^0,0)_{s_0} = (0,(-1)^{\ell(w)} c^0,0)_{w s_0} \in h_0^1(\widetilde{W})  \ \ \text{for any $w$ as above},
\end{equation*}
and
\begin{equation*}
  \tau_w e_{\id}\cdot  (0,0,c^+)_{s_1 s_0} \in
    \begin{cases}
    F_{\ell(w)-2} E^1 + h_0^1(\widetilde{W})  & \text{for any $w$ as above with $m \geq 1$},  \\
    \mathbb{F}_p^\times e_{\id^{-1}} \cdot (c^+,0,0)_{s_0} + h_0^1(\widetilde{W})  & \text{if $w = \omega s_1$},  \\
    \mathbb{F}_p^\times e_{\id}\cdot  (0,0,c^+)_{\omega s_1 s_0}   & \text{if $w = \omega$}.
    \end{cases}
\end{equation*}
It follows that
\begin{equation}\label{f:list0}
  \tau_w \cdot \mathbf{x}^+ \in\begin{cases}
  k^\times (0,0,\mathbf{c})_w + F_{\ell(w)-2} E^1 + h_0^1(\widetilde{W})   & \text{if $w = \omega(s_0 s_1)^m$ with $m \geq 1$},  \\
  k^\times (\mathbf{c},0,0)_w + F_{\ell(w)-2} E^1 + h_0^1(\widetilde{W})   & \text{if $w = \omega s_1(s_0 s_1)^m$ with $m \geq 1$},  \\
  k^\times (\mathbf{c},0,0)_w + k^\times e_{\id^{-1}} (\mathbf{c},0,0)_{s_0} + h_0^1(\widetilde{W})   & \text{if $w = \omega s_1$},   \\
  k^\times (0,0,\mathbf{c})_w + k^\times e_{\id} (0,0,\mathbf{c})_{\omega s_1 s_0} + h_0^1(\widetilde{W})   & \text{if $w = \omega$}.
  \end{cases}
\end{equation}
Similarly the elements $\tau_w$ with $w \in \widetilde{W}$ such that $\ell(w s_1) = \ell(w) + 1$ form a $k$-basis of $H/H\tau_{s_1}$; they are of the form $w = \omega (s_1 s_0)^m$ or $= \omega s_0 (s_1 s_0)^m$ with $m \geq 0$ and $\omega \in \Omega$. In this case we obtain
\begin{equation}\label{f:list1}
  \tau_w \cdot \mathbf{x}^- \in\begin{cases}
  k^\times (\mathbf{c},0,0)_w + F_{\ell(w)-2} E^1 + h_0^1(\widetilde{W})   & \text{if $w = \omega(s_1 s_0)^m$ with $m \geq 1$},  \\
  k^\times (0,0,\mathbf{c})_w + F_{\ell(w)-2} E^1 + h_0^1(\widetilde{W})   & \text{if $w = \omega s_0(s_1 s_0)^m$ with $m \geq 1$},  \\
  k^\times (0,0,\mathbf{c})_w + k^\times e_{\id}\cdot  (0,0,\mathbf{c})_{s_1} + h_0^1(\widetilde{W})   & \text{if $w = \omega s_0$},   \\
  k^\times (\mathbf{c},0,0)_w + k^\times e_{\id^{-1}} \cdot (\mathbf{c},0,0)_{\omega s_0 s_1} + h_0^1(\widetilde{W})   & \text{if $w = \omega$}.
  \end{cases}
\end{equation}
By comparing the lists \eqref{f:list0} and \eqref{f:list1} we easily see that the elements
\begin{equation*}
  \{\tau_w\cdot  \mathbf{x}^+ : \ell(w s_0) = \ell(w) + 1\} \cup \{\tau_w \cdot \mathbf{x}^- : \ell(w s_1) = \ell(w) + 1\}
\end{equation*}
in $E^1$ are $k$-linearly independent even in $E^1/h_0^1(\widetilde{W})$.
\end{proof}

\paragraph{Structure of $H_\zeta$-bimodule on $ H_\zeta / H_\zeta \tau_{s_0} \oplus H_\zeta / H_\zeta \tau_{s_1} $.}
In this paragraph, the only condition on $\mathfrak F$ is that it has residue field $\mathbb F_p$. Recall the involution $\upiota$ of $H$ defined in  \eqref{f:upiotadef}.

We consider the
 homomorphism of $k$-algebras $\kappa : H \rightarrow H_\zeta$ given by the composition of the involution $\upiota: H\rightarrow H$ and the inclusion $H\rightarrow H_\zeta$, the  element $-\tau_{\omega_{-1}}\zeta^{-1} \in Z(R)$ in the center of $H_\zeta$ and the character $\mu : \Omega \rightarrow k^\times, \omega_u\mapsto u^2$. Recall that as in Remark \ref{tll}, we may refer to the  idempotent corresponding to the latter as $e_{\id^2}$ instead of $e_\mu$.  As in  \S\ref{subsubsec:certain}, this yields a homomorphism of $k$-algebras  $\kappa_2 : H \rightarrow M_2(H_\zeta)$ and
   an $(H_\zeta,H)$-bimodule structure on $H_\zeta\oplus H_\zeta$ denoted by $(H_\zeta\oplus H_\zeta)[\kappa, -\tau_{\omega_{-1}}\zeta^{-1}, \mu]$ where $h\in H$ acts on $(\sigma^+,\sigma^-) \in H_\zeta\oplus H_\zeta$ via
$$((\sigma^+,\sigma^-) , h)\longmapsto (\sigma^+, \sigma^-)\kappa_2(h)\ .$$ We consider the composite map
$\kappa_2\circ \upiota^{-1}$. Again, it is a homomorphism of algebras  $H \rightarrow M_2(H_\zeta)$ and  it yields
   an $(H_\zeta,H)$-bimodule structure on $H_\zeta\oplus H_\zeta$ denoted by $(H_\zeta\oplus H_\zeta)^{\pm}$. We spell out below the action on $(\sigma^+,\sigma^-) \in H_\zeta\oplus H_\zeta$ of the generators $\upiota(\tau_{s_0}),\, \upiota(\tau_{s_1}), \, \tau_{\omega_u}$ for $u\in \mathbb F_p^\times$ of $H$
    \begin{align}\label{f:def-R}
(\sigma^+,\sigma^-) \upiota(\tau_{s_0}) & := (-\sigma^+ e_{\id^2}- \sigma^-\tau_{\omega_{-1}}\upiota(\tau_{s_1})\zeta^{-1}, 0)
    \nonumber  \\
  (\sigma^+,\sigma^-) \upiota(\tau_{s_1})& := (0, -\sigma^- e_{\id^{-2}}-\sigma^+\tau_{\omega_{-1}}\upiota(\tau_{s_0})\zeta^{-1})   \\
(\sigma^+,\sigma^-)  \tau_{\omega_u}& := (u^{-2} \sigma^+ \tau_{\omega_u}, u^2 \sigma^- \tau_{\omega_u}) \ .    \nonumber
\end{align}
One easily checks that \\$(\tau_{s_{0}},0)\upiota( \tau_{s_1}) =(0,\tau_{s_{1}})\upiota( \tau_{s_0}) =0$,\\
$(\tau_{s_{0}},0)\upiota( \tau_{s_0})=-e_{\id^{-2}}(\tau_{s_0}, 0)$ and
$(0,\tau_{s_{1}})\upiota( \tau_{s_1})=-e_{\id^{2}}(0,\tau_{s_1})$, and lastly\\
$(\tau_{s_{0}},0)\tau_{\omega_u}=u^{-2} \tau_{\omega_u^{-1}} (\tau_{s_0},0)$ and $(0,\tau_{s_1}) \tau_{\omega_u} = u^2 \tau_{\omega_u^{-1}} (0,\tau_{s_1})$.
  Hence this bimodule structure passes to the quotient $(H_\zeta/H_\zeta \tau_{s_0} \oplus H_\zeta/H_\zeta \tau_{s_1})^\pm$.
\begin{remark} In  $(H_\zeta/H_\zeta \tau_{s_0} \oplus H_\zeta/H_\zeta \tau_{s_1})^\pm$, we have
\begin{equation}\label{f:pm01}\tau_{s_0}(1,0)=(1,0)\tau_{s_0}=0\quad\text{ and }\tau_{s_1}(1,0)=(0,0)\tau_{s_1}=0\ .\end{equation}
The only non obvious statement is for the right actions. We prove it in the first case (it is actually a computation in $(H_\zeta \oplus H_\zeta)^\pm$):
\begin{align*}(1,0)\tau_{s_0}&=-(1,0)\upiota(\tau_{s_0})-(1,0)e_1=
(e_{\id^2}, 0)+\sum_u(1,0)\tau_{\omega_u}\cr&=(e_{\id^2}, 0)+\sum_u(u^{-2}\tau_{\omega_u},0)=(e_{\id^2}, 0)-(e_{\id^2},0)=0\ .
\end{align*}

\end{remark}
\begin{lemma}\label{lemma:zetazeta}
   For any $\sigma \in (H_\zeta/H_\zeta \tau_{s_0} \oplus H_\zeta/H_\zeta \tau_{s_1})^\pm$ we have $\zeta \sigma \zeta = \sigma$ 
   In particular, $(H_\zeta/H_\zeta \tau_{s_0} \oplus H_\zeta/H_\zeta \tau_{s_1})^\pm$ is an $(H_\zeta,H_\zeta)$-bimodule.
\end{lemma}

\begin{proof}
It suffices to show that $\zeta(1,0)\zeta \equiv (1,0)$ and $\zeta (0,1)\zeta \equiv (0,1)$. Here and in the following we write $\equiv$ and $=$, for greater clarity, if an equality holds in $\sigma \in (H_\zeta/H_\zeta \tau_{s_0} \oplus H_\zeta/H_\zeta \tau_{s_1})^\pm$ and $(H_\zeta \oplus H_\zeta)^\pm$, respectively. We give the computation in the first case:
\begin{align*}
  \zeta(1,0)\zeta & =   \zeta(1,0)\upiota(\tau_{s_1})\upiota(\tau_{s_0}) \text{ by \eqref{f:pm01}}
    \\&=\zeta(1,0)(0, -\tau_{\omega_{-1}}\upiota(\tau_{s_0})\zeta^{-1})\upiota(\tau_{s_0}) \cr&=
    \zeta(1,0)
 (\tau_{\omega_{-1}}\upiota(\tau_{s_0})\zeta^{-1}\tau_{\omega_{-1}}\upiota(\tau_{s_1})\zeta^{-1}, 0)=
\zeta (\upiota(\tau_{s_0})\upiota(\tau_{s_1})\zeta^{-2}, 0)\cr&=
\zeta (\zeta^{-2}(\zeta-\tau_{s_1}\tau_{s_0}), 0)\equiv \zeta (\zeta^{-1}, 0)=(1,0)\ .
\end{align*}
\end{proof}

\begin{lemma}\label{lemma:funnyright}\phantomsection
We have an  isomorphism of right $H_\zeta$-modules
\begin{equation*}
  \beta : H_\zeta / \tau_{s_0} H_\zeta \oplus H_\zeta / \tau_{s_1} H_\zeta \xrightarrow{\;\cong\;} (H_\zeta / H_\zeta \tau_{s_0} \oplus H_\zeta / H_\zeta \tau_{s_1})^\pm
\end{equation*}
sending $(1,0)$ and $(0,1)$ to $(1,0)$ and $(0,1)$, respectively. In particular, $(H_\zeta / H_\zeta \tau_{s_0} \oplus H_\zeta / H_\zeta \tau_{s_1})^\pm$
 is  a free  $k[\zeta^{\pm 1}]$-module of rank $4(p-1)$ on the left and on the right.

 \end{lemma}
 \begin{proof}
That the  rule  given to define $\beta$ yields a well defined module homomorphism is immediate from the fact that $(1,0)\tau_{s_0}=(0,1)\tau_{s_1}=0$ (see \eqref{f:pm01}).
To check the bijectivity we start by observing that, as a consequence of Lemma \ref{freeness}, a $k$-basis of $H_\zeta / \tau_{s_i} H_\zeta$ as well as $H_\zeta / H_\zeta \tau_{s_i}$ is given by
\begin{equation*}
  \{ \zeta^j \tau_{\omega_u} : j \in \mathbb{Z},u\in \mathbb F_q^\times\} \cup \{ \zeta^j \tau_{\omega_u} \upiota(\tau_{s_{1-i}}) : j \in \mathbb{Z},  u\in \mathbb F_q^\times\}
\end{equation*} where we use  the involution \eqref{f:upiotadef} of $H$.
It follows that $H_\zeta / \tau_{s_0} H_\zeta \oplus H_\zeta / \tau_{s_1} H_\zeta$ and $(H_\zeta / H_\zeta \tau_{s_0} \oplus H_\zeta / H_\zeta \tau_{s_1})^\pm$ both have the $k$-basis
\begin{equation*}
  \{ (\zeta^j \tau_{\omega_u},0) , \: (\zeta^j \tau_{\omega_u} \upiota(\tau_{s_1}),0) ,\: (0, \zeta^j \tau_{\omega_u}) , \: (0, \zeta^j \tau_{\omega_u} \upiota(\tau_{s_0})) : \quad j \in \mathbb{Z}, u\in \mathbb F_q^\times\} .
\end{equation*}
The image under $\beta$ of this set is
 \begin{equation*}
  \{ u^{-2} (\zeta^{-j} \tau_{\omega_u},0) , \:
  -u^{-2}(0,\zeta^{-j-1} \tau_{\omega_{-u}} \upiota(\tau_{s_0})),u^{2} (0,\zeta^{-j} \tau_{\omega_u}),   -u^{2}(\zeta^{-j-1}\tau_{\omega_{-u}}\upiota(\tau_{s_1}) ,0)
  \quad : j \in \mathbb{Z},u\in \mathbb F_q^\times\} .
\end{equation*}
which is  a basis for $(H_\zeta / H_\zeta \tau_{s_0} \oplus H_\zeta / H_\zeta \tau_{s_1})^\pm$.

The $k[\zeta^{\pm 1}]$-modules $H_\zeta / H_\zeta \tau_{s_0} \oplus H_\zeta / H_\zeta \tau_{s_1}$  and
  $H_\zeta / \tau_{s_0}H_\zeta  \oplus H_\zeta /  \tau_{s_1}H_\zeta$ are  free $k[\zeta^{\pm 1}]$-modules of rank $4(p-1)$ (respectively on the left and on the right). The last statement follows.

\end{proof}
\paragraph{On $\im(f^\pm)$.\label{par:fpm}}
 In this paragraph we assume that $\mathfrak{F} = \mathbb{Q}_p$ with $p \geq 5$ and that $\pi=p$.

\begin{proposition}\label{plusminus-part}

The map $f^\pm$  in  Lemma  \ref{lemma:fpm} yields  an injective homomorphism of $H$-bimodules \begin{equation*}
  (H_\zeta / H_\zeta \tau_{s_0} \oplus H_\zeta / H_\zeta \tau_{s_1})^\pm 
  \longrightarrow E^1
\end{equation*}   which we still denote by $f^\pm$.
Its image $\im(f^\pm)$ is contained in the kernel of the endomorphism $\zeta \cdot\id_{E^1}\cdot  \zeta - \id_{E^1}$ and is a sub-$H$-bimodule of $E^1$ on which $\zeta$ acts invertibly from the left and the right.
Furthermore, $\im(f^\pm)$ is a free   $k[\zeta^{\pm 1}]$-module of rank $4(p-1)$ on the left and on the right.
\end{proposition}

\begin{proof} From Lemma  \ref{lemma:fpm} we know that \begin{equation*}
 H_\zeta / H_\zeta \tau_{s_0} \oplus H_\zeta / H_\zeta \tau_{s_1} \xrightarrow{\; f^\pm=f_{{\bf x}^+} + f_{{\bf x}^-} \;} E^1
\end{equation*} is  an  injective homomorphism of  left $H$-modules
the image  of which  is contained in the kernel of the endomorphism $\zeta \cdot\id_{E^1}\cdot  \zeta - \id_{E^1}$.
The
 right $H$-equivariance  of
 \begin{equation*}
  (H_\zeta / H_\zeta \tau_{s_0} \oplus H_\zeta / H_\zeta \tau_{s_1})^\pm \overset{f^\pm}\longrightarrow
E^1
\end{equation*}
  is immediately seen by comparing the definition \eqref{f:def-R} with Lemma \ref{lemma:leftright}.
The last statement follows directly from  Lemma \ref{lemma:funnyright}

\medskip

In Prop.\ \ref{prop:kerf1star}  we will see that the image of $f^\pm$ coincides in fact with the kernel of $\zeta\cdot \id_{E^1} \cdot \zeta -  \id_{E^1}$.

\end{proof}

\begin{remark}\phantomsection\label{rem:offdiag-anti}
\begin{itemize}
  \item[1.] It follows from Remark \ref{rem:offdiag-Gamma-anti}-1 that the diagram
\begin{equation*}
  \xymatrix{
    (H_\zeta / H_\zeta \tau_{s_0} \oplus H_\zeta / H_\zeta \tau_{s_1})^\pm \ar[d]_{(\sigma^+,\sigma^-) \mapsto (\Gamma_\varpi(\sigma^-),\Gamma_\varpi(\sigma^+))} \ar[r]^-{f^\pm} & E^1 \ar[d]^{\Gamma_\varpi} \\
    (H_\zeta / H_\zeta \tau_{s_0} \oplus H_\zeta / H_\zeta \tau_{s_1})^\pm \ar[r]^-{f^\pm} & E^1   }
\end{equation*}
  is commutative.
  \item[2.]
 The maps
\begin{align*}\delta_0:H_\zeta/H_\zeta\tau_{s_0}&\longrightarrow  H_\zeta/H_\zeta\tau_{s_0}, h\longmapsto h(1 - e_{\id} - e_{\id} \zeta^{-1}) \cr \delta_1: H_\zeta/H_\zeta\tau_{s_1}&\longrightarrow  H_\zeta/H_\zeta\tau_{s_1}, h\longmapsto h(1 - e_{\id^{-1}} - e_{\id^{-1}} \zeta^{-1})\end{align*} are well defined isomorphisms of left $H_\zeta$-modules.

Note that on the component $H_\zeta(1-e_{\id})/H_\zeta\tau_{s_0}(1-e_{\id})$ (resp.   $H_\zeta(1-e_{\id^{-1}})/H_\zeta\tau_{s_1}(1-e_{\id^{-1}})$), the map $\delta_0$ (resp. $\delta_1$) is the identity map. On $H_\zeta e_{\id}/H_\zeta\tau_{s_0} e_{\id}$ (resp.   $H_\zeta e_{\id^{-1}}/H_\zeta\tau_{s_1} e_{\id^{-1}}$), the map $\delta_0$ (resp. $\delta_1$) is the multiplication by $\zeta^{-1}$.

\medskip

Consider
\begin{equation}(H_\zeta / H_\zeta \tau_{s_0} \oplus H_\zeta / H_\zeta \tau_{s_1})^\pm\xrightarrow{\anti\oplus \anti} H_\zeta / \tau_{s_0}H_\zeta  \oplus H_\zeta /  \tau_{s_1}H_\zeta\xrightarrow{\beta} (H_\zeta / H_\zeta \tau_{s_0} \oplus H_\zeta / H_\zeta \tau_{s_1})^\pm \end{equation}
We have
\begin{align*}f^\pm\circ \beta\circ (\anti\oplus\anti)(\sigma^+, \sigma^-)&=f^\pm\circ \beta(\anti(\sigma^+),\anti( \sigma^-))=f^\pm\circ \beta(\anti(\sigma^+),\anti( \sigma^-))\cr&={\bf x}^+\cdot \anti(\sigma^+)+{\bf x}^-\cdot \anti(\sigma^-)\text{ since $f^\pm\circ \beta$ is right $H$-equivariant .}\end{align*}
Let
\begin{equation}\label{f:antipm}\anti^\pm := \beta\circ (\anti\oplus\anti)\circ(\delta_0\oplus\delta_1).\end{equation}

\begin{align*}f^\pm\circ \anti^\pm(\sigma^+, \sigma^-)&={\bf x}^+\cdot \anti(\sigma^+(1 - e_{\id} - e_{\id} \zeta^{-1}))+{\bf x}^-\cdot \anti(\sigma^-(1 - e_{\id^{-1}} - e_{\id^{-1}} \zeta^{-1}))\cr&=
{\bf x}^+\cdot (1 - e_{\id^{-1}} - e_{\id^{-1}} \zeta^{-1}) \anti(\sigma^+)+{\bf x}^-\cdot (1 - e_{\id} - e_{\id} \zeta^{-1})\anti(\sigma^-)
\cr&=
 (1 - e_{\id} - e_{\id} \zeta) \cdot{\bf x}^+ \cdot \anti(\sigma^+)+(1 - e_{\id^{-1}} - e_{\id^{-1}} \zeta)\cdot {\bf x}^-\cdot \anti(\sigma^-) \text{by Lemma \ref{lemma:leftright}-1}\cr&=\anti({\bf x}^+) \cdot \anti(\sigma^+)+\anti({\bf x}^-)\cdot \anti(\sigma^-)\text{ by Remark \ref{rem:offdiag-Gamma-anti}-2}\cr&=\anti(\sigma^+\cdot {\bf x}^++\sigma^-\cdot {\bf x}^-)=\anti(f^\pm(\sigma^+, \sigma^-)) \ .
\end{align*}
It follows  that the diagram
\begin{equation*}
  \xymatrix{
    (H_\zeta / H_\zeta \tau_{s_0} \oplus H_\zeta / H_\zeta \tau_{s_1})^\pm \ar[d]_{\anti^\pm} \ar[r]^-{f^\pm} & E^1 \ar[d]^{\anti} \\
    (H_\zeta / H_\zeta \tau_{s_0} \oplus H_\zeta / H_\zeta \tau_{s_1})^\pm \ar[r]^-{f^\pm} & E^1   }
\end{equation*}
  is commutative.

\end{itemize}
\end{remark}

\section{Formulas for the left action of $H$ on $E^{d-1}$ when $G={\rm SL}_2(\mathbb Q_p) $, $p\neq 2,3$}

For the moment,  $G={\rm SL}_2(\mathfrak F)$ and $I$ is a Poincar\'e group of dimension $d$ (hence $p \geq 5$).

\subsection{\label{subsec:triples2}Elements of $E^{d-1}$ as triples}

Recall (see \eqref{f:dual}) the isomorphism  of $H$-bimodules
$ \Delta^{d-1}:E^{d-1}\rightarrow  {}^\anti((E^{1})^{\vee, f})^\anti$.
The left action of $h\in H$ on $\alpha\in {}^\anti((E^{1})^{\vee, f})^\anti\cong E^{d-1}$  is given by \begin{equation}\label{f:leftonEd-1}(h, \alpha)\mapsto \alpha(\anti(h)_-)\ .\end{equation} The anti-involution $\anti$ on $E^{d-1}$ corresponds to the transformation \begin{align} \label{f:antiEd-1}  {}^\anti((E^{1})^{\vee, f})^\anti&\longrightarrow  {}^\anti((E^{1})^{\vee, f})^\anti\cr \alpha&\longmapsto \alpha\circ \anti.\end{align}

\begin{proof}
We prove that for $\alpha_0\in E^{d-1}$ we have $\Delta^{d-1}(\anti(\alpha_0))=\Delta^{d-1}(\alpha_0)\circ \anti$ in $(E^{1})^{\vee, f}$. Let $\beta\in (E^{1})^{\vee, f}$. By definition of $\Delta^{d-1}$ we have
\begin{align*} [\Delta^{d-1}(\alpha_0)\circ \anti] (\beta)&=\eta\circ \trace^d(\alpha_0\cup \anti(\beta))= \eta\circ \trace^d(\anti(\anti(\alpha_0)\cup \beta))\textrm{ by \cite{Ext} Rmk.\ 6.2} \cr&= \eta\circ \trace^d(\anti(\alpha_0)\cup \beta)\textrm{ by \cite{Ext} Cor.\ 7.17} \cr&=[\Delta^{d-1}(\anti(\alpha_0))](\beta)\ .\end{align*}

\end{proof}

We will abbreviate  $ h^{d-1}(w) := H^{d-1}(I,\mathbf{X}(w))$ for $w \in \widetilde{W}$ and will identify it with $h^1(w)^\vee\subseteq{} ^\anti((E^{1})^{\vee, f})^\anti$.
Recall  from \eqref{tripdec} that an element $c$  in $h^1(w)\subset E^1$ may be seen as a triple  $(c^-,c^0,c^+)_w$ with $$ c^\pm \in \Hom(\mathfrak{O}/\mathfrak{M},k) \text{ and } c^0 \in \Hom((1+\mathfrak{M}) \big/ (1+\mathfrak{M}^{\ell(w)+1})(1+\mathfrak{M})^p ,k) \ .$$
For a given  finite dimensional ${\mathbb F_p}$-vector space $V$,
the $k$-dual of $\Hom_{\mathbb F_p}(V, k)$  identifies  canonically with
$V\otimes _{\mathbb F_p} k$ so we will see an element $\alpha$ of $(h^1(w))^\vee$
as  a triple \begin{equation}\label{tripdec'}(\alpha^-, \alpha^0,\alpha^+)_w\in \mathfrak{O}/\mathfrak{M}\otimes _{\mathbb F_p} k\times ((1+\mathfrak M)/(1+\mathfrak{M}^{\ell(w)+1})(1+\mathfrak{M})^p)\otimes _{\mathbb F_p} k\times \mathfrak{O}/\mathfrak{M}\otimes _{\mathbb F_p} k\end{equation} such that
$\alpha(c)=c^-(\alpha^-)+c^0(\alpha^0)+c^+(\alpha^+)\ .$
We still denote by $(\alpha^-, \alpha^0,\alpha^+)_w$ the image of this element in $h^{d-1}(w)$ via
the inverse of $\Delta^{d-1}$ and then we have

\begin{equation}\label{f:pairing'}(\alpha^-, \alpha^0,\alpha^+)_w\cup (c^-, c^0,c^+)_w=(c^-(\alpha^-)+c^0(\alpha^0)+c^+(\alpha^+))\,\phi_w\end{equation} where $\phi_w\in h^d(w)$  was defined in \S\ref{subsubsec:duality}.
Since $\anti$ respects the cup product  and since $\anti(\phi_w)=\phi_{w^{-1}}$ (\cite{Ext} Rmk.\ 6.2 and (8.2)), we obtain
from Lemma \ref{lemma:antievenodd} the following result:

\begin{lemma}\label{lemma:antialpha}
Let $w\in \widetilde W$  and  $\alpha=(\alpha^-, \alpha^0, \alpha^+)_w\in h^{d-1}(w)$.\\ If $\ell(w)$ is even then
 \begin{equation}\label{f:antievena}\anti(\alpha)=({{u^{-2}}}\alpha^-, \alpha^0, {{u^{2}}}\alpha^+)_{w^{-1}}.\end{equation}
  If $\ell(w)$ is odd then
 \begin{equation}\label{f:antiodda}\anti(\alpha)=(-{{u^{2}}}\alpha^+, -\alpha^0, -{{u^{-2}}}\alpha^-)_{w^{-1}}.\end{equation}
 where $u\in (\mathfrak{O}/\mathfrak{M})^\times$ is such that $\omega_u^{-1} w$ lies in the subgroup of $\widetilde W$ generated by $s_0$ and $s_1$.
\end{lemma}

From  \eqref{f:cup+Gamma}, \eqref{f:conjphi} and Lemma \ref{lemma:conjtrip} we obtain:

\begin{lemma} Let $w\in \widetilde W$ and $(\alpha^-, \alpha^0, \alpha^+)_w\in  h^{d-1}(w)$.
Its image  by conjugation by $\varpi$ defined in  \eqref{gammapi}
is
$$\Gamma_\varpi((\alpha^-, \alpha^0, \alpha^+)_w)=(\alpha^+, -\alpha^0, \alpha^-)_{\varpi w\varpi^{-1}}\in h^2(\varpi w\varpi^{-1}) \ . $$
\label{lemma:conjtripdual}
\end{lemma}

%
%
%
%
%
%
%
%
%
%
%
%
%
\noindent In the next lemma we refer to the notation in \S\ref{subsubsec:normalize}.

\begin{lemma}
Assume  $G={\rm SL}_2(\mathbb Q_p) $, $p\neq 2,3$. For $w\in \widetilde W$, $\ell(w)\geq 1$ we have
\begin{align}\label{f:pairing}
   (0,\upalpha^0,0)_w & = -(\c,0,0)_w\cup(0,0,\c)_w, \nonumber \\
   (\upalpha,0,0)_w & = (0,\c^0,0)_w\cup(0,0,\c)_w, \\
   (0,0,\upalpha)_w & = (\c,0,0)_w\cup(0,\c^0,0)_w \ .  \nonumber
\end{align}
\end{lemma}
\begin{proof}
By definition, $(0,\upalpha^0,0)_w$ is the unique element in $h^2(w)$ such that
\begin{align*}
   & \eta\circ\trace^d\big((0,\upalpha^0,0)_w\cup(0,\c^0,0)_w \big)=\c^0(\upalpha^0)=1, \\
   & \eta\circ\trace^d\big((0,\upalpha^0,0)_w\cup(\c,0,0)_w \big)=0,   \\
   & \eta\circ\trace^d\big((0,\upalpha^0,0)_w\cup(0,0,\c)_w \big)=0,
\end{align*}
namely $(0,\upalpha^0,0)_w\cup(0,\c^0,0)_w=\phi_w$ while
$(0,\upalpha^0,0)_w\cup(\c,0,0)_w=(0,\upalpha^0,0)_w\cup(0,0,\c)_w=0.$
By \eqref{f:orientation}, we obtain the first formula of the lemma. The other formulas are obtained similarly.
%
%
%
%
%

\end{proof}

\noindent For any subset $U \subseteq \widetilde{W}$ we define as in \S\ref{triples} the $k$-subspaces
\begin{equation*}
  h^{d-1}_-(U) := \oplus_{w \in U} h^{d-1}_-(w), \quad h^{d-1}_0(U) := \oplus_{w \in U} h^{d-1}_0(w),\quad \text{and}\ h^{d-1}_+(U) := \oplus_{w \in U} h^{d-1}_+(w)
\end{equation*}
of $h^{d-1}$.  We also let  $h^{d-1}_\pm(U):=h^{d-1}_-(U) \oplus h^{d-1}_+(U) $.

\subsection{Left action of  $\tau_\omega$  on $E^{d-1}$ for $\omega\in \Omega $}

Let $w\in \widetilde W$.
The action of  $\tau_\omega$ on the left on  an element  $\alpha\in h^{d-1}(w)\subseteq E^{d-1}$ was given at the beginning of \S\ref{subsec:omega}. Here we make this action explicit when $\alpha$ is given by a triple
$$
\alpha=(\alpha^-, \alpha^0, \alpha^+)_w\in (h^1(w))^\vee\subset   {}^\anti((E^{1})^{\vee, f})^\anti\cong E^{d-1}
$$
as in  \eqref{tripdec'}.
For $u\in (\mathfrak O/\mathfrak M)^\times$, we compute  $\tau_{\omega_u}\cdot \alpha\in
h^1(\omega_u w)$. For $c=(c^-,c^0, c^+)_{\omega_u w}\in h^1(\omega _uw)$
we have $(\tau_{\omega_u}\cdot\alpha)(c)=c^-(u^{2}\alpha^-)+ c^0(\alpha^0)+c^+(u^{-2}\alpha^+)$
(see \eqref{f:leftomega})
 therefore
\begin{equation} \label{f:leftomegaE2}\tau_{\omega_u}\cdot(\alpha^-, \alpha^0, \alpha^+)_w=(u^{2}\alpha^-, \alpha^0, u^{-2}\alpha^+)_{\omega_uw}\ .
\end{equation}
In particular, for $s\in \{s_0, s_1\}$ we have (compare with \eqref{f:lefts2}) \begin{equation}\label{f:lefts2'}\tau_{s^2}\cdot(\alpha^-, \alpha^0, \alpha^+)_w=(\alpha^-, \alpha^0, \alpha^+)_{s^2w} \ .\end{equation}

\begin{remark}
Using \eqref{f:leftonEd-1} and  the formulas in \S\ref{subsec:commuel}, we have for  $w\in \widetilde W$ and
 $\alpha=(\alpha^-, \alpha^0, \alpha^+)_w\in h^{d-1}(w)$:
\begin{equation}\label{f:condeven'}
(\ell(w)\text{ even})):\quad  e_\lambda\cdot  \alpha  = \alpha \cdot e_\mu   \textrm{ if and only if }\left\lbrace\begin{array}{l}\alpha ^-=\mu^{-1}\lambda(\omega_u ) \alpha ^-(u^{-2}\,_-)\cr \cr \alpha ^0=\mu^{-1}\lambda(\omega_u ) \alpha ^0\cr\cr
\alpha ^+=\mu^{-1}\lambda(\omega_u ) \alpha ^+(u^{2}\,_-).
 \end{array}\right. \textrm{ for any $u\in(\mathfrak O/\mathfrak M)^\times$}
\end{equation}
and
\begin{equation}\label{f:condodd'}
(\ell(w)\text{ odd})):\quad e_\lambda\cdot \alpha  = \alpha \cdot e_\mu  \textrm{ if and only if } \left\lbrace\begin{array}{l}\alpha ^-=\mu\lambda(\omega_u ) \alpha ^-(u^{-2}\,_-)\cr \cr \alpha ^0=\mu\lambda(\omega_u ) \alpha^0\cr\cr
\alpha ^+=\mu\lambda(\omega_u ) \alpha ^+(u^{2}\,_-).
 \end{array}\right. \textrm{ for any $u\in(\mathfrak O/\mathfrak M)^\times$}.
\end{equation}
\end{remark}

\subsection{\label{subsec:HonE2}Left action of   $H$ on $ E^{2}$ when $G={\rm SL}_2(\mathbb Q_p) $, $p\neq 2,3$}

Suppose that  $G={\rm SL}_2(\mathbb Q_p) $, $p\neq 2,3$ and $\pi=p$. Then $d=3$.
The isomorphism $\iota$ was defined in \eqref{f:iota}. The following proposition is proved in \S\ref{proof:theformulas'}. Together with \eqref{f:leftomegaE2}, these formulas give the left action of $H$ on $E^2$.

\begin{proposition} \label{prop:theformulas'}
Let $w\in{\widetilde W}$, $\omega\in \Omega$ and $\alpha=(\alpha^-, \alpha^0, \alpha^+)_w\in (h^1(w))^\vee$ seen as an element of $E^{2}$. We have:
\begin{multline*}
 \tau_{s_0}\cdot (\alpha^-, \alpha^0, \alpha^+)_w =  \cr
  \begin{cases}
  (-\alpha^+,0,0)_{s_0w} &  \text{if  $w\in \widetilde W^0$ with $\ell(w)\geq 1$,} \cr
     e_1 \cdot (-\alpha^-, -\alpha^0, -\alpha^+)_w + e_{\id}\cdot (2\iota(\alpha^0),0,0)_w + (-\alpha^+,  {-} \alpha^0,0)_{s_0w}   & \text{if  $w\in \widetilde W^1$ with  $\ell(w) \geq 2$,} \cr
     e_1 \cdot (-\alpha^-, -\alpha^0, -\alpha^+)_w + e_{\id} \cdot ( 2\iota(\alpha^0),-\iota^{-1}(\alpha^+),0)_w & \cr
         \quad\quad\quad\quad\quad\quad\quad\quad\quad\quad\quad + e_{\id^2} \cdot( \alpha^+, 0,0)_w+(-\alpha^+,0,0)_{s_0w}
      & \text{if  $w\in \widetilde W^1$ with  $\ell(w) =1$.}
  \end{cases} \cr
  \shoveleft{\tau_{s_1}\cdot (\alpha^-, \alpha^0, \alpha^+)_w =} \cr
  \begin{cases}
  (0,0,-\alpha^-)_{s_1w }&  \text{ if  $w\in \widetilde W^1$ with $\ell(w)\geq 1$,} \cr
   -e_1 \cdot (\alpha^-, \alpha^0, \alpha^+)_w+e_{\id^{-1}}\cdot (0,0,-2\iota(\alpha^0))_w +   (0,  {-} \alpha^0,-\alpha^-)_{s_1w} & \text{if  $w\in \widetilde W^0$ with  $\ell(w) \geq 2$,} \cr
   -e_1  \cdot(\alpha^-, \alpha^0, \alpha^+)_w+e_{\id^{-1}}\cdot (0, \iota^{-1}(\alpha^-),-2\iota(\alpha^0))_w & \cr
   \quad\quad\quad\quad\quad\quad\quad\quad\quad\quad + e_{\id^{-2}} \cdot ( 0, 0,\alpha^-)_w +   (0,  0,-\alpha^-)_{s_1w}
      & \text{if  $w\in \widetilde W^0$ with  $\ell(w) =1$.}
  \end{cases}\cr
\shoveleft{\tau_{s_0} \cdot (\alpha^-, 0, \alpha^+)_\omega = (-\alpha^+, 0, 0)_{s_0\omega}}  \cr
\shoveleft{\tau_{s_1}\cdot(\alpha^-, 0, \alpha^+)_\omega=(0,0,-\alpha^-)_{s_1\omega}.}\\
\end{multline*}
\end{proposition}

\begin{corollary}\label{coro:zetaonEd-1}
Let $w\in{\widetilde W}$, $\omega\in \Omega$ and $\alpha=(\alpha^-, \alpha^0, \alpha^+)_w\in (h^1(w))^\vee$ seen as an element of $E^{2}$.\\
\begin{equation*}
\zeta\cdot (\alpha^-,0,\alpha^+)_\omega =
(\alpha^-, 0,0)_{s_0s_1\omega}+ (0,0,\alpha^+)_{s_1 s_0\omega} +
 e_1\cdot (-\alpha^+,0,0)_{s_0\omega}+e_1 \cdot(0,0,-\alpha^-)_{s_1\omega} + e_1\cdot (\alpha^-,0,\alpha^+)_\omega .
\end{equation*}
\begin{equation*}
\zeta\cdot  (\alpha^-,\alpha^0,\alpha^+)_w =
  \begin{cases}
 (\alpha^-,0,0)_{s_0s_1w}+ (0,  0,\alpha^+)_{s_1s_0w}+e_1\cdot (-\alpha^+,  0,0)_{s_0w}
\cr\quad\quad\quad+e_{\id^{-1}}\cdot ( 0,0,-2\iota(\alpha^0))_{s_1w} +e_{\id^{-2}}\cdot ( 0, 0,-\alpha^+)_{s_1w}
  &  \text{ if  $w\in s_0\Omega$},\cr
  (\alpha^-,  0,0)_{s_0s_1w}
+(0,0,\alpha^+)_{s_1s_0w}+ e_1\cdot (0,  0,-\alpha^-)_{s_1w}\cr\quad\quad\quad+e_{\id}\cdot (2\iota(\alpha^0), 0,0)_{s_0w}+e_{\id^2}\cdot ( -\alpha^-, 0, 0,)_{s_0w}
   & \text{ if  $w\in s_1\Omega $}
 \end{cases}
\end{equation*}
\begin{equation*}
\zeta \cdot(\alpha^-,\alpha^0,\alpha^+)_w =
\begin{cases}
 (\alpha^-,0,0)_{s_0s_1w} +  (0,  0,\alpha^+)_{s_1s_0w}  + e_{\id^{-1}}\cdot (0,0,-2\iota(\alpha^0))_{s_1w} & \cr\quad\quad + e_{\id^{-1}} \cdot (0,- \iota^{-1}(\alpha^+),2\iota(\alpha^0))_{s_0w}
 + e_{\id^{-2}} \cdot ( 0, 0,-\alpha^+)_{s_0w} & \text{if  $w\in{\widetilde W}^1$, $\ell(w)=2$},\cr (0,0,\alpha^+)_{s_1s_0w}+ (\alpha^-,  0,0)_{s_0s_1w}+
e_{\id}\cdot (2\iota(\alpha^0),0,0)_{s_0w}
&\cr\quad\quad+e_{\id}\cdot ( -2\iota(\alpha^0),\iota^{-1}(\alpha^-),0)_{s_1w}  +e_{\id^2} \cdot (- \alpha^-, 0,0)_{s_1w}
& \text{if $w\in{\widetilde W}^0$, $\ell(w)=2$,}
  \end{cases}
\end{equation*}
\begin{equation*}  \zeta \cdot (\alpha^-,\alpha^0,\alpha^+)_w =
\begin{cases}
 (\alpha^-,0,0)_{s_0s_1w}+(0,  \alpha^0,\alpha^+)_{s_1s_0w} &\cr\quad\quad
  + e_{\id^{-1}}\cdot (0,0,-2\iota(\alpha^0))_{s_1w}+
e_{\id^{-1}}\cdot (0,0,2\iota(\alpha^0))_{s_0w}
  & \text{ if  $w\in{\widetilde W}^1$, $\ell(w)\geq 3$},\cr
  (0,0,\alpha^+)_{s_1s_0w} +(\alpha^-,   \alpha_0,0)_{s_0s_1w}
  &\cr\quad\quad
  +
  e_{\id}\cdot (2\iota(\alpha^0),0,0)_{s_0w}+ e_{\id}\cdot (-2\iota(\alpha^0),0,0)_{s_1w}
& \text{if  $w\in{\widetilde W}^0$, $\ell(w)\geq 3$.}
\end{cases}
\end{equation*}
\end{corollary}

\section{\label{sec:zetatorsion}$k[\zeta]$-torsion in $E^*$  when $G={\rm SL}_2(\mathbb Q_p)$, $p\neq 2,3$}

In this whole section $G={\rm SL}_2(\mathfrak F)$.\\\\
\textbf{A)} Without any assumption on $\mathfrak F$, we know that $E^0$ is a free left (resp. right) $k[\zeta]$-module (Lemma \ref{freeness}). Therefore it is $k[\zeta]$-torsionfree on the left (resp. right).\\\\
\textbf{B)} Here we suppose that  the group $I$ is  torsionfree and its dimension as a  Poincar\'e group is $d$. We study the $k[\zeta]$-torsion in $E^d$. We know by Remark \ref{rema:centralizeEd} that the left and right actions of  $\zeta$ on $E^d$ coincide.
Recall that we have the following isomorphisms of  $H$-bimodules
\begin{equation}\label{f:keydecompEd}
E^d\cong\ker(\EuScript S^d)\oplus \chi_{triv}
\end{equation}
and by Proposition \ref {prop:kerS-injhull}, we have $\ker(\trace^d)\cong \bigcup_m (H/\zeta^m H)^\vee$ as $H$-bimodules.
Therefore $E^d$ is the direct sum of its one-dimensional subspace of $(\zeta-1)$-torsion  and of its subspace
 $\ker(\trace^d)$ of $\zeta$-torsion.
  This applies in particular when $G={\rm SL}_2(\mathbb Q_p)$, $p\neq 2,3$ and $d=3$. \\\\
\textbf{C)}   We study the  $k[\zeta]$-torsion  in $E^1$.

\begin{lemma}\label{lemma:notorsion}
 Suppose that $G={\rm SL}_2(\mathfrak F)$.
\begin{itemize}
\item[i]  Suppose that $p\neq 2$. For any $P \in k[X]$ such that $P(0) \neq 0$ there is no left (resp.\ right) $P(\zeta)$-torsion in $E^1$.
\item[ii.]  If $\mathfrak F=\mathbb Q_p$, given any $0 \neq P \in k[X]$, there is no left (resp.\ right) $P(\zeta)$-torsion in $E^1$.
\end{itemize}
\end{lemma}
\begin{proof}
Let $0 \neq P \in k[X]$. Suppose that we know that $(\X/ \X P(\zeta))^I\cong H/ H P(\zeta)$. Then the exact sequence of $(G,H)$-bimodules
$0\rightarrow \X \xrightarrow{\cdot P(\zeta) } \X \rightarrow \X/\X P(\zeta) \rightarrow 0$ induces the long exact sequence of $H$-bimodules
$$
0 \rightarrow E^1 \xrightarrow{\cdot P(\zeta) } E^1 \rightarrow H^1(I, \X/\X P (\zeta)) \rightarrow E^2 \rightarrow ...
$$
In particular, there is no right $P(\zeta)$-torsion in $E^1$. Since  $P(\zeta) \cdot c=\anti(\anti(c)\cdot P(\zeta)$ for any $c\in E^*$,  there is no left $P(\zeta)$-torsion in $E^1$ either.

i. For any field extension $k'/k$ and any $V \in \Mod(G)$ we have $(V \otimes_k k')^I = V^I \otimes_k k'$. Therefore we may assume that $\mathbb F _q\subseteq  k$ (and that $p \neq 2$). Suppose that $P(0) \neq 0$. Then $H/H P(\zeta)$ is an $H_\zeta$-module. Hence by \cite{embed} Thm. 3.33  we know that $(\X/\X P(\zeta))^I \cong H/HP(\zeta)$.

ii.  Suppose  $\mathfrak F=\mathbb Q_p$. Then $(\X/\X P (\zeta))^I\cong H/ H P(\zeta)$ (see \S\ref{subsubsec:eqcat}).
\end{proof}

\noindent\textbf{D)}  Here we suppose that  the group $I$ is  torsionfree and its dimension as a  Poincar\'e group is $d$.
We study the  $k[\zeta]$-torsion subspace in $E^{d-1}$.  \\ Let $i\in\{0, \dots, d\}$ and $\ell\geq 1$.
Recall that the left  action of $\zeta$ on ${}^\anti((E^i)^{\vee, f})^{\anti}\cong E^{d-i}$ is given by
$(\zeta, \varphi)\mapsto \varphi(\zeta\cdot _-)\: : \: E^i\rightarrow k.$
In particular,
$\coker( \zeta^\ell\cdot : E^{i}\rightarrow E^{i})=\{0\}$ implies $\ker(\zeta^\ell \cdot  : E^{d-i}\rightarrow E^{d-i})=\{0\}$. We explore the converse implication in
the lemma below  where we refer to the decreasing filtration $(F^m E^i)_{m\geq 0}$ introduced in \S\ref{subsubsec:fil}.

\begin{lemma}\label{lemma:obvious}
Suppose that  $G= {\rm SL}_2(\mathfrak F)$ and $I$ is a Poincar\'e group of dimension $d$.
Let $i\in\{0,\dots, d\}$. Suppose that there is $m\geq 0$ such that $ \zeta^\ell\cdot E^i \supseteq F^m E^i$, then we have an isomorphism of $H$-bimodules:
$$\ker(\zeta^\ell\cdot : E^{d-i}\rightarrow E^{d-i})\cong {}^\anti((E^i/\zeta^\ell\cdot E^i)^\vee)^\anti\ .$$ In particular,   $\ker(\zeta^\ell\cdot : E^{d-i}\rightarrow E^{d-i})=0$ \textrm{ if and only if }
$\coker(\zeta^\ell\cdot : E^{i}\rightarrow E^{i})=0.$
The same statements are valid for the right action of $\zeta^\ell$.
\end{lemma}

\begin{proof}
The kernel of  the
left  action of $\zeta^\ell$ on ${}^\anti((E^i)^{\vee,f})^{\anti}$ is the space of all $\varphi\in (E^i)^{\vee,f}$ which are trivial on $\zeta^\ell\cdot E^i$. Suppose that there is $m\geq 0$ such that $\zeta^\ell\cdot E^i\supseteq F^m E^i$. Then any
$\varphi\in (E^i)^{\vee}$ which is trivial on $\zeta^\ell\cdot E^i$ lies in $(E^i)^{\vee,f}$. Therefore, the kernel of  the
right  action of $\zeta^\ell$ on ${}^\anti((E^i)^{\vee,f})^{\anti}$ is the space of all $\varphi\in (E^i)^{\vee}$ which are trivial on $\zeta^\ell\cdot E^i$, namely
$\ker(\cdot \zeta^\ell : {}^\anti((E^i)^{\vee,f})^{\anti} \rightarrow {}^\anti((E^i)^{\vee,f})^{\anti})=  {}^\anti((E^i/\zeta^\ell\cdot E^i)^\vee)^{\anti}\ .$
\end{proof}

\begin{remark} \label{rema:Edtorsion}
It is easy to check that $\zeta^\ell\cdot E^0\supset  \zeta^\ell\cdot F^1E^0= F^{2\ell+1} E^0$. So we recover $\ker(\zeta^\ell \cdot : E^{d}\rightarrow E^{d})\cong {}^\anti((H/\zeta^\ell H )^\vee)^\anti\ $ which  is isomorphic to
 $(H/\zeta^\ell H )^\vee$.
(compare with \textbf{B)} above).
\end{remark}

Using Corollary \ref{coro:zetaE1F}  we obtain immediately:

\begin{corollary} \label{cor:torsionE2}
Suppose that $G={\rm SL}_2(\mathbb Q_p)$, $p\neq 2,3$. We have an isomorphism of $H$-bimodules:
$$\ker(\zeta \cdot  : E^{2}\rightarrow E^{2})\cong {}^\anti((E^1/\zeta\cdot E^1)^\vee)^\anti \ .$$
\end{corollary}

\begin{remark}
We will see in Proposition  \ref{prop:structurekerg2} that this space is nontrivial.
\end{remark}


\begin{lemma}\label{lemma:notorsion2}
  Suppose that $G={\rm SL}_2(\mathbb Q_p)$, $p\neq 2,3$  and  $\pi=p$. There is no left (resp. right) $P(\zeta)$-torsion in $E^2$ for any $P \in k[X]$ with $P(0) \neq 0$.
  \end{lemma}
\begin{proof}
 We may prove the assertion after a base extension of $k$. Hence it suffices to consider the case $P(X) = X - a$ for some $a \in k^\times$.
As in the  proof of Lemma \ref{lemma:notorsion}, it is enough to prove that there is no left $(\zeta-a)$-torsion in $E^2$
or equivalently
that there is no right $(\zeta-a)$-torsion $(E^1)^{\vee,f}$ (see  \eqref{f:dual}).
We prove that  for a given $m\geq 1$, we have
$$(\zeta-a)\cdot E^1+ F^m E^1= E^1.$$
By our assumption that $\pi=p$,  we may use the formulas of Cor.\ \ref{coro:form-zeta}.
\begin{itemize}
\item  If  $w\in \widetilde W^0$, $\ell(w)\geq 2$, we have $(\zeta-a)\cdot  (c^-,c^0,0)_w= (c^-, c^0,0)_{s_1s_0w}-a\cdot  (c^-,c^0,0)_w$
and  if $\ell(w)\geq 1$, we have $(\zeta-a)\cdot  (c^-,0,0)_w= (c^-, 0,0)_{s_1s_0w}-a\cdot  (c^-,0,0)_w$.\\
So by induction   $h^1_-(\widetilde W^{0,\ell\geq 1})+h^1_0(\widetilde W^{0,\ell\geq 2})\subseteq  (\zeta-a)\cdot E^1+ F^m E^1$. Using conjugation by $\varpi$, we  have proved  $h^1_-(\widetilde W^{0,\ell\geq 1})+ h^1_0(\widetilde W^{\ell\geq 2})+h^1_+(\widetilde W^{1,\ell\geq 1})\subseteq  (\zeta-a)\cdot E^1+ F^m E^1$.

\item   If $\ell(w)\geq 3 $
we have,
$(\zeta - a) \cdot (0,0,c^+)_w \in  (0,0,c^+)_{s_0 s_1w}- a(0,0,c^+)_w  +
h^1_0(\widetilde W^{\ell\geq 2})$ if $w\in \widetilde W^0$
therefore   $h^1_-(\widetilde W^{1, \ell\geq 1})+h^1_+(\widetilde W^{0, \ell\geq 1}) \subseteq  (\zeta-a)\cdot E^1+ F^m E^1$ by induction and conjugation by $\varpi$. \end{itemize}
So at this point we have  $h^1_0(\widetilde W^{\ell\geq 2})+h^1_\pm(\widetilde W^{\ell\geq 1})\subseteq  (\zeta-a)\cdot E^1+ F^m E^1$.
\begin{itemize}
\item  But  if $\ell(w)=1$ we have   $(\zeta -a) (0,c^0,0)_w\in
 - a(0,c^0,0)_w+ h^1_0(\widetilde W^{\ell\geq 2})+h^1_\pm(\widetilde W^{ \ell\geq 1})$  \\
   so $h^1_0(\widetilde W)+h^1_\pm(\widetilde W^{\ell\geq 1})\subseteq  (\zeta-a)\cdot E^1+ F^m E^1$.
\item Lastly,
$(c^-,0,c^+)_{\omega}\in (\zeta-a)\cdot (c^-,0,0)_{s_0s_1\omega}+(\zeta-a)\cdot (0,0,c^+)_{s_1s_0\omega}+ h^1_0(\widetilde W)+h^1_\pm(\widetilde W^{\ell\geq 1})$   for $\omega\in \Omega$. So  $h^1_0(\widetilde W)+h^1_\pm(\widetilde W)\subseteq  (\zeta-a)\cdot E^1+ F^m E^1$.
\end{itemize}
\end{proof}

\section{\label{sec:structure}Structure of $E^1$ and $E^2$ when $G={\rm SL}_2(\mathbb Q_p)$, $p\neq 2,3$}\label{sec:structure}

\subsection{Preliminaries}\label{subsec:fg}

We define the following endomorphisms of $H$-bimodules  of $E^*$:
$$
f:= \zeta\cdot \id_{E^*}\cdot\zeta -  \id_{E^*}:  c\mapsto \zeta \cdot c\cdot \zeta- c
$$
and
$$
g:= \zeta\cdot \id_{E^*} -  \id_{E^*}\cdot \zeta: c\mapsto \zeta \cdot c- c\cdot \zeta.
$$
We will restrict them to the graded pieces $E^i$ and  will then  use the notation $f_i$ and $g_i$. The following remarks are easy to check. Here $G={\rm SL}_2(\mathfrak F)$.

\begin{remark}\phantomsection\label{rema:calcufg}
\begin{itemize}
\item[i.] $f$ and $g$ commute. In fact,  $$f\circ g=(\zeta^2+1)\cdot \id_{E^*}\cdot \zeta-\zeta\cdot \id_{E^*}\cdot  (\zeta^2+1)= g\circ f.$$
\item[ii.] It is clear that the left (resp. right) action of $\zeta$ on $\ker(f)$ induces a bijective map. Hence $\ker(f)$ is naturally a $H_\zeta$-bimodule.
\item[iii.]  We have the following inclusions of subalgebras of $E^*$:
$$\ker(g)\subseteq \ker(f)+\ker(g)\subseteq E^*\ .$$
We have indeed $\ker(f)\cdot \ker(f)\subseteq\ker(g)$ as well as  $\ker(f)\cdot \ker(g)\subseteq \ker(f)$  and $\ker(g)\cdot \ker(f)\subseteq \ker(f)$.
\item[iv.] The spaces $\ker(f)$ and $\ker(g)$ are stable by conjugation by $\varpi$ (see \eqref{gammapi} and use that $\Gamma_\varpi(\zeta) = \zeta$).

\item[v.]  The spaces $\ker(f)$ and $\ker(g)$ are stable  by $\anti$  (use that $\anti(\zeta)=\zeta$).
\end{itemize}
\end{remark}

\begin{lemma} \label{lemma:cap}Suppose  $G={\rm SL _2}(\mathfrak F)$. We have
\begin{itemize}
\item[i.]   $\ker({f_0})=\{0\}$ and $\ker({g_0})=E^0$ .
\item[ii.]   If   $I$ is a Poincar\'e group of dimension $d$, then $\ker({g_d})=E^d$ and
 $\ker({f_d})\cong \chi_{triv}$ as a left (resp. right) $H$-module.

\item[iii.] Suppose that $p\neq 2$  or $\mathfrak F = \mathbb Q_p$. Then  $\ker({f_1})\cap  \ker({g_1})=\{0\}$.
\item[iv.] Suppose that $\mathfrak F= \mathbb Q_p$ with $p\neq 2,3$.  Assume $\pi=p$. Then  $\ker({{f_2}})\cap  \ker({g_2})=\{0\}$.
\end{itemize}
\end{lemma}
\begin{proof}
The first point is clear, using in particular the freeness of $H$ as a $k[\zeta]$-module.
For the second point:  we saw in \S\ref{sec:zetatorsion}\textbf{B)} that
$\zeta$ centralizes the elements in $E^d$, therefore
 $\ker({g_d})= E^d$ and the kernel of ${f_d}$ coincides with the kernel of the action of $\zeta^2-1$ on $E^d$. But
$E^d$ is the direct sum of its one-dimensional subspace of $(\zeta-1)$-torsion   and of its subspace  $\zeta$-torsion. So $\ker({f_d})$ coincides with the subspace of
$(\zeta-1)$-torsion and is isomorphic to $\chi_{triv}$ as a left (resp. right) $H$-module.

The last two points  come from the fact that  for  any $i$ the space
$\ker(f_i)\cap  \ker(g_i)$ is contained in the $\zeta^2-1$ torsion space in $E^i$. But for  $i=1,2$ and under the respective hypotheses, this torsion space is  trivial  by Lemmas \ref{lemma:notorsion} and
 \ref{lemma:notorsion2}.
\end{proof}

\subsection{Structure of $E^1$}

 We suppose that $G={\rm SL}_2(\mathbb Q_p)$ with $p\neq 2,3$  and we choose $\pi=p$.
Here we focus on the graded piece $E^1$  and work with the endomorphisms of $H$-bimodules
$$
{f_1}:= \zeta\cdot \id_{E^1}\cdot \, \zeta -  \id_{E^1}:  c\mapsto \zeta \cdot c\cdot \zeta- c
$$
and
$$
{g_1}:= \zeta\cdot \id_{E^1} -  \id_{E^1}\cdot \,\zeta: c\mapsto \zeta \cdot c- c\cdot \zeta.
$$

\subsubsection{\label{subsubsec:kerg1}On $\ker(g_1)$}
In Prop.\ \ref{F1part} we  established the injectivity  of the $H$-bimodule homomorphism
\begin{equation}\label{f:isokerg}
  f_{(\mathbf{x}_0,\mathbf{x}_1)} : F^1 H \longrightarrow \ker(g_1) \ .
\end{equation}

\begin{proposition}\label{prop:isokerg}
Assume $G={\rm SL _2}(\mathbb Q_p)$ with $p\neq 2,3$   and  $\pi=p$. The map \eqref{f:isokerg} is bijective, so $\ker(g_1)$ is isomorphic to $F^1H$ as an $H$-bimodule. In particular,
   as a left (resp. right) $k[\zeta]$-module, $\ker(g_1)$ is free of rank $4(p-1)$.
\end{proposition}
\begin{proof}
It is immediate from Prop.\ \ref{F1part}-i that  $E^1 = \im(f_{(\mathbf{x}_0,\mathbf{x}_1)}) \oplus h_\pm ^1(\widetilde W)$. Therefore we only need to check that $g_1$ is injective on $h_\pm ^1(\widetilde W)$. From \S\ref{subsubsec:fil} we know that, for  $n\geq 0$, we have
\begin{equation*}
  \zeta\cdot  F_n E^1 +  F_n E^1\cdot\zeta \subseteq F_{n+2} E^1 \quad\text{and hence}\quad  g_1(F_n E^1) \subseteq F_{n+2} E^1 \ .
\end{equation*}
But Lemma \ref{lemma:calculsg} tells us that modulo $F_{\ell(w)+1} E^1$ we have
\begin{equation*}
  g_1(  (c^-, 0, c^+)_w) \equiv
  \begin{cases}
    (0, 0, c^+)_{s_0 s_1 w} - (c^-, 0, 0)_{s_0 s_1 w} & \text{if $w\in W^{1,\ell \geq 1}$},  \\
    (c^-,0 , 0)_{s_1 s_0 w} - (0, 0, c^+)_{s_1 s_0 w} & \text{if $w\in W^{0,\ell \geq 1}$},  \\
    (0, 0, c^+)_{s_0 s_1 \omega} + (c^-, 0, 0)_{s_1 s_0 \omega} - (c^-, 0, 0)_{s_0 s_1 \omega} - (0, 0, c^+)_{s_1 s_0 \omega} & \text{if $w = \omega \in \Omega$}.
  \end{cases}
\end{equation*}
This shows that $g_1$ is injective on $h_\pm ^1(\widetilde W)$.
\end{proof}

\begin{remark}\label{rem:centralizer E1}
  The above proposition implies in particular that $\ker(g_1)$ is the centralizer in $E^1$ of the full center $Z$ of $H$.
\end{remark}
\subsubsection{On $\ker(f_1)$}
In Prop.\ \ref{plusminus-part} we  introduced and established the injectivity of the $H_\zeta$-bimodule homomorphism
\begin{equation}\label{f:isohpm}
  f^\pm : (H_\zeta/ H_\zeta \tau_{s_0} \oplus H_\zeta / H_\zeta \tau_{s_1})^\pm \longrightarrow \ker(f_1) \ .
\end{equation}
To show that this map is actually also  surjective we need to introduce the vector subspace $\mathfrak{V} \subseteq E^1$ with basis
\begin{align}\label{f:basisV}
x := e_{\id} \cdot (0,0,\c)_{1} \cdot  e_{\id^{-1}}, & \quad e_{\id} \cdot (0,0,\c)_{s_1} \cdot e_{\id}= x \cdot \tau_{s_1}, \\
y := e_{\id^{-1}} \cdot (\c,0,0)_{1} \cdot  e_{\id}, & \quad e_{\id^{-1}}\cdot (\c,0,0)_{s_0} \cdot e_{\id^{-1}} = y \cdot \tau_{s_0}. \nonumber
\end{align}
Temporarily we put
\begin{equation*}
  \mathfrak{U} := \mathfrak{V} + \im(f_{(\mathbf{x}_0,\mathbf{x}_1)}) + \im(f^\pm) \ .
\end{equation*}
But note that $\im(f_{(\mathbf{x}_0,\mathbf{x}_1)}) + \im(f^\pm) = \im(f_{(\mathbf{x}_0,\mathbf{x}_1)}) \oplus \im(f^\pm)$ by Lemma \ref{lemma:cap}-iii.

\begin{lemma}\label{lemma:zetaV}
   We have:
\begin{itemize}
  \item[a)] $(x\cdot  \tau_{s_1})\cdot \tau_{s_1} = 0$ and $(y\cdot  \tau_{s_0})\cdot \tau_{s_0} = 0$;
  \item[b)] $x\cdot  \tau_{s_0} = 0$ and $y \cdot \tau_{s_1} = 0$;
  \item[c)] $\tau_{s_0} \cdot x = 0 = \tau_{s_0}\cdot  (x\cdot  \tau_{s_1})$ and $\tau_{s_1} \cdot y = 0 = \tau_{s_1}\cdot  (y \cdot \tau_{s_0})$;
  \item[d)] $\tau_{s_1} \cdot x = y\cdot  \tau_{s_0} + e_{\id^{-1}} \tau_{s_1} f_{{\bf x}^+}(1)$,  $\tau_{s_0} \cdot  y = x \cdot  \tau_{s_1} + e_{\id} \tau_{s_0}\cdot f_{{\bf x}^-}(1)$;
  \item[e)] $\zeta\cdot  x - x = e_{\id} \tau_{s_0 s_1} \cdot f_{{\bf x}^+}(1) + 2 e_{\id} \cdot f_{(\mathbf{x}_0,\mathbf{x}_1)}(\tau_{s_0})$;
  \item[f)] $\zeta\cdot  y - y =  e_{\id^{-1}} \tau_{s_1 s_0} \cdot  f_{{\bf x}^-}(1) + 2 e_{\id^{-1}} \cdot f_{(\mathbf{x}_0,\mathbf{x}_1)}(\tau_{s_1})$;
  \item[g)] $x\cdot  \zeta - x, y\cdot  \zeta - y \in \im(f_{(\mathbf{x}_0,\mathbf{x}_1)}) \oplus \im(f^\pm)$;
  \item[h)] $(x \cdot \tau_{s_1}) \cdot \tau_{s_0} - x , (y\cdot  \tau_{s_0}) \cdot \tau_{s_1} - y \in \im(f_{(\mathbf{x}_0,\mathbf{x}_1)}) \oplus \im(f^\pm)$;
  \item[i)] $\tau_{s_1} \cdot (x\cdot  \tau_{s_1}) - y \in \im(f_{(\mathbf{x}_0,\mathbf{x}_1)}) \oplus \im(f^\pm)$, $\tau_{s_0} \cdot (y \cdot \tau_{s_0}) - x \in \im(f_{(\mathbf{x}_0,\mathbf{x}_1)}) \oplus \im(f^\pm)$;
  \item[j)] $\mathfrak{U}$ is a sub-$H$-bimodule of $E^1$.
\end{itemize}
\end{lemma}
\begin{proof}
a) is obvious. For the subsequent computations it is useful to note that we have
\begin{equation*}
  x = e_{\id}\cdot  (0,0,\c)_{1} ,\ x \cdot \tau_{s_1} = e_{\id} \cdot (0,0,\c)_{s_1} ,\ y =  e_{\id^{-1}}\cdot  (\c,0,0)_{1} ,\ y \cdot \tau_{s_0} =  e_{\id^{-1}}\cdot  (\c,0,0)_{s_0} \ .
\end{equation*}
We also recall that $\im(f_{(\mathbf{x}_0,\mathbf{x}_1)}) \oplus \im(f^\pm)$ is a sub-$H$-bimodule of $E^1$.

Points b), c), d), e), and f) are a straightforward computation based on the formulas in Prop.\ \ref{prop:theformulas}-g) follows from e) and f) by applying $\anti$. By b) we have $x \cdot  \zeta = (x \cdot  \tau_{s_1}) \cdot \tau_{s_0}$ and $y \cdot  \zeta = (y\cdot  \tau_{s_0}) \cdot  \tau_{s_1}$; hence h) follows from g). i) follows from d) and h). j) follows from a) - d), h), and i).
\end{proof}

\begin{lemma}\label{lemma:V}
   We have $E^1 = \im(f_{(\mathbf{x}_0,\mathbf{x}_1)}) \oplus \im(f^\pm) \oplus \mathfrak V = \ker(g_1) \oplus \im(f^\pm) \oplus \mathfrak V$.
\end{lemma}
\begin{proof}
We remind the reader of the following consequences of \eqref{f:leftomega} which we will silently use in the following:
\begin{equation*}
  e_{\id}\cdot  (0,0,\mathbf{c})_{\omega_u w} = u^{-1} e_{\id} \cdot  (0,0,\mathbf{c})_w  \quad\text{and}\quad   e_{\id^{-1}}\cdot (\mathbf{c},0,0)_{\omega_u w} = u e_{\id^{-1}}\cdot (\mathbf{c},0,0)_w
\end{equation*}
for any $w \in \widetilde{W}$ and $u \in \mathbb{F}_p^\times$. We also recall, using \eqref{f:condeven} and \eqref{f:condodd}  that
\begin{equation*}
  x = e_{\id} \cdot  (0,0,\c)_{1} ,\ y =  e_{\id^{-1}} \cdot  (\c,0,0)_{1} ,\ x \cdot \tau_{s_1} = e_{\id}\cdot  (0,0,\c)_{s_1} ,\ y \cdot \tau_{s_0} = e_{\id^{-1}}\cdot  (\c,0,0)_{s_0} \ .
\end{equation*}
Prop.\ \ref{F1part}-i tells us that
\begin{align}\label{f:imfxx}
  \im(f_{(\mathbf{x}_0,\mathbf{x}_1)}) & = h_0^1(\widetilde{W}^{\ell \geq 2}) \oplus
         (\oplus_{u \in \mathbb{F}_p^\times} k \big( (0,\mathbf{c}\iota,0)_{s_1 \omega_u} - u^{-1} e_{\id} \cdot (0,0,\mathbf{c})_1 \big))  \nonumber \\
   &   \qquad\qquad\qquad\qquad\qquad\qquad\qquad\qquad   \oplus (\oplus_{u \in \mathbb{F}_p^\times} k \big( (0,\mathbf{c}\iota,0)_{s_0 \omega_u} + u e_{\id^{-1}} \cdot (\mathbf{c},0,0)_1 \big))  \\
        & =  h_0^1(\widetilde{W}^{\ell \geq 2}) \oplus (\oplus_{u \in \mathbb{F}_p^\times} k \big( (0,\mathbf{c}\iota,0)_{s_1 \omega_u} - u^{-1} x \big)) \oplus (\oplus_{u \in \mathbb{F}_p^\times} k \big( (0,\mathbf{c}\iota,0)_{s_0 \omega_u} + u y \big)) .  \nonumber
\end{align}
This implies
\begin{equation}\label{f:directsum1}
  \mathfrak{U} \supseteq \im(f_{(\mathbf{x}_0,\mathbf{x}_1)}) \oplus k x \oplus k y = h_0^1(\widetilde{W}) \oplus k x \oplus k y \ .
\end{equation}
Next we observe that, by Lemma \ref{lemma:righteasy}, we have
\begin{align*}
   & (0,0,\mathbf{c})_w = (0,0,\mathbf{c})_1 \cdot\tau_w  \quad\text{and}\quad  e_{\id} \cdot(0,0,\mathbf{c})_w = x\cdot \tau_w    \quad\text{for any $w \in \widetilde{W}^0$, and}  \\
   & (\mathbf{c},0,0)_w = (\mathbf{c},0,0)_1 \cdot \tau_w  \quad\text{and}\quad  e_{\id^{-1}} \cdot (\mathbf{c},0,0)_w = y \cdot\tau_w    \quad\text{for any $w \in \widetilde{W}^1$}.
\end{align*}
Furthermore, Prop.\ \ref{prop:theformulas} implies
\begin{align*}
  & (0,0,\mathbf{c})_w =
  \begin{cases}
  \tau_w \cdot(0,0,\mathbf{c})_1 \\
  - \tau_w\cdot (\mathbf{c},0,0)_1
  \end{cases}
  \text{and}\quad
  e_{\id} \cdot(0,0,\mathbf{c})_w  =
  \begin{cases}
  \tau_w \cdot x  & \text{if $w = (s_0 s_1)^m$ with $m \geq 0$},  \\
  - \tau_w \cdot y  & \text{if $w = (s_0 s_1)^m s_0$ with $m \geq 0$},
  \end{cases}          \\
  & (\mathbf{c},0,0)_w  =
  \begin{cases}
  \tau_w\cdot  (\mathbf{c},0,0)_1  \\
  - \tau_w \cdot (0,0,\mathbf{c})_1
  \end{cases}
  \text{and}\quad
   e_{\id^{-1}} \cdot (\mathbf{c},0,0)_w  =
   \begin{cases}
  \tau_w\cdot  y  & \text{if $w = (s_1 s_0)^m$ with $m \geq 0$},  \\
  - \tau_w \cdot x  & \text{if $w = (s_1 s_0)^m s_1$ with $m \geq 0$}.
  \end{cases}
\end{align*}
It follows that
\begin{align}
  H \cdot \big( k (0,0,\mathbf{c})_1 \oplus k (\mathbf{c},0,0)_1 \big) \cdot H & \supseteq h_-^1(\widetilde{W}) \oplus h_+^1(\widetilde{W}) \qquad\text{and}  \label{f:directsum2} \\
     \mathfrak{U} \supseteq H \cdot \mathfrak{V} \cdot H & \supseteq e_{\id^{-1}} h_-^1(\widetilde{W}) \oplus e_{\id} h_+^1(\widetilde{W})   \ . \label{f:directsum3}
\end{align}
By looking at the definition of $\mathbf{x}^\pm$ and using \eqref{f:directsum1} and \eqref{f:directsum3} we see that $(0,0,\mathbf{c})_1 , (\mathbf{c},0,0)_1 \in \mathfrak{U}$. So \eqref{f:directsum2} implies that $h_-^1(\widetilde{W}) \oplus h_+^1(\widetilde{W}) \subseteq \mathfrak{U}$, and together with \eqref{f:directsum1} we obtain $\mathfrak{U} = E^1$.

It remains to check that $\mathfrak{V} \cap (\im(f_{(\mathbf{x}_0,\mathbf{x}_1)}) \oplus \im(f^\pm)) = 0$. If $z = r_1 x + r_2 y + r_3 x \tau_{s_1} + r_4 y \tau_{s_0} \in \mathfrak{V}$ with $r_i \in k$ is an arbitrary element then $e_{\id}\cdot z \cdot e_{\id^{-1}} = r_1 x$, $e_{\id^{-1}}\cdot  z \cdot e_{\id} = r_2 y$, $e_{\id} \cdot z \cdot e_{\id} = r_3 x\cdot  \tau_{s_1}$, and $e_{\id^{-1}}\cdot z \cdot e_{\id^{-1}} = r_4 y \cdot \tau_{s_0}$. Hence it suffices to show that none of the elements $x$, $y$, $x \cdot \tau_{s_1}$, $y \cdot \tau_{s_0}$ is contained in $\im(f_{(\mathbf{x}_0,\mathbf{x}_1)}) \oplus \im(f^\pm)$. Obviously we need to check this only for $x\cdot \tau_{s_1}$ and $y\cdot \tau_{s_0}$. Using \eqref{f:id-id} we deduce from \eqref{f:imfxx} that
\begin{equation}\label{f:idid-1}
  e_{\id} \cdot\im(f_{(\mathbf{x}_0,\mathbf{x}_1)})\cdot e_{\id} = e_{\id} \cdot h_0^1(\widetilde{W}^{\ell \geq 2}) \cdot e_{\id}  \quad\text{and}\quad e_{\id^{-1}} \cdot  \im(f_{(\mathbf{x}_0,\mathbf{x}_1)}) \cdot e_{\id^{-1}} = e_{\id^{-1}} \cdot h_0^1(\widetilde{W}^{\ell \geq 2})\cdot e_{\id^{-1}} \ .
\end{equation}
On the other hand we have $\im(f_{{\bf x}^-}) = (H/H\tau_{s_1}) \cdot   \mathbf{x}^- \cdot   \zeta^{\mathbb{N}_0}$ by Lemma \ref{lemma:leftright}, Remark \ref{rem:offdiag-Gamma-anti}-4, and the definition of $\mathbf{x}^-$ before \eqref{f:defifpm} and $\im(f_{{\bf x}^+}) = (H/H\tau_{s_0})  \cdot  \mathbf{x}^+ \cdot  \zeta^{\mathbb{N}_0}$. Using \eqref{f:id-id} we check that
\begin{equation}\label{f:x-y-id}
  \mathbf{x}^+ \cdot   e_{\id} = e_{\id^3}  \cdot  (0,0,\mathbf{c})_1 ,\  \mathbf{x}^+  \cdot  e_{\id^{-1}} = e_{\id}  \cdot  \mathbf{x}^+ ,\ \mathbf{x}^-  \cdot  e_{\id} = e_{\id^{-1}}  \cdot  \mathbf{x}^- ,\ \mathbf{x}^-  \cdot  e_{\id^{-1}} = e_{\id^{-3}}  \cdot  (\mathbf{c},0,0)_1 \ .
\end{equation}
For $i = 0, 1$ we temporarily put $H^{even/odd,i} := \oplus_{\ell(w)\, even/odd,\, \ell(w s_i) = \ell(w) + 1} k \tau_w$, so  $H = H^{even,i} \oplus H^{odd,i} \oplus H\tau_{s_i}$.

We claim that
\begin{equation}\label{f:idid-2}
  e_{\id}  \cdot \im(f_{{\bf x}^+})  \cdot  e_{\id} \subseteq e_{\id}  \cdot  h_-^1(\widetilde{W}) \cdot   e_{\id}   \quad\text{and}\quad  e_{\id^{-1}}  \cdot  \im(f_{{\bf x}^-}) \cdot   e_{\id^{-1}} \subseteq e_{\id^{-1}}  \cdot  h_+^1(\widetilde{W})  \cdot  e_{\id^{-1}} \ .
\end{equation}
For any $n \geq 0$ we compute
\begin{gather*}
  e_{\id}  (H/H\tau_{s_0})  \cdot  \mathbf{x}^+  \cdot \zeta^n e_{\id}  \qquad\qquad\qquad\qquad\qquad\qquad\qquad\qquad\qquad\qquad\qquad\qquad\qquad\qquad\qquad\qquad  \\
\begin{split}
   & = e_{\id} (H/H\tau_{s_0})  \cdot  \mathbf{x}^+ \cdot  (\tau_{s_0} \tau_{s_1} + \tau_{s_1} \tau_{s_0})^n  e_{\id}  \\
   & =  e_{\id} (H/H\tau_{s_0})  \cdot \mathbf{x}^+  \cdot (\tau_{s_1} \tau_{s_0})^n  e_{\id} \quad \text{by Remark \ref{rem:offdiag-Gamma-anti}}     \\
   & =  e_{\id} (H/H\tau_{s_0}) \cdot  \mathbf{x}^+  \cdot  e_{\id} (\tau_{s_1} \tau_{s_0})^n = e_{\id} (H/H\tau_{s_0}) e_{\id^3}  \cdot (0,0,\mathbf{c})_1  \cdot (\tau_{s_1} \tau_{s_0})^n \quad \text{by \eqref{f:x-y-id} }   \\
   & = e_{\id} (H/H\tau_{s_0}) e_{\id^3}  \cdot (0,0,\mathbf{c})_{(s_1 s_0)^n}  \quad\text{by Lemma \ref{lemma:righteasy}-i}  \\
   & = e_{\id} H^{even,0} e_{\id^3}  \cdot (0,0,\mathbf{c})_{(s_1 s_0)^n} + e_{\id} H^{odd,0} e_{\id^3}  \cdot (0,0,\mathbf{c})_{(s_1 s_0)^n}  \\
   & = H^{even,0} e_{\id} e_{\id^3} \cdot  (0,0,\mathbf{c})_{(s_1 s_0)^n} + H^{even,1} e_{\id^{-3}} e_{\id}  \tau_{s_1}  \cdot (0,0,\mathbf{c})_{(s_1 s_0)^n}  \\
   & = H^{even,1} e_{\id^{-3}} e_{\id}  \tau_{s_1}  \cdot (0,0,\mathbf{c})_{(s_1 s_0)^n}  \quad\text{since $\id^3 \neq \id$}   \\
   & = H^{even,1} e_{\id^{-3}} e_{\id}  \cdot (-\mathbf{c},0,0)_{s_1(s_1 s_0)^n}  \quad\text{by Prop.\ \ref{prop:theformulas}}   \\
   & = e_{\id^{-3}} e_{\id} H^{even,1}  \cdot (-\mathbf{c},0,0)_{s_1(s_1 s_0)^n} \ .
\end{split}
\end{gather*}
If $p > 5$ then $e_{\id^{-3}} \neq e_{\id}$ and we obtain $e_{\id} \cdot  \im(f_{{\bf x}^+})  \cdot  e_{\id} = 0$. In general we deduce from Prop.\ \ref{prop:theformulas} that we have
\begin{equation*}
  e_{\id} \tau_{(s_1 s_0)^n} \cdot  (\mathbf{c},0,0)_w = e_{\id} \cdot  (\mathbf{c},0,0)_{(s_1 s_0)^n w}   \qquad\text{for any $n \geq 0$ and any $w \in \widetilde{W}$}.
\end{equation*}
This establishes the inclusion $e_{\id} \cdot  \im(f_{{\bf x}^+})  \cdot e_{\id} \subseteq e_{\id}  \cdot h_-^1(\widetilde{W}) e_{\id}$. The computation for the other inclusion of \eqref{f:idid-2} is entirely analogous.

Next we claim that
\begin{equation}\label{f:idid-3}
   e_{\id^{-1}} \cdot   \im(f_{{\bf x}^+})  \cdot  e_{\id^{-1}} \subseteq e_{\id^{-1}}  \cdot  h_0^1(\widetilde{W}) e_{\id^{-1}}   \quad\text{and}\quad  e_{\id} \cdot  \im(f_{{\bf x}^-})  \cdot e_{\id} \subseteq e_{\id}  \cdot h_0^1(\widetilde{W})  \cdot e_{\id} \ .
\end{equation}
We have
\begin{align*}
  e_{\id^{-1}}\cdot  \im(f_{{\bf x}^+})\cdot  e_{\id^{-1}} & =  e_{\id^{-1}} (H/H\tau_{s_0})\cdot  \mathbf{x}^+ \cdot \zeta^{\mathbb{N}_0} e_{\id^{-1}} =  e_{\id^{-1}} (H/H\tau_{s_0}) \cdot \mathbf{x}^+\cdot  e_{\id^{-1}} \zeta^{\mathbb{N}_0} e_{\id^{-1}}   \\
   & =  e_{\id^{-1}} (H/H\tau_{s_0}) e_{\id}  \cdot \mathbf{x}^+\cdot   \zeta^{\mathbb{N}_0} e_{\id^{-1}}   \quad\text{by \eqref{f:x-y-id}}  \\
   & =  e_{\id^{-1}} H^{odd,0} e_{\id}  \cdot \mathbf{x}^+\cdot   \zeta^{\mathbb{N}_0} e_{\id^{-1}}   \\
   & = \sum_{m \geq 0} k e_{\id^{-1}} \tau_{s_1} (\tau_{s_0} \tau_{s_1})^m  \cdot \mathbf{x}^+\cdot   \zeta^{\mathbb{N}_0} e_{\id^{-1}}   \\
   & = \sum_{m \geq 0} k e_{\id^{-1}} \tau_{s_1} \zeta^m  \cdot \mathbf{x}^+\cdot   \zeta^{\mathbb{N}_0} e_{\id^{-1}} \quad\text{since $\tau_{s_0}  \cdot \mathbf{x}^+   = 0$} \\
   & = \sum_{m \geq 0} k \zeta^m e_{\id^{-1}} \tau_{s_1}  \cdot \mathbf{x}^+\cdot   \zeta^{\mathbb{N}_0} e_{\id^{-1}} \ .
\end{align*}
Therefore we must show that $\zeta^m e_{\id^{-1}} \tau_{s_1}  \cdot \mathbf{x}^+\cdot   \zeta^n e_{\id^{-1}} \in e_{\id^{-1}} \cdot  h_0^1(\widetilde{W})\cdot  e_{\id^{-1}}$ for any $m \geq 0$ and $n \geq 0$.  But since $\bf{x}^+ = \zeta \cdot \bf{x}^+ \cdot \zeta$ it suffices to consider $n \geq 1$. Using Prop.\ \ref{prop:theformulas} and \eqref{f:lefts2} we first compute $e_{\id^{-1}} \tau_{s_1}  \cdot \mathbf{x}^+= - e_{\id^{-1}}\cdot  ((\mathbf{c},0,0)_{s_1} + (\mathbf{c},0,0)_{s_0})$. Next, using \eqref{f:lefts2}, Prop.\ \ref{prop:theformulas}, and Lemma \ref{lemma:antievenodd}, we compute
\begin{align*}
  (\mathbf{c},0,0)_{s_1}\cdot  \zeta^n e_{\id^{-1}} & = \anti(e_{\id} \zeta^n\cdot  \anti((\mathbf{c},0,0)_{s_1})) = \anti(e_{\id} \zeta^n \cdot (0,0,-\mathbf{c})_{s_1^{-1}}) = \anti(e_{\id} \zeta^n\cdot  (0,0,\mathbf{c})_{s_1})    \\
   & = \anti(e_{\id} \zeta^{n-1}\cdot ((0,-2\mathbf{c}\iota,0)_{s_0 s_1} + (0,0,\mathbf{c})_{s_0 s_1^2})) = - \anti(e_{\id} \zeta^{n-1}\cdot ((0,2\mathbf{c}\iota,0)_{s_0 s_1} + (0,0,\mathbf{c})_{s_0}))   \\
   & = - \anti(e_{\id}\cdot  ((0,2\mathbf{c}\iota,0)_{(s_0 s_1)^n} + (0,0,\mathbf{c})_{(s_0 s_1)^{n-1}s_0}))   \\
   & = - ((0,2\mathbf{c}\iota,0)_{(s_1 s_0)^n} - (\mathbf{c},0,0)_{(s_0 s_1)^{n-1}s_0 s_0^2}) \cdot e_{\id^{-1}}
\end{align*}
and
\begin{align*}
  (\mathbf{c},0,0)_{s_0}\cdot  \zeta^n e_{\id^{-1}} & = \anti(e_{\id} \zeta^n \cdot\anti((\mathbf{c},0,0)_{s_0})) = \anti(e_{\id} \zeta^n \cdot(0,0,-\mathbf{c})_{s_0^{-1}}) = \anti(e_{\id} \zeta^n \cdot (0,0,\mathbf{c})_{s_0}) \\
   & = \anti(e_{\id}  \cdot(0,0,\mathbf{c})_{(s_0 s_1)^n s_0}) = - (\mathbf{c},0,0)_{(s_0 s_1)^n s_0 s_0^2}\cdot e_{\id^{-1}} \ .
\end{align*}
It follows that $e_{\id^{-1}} \tau_{s_1}  \cdot \mathbf{x}^+\cdot   \zeta^n e_{\id^{-1}} = e_{\id^{-1}} \cdot(0,2\mathbf{c}\iota,0)_{(s_1 s_0)^n} \cdot e_{\id^{-1}}$ and then
\begin{equation*}
  \zeta^m e_{\id^{-1}} \tau_{s_1}  \cdot \mathbf{x}^+\cdot   \zeta^n e_{\id^{-1}} = e_{\id^{-1}}\cdot (0,2\mathbf{c}\iota,0)_{(s_1 s_0)^{m+n}} \cdot e_{\id^{-1}} \in e_{\id^{-1}}\cdot h_0^1(\widetilde{W})\cdot e_{\id^{-1}}
\end{equation*}
by Cor.\ \ref{coro:form-zeta}. Again the other inclusion is shown analogously or by applying $\Gamma_\varpi$. We conclude that
\begin{equation*}
  e_{\id} \cdot (\im(f_{(\mathbf{x}_0,\mathbf{x}_1)}) \oplus \im(f^\pm)) \cdot e_{\id} \subseteq e_{\id} \cdot  (h_-^1(\widetilde{W}) \oplus h_0^1(\widetilde{W})) \cdot e_{\id} \quad\text{whereas $x\cdot \tau_{s_1} \in e_{\id} \cdot h_+^1(s_1) \cdot e_{\id}$}
\end{equation*}
and
\begin{equation*}
  e_{\id^{-1}} \cdot (\im(f_{(\mathbf{x}_0,\mathbf{x}_1)}) \oplus \im(f^\pm)) \cdot e_{\id^{-1}} \subseteq e_{\id^{-1}}\cdot  (h_0^1(\widetilde{W}) \oplus h_+^1(\widetilde{W}))\cdot  e_{\id^{-1}} \quad\text{whereas $y\cdot \tau_{s_0} \in e_{\id^{-1}} \cdot h_-^1(s_0) \cdot e_{\id^{-1}}$}.
\end{equation*}
This concludes the proof of the first equality of Lemma \ref{lemma:V}. The second equality then follows from Prop.\ \ref{prop:isokerg}.
\end{proof}
\begin{proposition}\label{prop:kerf1star}
   Suppose $G = {\rm SL _2}(\mathbb Q_p)$ with $p\neq 2,3$ and $\pi=p$. We have:
\begin{itemize}
  \item[i.] The map $f^{\pm}$ described in  \eqref{f:isohpm} is bijective;
  \item[ii.] $f_1 \circ g_1 = g_1 \circ f_1 = 0$ on $E^1$.
\end{itemize}
In particular (cf.\ Remark \ref{rem:offdiag-anti}), as a left (resp.\ right) $k[\zeta^{\pm 1}]$-module, $\ker(f_1)$ is free of rank $4(p-1)$.
\end{proposition}
\begin{proof}
By \cite{embed} Remark 3.2.ii we have
\begin{equation*}
  \zeta \tau_w \zeta = \zeta^2 \tau_w =
  \begin{cases}
  \tau_{(s_0 s_1)^2 w}  & \text{if $w \in \widetilde{W}^{1,\ell \geq 1}$},  \\
  \tau_{(s_1 s_0)^2 w}  & \text{if $w \in \widetilde{W}^{0,\ell \geq 1}$}.
  \end{cases}
\end{equation*}
We deduce that
\begin{align*}
  f_1 (f_{(\mathbf{x}_0,\mathbf{x}_1)}(\tau_w)) & = f_{(\mathbf{x}_0,\mathbf{x}_1)}(\zeta^2 \tau_w) - f_{(\mathbf{x}_0,\mathbf{x}_1)}(\tau_w)  \\
  & =
  \begin{cases}
  f_{(\mathbf{x}_0,\mathbf{x}_1)}(\tau_{(s_0 s_1)^2 w}) - f_{(\mathbf{x}_0,\mathbf{x}_1)}(\tau_w) & \text{if $w \in \widetilde{W}^{1,\ell \geq 1}$},  \\
  f_{(\mathbf{x}_0,\mathbf{x}_1)}(\tau_{(s_1 s_0)^2 w}) - f_{(\mathbf{x}_0,\mathbf{x}_1)}(\tau_w)  & \text{if $w \in \widetilde{W}^{0,\ell \geq 1}$}
  \end{cases}
\end{align*}
and, using Prop.\ \ref{F1part}-i, see that
\begin{equation}\label{f:F1-length}
  f_1 (f_{(\mathbf{x}_0,\mathbf{x}_1)}(\tau_w)) \in
  \begin{cases}
  k^\times (0,\mathbf{c}\iota,0)_{(s_0 s_1)^2 w} + F_{\ell(w)+3}E^1  & \text{if $w \in \widetilde{W}^{1,\ell \geq 1}$},  \\
  k^\times (0,\mathbf{c}\iota,0)_{(s_1 s_0)^2 w} + F_{\ell(w)+3}E^1  & \text{if $w \in \widetilde{W}^{0,\ell \geq 1}$}.
  \end{cases}
\end{equation}
On the other hand we observe that
\begin{align*}
  f_1(x) & = \zeta \cdot x  \cdot  \zeta - x = e_{\id}  \cdot  (0,0,\mathbf{c})_{s_0 s_1}   \qquad\text{by Lemma \ref{lemma:zetaV}-i and Prop.\ \ref{prop:theformulas}} \\
   & = - e_{\id} \cdot  (0,2\mathbf{c}\iota,0)_{s_0 s_1 s_0} - e_{\id} \cdot  (0,2\mathbf{c}\iota,0)_{s_0}  \qquad\text{by \eqref{f:Vzeta}}  \\
   & \in F_3 E^1 \cap \im(f_{(\mathbf{x}_0,\mathbf{x}_1)})  \qquad\text{by Prop.\ \ref{F1part}-i}
\end{align*}
and, by an analogous computation, $f_1(y) \in F_3 E^1 \cap \im(f_{(\mathbf{x}_0,\mathbf{x}_1)})$ as well. By Prop.\ \ref{prop:yo} we conclude that $f_1(\mathfrak{V}) \subseteq F_4 E^1 \cap \im(f_{(\mathbf{x}_0,\mathbf{x}_1)}) = F_4 E^1 \cap \ker(g_1)$ using Prop.\ \ref{prop:isokerg}. This together with \eqref{f:F1-length} shows that $f_1$ is injective on
$\im(f_{(\mathbf{x}_0,\mathbf{x}_1)}) \oplus  \mathfrak{V}$. Lemma \ref{lemma:V} then implies that $\im(f^\pm) = \ker(f_1)$, which establishes Point i of the proposition. Furthermore,  we have $f_1(\ker(g_1)) \subseteq \ker(g_1)$ since $f_1$ and $g_1$ commute (Remark \ref{rema:calcufg}-i). The fact that $f_1(\mathfrak{V}) \subseteq \ker(g_1)$ then shows, again invoking Lemma \ref{lemma:V}, that $f_1 (E^1) \subseteq \ker(g_1)$ which amounts to our assertion ii.
\end{proof}

\begin{remark}\label{rema:1-gamma0pm}
  $(1 - e_{\gamma_0}) \cdot  \ker(f_1) = (1 - e_{\gamma_0})\cdot   h_\pm^1(\widetilde{W})$.
\end{remark}
\begin{proof}
We deduce from Cor.\ \ref{coro:form-zeta} that left multiplication by $\zeta$ preserves $(1 - e_{\gamma_0}) \cdot  h_\pm^1(\widetilde{W})$ as well as $h_\pm^1(\widetilde{W}) \cdot  (1 - e_{\gamma_0})$; for the latter use in addition that $e_{\gamma_0}$ centralizes $h_0^1(\widetilde{W})$ by \eqref{f:id-id}. Applying $\anti$, which preserves $h_\pm^1(\widetilde{W})$ by Lemma \ref{lemma:antievenodd}, one sees that also right multiplication by $\zeta$ preserves $(1 - e_{\gamma_0}) \cdot  h_\pm^1(\widetilde{W})$. We now compute
\begin{align*}
  (1 - e_{\gamma_0})\cdot   \ker(f_1) & = (1 - e_{\gamma_0})\cdot   \im(f^\pm)  \quad\text{by Prop.\ \ref{prop:kerf1star}-i}  \\
   & = (1 - e_{\gamma_0}) H  \cdot \mathbf{x}^-\cdot   \zeta^{\mathbb{N}} + (1 - e_{\gamma_0}) H  \cdot \mathbf{x}^+\cdot   \zeta^{\mathbb{N}}  \\
   & = H (1 - e_{\gamma_0})  \cdot \mathbf{x}^-\cdot   \zeta^{\mathbb{N}} + H (1 - e_{\gamma_0})  \cdot \mathbf{x}^+\cdot   \zeta^{\mathbb{N}}   \\
   & =  H (1 - e_{\gamma_0}) \cdot  (\mathbf{c},0,0)_1 \cdot  \zeta^{\mathbb{N}} + H (1 - e_{\gamma_0}) \cdot (0,0,\mathbf{c})_1 \cdot \zeta^{\mathbb{N}}  \\
   & =  (1 - e_{\gamma_0}) H\cdot   (\mathbf{c},0,0)_1 \cdot  \zeta^{\mathbb{N}} + (1 - e_{\gamma_0}) H \cdot  (0,0,\mathbf{c})_1 \cdot \zeta^{\mathbb{N}}  \\
   & =  (1 - e_{\gamma_0}) \cdot  h_-^1(\widetilde{W})\cdot   \zeta^{\mathbb{N}} + (1 - e_{\gamma_0})\cdot   h_+^1(\widetilde{W})\cdot  \zeta^{\mathbb{N}}   \quad\text{by \eqref{f:leftomega} and Prop.\ \ref{prop:theformulas}} \\
   & \subseteq (1 - e_{\gamma_0}) \cdot  h_\pm^1(\widetilde{W})   \quad\text{by the initial consideration}  .
\end{align*}
Since $(1 - e_{\gamma_0}) \cdot  \mathfrak{V} = 0$ we conclude from Lemma \ref{lemma:V} the right hand equality in
\begin{equation*}
  (1 - e_{\gamma_0}) \cdot h_0^1(\widetilde{W}) \oplus (1 - e_{\gamma_0})\cdot   h_\pm^1(\widetilde{W}) = (1 - e_{\gamma_0}) \cdot  E^1 = (1 - e_{\gamma_0}) \cdot  \im(f_{(\mathbf{x}_0,\mathbf{x}_1)})  \oplus (1 - e_{\gamma_0}) \cdot  \im(f^\pm)   \ .
\end{equation*}
The left hand summands are equal by Remark \ref{rema:1-gamma0F1}, of  the right hand summands one contains the other by the above calculation since $\im(f^\pm) =\ker(f_1)$. Hence the right hand summands must be equal as well.
\end{proof}

\subsubsection{Structure of $E^1$ as an $H$-bimodule.}
Recall the central idempotent $e_{\gamma_0} = e_{\id} + e_{\id^{-1}}$ in $H$.

\begin{proposition}\label{prop:structure1}
   Let $G={\rm SL}_2(\mathbb Q_p)$ with  $p\neq 2,3 $ and assume $\pi=p$. We have the following.
\begin{enumerate}
  \item As an $H$-bimodule, $E^1$ sits in an exact sequence of the form
\begin{equation*}
  0 \rightarrow \ker({f_1})\oplus \ker({g_1}) \rightarrow  E^1\rightarrow E^1/ \ker({f_1})\oplus \ker({g_1}) \rightarrow 0
\end{equation*}
where $E^1/ \ker({f_1})\oplus \ker({g_1})$ is  a $4$-dimensional $H$-bimodule.
  \item As a left (resp. right) $H$-module, $E^1/\ker({f_1})\oplus \ker({g_1})$ is isomorphic  to the direct sum of two copies of a simple $2$-dimensional left (resp. right) $H$-module on which $\zeta$ and $e_\go$ act by $1$.
  \item $E^1 / \ker(g_1)$ is an $H_\zeta$-bimodule.
\end{enumerate}
\end{proposition}
\begin{proof}
The first assertion follows from Lemma \ref{lemma:V} and Prop.\ \ref{prop:kerf1star}-i. As observed before we trivially have $f_1(\ker(g_1)) \subseteq \ker(g_1)$. Hence $f_1$ induces a well defined endomorphism of $E^1 /\ker(g_1)$. But Prop.\ \ref{prop:kerf1star}-ii implies that this latter map is actually  the zero map. It follows that $z \equiv \zeta \cdot z \cdot \zeta \bmod \ker(g_1)$ for any $z \in E^1$, which implies the third assertion.

It remains to determine the module structure of the $4$-dimensional quotient $E^1 / \ker({f_1})\oplus \ker({g_1})$ which has as a $k$-basis the cosets of $x$, $y$, $x\cdot  \tau_{s_1}$, and $y\cdot  \tau_{s_0}$. Obviously $e_\go$ acts by $1$ on these elements from the left and the right. It follows from Lemma \ref{lemma:zetaV} that $\zeta$ acts by $1$ from the left and the right on this quotient. The same lemma also implies that $x$ and $x\cdot  \tau_{s_1}$ generate a 2-dimensional right $H$-submodule in $E^1 / \ker({f_1})\oplus \ker({g_1})$. It is necessarily a simple module because the only one-dimensional modules on which $e_{\gamma_0}$ acts by $1$ are supersingular namely annihilated by $\zeta$ (see \eqref{f:quad}). Correspondingly one sees that $y$ and $y \cdot \tau_{s_0}$ generate another 2-dimensional simple right $H$-submodule in $E^1 / \ker({f_1})\oplus \ker({g_1})$. It is easy to check that these two simple right modules are isomorphic to each other via the map $x \mapsto y \cdot  \tau_{s_1}, x \cdot  \tau_{s_0} \mapsto y$. This proves in particular that $ E^1/\ker(f_1)+\ker(g_1)$ is semisimple isotypic as a right $H$-module, and therefore also as a left $H$-module using $\anti$.
\end{proof}

\subsection{Structure of $E^2$}

We still assume that $G={\rm SL}_2(\mathbb Q_p)$ with $p\neq 2,3$ and that $\pi=p$. Here we focus on the graded piece $E^2$ and work with the endomorphisms of $H$-bimodules
\begin{equation*}
  {f_2}:= \zeta\cdot \id_{E^1}\cdot \, \zeta -  \id_{E^2}:  c\mapsto \zeta \cdot c\cdot \zeta- c \quad\text{
and }\quad
{g_2}:= \zeta\cdot \id_{E^1} -  \id_{E^2}\cdot \,\zeta: c\mapsto \zeta \cdot c- c\cdot \zeta
\end{equation*}
as introduced in \S\ref{subsec:fg}.
By  Prop.\ \ref{prop:structure1} we have an exact sequence of $H$-bimodules
\begin{equation*}
  0 \longrightarrow \ker({f_1}) \oplus \ker({g_1}) \longrightarrow  E^1 \longrightarrow  E^1/(\ker({f_1}) \oplus \ker({g_1})) \longrightarrow 0
\end{equation*}
where $E^1/(\ker({f_1})\oplus \ker({g_1}))$ is  a $4$-dimensional $H$-bimodule. Passing to duals, this gives an exact sequence of $H$-bimodules
\begin{equation*}
  0 \longrightarrow (E^1/(\ker({f_1}) \oplus \ker({g_1})))^\vee \longrightarrow (E^1)^\vee \longrightarrow (\ker({f_1})\oplus \ker({g_1}))^\vee \longrightarrow 0 \ .
\end{equation*}
We define the sub-$H$-bimodules \begin{center}$(E^1)^\vee_{f_1} := \{ \xi \in (E^1)^\vee : \xi | \ker(g_1) = 0\}$ and $(E^1)^\vee_{g_1} := \{ \xi \in (E^1)^\vee : \xi | \ker(f_1) = 0\}$.\end{center} Then
\begin{equation*}
  (E^1)^\vee = (E^1)^\vee_{f_1} + (E^1)^\vee_{g_1}  \quad\text{and}\quad   (E^1)^\vee_{f_1} \cap (E^1)^\vee_{g_1} = (E^1/(\ker({f_1}) \oplus \ker({g_1})))^\vee \ .
\end{equation*}

\begin{lemma}\label{lemma:dualkers}
The composed map
\begin{equation*}
(E^1)^{\vee,f} \xrightarrow{\subseteq} (E^1)^\vee \longrightarrow (\ker({f_1}) \oplus \ker({g_1}))^\vee
\end{equation*}
is injective.
\end{lemma}
\begin{proof}
We have to prove, for $m \geq 1$, that
\begin{equation*}
  \ker(f_1) + \ker(g_1) + F^m E^1 = E^1 \ .
\end{equation*}
Because of Lemma \ref{lemma:V} this boils down to proving that $x$, $x \cdot \tau_{s_1}$, $y$, $y \cdot \tau_{s_0}$ all lie in $\im(f_{(\mathbf{x}_0,\mathbf{x}_1)}) \oplus \im(f^\pm) \oplus F^m E^1$. Since $y=\Gamma_\varpi(x)$ it is enough to prove this for $x = e_{\id} \cdot (0,0,\c)_{1}$ and $x \cdot \tau_{s_1}= e_{\id}\cdot  (0,0,\c)_{s_1}$.
By  Lemma \ref{lemma:zetaV}-e) we know that $\zeta^m\cdot  x - x$ and $\zeta^m \cdot x\cdot  \tau_{s_1} - x\cdot \tau_{s_1}$ lie in $\ker(f_1) + \ker(g_1)$  for any $m \geq 1$. But, using Cor.\ \ref{coro:form-zeta}, we have $\zeta^m\cdot  x = e_{\id} \zeta^m\cdot  (0,0,\c)_{1} = e_{\id} \cdot (0,0,\mathbf{c})_{(s_0 s_1)^m} \in F^{2m} E^1$ and then $\zeta^m \cdot x\cdot  \tau_{s_1} = e_{\id}\cdot  (0,0,\mathbf{c})_{(s_0 s_1)^m} \tau_{s_1} \in F^{2m-1} E^1$ by applying $\anti$ and using Prop.\ \ref{prop:theformulas}.
\end{proof}

We put $K_{f_1} := (E^1)^{\vee,f} \cap (E^1)^\vee_{f_1}$ and $K_{g_1} := (E^1)^{\vee,f} \cap (E^1)^\vee_{g_1}$. Because of Lemma \ref{lemma:dualkers} we have $K_{f_1} \oplus K_{g_1} \subseteq (E^1)^{\vee,f}$. Since $K_{f_1}$ and $K_{g_1}$ inject into $\ker(f_1)^\vee$ and $\ker(g_1)^\vee$, respectively, we have $\zeta \cdot \eta\cdot  \zeta = \eta$ for $\eta \in K_{f_1}$ and $\zeta\cdot  \eta = \eta \cdot  \zeta$ for $\eta \in K_{g_1}$.

\begin{lemma}\label{lemma:decompE1dual}
   $(E^1)^{\vee, f} = K_{f_1} \oplus K_{g_1}$.
\end{lemma}
\begin{proof}
Let $\xi\in (E^1)^{\vee, f}$. We claim that there exists a linear map $\eta \in K_{g_1}$ such that $\eta\vert_{\ker({g_1})} = \xi\vert_{\ker({g_1})}$. This implies that
 $\xi-\eta \in K_{f_1}$.
\begin{itemize}
\item  Suppose $\xi=\xi  (1-e_{\gamma_0})$. Then we can see $\xi$ as an element in $((1-e_{\gamma_0}) E^1)^{\vee, f}$.
Since $(1-e_{\gamma_0}) E^1=(1-e_{\gamma_0}) \ker({f_1})\oplus (1-e_{\gamma_0})\ker({g_1})$ where $(1-e_{\gamma_0}) \ker({f_1})=(1-e_{\gamma_0}) h_\pm ^1(\widetilde W)$ by Remark \ref{rema:1-gamma0pm} and $(1-e_{\gamma_0}) \ker({g_1})=(1-e_{\gamma_0}) h_0^1(\widetilde W)$ by Remark \ref{rema:1-gamma0F1} and Prop.\ \ref{prop:isokerg},  we may define $\eta$
to be zero on $(1-e_{\gamma_0}) \ker({f_1})$ and $\eta\vert_{(1-e_{\gamma_0}) \ker({g_1})} = \xi\vert_{(1-e_{\gamma_0}) \ker({g_1})}$.
\item  Suppose $\xi= (1-e_{\gamma_0}) \xi  e_{\gamma_0}$. Then we can see $\xi$ as an element in $(e_{\gamma_0} E^1 (1-e_{\gamma_0}))^{\vee, f}$.
Since $e_{\gamma_0} \ker (g_1) (1-e_{\gamma_0})=0$ by Remark \ref{rema:1-gamma0F1} and Prop.\ \ref{prop:isokerg}, the linear form $\xi $ is already in $K_{f_1}$.
\item Now suppose  $\xi= e_{\gamma_0} \xi e_{\gamma_0}$. We  may  consider separately  two cases, namely  $\xi= e_{\id} \xi  e_{\id}$ and $\xi= e_{\id} \xi e_{\id^{-1}}$ (the other cases following by conjugation by $\varpi$). We treat the first case, the second one being  similar.
 If  $\xi= e_{\id} \xi e_{\id}$, then we can see $\xi$ as a linear map on
$e_{\id }  E^1  e_{\id}$ (recall that we are working in the $H$-bimodule $(E^1)^{\vee}$).
By Lemma \ref{lemma:V} and \eqref{f:basisV} we have
\begin{equation*}
  e_{\id}  E^1 e_{\id}=e_{\id} (\ker({f_1}) \oplus \ker({g_1}))  e_{\id}\oplus  k e_{\id} (0,0,\c)_{s_1} \ .
\end{equation*}
Define the linear map $\eta: E^1\rightarrow k$ by
\begin{align*}
   & \eta\vert_{e_{\id} \ker(f_1) e_{\id}} : = 0, \
            \eta\vert_{e_{\id} \ker(g_1) e_{\id}} := \xi\vert_{e_{\id} \ker(g_1) e_{\id}}, \ \text{and} \\
   & \eta(e_{\id} (0,0, \c)_{s_1}) := \sum_{j=1}^{+\infty}\xi\big(e_{\id} ( 0,  2\c\iota, 0)_{ (s_0s_1)^{j}}\big) \ ,
\end{align*}
which is well defined because $\xi \in (E^1)^{\vee, f}$.
From \eqref{f:id-id} we have
\begin{equation*}
  e_{\id}  E^1 e_{\id}= e_{\id} h^1_0(\widetilde W^{even}) +  e_{\id}  h^1_+(\widetilde W^{odd})\ .
\end{equation*}
It remains to check that $\eta \in (E^1)^{\vee, f}$. Since  $h^1_0(\widetilde W^{\ell\geq 2})$ is contained in $\ker(g_1)$ by Prop.\ \ref{F1part}, we only need to check that $\eta$ is trivial on $e_{\id}  \cdot h^1_+(\widetilde W^{odd, \ell\geq m})$ for $m$ large enough.  From Cor.\ \ref{coro:form-zeta} we deduce that $\zeta^{m+1}  \cdot  x  \cdot \tau_{s_1} = \zeta^{m+1} e_{\id}  \cdot (0,0,\mathbf{c})_{s_1} = - e_{\id} (0,2\mathbf{c}\iota,0)_{(s_0 s_1)^{m+1}} - e_{\id} \cdot  (0,0,\mathbf{c})_{(s_0 s_1)^m s_0}$ for any $m \geq 0$. Hence
\begin{gather*}
  e_{\id}  \cdot  (0,0,\mathbf{c})_{(s_0 s_1)^m s_0}   \qquad\qquad\qquad\qquad\qquad\qquad\qquad\qquad\qquad\qquad\qquad\qquad\qquad\qquad\qquad\qquad \\
\begin{split}
   & = - \zeta^{m+1}  \cdot x  \cdot  \tau_{s_1} - e_{\id}  \cdot  (0,2\mathbf{c}\iota,0)_{(s_0 s_1)^{m+1}} \\
   & = - e_{\id} \cdot  (0,0,\mathbf{c})_{s_1} - (\zeta^{m+1}  \cdot x \cdot  \tau_{s_1} - x  \cdot \tau_{s_1}) - e_{\id} \cdot  (0,2\mathbf{c}\iota,0)_{(s_0 s_1)^{m+1}}   \\
   & = - e_{\id}  \cdot (0,0,\mathbf{c})_{s_1} - (\sum_{j=0}^m \zeta^j) (\zeta \cdot  x - x) \cdot \tau_{s_1} - e_{\id} \cdot  (0,2\mathbf{c}\iota,0)_{(s_0 s_1)^{m+1}}   \\
   & \in \ker(f_1) - e_{\id}  \cdot (0,0,\mathbf{c})_{s_1} + (\sum_{j=0}^m \zeta^j) e_{\id}  \cdot (0,2\mathbf{c}\iota,0)_{s_0} \cdot  \tau_{s_1} - e_{\id}  \cdot (0,2\mathbf{c}\iota,0)_{(s_0 s_1)^{m+1}} \\
   &  \qquad\qquad\qquad\qquad\qquad\qquad\qquad\qquad\qquad\qquad\qquad\qquad\qquad\qquad\qquad\qquad\text{by Lemma \ref{lemma:zetaV}-e)}   \\
   & = \ker(f_1) - e_{\id} \cdot  (0,0,\mathbf{c})_{s_1} + (\sum_{j=0}^m \zeta^j) e_{\id} \cdot  (0,2\mathbf{c}\iota,0)_{s_0 s_1}  - e_{\id}  \cdot (0,2\mathbf{c}\iota,0)_{(s_0 s_1)^{m+1}} \\
   &  \qquad\qquad\qquad\qquad\qquad\qquad\qquad\qquad\qquad\qquad\qquad\qquad\qquad\qquad\qquad\qquad\text{by Lemma \ref{lemma:righteasy}-i}   \\
   & =  \ker(f_1) - e_{\id} \cdot  (0,0,\mathbf{c})_{s_1} + (\sum_{j=0}^m e_{\id}  \cdot (0,2\mathbf{c}\iota,0)_{(s_0 s_1)^{j+1}}  - e_{\id}  \cdot (0,2\mathbf{c}\iota,0)_{(s_0 s_1)^{m+1}} \\
   &  \qquad\qquad\qquad\qquad\qquad\qquad\qquad\qquad\qquad\qquad\qquad\qquad\qquad\qquad\qquad\qquad\text{by Cor.\ \ref{coro:form-zeta}}   \\
   & =  \ker(f_1) - e_{\id}  \cdot (0,0,\mathbf{c})_{s_1} + (\sum_{j=1}^m e_{\id}  \cdot (0,2\mathbf{c}\iota,0)_{(s_0 s_1)^j} \ .
\end{split}
\end{gather*}
Since $\eta$ is zero on $\ker(f_1)$ it follows that
\begin{equation*}
   \eta(e_{\id} \cdot  (0,0,\c)_{(s_0s_1)^{m}s_0}) = \eta(- e_{\id} \cdot  (0,0, \c)_{s_1 }+ \sum_{j=1}^{m}e_{\id}  \cdot ( 0,  2\c\iota, 0)_{ (s_0s_1)^{j}}) = - \xi( \sum_{j=m+1}^{\infty}e_{\id}  \cdot ( 0,  2\c\iota, 0)_{(s_0 s_1)^{j}})  .
\end{equation*}
An analogous computation gives
\begin{equation*}
   \eta(e_{\id}  \cdot (0,0,\c)_{(s_1s_0)^{m}s_1}) = \eta (e_{\id}  \cdot (0,0, \c)_{s_1} - \sum_{j=1}^{m}e_{\id}  \cdot ( 0,  2\c\iota, 0)_{ (s_0s_1)^{j}}) = \xi( \sum_{j=m+1}^{\infty}e_{\id}  \cdot ( 0,  2\c\iota, 0)_{(s_0 s_1)^{j}})  .
\end{equation*}
Both are zero for $m$ large enough.
\end{itemize}
\end{proof}

Recall from  \eqref{f:dual} that we have an isomorphism of $H$-bimodules
\begin{equation}\label{f:dual12}
    E^2 \xrightarrow{\;\cong\;} {}^\anti((E^1)^{\vee,f})^\anti.
\end{equation}

\begin{proposition}\label{prop:structure2} Suppose $G={\rm SL}_2(\mathbb Q_p)$ with $p\neq 2,3$ and $\pi=p$.
  Via the isomorphism \eqref{f:dual12}, we have $\ker({{f_2}}) \cong K_{f_1}$ and $\ker({g_2}) \cong K_{g_1}$ and as  $H$-bimodules
\begin{equation*}
   E^2= \ker({{f_2}})\oplus \ker({g_2}).
\end{equation*}
In particular, $f_2\circ g_2= g_2\circ f_2=0$.
\end{proposition}
\begin{proof}
 Let us denote the isomorphism \eqref{f:dual12} temporarily by $\mathfrak{j}$. We had observed already that $\zeta \eta \zeta = \eta$ for $\eta \in K_{f_1}$ and $\zeta \eta = \eta \zeta$ for $\eta \in K_{g_1}$. It follows that $\mathfrak{j}^{-1}(K_{f_1}) \subseteq \ker({{f_2}})$ and $\mathfrak{j}^{-1}(K_{g_1}) \subseteq \ker({g_2})$. We also know from Lemma  \ref{lemma:cap}-iv that $\ker({{f_2}}) \cap \ker({g_2}) = \{0\}$. Therefore, our assertion is a consequence of Lemma \ref{lemma:decompE1dual}.
\end{proof}

From Lemma \ref{lemma:cap}-i-ii, Propositions \ref{prop:kerf1star}-ii and \ref{prop:structure2} we get:

\begin{corollary}\label{coro:fg}
Under the same assumptions, we have
we have $f\circ g= g\circ f=0$ on $E^*$.
\end{corollary}

In the following two sections we determine the $H$-bimodule structure of the two summands $\ker(g_2)$ and $\ker(f_2)$.

\subsubsection{On $\ker(g_2)$}\label{subsubsec:kerg2}

The surjective restriction map $(E^1)^\vee \longrightarrow \ker(g_1)^\vee$ induces the injective map of $H$-bimodules
\begin{equation*}
  \ker(g_2) \cong {}^\anti(K_{g_1})^\anti \longrightarrow  {}^\anti(\ker(g_1)^\vee)^\anti \ .
\end{equation*}
We have to determine the image of this map. From Prop.\ \ref{F1part}-i we know that $h_0^1(\widetilde{W}^{\ell \geq 2}) \subseteq \ker(g_1) \subseteq h_0^1(\widetilde{W}) \oplus h_\pm^1(\Omega)$. Hence the decreasing filtration
\begin{equation*}
  F^n \ker(g_1) :=
  \begin{cases}
  \ker(g_1) & \text{if $n = 1$}, \\
  h_0^1(\widetilde{W}^{\ell \geq n}) & \text{if $n \geq 2$}
  \end{cases}
\end{equation*}
is well defined as well as the corresponding finite dual
\begin{equation*}
  \ker(g_1)^{\vee,f} := \bigcup_{n \geq 1} (\ker(g_1)/F^n \ker(g_1))^\vee \ .
\end{equation*}
If $\xi \in (E^1)^{\vee,f}$ satisfies $\xi | F^n E^1 = 0$ for some $n \geq 2$ then obviously $\xi_{|\ker(g_1)} | F^n \ker(g_1) = 0$ and hence $\xi_{|\ker(g_1)} \in \ker(g_1)^{\vee,f}$. Vice versa, let $\eta \in \ker(g_1)^{\vee,f}$ such that $\eta | F^n \ker(g_1) = 0$ for some $n \geq 2$. We first choose an extension $\dot\eta$ of $\eta$ to $h_0^1(\widetilde{W}^{\ell \geq 2})$ and then extend $\dot\eta$ further to $\ddot\eta$ on $E^1$ by setting $\ddot\eta | h_\pm^1(\widetilde{W}^{\ell \geq 1}) := 0$. then clearly $\ddot\eta | F^n E^1 = 0$, i.e., $\ddot\eta \in (E^1)^{\vee,f}$. This shows that our $\eta$ has an extension in $(E^1)^{\vee,f}$. By Prop.\ \ref{prop:structure2} it then must also have an extension $\xi \in (E^1)^{\vee,f}$ which satisfies $\xi_{| \ker(f_1)} = 0$, i.e., $\xi \in K_{g_1}$. We see that the above restriction map induces an isomorphism of $H$-bimodules
\begin{equation}\label{f:kerg2}
  \ker(g_2) \cong {}^\anti(K_{g_1})^\anti \xrightarrow{\;\cong\;}  {}^\anti(\ker(g_1)^{\vee,f})^\anti \ .
\end{equation}

\begin{proposition}\label{prop:structurekerg2}  Suppose $G={\rm SL}_2(\mathbb Q_p)$ with $p\neq 2,3$ and $\pi=p$.
 The space  $\ker({g_2})$ is the subspace of $\zeta$-torsion in $E^2$ on the left and on the right. We have an isomorphism of $H$-bimodules
\begin{equation}\label{f:compisokerg2}
         \ker(g_2) \cong (F^1H)^{\vee,f}\cong \bigcup_{n\geq 1} (F^1H/\zeta^n F^1 H )^{\vee} \ .
\end{equation}
In particular, $\ker(g_2)$  is $k[\zeta]$-divisible.
\end{proposition}
\begin{proof}
Prop.\ \ref{F1part}-i makes it directly visible that the isomorphism of $H$-bimodules $F^1 H \cong \ker(g_1)$ in Prop.\ \ref{prop:isokerg} respects the filtrations on both sides. Combined with \eqref{f:kerg2} we therefore obtain an isomorphism of $H$-bimodules
\begin{equation*}
   \ker({g_2}) \cong {}^\anti (( F^1H)^{\vee, f})^\anti\cong ( {}^\anti (F^1H )^\anti)^{\vee, f} \cong  ( F^1H)^{\vee, f}
\end{equation*}
where the last isomorphism is induced by $\anti: H\rightarrow H$. Since $\zeta^n\cdot F^1H=F^{2n+1}H$ for $n\geq 1$ by \cite{embed} Remark 3.2.ii, we also have
\begin{equation*}
  \ker({g_2}) \cong \bigcup_{n\geq 1} (F^1H/\zeta^n F^1 H )^{\vee} \ .
\end{equation*}
In particular, this makes visible that $\ker(g_2)$ is $\zeta$-torsion. On the other hand $\ker(f_2)$ does not contain any left or right $\zeta$-torsion since it is an $H_\zeta$-bimodule. It therefore follows from Prop.\ \ref{prop:structure2} that $\ker(g_2)$ is the full subspace of left (or right) $\zeta$-torsion in $E^2$. By Lemma \ref{freeness} $F^1 H$ is a finitely generated free $k[\zeta]$-module. Hence $(F^1H/\zeta^n F^1 H )^{\vee} \cong k[\zeta^{\pm 1}]/k[\zeta] \otimes_{k[\zeta]} F^1$ as a $k[\zeta]$-module, which shows that $\ker(g_2)$ is $k[\zeta]$-divisible.
\end{proof}

\begin{corollary}\label{coro:f1timesg2} Under the same assumptions, we have
   $\ker(f_1) \cdot \ker(g_2) = 0 = \ker(g_2) \cdot \ker(f_1)$.
\end{corollary}
\begin{proof}
Let $a \in \ker(f_1)$ and $b \in \ker(g_2)$. By Prop.\ \ref{prop:structurekerg2} we find an $m \geq 1$ such that $\zeta^m\cdot  b = 0 = b \cdot \zeta^m$. Then $a\cdot b = \zeta^m \cdot a \cdot \zeta^m \cdot b = 0 = b \cdot \zeta^m \cdot a \cdot \zeta^m$.
\end{proof}

\subsubsection{On $\ker(f_2)$}\label{subsubsec:kerf2}

We proceed in a way which is entirely analogous to section \S\ref{subsubsec:Hzeta-inside}. Consider the following elements of $E^2$:
\begin{align*}
  \a^+ & := (\upalpha, 0,0)_{1}- e_{\id}\cdot (0,\iota^{-1} \upalpha, 0)_{s_0}=(\upalpha, 0,0)_{1}-  (0,\iota^{-1} \upalpha, 0)_{s_0}\cdot e_{\id^{-1}} \ \text{and} \\
  \a^- & := (0, 0,\upalpha)_{1}+ e_{\id^{-1}}\cdot (0, \iota^{-1}\upalpha, 0)_{s_1}=(0, 0,\upalpha)_{1}+ (0, \iota^{-1}\upalpha, 0)_{s_1}\cdot e_{\id}
\end{align*}
where $\upalpha$ is chosen as in \eqref{f:normalize} (see also \eqref{f:condodd'}).
It is easy to verify that
\begin{equation}\label{f:invJ}
  \anti(\a^+) = \a^+   \quad\text{and}\quad  \anti(\a^-) = \\a^-
\end{equation}
using Lemma \ref{lemma:antialpha} and \eqref{f:lefts2'}. In order to check that $\a^+$ lies in $\ker(f_2)$ we compute
\begin{align*}
   \a^+\cdot \zeta & = \anti(\zeta \cdot \anti(\a^+)) = \anti(\zeta \cdot \a^+)  \\
      & = \anti((\upalpha, 0,0)_{s_0s_1} + e_1\cdot (0,0,-\upalpha)_{s_1} + e_1\cdot (\upalpha, 0,0)_1)   \quad\text{ by Cor. \ref{coro:zetaonEd-1}}  \\
      & = (\upalpha, 0,0)_{s_1s_0} + (\upalpha, 0,0)_{s_1^{-1}} \cdot e_1 + (\upalpha, 0,0)_1\cdot e_1    \quad\text{ by Lemma \ref{lemma:antialpha}}  \\
      & = (\upalpha, 0,0)_{s_1s_0} + \tau_{\omega_{-1}} \cdot (\upalpha, 0,0)_{s_1} \cdot e_1 + (\upalpha, 0,0)_1 \cdot e_1    \quad\text{ by \eqref{f:lefts2'}}  \\
      & = (\upalpha, 0,0)_{s_1s_0} + e_{\id^2}\cdot (\upalpha, 0,0)_{s_1} + e_{\id^2}\cdot (\upalpha, 0,0)_1   \quad\text{ by \eqref{f:condeven'} and \eqref{f:condodd'}}.
\end{align*}
Hence
\begin{align*}
   \zeta \cdot \a^+ \cdot \zeta & = \zeta \cdot \big((\upalpha, 0,0)_{s_1s_0} + e_{\id^2} \cdot (\upalpha, 0,0)_{s_1} + e_{\id^2} \cdot (\upalpha, 0,0)_1 \big) \\
        & = (\upalpha,  0,0)_{1} + e_{\id}\cdot ( 0,\iota^{-1}(\upalpha),0)_{s_1^2 s_0}  + e_{\id^2}(- \upalpha, 0,0)_{s_1^2s_0}  \\
         & \qquad\qquad + e_{\id^2}\cdot  (\upalpha,  0,0)_{s_0s_1^2} + e_{\id^2}\cdot ( -\upalpha, 0, 0,)_{s_0s_1}
               + e_{\id^2}\cdot (\upalpha, 0,0)_{s_0s_1}    \quad\text{ by Cor. \ref{coro:zetaonEd-1}}  \\
        & = (\upalpha,  0,0)_{1} + e_{\id}\cdot ( 0,\iota^{-1}(\upalpha),0)_{s_1^2 s_0}  \\
        & = (\upalpha,  0,0)_{1} - e_{\id}\cdot ( 0,\iota^{-1}(\upalpha),0)_{s_0}    \quad\text{ by \eqref{f:lefts2'}}  \\
        & = \a^+\ .
\end{align*}
Using Lemma \ref{lemma:conjtripdual} we notice that $\Gamma_\varpi(\a^+)=\a^-$. Hence Remark \ref{rema:calcufg}-iv implies that also $\a^-\in \ker(f_2)$. As in Lemma \ref{lemma:extension} we therefore have the homomorphism of left $H_\zeta$-modules
\begin{equation}\label{f:isokerf2-a}
  H_\zeta \oplus H_\zeta \xrightarrow{f_{\a^+} + f_{\a^-}}  \ker(f_2)
\end{equation}
sending $(1,0)$ and $(0,1)$ to $\a^+$ and $ \a^-$, respectively.

\begin{remark}\label{rema:easyonapm}
Let $w\in \widetilde W$ with  $\ell(w)\geq 1$. From Proposition \ref{prop:theformulas'} we obtain
\begin{align}\label{f:easya}
\tau_w\cdot \a^+ & = \begin{cases} 0&\text{if $w^{-1}\in \widetilde W^1$} \\
(0,0,-\upalpha)_w&\text{if $w^{-1}\in \widetilde W^0$, $\ell(w)$ odd} \\
(\upalpha,0,0)_w&\text{if $w^{-1}\in \widetilde W^0$, $\ell(w)$ even}
\end{cases}   \nonumber \\
\text{and} &  \\
\tau_w\cdot \a^- & = \begin{cases} 0&\text{if $w^{-1}\in \widetilde W^0$} \\
(-\upalpha,0,0)_w&\text{if $w^{-1}\in \widetilde W^1$, $\ell(w)$ odd} \\
(0,0,\upalpha)_w&\text{if $w^{-1}\in \widetilde W^1$, $\ell(w)$ even} \ .
\end{cases}   \nonumber
\end{align}

\end{remark}

\begin{lemma}\label{lemma:leftright2}
\begin{itemize}
\item[1.] For any $u \in \mathbb{F}_p^\times$ we have  $\a^+ \cdot  \tau_{\omega_u}=u^{-2} \tau_{\omega_u}\cdot \a^+$ and $\a^- \cdot \tau_{\omega_u} = u^2 \tau_{\omega_u}\cdot  \a^-$.
\item[2.]  We have
$\a^+\cdot \tau_{s_0}=\tau_{s_0}\cdot\a^+=0$ and $ \a^-\cdot \tau_{s_1}=\tau_{s_1}\cdot \a^-=0$.

\item[3.]  We have
\begin{align*}
  \mathbf{a}^+\cdot \upiota (\tau_{s_1}) & =  -  \tau_{\omega_{-1}} \upiota( \tau_{s_0})\cdot \mathbf{a}^-  \cdot \zeta   \quad\text{and}  \\
  \mathbf{a}^- \cdot \upiota (\tau_{s_0}) & = - \tau_{\omega_{-1}} \upiota(\tau_{s_1})\cdot  \mathbf{a}^+\cdot  \zeta \ .
\end{align*}\end{itemize}
\end{lemma}
\begin{proof}
1. Using using \eqref{f:leftomega}, \eqref{f:rightomegaE1}  we compute:
\begin{align*}
  \a^+ \cdot  \tau_{\omega_u} &  =\anti(\tau_{\omega_u^{-1}}\cdot   ((\upalpha, 0,0)_{1}+ e_{\id} \cdot(0,\iota^{-1} \upalpha, 0)_{s_0^{-1}}))\\&=
\anti(  (u^{-2}\upalpha, 0,0)_{\omega_u^{-1}}+ u^{-1}e_{\id} \cdot(0,\iota^{-1} \upalpha, 0)_{s_0^{-1}})\\&=
   (\upalpha, 0,0)_{\omega_u}- u^{-1} (0,\iota^{-1} \upalpha, 0)_{s_0}\cdot e_{\id^{-1}}\\&=
   u^{-2} (\tau_{\omega_u}\cdot (\upalpha, 0,0)_{1}-\tau_{\omega_u} e_{\id}  \cdot (0,\iota^{-1} \upalpha, 0)_{s_0})\\&=
   u^{-2} \tau_{\omega_u}\cdot \a^+
\end{align*}
and, by an analogous computation (or by conjugation by $\varpi$), we obtain the second claim  of Point 1.\\
2. Point 2 follows from \eqref{f:easya} and  \eqref{f:invJ}.\\
3.
 We check the first identity.
Since $\a^-,\a^+\in\ker(f_2)$, we may as well check the following \begin{equation}\label{f:sc}
  - \zeta\cdot  \mathbf{a}^+\cdot  (\tau_{s_1} +e_1)=
     (\tau_{s_0^{-1}}  + e_1 )\cdot \mathbf{a}^-
\end{equation}
For the left hand side, we have using Lemma \ref{lemma:antialpha},  Prop.\ \ref{prop:theformulas'},  \eqref{f:lefts2'} and \eqref{f:condeven'},\eqref{f:condodd'}
\begin{align*}\mathbf{a}^+\cdot  (\tau_{s_1} +e_1)&=\anti((\tau_{s_1}+e_1)\cdot ((\upalpha, 0,0)_{\omega_{-1}}- e_{\id}\cdot (0,\iota^{-1} \upalpha, 0)_{s_0^{-1}}))\cr&=\anti((0,0,-\upalpha)_{s_1^{-1}}+ e_1\cdot (\upalpha, 0,0)_1)=
(\upalpha,0,0)_{s_1}+ e_{\id^2}\cdot (\upalpha, 0,0)_1
\end{align*} and then using Corollary \ref{coro:zetaonEd-1}:
\begin{align*}
   -\zeta\cdot \mathbf{a}^+ \cdot  (\tau_{s_1} +e_1)& = -(\upalpha,  0,0)_{s_0s_1^2}
               + e_1\cdot (0,  0,\upalpha)_{s_1^2}+e_{\id^2}\cdot ( \upalpha, 0, 0,)_{s_0s_1} -e_{\id^2}\cdot ( \upalpha, 0, 0,)_{s_0s_1} \\
              & = -(\upalpha,  0,0)_{s_0^{-1}} + e_1\cdot (0,  0,\upalpha)_{1}
\end{align*}
For the right hand side we have, using Remark \ref{rema:easyonapm}:
\begin{align*}
  \tau_{s_0^{-1}} \cdot  \mathbf{a}^- =(-\upalpha,0,0)_{s_0^{-1}}, \quad
  e_1\cdot  \mathbf{a}^- =  e_1\cdot (0,0,\upalpha)_1 \ .
   \end{align*}
By adding up, we see that  \eqref{f:sc} holds.

\end{proof}
By Lemma \ref{lemma:leftright2}-2, the map \eqref{f:isokerf2-a} factors through a homomorphism of left $H_\zeta$-modules
\begin{equation}\label{f:isokerf2}
H_\zeta/ H_{\zeta}\tau_{s_0}  \oplus H_\zeta/ H_{\zeta}\tau_{s_1} \xrightarrow{f_{\a^+}+ f_{\a^-}}  \ker(f_2)\ .
\end{equation}
\begin{proposition}\label{prop:kerf2}  Suppose $G = {\rm SL _2}(\mathbb Q_p)$ with $p\neq 2,3$ and $\pi=p$.
   The map \eqref{f:isokerf2}  induces an isomorphism  of $H_\zeta$-bimodules
\begin{equation}
\label{f:mapkerf2}  (H_\zeta / H_\zeta \tau_{s_0} \oplus H_\zeta / H_\zeta \tau_{s_1})^\pm \xrightarrow{\simeq} \ker(f_2)
\end{equation}
\end{proposition}
\begin{proof}
We need to verify that the map is bijective and
 right $H$-equivariant. We may compare with the proof of Proposition \ref{plusminus-part}. Just like in  that proof, the right $H$-equivariance
 is  seen by comparing the definition \eqref{f:def-R} with
 Lemma \ref{lemma:leftright2}.

Concerning the injectivity we first observe that it suffices to check the injectivity of the restriction of the map to $H/H\tau_{s_0} \oplus H/H\tau_{s_1}$. The elements $\tau_w$ with $w \in \widetilde{W}$ such that $\ell(w s_0) = \ell(w) + 1$ from a $k$-basis of $H/H\tau_{s_0}$; they are of the form $w = \omega(s_0 s_1)^m$ or $= \omega s_1 (s_0 s_1)^m$ with $m \geq 0$ and $\omega \in \Omega$. Using \eqref{f:leftomegaE2} and \eqref{f:easya}, we see that
\begin{equation}\label{f:list0a}
  \tau_w \cdot \mathbf{a}^+ \in\begin{cases}
  k^\times (\upalpha,0,0)_w   & \text{if $w = \omega(s_0 s_1)^m$ with $m \geq 1$},  \\
  k^\times (0,0,\upalpha)_w & \text{if $w = \omega s_1(s_0 s_1)^m$ with $m \geq 0$},  \\
  k^\times (\upalpha,0,0)_w + k^\times e_{\id} (0,\iota^{-1}\upalpha,\mathbf{c},0)_{s_0}     & \text{if $w = \omega$}.
  \end{cases}
\end{equation}
Similarly the elements $\tau_w$ with $w \in \widetilde{W}$ such that $\ell(w s_1) = \ell(w) + 1$ form a $k$-basis of $H/H\tau_{s_1}$; they are of the form $w = \omega (s_1 s_0)^m$ or $= \omega s_0 (s_1 s_0)^m$ with $m \geq 0$ and $\omega \in \Omega$. In this case we obtain
\begin{equation}\label{f:list1a}
  \tau_w \cdot \mathbf{a}^- \in\begin{cases}
  k^\times (0,0,\upalpha)_w   & \text{if $w = \omega(s_1 s_0)^m$ with $m \geq 1$},  \\
  k^\times (\upalpha,0,0)_w & \text{if $w = \omega s_0(s_1 s_0)^m$ with $m \geq 0$},  \\
  k^\times (0,0,\upalpha)_w + k^\times e_{\id^{-1}} (0,\iota^{-1}\upalpha,\mathbf{c},0)_{s_1}     & \text{if $w = \omega$}.
  \end{cases}
\end{equation}

By comparing the lists \eqref{f:list0a} and \eqref{f:list1a} we easily see that the elements
\begin{equation*}
  \{\tau_w\cdot  \a^+ : \ell(w s_0) = \ell(w) + 1\} \cup \{\tau_w \cdot \mathbf{a}^- : \ell(w s_1) = \ell(w) + 1\}
\end{equation*}
in $E^2$ are $k$-linearly independent. This concludes the proof of the injectivity.
For the surjectivity,  we first list the following results which we obtain from  \eqref{f:list0a} and \eqref{f:list1a} by applying $\anti$:
\begin{equation}\label{f:list0a}
 \mathbf{a}^+ \cdot \tau_{w^{-1}}\in\begin{cases}
  k^\times (\upalpha,0,0)_w   & \text{if $w = \omega(s_0 s_1)^m$ with $m \geq 1$},  \\
  k^\times (0,0,\upalpha)_w & \text{if $w = \omega s_1(s_0 s_1)^m$ with $m \geq 0$},  \\
  k^\times (\upalpha,0,0)_w + k^\times e_{\id} (0,\iota^{-1}\upalpha,\mathbf{c},0)_{s_0}     & \text{if $w = \omega$}.
  \end{cases}
\end{equation}
Similarly the elements $\tau_w$ with $w \in \widetilde{W}$ such that $\ell(w s_1) = \ell(w) + 1$ form a $k$-basis of $H/H\tau_{s_1}$; they are of the form $w = \omega (s_1 s_0)^m$ or $= \omega s_0 (s_1 s_0)^m$ with $m \geq 0$ and $\omega \in \Omega$. In this case we obtain
\begin{equation}\label{f:list1a}
  \tau_w \cdot \mathbf{a}^- \in\begin{cases}
  k^\times (0,0,\upalpha)_w   & \text{if $w = \omega(s_1 s_0)^m$ with $m \geq 1$},  \\
  k^\times (\upalpha,0,0)_w & \text{if $w = \omega s_0(s_1 s_0)^m$ with $m \geq 0$},  \\
  k^\times (0,0,\upalpha)_w + k^\times e_{\id^{-1}} (0,\iota^{-1}\upalpha,\mathbf{c},0)_{s_1}     & \text{if $w = \omega$}.
  \end{cases}
\end{equation}

We  gather the following arguments:
\begin{itemize}
\item[-] A basis for $\ker(g_1)$ is  given by the set of all $f_{(x_0, x_1)}(\tau_w)$, $w\in \widetilde W$, $\ell(w)\geq 1$. These elements are  spelled out in Proposition \ref{F1part}. From these formulas, we see that an element in $\ker(f_2)$ lies necessarily in the space $h^2_\pm(\widetilde W^{\ell\geq 2})+ h^2(s_1\Omega)+ h^2(s_0\Omega)+h^2(\Omega)$.
\item[–] From  \eqref{f:list0a} and \eqref{f:list1a},  we deduce  that
$h^2_-(\widetilde{W}^{1,  \ell\geq 1}) + h^2_+(\widetilde{W}^{0, \ell\geq 1})=\sum_{w\in\widetilde W, \ell(w)\geq 1}k \tau_w\cdot \a^-+k \tau_w\cdot \a^+$ is contained in the image of the map of the proposition.

\begin{itemize}
\item  So it is contained in
 $\ker(f_2)$ which is invariant under $\anti$. Therefore by Lemma \ref{lemma:antialpha}, the whole space   $h^2_\pm(\widetilde W^{\ell\geq 1})$ is contained in $\ker(f_2)$.

\item  But this map is also right $H$-equivariant. So for $w\in \widetilde W$ with length $\geq 1$,
the elements $\a^+\cdot \tau_{w^{-1}}=\anti(\tau_w\cdot \a^+)$ and  $\a^-\cdot \tau_{w^{-1}}=\anti(\tau_w\cdot \a^-)$ also lie in this image (see \eqref{f:invJ}).
Therefore  the whole space   $h^2_\pm(\widetilde W^{\ell\geq 1})$ is contained in the image of the map.

\end{itemize}
\item[-] The component in
$$h^2(\Omega)+h^2_0(s_1\Omega)+h^2(s_0\Omega)$$ of $\ker(f_2)$ is spanned by all $\tau_{\omega}\cdot \a^+$ and
$\tau_{\omega}\cdot \a^-$ for $\omega\in \Omega$.

To verify this statement we notice,  using the third lines of   \eqref{f:list0a} and \eqref{f:list1a}, that it is equivalent to saying that the component in $h^2_0(s_1\Omega)+h^2(s_0\Omega)$ of $\ker(f_2)$  is zero. But the latter  follows easily from the formulas for $f_{(x_0, x_1)}(s_\epsilon\tau_\omega)$, $\omega\in \Omega$, $\epsilon=0,1$  given in Proposition \ref{F1part}.

\item[-] We have proved that $\ker(f_2)=h^2_\pm(\widetilde W^{\ell\geq 1})\oplus \bigoplus_{\omega\in \Omega}k\,\tau_{\omega}\cdot  \a^-\oplus k\,\tau_{\omega}\cdot\a^+$  and this space is contained in the image of the map.
\end{itemize}
\end{proof}

\begin{corollary}\phantomsection\label{cor:structurekerf2}
Suppose $G = {\rm SL _2}(\mathbb Q_p)$ with $p\neq 2,3$ and $\pi=p$.

\begin{itemize}
  \item[i.] The $H_\zeta$-bimodules $\ker({{f_1}})$ and $\ker({{f_2}})$ are isomorphic.
  \item[ii.] $\ker({{f_2}})$ is a free  $k[\zeta^{\pm 1}]$-module of rank $4(p-1)$ on the left and on the right.
\end{itemize}
\end{corollary}
\begin{proof}
Combine Propositions \ref{prop:kerf1star} and \ref{prop:kerf2}.
\end{proof}

%

\begin{remark}\phantomsection\label{rem:offdiag-anti'}
\begin{itemize}
  \item[1.] It follows from $\Gamma_\varpi(\a^+)=\a^-$ (see also Remark \ref{rema:calcufg}-iv) that the diagram
\begin{equation*}
  \xymatrix{
    (H_\zeta / H_\zeta \tau_{s_0} \oplus H_\zeta / H_\zeta \tau_{s_1})^\pm \ar[d]_{(\sigma^+,\sigma^-) \mapsto (\Gamma_\varpi(\sigma^-),\Gamma_\varpi(\sigma^+))} \ar[r]^-{\eqref{f:mapkerf2} } & \ker(f_2)\ar[d]^{\Gamma_\varpi} \\
    (H_\zeta / H_\zeta \tau_{s_0} \oplus H_\zeta / H_\zeta \tau_{s_1})^\pm \ar[r]^-{\eqref{f:mapkerf2} } & \ker(f_2)   }
\end{equation*}
  is commutative.
  \item[2.] It follows from  \eqref{f:invJ}  (see also Remark \ref{rema:calcufg}-v) that the diagram
\begin{equation*}
  \xymatrix{
    (H_\zeta / H_\zeta \tau_{s_0} \oplus H_\zeta / H_\zeta \tau_{s_1})^\pm \ar[d]_{\beta\circ(\anti\oplus \anti)} \ar[r]^-{\eqref{f:mapkerf2}} & \ker(f_2) \ar[d]^{\anti} \\
    (H_\zeta / H_\zeta \tau_{s_0} \oplus H_\zeta / H_\zeta \tau_{s_1})^\pm \ar[r]^-{\eqref{f:mapkerf2}} & \ker(f_2) \ . }
\end{equation*}
Compare with Remark \ref{rem:offdiag-anti}-2. The maps in the diagram are all bijective.

\end{itemize}
\end{remark}

%

\section{\label{sec:valuesH*}On the left $H$-module $H^*(I, V)$ when $G={\rm SL}_2(\mathbb Q_p)$ with $p\neq 2,3$ and $V$ is of finite length}

We suppose that $G={\rm SL}_2(\mathbb Q_p)$ with $p\neq 2,3$. The goal of this section is to investigate the cohomology $H^*(I,V) = \Ext_{\Mod(G)}^*(\mathrm{X},V)$ for any finite length representation $V$ in $\Mod(G)$.

\begin{remark}
  Recall that our assumption on $G$ guarantees that the pro-$p$ Iwahori subgroup $I$ has cohomological dimension $3$. We therefore have $H^i(I,V) = 0$ for $i \geq 4$ and any $V$ in $\Mod(G)$.
\end{remark}

In a \textbf{first step} we fix a nonzero polynomial $Q \in k[X]$ and consider the smooth $G$-representation $\mathbf{X}/\mathbf{X} Q(\zeta)$. Since $H$ is free over $k[\zeta]$ (Lemma \ref{freeness}), right multiplication by   $Q(\zeta)$ induces an injective map on $\X^I$ and therefore on  $\mathbf{X}$. So  we have the short exact sequence of smooth $G$-representations
\begin{equation*}
  0 \rightarrow \mathbf{X} \xrightarrow{\cdot Q(\zeta)} \mathbf{X} \longrightarrow \mathbf{X}/\mathbf{X}Q(\zeta) \rightarrow 0 \ .
\end{equation*}
Hence we obtain the long exact cohomology sequence (of $H$-bimodules)
\begin{align}\label{f:coh-seq}
  0 & \longrightarrow E^0 \xrightarrow{\cdot Q(\zeta)} E^0 \longrightarrow (\mathbf{X}/\mathbf{X} Q(\zeta))^I  \longrightarrow  E^1 \xrightarrow{\cdot Q(\zeta)} E^1 \longrightarrow H^1(I,\mathbf{X}/\mathbf{X} Q(\zeta)) \\
    & \longrightarrow E^2 \xrightarrow{\cdot Q(\zeta)} E^2 \longrightarrow H^2(I,\mathbf{X}/\mathbf{X} Q(\zeta))   \longrightarrow E^3 \xrightarrow{\cdot Q(\zeta)} E^3 \longrightarrow H^3(I,\mathbf{X}/\mathbf{X} Q(\zeta)) \longrightarrow 0            \nonumber
\end{align}
and therefore the short exact sequences
\begin{equation}\label{f:coh-seq2}
  0 \rightarrow E^i/E^iQ(\zeta) \longrightarrow H^i(I,\mathbf{X}/\mathbf{X} Q(\zeta)) \longrightarrow \ker (E^{i+1} \xrightarrow{\cdot Q(\zeta)} E^{i+1} ) \rightarrow 0 \ .
\end{equation}
Note that all three terms in these short exact sequences are annihilated by $Q(\zeta)$ from the right. Next we collect in the following proposition what we have proved in the previous sections about $E^*$ as a left or a right $k[\zeta]$-module.

\begin{proposition}\label{collect-zeta}
   As left or right $k[\zeta]$-modules we have the following isomorphisms (for 2. and 3. we need $\pi=p$):
\begin{enumerate}
  \item $H \cong k[\zeta]^{4(p-1)}$;
  \item $E^1 \cong k[\zeta^{\pm 1}]^{4(p-1)} \oplus k[\zeta]^{4(p-1)}$;
  \item $E^2 \cong k[\zeta^{\pm 1}]^{4(p-1)} \oplus \big((k[\zeta^{\pm 1}]/k[\zeta]\big)^{4(p-1)}$;
  \item $E^3 \cong k \oplus \big((k[\zeta^{\pm 1}]/k[\zeta]\big)^{4(p-1)}$ with $\zeta$ acting by $1$ on the summand $k$.
\end{enumerate}
\end{proposition}
\begin{proof}
1. See Lemma \ref{freeness}.

4. According to \eqref{f:keydecompEd} and Prop.\ \ref{prop:kerS-injhull} we have
\begin{equation*}
  E^3 \cong k \oplus \bigcup_{m \geq 1} (H/\zeta^m H)^\vee  \qquad\text{as $H$-bimodules}.
\end{equation*}
Using 1. we obtain
\begin{equation*}
  \bigcup_{m \geq 1} (H/\zeta^m H)^\vee \cong \big( \bigcup_{m \geq 1} (k[\zeta]/\zeta^m k[\zeta])^\vee \big)^{4(p-1)} \cong \big( \bigcup_{m \geq 1} (\frac{1}{\zeta^m}k[\zeta]/k[\zeta])^\vee \big)^{4(p-1)} \cong \big( k[\zeta^{\pm 1}]/k[\zeta] \big)^{4(p-1)} \ .
\end{equation*}

3. By Propositions \ref{prop:structure2} and \ref{prop:structurekerg2} and Corollary  \ref{cor:structurekerf2}, we have $E^2 = A \oplus B$ with $A \cong k[\zeta^{\pm 1}]^{4(p-1)}$ and $B \cong \bigcup_{m \geq 1} (F^1 H/\zeta^m F^1 H)^\vee$, the latter even as an $H$-bimodule. But $F^1 H$ is of finite codimension in $H$. Hence the elementary divisor theorem implies that also $F^1 H \cong k[\zeta]^{4(p-1)}$. Therefore the same computation as in the proof of 4. above shows that $B \cong \big((k[\zeta^{\pm 1}]/k[\zeta]\big)^{4(p-1)}$.

2. According to Propositions \ref{prop:structure1}, \ref{prop:isokerg}, and \ref{prop:kerf1star} the $H$-bimodule $E^1$ has the two sub-$H$-bimodules $A := \ker(f_1)$ and $B := \ker(g_1)$ which have the following properties:
\begin{itemize}
  \item[a.] $A \oplus B \subseteq E^1$ with $E^1/A \oplus B$ being $4$-dimensional;
  \item[b.] $A \cong k[\zeta^{\pm 1}]^{4(p-1)}$ and $B \cong F^1 H \cong k[\zeta]^{4(p-1)}$ as left or as right $k[\zeta]$-modules;
  \item[c.] $E^1/B$ is a $k[\zeta^{\pm 1}]$-module;
  \item[d.] $\zeta$ acts on $E^1/A \oplus B$ from the left and from the right by $1$.
\end{itemize}
We give the argument for the left $k[\zeta]$-action, the other case being entirely analogous. Again the elementary divisor theorem implies that $E^1/A$ as a $k[\zeta]$-module is of the form $E^1/A = F \oplus \bar{D}$ with $F$ being free of rank $4(p-1)$ and $\bar{D}$ being finite dimensional. Since the natural map $\bar{D} \hookrightarrow E^1/A \oplus B$ is injective $\zeta$ must act by $1$ on $\bar{D}$. Suppose that $\bar{D} = 0$. Then we have the short exact sequence $0 \rightarrow A \rightarrow E^1\rightarrow F \rightarrow 0$ which splits since $F$ is free. We therefore assume in the following that $\bar{D} \neq 0$, and we let $D \subset E^1$ denote the preimage of $\bar{D}$ in $E^1$. Then $\zeta$ acts bijectively on $D$ which therefore is a $k[\zeta^{\pm 1}]$-module, which contains the free $k[\zeta^{\pm 1}]$-module $A$ with a finite dimensional quotient. Applying this time the elementary divisor theorem to the $k[\zeta^{\pm 1}]$-module $D$ we see that it must be of the form $D = F' \oplus D'$ with $F' \cong k[\zeta^{\pm 1}]^{4(p-1)}$ and finite dimensional $D'$. This $D'$ then is a $k[\zeta]$-submodule of $E^1$ on which $\zeta$ acts by $1$ so that $(\zeta - 1) D' = 0$. It therefore follows from Lemma \ref{lemma:notorsion}.ii that $D' = 0$. Hence we have a short exact sequence $0 \rightarrow F' \rightarrow E^1 \rightarrow F \rightarrow 0$, which also must split.
\end{proof}

\begin{lemma}\label{ker-coker-zeta}
   The multiplication by $Q(\zeta)$ on $k[\zeta]$ and on $k[\zeta^{\pm 1}]$ has zero kernel and a finite dimensional cokernel whereas on $k[\zeta^{\pm 1}]/k[\zeta]$ it has a finite dimensional kernel and zero cokernel.
\end{lemma}
\begin{proof}
The only part of the statement which might not be entirely obvious is the surjectivity of the multiplication on $k[\zeta^{\pm 1}]/k[\zeta]$. This is clear if $Q(\zeta)$ is a power of $\zeta$. We therefore assume that $Q(\zeta)$ is prime to $\zeta$. But then $k[\zeta]/Q(\zeta) k[\zeta] = k[\zeta^{\pm 1}] / Q(\zeta) k[\zeta^{\pm 1}]$.
\end{proof}
In the three next statements, $G={\rm SL}_2(\mathbb Q_p)$ with $p\neq 2,3$ and we assume $\pi=p$.

\begin{corollary}\label{ker-coker-zeta-E}
   The multiplication by $Q(\zeta)$ from the right on $E^*$ has finite dimensional kernel and cokernel.
\end{corollary}

Using \eqref{f:coh-seq2} we deduce the following result.

\begin{corollary}\label{Q-finite-dim}
  The $k$-vector space $H^*(I,\mathbf{X}/\mathbf{X} Q(\zeta))$ is finite dimensional.
\end{corollary}

Next we consider the left $k[\zeta]$-action on $H^*(I,\mathbf{X}/\mathbf{X} Q(\zeta))$. For this we introduce the polynomial $P(X) := Q(X)Q(\frac{1}{X})X^{\deg(Q)}$.

\begin{proposition}\label{left-torsion}
   $H^*(I,\mathbf{X}/\mathbf{X} Q(\zeta))$ is left $P(\zeta)$-torsion.
\end{proposition}
\begin{proof}
We start with the following observation.
By Corollary \ref{coro:fg}, we know that  for any $x \in E^*$, we have $\zeta\cdot x\cdot \zeta-x\in \ker(g)$.
We deduce, for any $m\geq 0$ and $0 \leq i \leq m$, that
$ \zeta^m\cdot  x \cdot  \zeta^i \equiv \zeta^{m-i} \cdot  x  \bmod \ker(g)$
We choose  $m$ to be $2\deg(Q)$ which is $\geq \deg(P)$.
The coefficients of the polynomial $P=\sum_{i=0}^{m} a_i X^i$ satisfy $a_{m-i} = a_i$ for any $i$.
For $x\in E^*$, have
\begin{align}
  P(\zeta) \cdot x -  \zeta^m\cdot  x\cdot  P(\zeta) & = \sum_{i=0}^m a_i \zeta^i \cdot x - \sum_{i=0}^m  a_i \zeta^m\cdot x\cdot  \zeta^i    \equiv  \sum_{i=0}^m a_i \zeta^i \cdot x - \sum_{i=0}^m  a_i \zeta^{m-i}\cdot x     \bmod \ker(g) \cr
   &  \sum_{i=0}^m a_i \zeta^i \cdot x - \sum_{i=0}^m  a_{m-i} \zeta^{m-i}\cdot x  \equiv 0   \bmod \ker(g) \ .\label{f:preliore}
\end{align}
Now we prove the proposition. Because of the exact sequences \eqref{f:coh-seq2} it suffices to show that $E^*/E^* Q(\zeta)$ and $\ker (E^* \xrightarrow{\cdot Q(\zeta)} E^*)$ are left $P(\zeta)$-torsion. Obviously both modules are annihilated by $P(\zeta)$ from the right.
That $\ker (E^* \xrightarrow{\cdot Q(\zeta)} E^*)$ is  of left $P(\zeta)$-torsion follows from the above observation: suppose $x\cdot Q(\zeta)=0$, then $x\cdot P(\zeta)=0$ and $ P(\zeta) \cdot x \in \ker(g)$ so $ P(\zeta)^2\cdot  x  = P(\zeta)\cdot x \cdot P(\zeta)=0$.
Now let $x\in E^*$. From \eqref{f:preliore}, we deduce that
$P(\zeta)^2 \cdot x -  \zeta^mP(\zeta)\cdot  x\cdot  P(\zeta)=  P(\zeta) \cdot x \cdot  P(\zeta) -  \zeta^m\cdot  x\cdot  P(\zeta)^2$
so
$$P(\zeta)^2 \cdot x =  (\zeta^m P(\zeta)\cdot  x+  P(\zeta) \cdot x  -  \zeta^m\cdot  x\cdot  P(\zeta))\cdot   P(\zeta)\in E^*\cdot Q(\zeta).$$
This shows that $E^*/E^* Q(\zeta)$ is left $P(\zeta)$-torsion.
\end{proof}

\begin{remark}\label{Ore}
  The formula \eqref{f:preliore} actually holds true for any nonzero polynomial $P(X) \in k[X]$ with the property that $X^m P(\frac{1}{X}) = P(X)$ for some integer $m \geq \deg(P)$. It shows that, for any $x \in E^*$ and any $j \geq 1$, we have
\begin{equation*}
  P(\zeta)^j \cdot x \equiv \zeta^{mj} \cdot x \cdot P(\zeta)^j  \bmod \ker(g)
\end{equation*}
and symmetrically
\begin{equation*}
   x \cdot P(\zeta)^j \equiv  P(\zeta)^j  \cdot x \cdot \zeta^{mj}  \bmod \ker(g) \ .
\end{equation*}
This easily implies that the multiplicative subset $\{P(\zeta)^n : n \geq 0\}$ of $H = E^0$ satisfies the left and right Ore conditions inside the full algebra $E^*$. Therefore the corresponding classical ring of fractions $E^*_{P(\zeta)}$ exists. This applies in particular to $P(X) = X$ so that $H_\zeta$ is part of the larger ring $E^*_\zeta$. We will come back to these localizations elsewhere.
\end{remark}

\begin{lemma}\phantomsection\label{sub-quot}
\begin{itemize}
  \item[i.] $\Mod^I(G)$ is closed under the formation of subrepresentations and quotient representations.
  \item[ii.] The functor $V \longrightarrow V^I$ is exact on $\Mod^I(G)$.
\end{itemize}
\end{lemma}
\begin{proof}
i. For quotient representations the assertion is obvious. For a subrepresentation $U$ of a representation $V$ in $\Mod^I(G)$ we consider the commutative diagram
\begin{equation*}
  \xymatrix{
     0 \ar[r] & \mathbf{X} \otimes_H U^I \ar[d] \ar[r] & \mathbf{X} \otimes_H V^I \ar[d]_{\cong} \ar[r] & \mathbf{X} \otimes_H (V/U)^I \ar[d]_{\cong} & \\
    0 \ar[r] & U \ar[r] & V \ar[r] & V/U \ar[r] & 0.   }
\end{equation*}
The upper horizontal row is exact by the left exactness of the functor $(-)^I$ and the fact that $\mathbf{X}$ is projective as a (right) $H$-module (cf.\ the proof of \cite{embed} Prop.\ 3.25). By the equivalence of categories in \S\ref{subsubsec:eqcat} the middle and right perpendicular arrows are isomorphisms. Hence the left one is an isomorphism as well. This shows that $U$ lies in $\Mod^I(G)$.\\
ii. This a consequence of the equivalence of categories in \S\ref{subsubsec:eqcat}.
\end{proof}

\begin{lemma}\label{subquotients}
  The $G$-representation $\mathbf{X}/\mathbf{X} Q(\zeta)$ is of finite length. Furthermore, the following sets of isomorphism classes of $G$-representations coincide:
\begin{itemize}
  \item[a.] irreducible smooth $G$-representations $V$ such that $Q(\zeta) V^I = 0$;
  \item[b.] irreducible quotient representations of $\mathbf{X}/\mathbf{X} Q(\zeta)$;
  \item[c.] irreducible subquotient representations of $\mathbf{X}/\mathbf{X} Q(\zeta)$.
\end{itemize}
\end{lemma}
\begin{proof}
First of all we have, by Lemma \ref{sub-quot}, that $(\mathbf{X}/\mathbf{X} Q(\zeta))^I = H/HQ(\zeta)$. This is finite dimensional over $k$ by Prop.\ \ref{collect-zeta}.1 and hence is an $H$-module of finite length. The equivalence of categories in \S\ref{subsubsec:eqcat} then implies that $\mathbf{X}/\mathbf{X} Q(\zeta)$ is of finite length.

Also by Lemma \ref{sub-quot} the $H$-module $V^I$, for any irreducible subquotient $V$ of $\mathbf{X}/\mathbf{X} Q(\zeta)$, is a subquotient of $H/HQ(\zeta)$ and hence satisfies $Q(\zeta) V^I = 0$. On the other hand consider any irreducible smooth $G$-representation $V$ such that $Q(\zeta)V^I = 0$. By the equivalence of categories $V^I$ is a simple $H$-module. We therefore have a surjection $H \twoheadrightarrow V^I$, which  factors over $H/HQ(\zeta)$ and then gives rise to a surjection $\mathbf{X}/\mathbf{X} Q(\zeta) = \mathbf{X} \otimes_H H/H Q(\zeta) \twoheadrightarrow \mathbf{X} \otimes_H V^I = V$.
\end{proof}

\begin{remark}\label{rem:Q}
   As pointed out in the proof of the previous lemma the $H$-module $V^I$ is finite dimensional for any irreducible smooth $G$-representation $V$. Hence there always is a nonzero polynomial $Q \in k[X]$ such that $Q(\zeta) V^I = 0$.
\end{remark}

Combining all of the above we may now establish in a \textbf{second step} our main result.

\begin{theorem}\label{cohomology}
Let $G={\rm SL}_2(\mathbb Q_p)$ with $p\neq 2,3$. For any representation $V$ of finite length in $\Mod(G)$ we have:
\begin{itemize}
  \item[i.] The $k$-vector space $H^*(I,V)$ is finite dimensional;
  \item[ii.] Assume $\pi=p$. If $V$ lies in $\Mod^I(G)$ and $Q(\zeta) V^I = 0$ for some nonzero polynomial $Q \in k[X]$, then the left $H$-module $H^*(I,V)$ is $P(\zeta)$-torsion for the polynomial $P(X) := Q(X)Q(\frac{1}{X})X^{\deg(Q)}$.
\end{itemize}
\end{theorem}
\begin{proof}
By a straightforward induction using the long exact cohomology sequence as well as Lemma \ref{sub-quot} (for ii.) we may assume that $V$ is irreducible. According to Remark \ref{rem:Q} and Lemma \ref{subquotients} we then have a surjection $\mathbf{X}/\mathbf{X}Q(\zeta) \twoheadrightarrow V$. Because of the bound $3$ for the cohomological dimension of $I$ this surjection induces a surjection $H^3(I,\mathbf{X}/\mathbf{X}Q(\zeta)) \twoheadrightarrow H^3(I,V)$. It therefore follows from Cor.\ \ref{Q-finite-dim} and Prop.\ \ref{left-torsion} that $H^3(I,V)$ is finite dimensional and $P(\zeta)$-torsion. By Lemma \ref{subquotients} this reasoning, in fact, applies to any irreducible subquotient representation of $\mathbf{X}/\mathbf{X}Q(\zeta)$. By another induction we obtain that $H^3(I,U)$ is finite dimensional and $P(\zeta)$-torsion for any subquotient representation $U$ of $\mathbf{X}/\mathbf{X}Q(\zeta)$. Going back to our original $V$ we write it as part of an exact sequence $0 \rightarrow U \rightarrow \mathbf{X}/\mathbf{X}Q(\zeta) \rightarrow V \rightarrow 0$, which gives rise to an exact sequence of $H$-modules
\begin{equation*}
  H^2(I,\mathbf{X}/\mathbf{X}Q(\zeta)) \longrightarrow H^2(I,V) \longrightarrow H^3(I,U) \ .
\end{equation*}
Using again Cor.\ \ref{Q-finite-dim} and Prop.\ \ref{left-torsion} it follows that $H^2(I,V)$ is finite dimensional and $P(\zeta)$-torsion. By repeating the above argument we obtain that $H^2(I,U)$ is finite dimensional and $P(\zeta)$-torsion for any subquotient $U$. Reasoning  inductively downwards (w.r.t.\ the cohomology degree $i$) we finally deduce our full assertion.
\end{proof}

Over an algebraically closed field $k$ we refer to \cite{OV} \S 5 for the notion of an irreducible admissible supercuspidal representation. Note that for our group $G$ every irreducible representation is admissible as a consequence of the equivalence of categories in \S \ref{subsubsec:eqcat}. We extend this notion as follows to arbitrary $k$. Let $V$ be an irreducible representation in $\Mod(G)$. By this equivalence of categories $V^I$ is a finite dimensional $H$-module. Hence, if $\bar{k}$ denotes an algebraic closure of $k$, the base extension $\bar{k} \otimes_k V$ is still generated by its $I$-fixed vectors and $(\bar{k} \otimes_k V)^I = \bar{k} \otimes_K V^I$ is a finite dimensional $\bar{k} \otimes_k H$-module. The equivalence of categories over $\bar{k}$ therefore implies that $\bar{k} \otimes_k V$ is a representation of finite length of $G$ over $\bar{k}$. We will call $V$ \textit{supersingular} if all irreducible constituents of $\bar{k} \otimes_k V$ are supersingular in the sense of \cite{OV} \S 5.

\begin{corollary}\label{coro:supersingular}
  Let $G={\rm SL}_2(\mathbb Q_p)$ with $p\neq 2,3$. An irreducible representation $V$ in $\Mod(G)$ is supersingular if and only if the left $H$-module $H^*(I,V)$ is supersingular.\end{corollary}
\begin{proof}
It is shown in \cite{OV} Thm. 5.3 that, when $k$ is algebraically closed, an irreducible (admissible) representation $V_0$ in $\Mod(G)$ is supersingular if and only $V_0^I$ is  $\zeta$-torsion, namely if and only $V_0^I$  is supersingular. Hence $V$ is supersingular if and only if $V_0^I$ is $\zeta$-torsion for all irreducible constituents $V_0$ of $\bar{k} \otimes_k V$. By Lemma \ref{sub-quot} the latter is equivalent to $(\bar{k} \otimes_k V)^I$ being $\zeta$-torsion hence to $V^I$ being $\zeta$-torsion, i.e., being supersingular (see \S\ref{subsubsec:supersing}). But by the equivalence of categories in \S\ref{subsubsec:eqcat} the $H$-module $V^I$ is simple. If it is $\zeta$-torsion it must satisfy $\zeta V^I = 0$. So we apply Thm.\ \ref{cohomology}.ii with $Q := X$ to see that then all of $H^*(I,V)$ is $\zeta$-torsion and hence supersingular.
\end{proof}

We remind the reader that in Prop.\ \ref{prop:HdVzero} we had determined for which irreducible representations $V$ the top cohomology $H^d(I,V)$ vanishes.

\section{The commutator  in $E^*$ of the center of $H$ when $G={\rm SL}_2(\mathbb Q_p)$, $p\neq 2,3$}\label{sec:commutator}
We assume in this section that $G={\rm SL}_2(\mathbb Q_p)$, $p\neq 2,3$ and  $\pi=p$.
Recall that we denote by $Z$ the center of $H$. In this section we consider the subalgebra
$$\CZ=\{\mathscr E\in E^*, \: z \cdot \mathscr E= \mathscr E\cdot z \, \quad \forall z\in Z \}$$ of $E^*$. We are going to describe the product in this algebra.
We denote by $\cz{i}$ its $i^{\text{th}}$ graded piece.

\begin{proposition}
$\CZ$ coincides with the commutator of $\zeta$ in $E^*$ namely with $\ker(g)$:
$$
\CZ=\{\mathscr E\in E^*, \: \zeta \cdot \mathscr E= \mathscr E\cdot \zeta \}.
$$
\end{proposition}
\begin{proof}
As $H$-bimodules, we have
\begin{equation*}
  \ker(g^0)\cong  H  , \ \ker(g^1) \cong F^1 H , \ \text{and}\ \ker(g_2)\cong (F^1H )^{\vee,f}\cong
\bigcup_{n\geq 1} (F^1H/\zeta^n F^1 H )^{\vee}
\end{equation*}
(see Propositions \ref{prop:isokerg} and \ref{prop:structurekerg2}). So
these spaces are contained in $\CZ$. Lastly we explained in \S\ref{sec:zetatorsion}\textbf{B)} that the elements of $Z$ centralize the elements of $E^3 =\ker(g^3)$.
\end{proof}

We recall some notations  and results from \S\ref{subsubsec:fduals},  \S\ref{subsubsec:kerg1} and   \S\ref{subsubsec:kerg2}:
\begin{itemize}
\item $\cz{0}=H$,
\item We have an isomorphism of  $H$-bimodules
$  f_{(\mathbf{x}_0,\mathbf{x}_1)} : F^1 H \longrightarrow \cz{1}$. We keep track of its inverse
\begin{equation}\label{f:cz1}f_{(\mathbf{x}_0,\mathbf{x}_1)}^{-1}:\cz{1}\overset{\simeq}\longrightarrow F^1H\:, \ .\end{equation}
\item
We  have an isomorphism of $H$-bimodules (see \eqref{f:kerg2})
\begin{equation}\label{f:cz2}\cz{2}\overset{\simeq}\longrightarrow {}^{\anti}((F^1H)^{\vee,f})^\anti
\end{equation}
 and we denote by $\upalpha^\star_w$ the preimage of $\tau^\vee_w\vert_{F^1H}$ by this map for $w\in \widetilde W$, $\ell(w)\geq 1$.
The set of all these $\upalpha_w^\star s$ forms a basis of $\cz{2}$.
\item $\cz{3}\cong{}^\anti (H^{\vee,f})^\anti $ as $H$-bimodules. As in \S\ref{subsubsec:fduals}, the element in $E^3$ corresponding to $\tau_w^\vee$
is denoted by $\phi_w$.
\end{itemize}

\begin{remark}\label{rema:actionstar0}
Let $w\in{\widetilde W}$ with $\ell(w)\geq 1$, $\omega\in \Omega$. Using  formulas \eqref{f:leftrightHd}, we obtain immediately
\begin{align*}
 \tau_{\omega}\cdot \upalpha ^\star_w=&\, \upalpha ^\star_{\omega w}\cr
 \tau_{s_\epsilon}\cdot\upalpha ^\star_w=
  &\begin{cases} 0&  \text{ if  $w\in \widetilde W^\epsilon$ with $\ell(w)\geq 1$,}\cr
-e_1\cdot  \upalpha ^\star_w      + \upalpha ^\star_{s_0w}   & \text{ if  $w\in \widetilde W^{1-\epsilon}$ with  $\ell(w) \geq 2$, }\cr
-e_1\cdot \upalpha ^\star_w
      & \text{ if  $w\in \widetilde W^{1-\epsilon}$ with $\ell(w)= 1$}.\cr\end{cases}\cr
            \zeta\cdot \upalpha ^\star_w=
&\begin{cases} 0&  \text{if  $\ell(w)\leq 2$,}\cr
\upalpha ^\star_{s_{\epsilon}s_{1-\epsilon} w}   & \text{ if  $w\in \widetilde W^\epsilon$ with  $\ell(w) \geq 3$.}
  \end{cases}
\end{align*}

 \end{remark}

 \begin{remark}\phantomsection\label{rema:conjkerg2}
\begin{itemize}
\item[i.]  We have   $\upalpha^\star_w\cup f_{(\mathbf{x}_0,\mathbf{x}_1)}^{-1}(\tau_w)=\delta_{v,w}\phi_w$
for all $v,w\in\widetilde W$ with $\ell(v),\ell(w)\geq 1$.
\begin{itemize}\item In particular, using \eqref{f:conjphi} and Proposition \ref{F1part}-v we see that the image of
$\upalpha^\star_w$ by conjugation by $\varpi$ is $\upalpha^\star_{\varpi w\varpi^{-1}}$.
\item Using Proposition \ref{F1part}-iv and recalling by \cite{Ext} (89) (8.2)  that $\anti(\phi_w)=\phi_{w^{-1}}$, we deduce (see also \cite{Ext} Rmk. 6.2) that $\anti(\upalpha^\star_w)=-\upalpha^\star_{w^{-1}}$.
\end{itemize}

\item[ii.]  Recall that the element $\upalpha^0\in 1+p\mathbb  Z_p/1+ p^2\mathbb Z_p$ was chosen in \eqref{f:normalize}.
For $w\in \widetilde W$ with $\ell(w)\geq 1$, there is a unique element
 in  $\ker(g_2)$ which, when seen as a linear form in $(E^1)^{\vee,f}$,    coincides with $(0, \upalpha^0, 0)_w$ if $w\in \widetilde W^0$ (resp. -$(0, \upalpha^0, 0)_w$ if $w\in \widetilde W^1$) on $\ker(g_1)$
(see Lemma \ref{lemma:decompE1dual} and
Proposition \ref{prop:structure2}). By Proposition \ref{F1part}-i, this element is $\upalpha^\star_w$. By definition, it is zero on $\ker(f_1)$.

When $w\in \widetilde W^0$, the element  $\upalpha_w^\star-(0, \upalpha^0, 0)_w$
is an element of $\ker(f_2)$ which coincides with $-(0, \upalpha^0, 0)_w$ on $\ker(f_1)$. But
Remark \ref{rema:1-gamma0pm} implies that
$(1-e_{\gamma_0})\cdot (0, \upalpha^0, 0)_w$ is trivial on $\ker(f_1)$. Therefore, and  by conjugation by $\varpi$ (Lemma \ref{lemma:conjtripdual}),
\begin{align}\label{f:explicitalphastar}
 & \upalpha_w^\star-(0, \upalpha^0, 0)_w\in e_{\gamma_0}\cdot \ker(f_2) \text{ if $w\in \widetilde W^0$} \\
 & \upalpha_w^\star+(0, \upalpha^0, 0)_w\in e_{\gamma_0}\cdot \ker(f_2) \text{ if $w\in \widetilde W^1$} \ .  \nonumber
\end{align}
\end{itemize}
\end{remark}

\subsection{The product $(\cz{1}, \cz{1})\rightarrow \cz{2}$}

Recall using  \eqref{f:mm+1} that we have an homomorphism of $H$-bimodules
\begin{align}
F^1H&\longrightarrow  {}^{\anti}((F^1H/F^{2}H)^\vee)^\anti   \\
\tau_w  & \longmapsto
\begin{cases}
    -\tau_w^\vee\vert_{F^1H} & \text{ if $\ell(w)=1$}, \\
     0 & \text{ if $\ell(w)\geq 2$}
\end{cases} \nonumber
\end{align}
which is trivial on $F^2H$.
Identifying $(F^1H/F^2H)^\vee$ with the sub-$H$-bimodule of the linear forms in $(F^1H)^{\vee,f}$ which are trivial on $F^2H$, we  obtain an  homomorphism of $H$-bimodules
\begin{equation}\label{f:prod11}
\begin{array}
{ccc}F^1H\otimes _H F^1H&\longrightarrow &  {}^{\anti}((F^1H/F^{2}H)^\vee)^\anti\hookrightarrow{}^{\anti}((F^1H)^{\vee,f})^\anti\cr \tau_v\otimes  \tau_w&\longrightarrow & \begin{cases}
   -\tau_v\cdot \tau_w^\vee\vert_{F^1H} &\text{ if $\ell(w)=1$}\cr 0 &\text{ if $\ell(w)\geq 2$ .}
\end{cases}
\end{array}
\end{equation}

\begin{remark}\label{rema:prodF1}
Let $v,w\in\widetilde W$ with length $\geq 1$,  $\omega, \omega'\in\Omega$ and $\epsilon\in\{0,1\}$. Using \eqref{f:leftrightHd}, we see that the map above has the following outputs:
\begin{align*}
 \tau_{\omega s_\epsilon}\otimes \tau_{\omega' s_\epsilon}&\longmapsto   e_1\cdot \tau_{ s_\epsilon}^\vee\vert_{F^1H}= \tau_{ s_\epsilon}^\vee\vert_{F^1H}\cdot e_1\cr
 \tau_{\omega s_\epsilon}\otimes \tau_{\omega' s_{1-\epsilon}}&\longmapsto   0\cr&\quad\text{and}\cr
 \tau_{v}\otimes \tau_{w}&\longmapsto  0\text{ if $\ell(v)\geq 2$ or $\ell(w)\geq 2$.}
\end{align*}
We see that \eqref{f:prod11} is a symmetric bilinear map onto a 2-dimensional $k$-vector space.
\end{remark}

\begin{proposition}\label{prop:11} Assume that $G={\rm SL}_2(\mathbb Q_p)$, $p\neq 2,3$ and  $\pi=p$. We have a commutative diagram of $H$-bimodules
\begin{equation}
  \xymatrix{
    \cz{1}\otimes_H\cz{1} \ar[d]^{\cong}_{\eqref{f:cz1}\otimes \eqref{f:cz1}} \ar[rrr]^{\quad\quad\text{Yoneda product}} && &\cz{2}\ar[d]^{\cong}_{\eqref{f:cz2}} \\
F^1H\otimes _H F^1H \ar[rrr]^{\eqref{f:prod11}} & && {}^{\anti}((F^1H)^{\vee,f})^\anti }
\end{equation}
\end{proposition}
\begin{proof}
Because of the isomorphism \eqref{f:cz1}, the $H$-bimodule
   $ \cz{1}\otimes_H\cz{1}$ is generated by the elements of the form $f_{(\mathbf{x}_0,\mathbf{x}_1)}^{-1}(\tau_{s_\epsilon})\otimes f_{(\mathbf{x}_0,\mathbf{x}_1)}^{-1}(\tau_{s_\epsilon'})=\mathbf{x}_\epsilon\otimes \mathbf{x}_{\epsilon'}$ for $\epsilon, \epsilon'\in\{0,1\}$. Therefore, using  Remark \ref{rema:actionstar0},
it is enough to prove that
\begin{equation*}
  \mathbf{x}_\epsilon\cdot  \mathbf{x}_{1-\epsilon} = 0   \qquad\text{and}\qquad   \mathbf{x}_\epsilon\cdot \mathbf{x}_{\epsilon}=e_1\cdot \upalpha_{s_{\epsilon}}^\star \ .
\end{equation*}
We verify these identities now. In the calculations below, we use formulas  \eqref{f:leftomega}, \eqref{f:rightomegaE1}, \eqref{f:condeven}  the definition of the idempotents \eqref{defel}, Proposition \ref{prop:theformulas}, Lemma \ref{lemma:righteasy}-i and  Proposition \ref{prop:yo}.
\begin{itemize}
\item
First we check that
\begin{align*}
\mathbf{x}_0\cdot  \mathbf{x}_{1}
=&
 -((0,\c^0, 0)_{s_0}+e_{\id^{-1}} \cdot (  \c^0\iota^{-1}, 0,0) _{1})\cdot
((0,\c^0, 0)_{s_1}- ( 0,  0, \c^0\iota^{-1}) _{1}\cdot e_{\id^{-1}})\cr=&
-(0,\c^0, 0)_{s_0}\cdot  (0,\c^0, 0)_{s_1}
+(0,\c^0, 0)_{s_0} \cdot ( 0,  0, \c^0\iota^{-1}) _{1}\cdot   e_{\id^{-1}} \cr&
-e_{\id^{-1}} \cdot (  \c^0\iota^{-1}, 0,0) _{1}\cdot
(0,\c^0, 0)_{s_1}+ e_{\id^{-1}} \cdot (  \c^0\iota^{-1}, 0,0) _{1}\cdot   ( 0,  0, \c^0\iota^{-1}) _{1} \cdot  e_{\id^{-1}}
\cr=&
-((0,\c^0, 0)_{s_0}\cdot \tau_{s_1}\cup \tau_{s_0}\cdot  (0,\c^0, 0)_{s_1})
+((0,\c^0, 0)_{s_0} \cup \tau_{s_0}\cdot ( 0,  0, \c^0\iota^{-1}) _{1}))\cdot   e_{\id^{-1}} \cr&
-e_{\id^{-1}} \cdot( (  \c^0\iota^{-1}, 0,0) _{1}\cdot  \tau_{s_1}\cup
(0,\c^0, 0)_{s_1})+ e_{\id^{-1}} \cdot ((  \c^0\iota^{-1}, 0,0) _{1}\cup   ( 0,  0, \c^0\iota^{-1}) _{1} )\cdot  e_{\id^{-1}}
\cr=&
(0,\c^0, 0)_{s_0s_1}\cup   (0,\c^0, 0)_{s_0s_1}
+ e_{\id^{-1}} \cdot ((  \c^0\iota^{-1}, 0,0) _{1}\cup   ( 0,  0, \c^0\iota^{-1}) _{1} )\cdot  e_{\id^{-1}}\cr=&
0 \text{ by  Example \ref{ex:H*I}} .
\end{align*}
Likewise, by conjugation by $\varpi$ (see Proposition \ref{F1part}-v) we have $\mathbf{x}_1\cdot  \mathbf{x}_{0}=0$.
\item
Next we compute
\begin{align*}
 \mathbf{x}_0\cdot  \mathbf{x}_{0} & =
 [(0,\c^0, 0)_{s_0}+e_{\id^{-1}} \cdot (  \c^0\iota^{-1}, 0,0) _{1}]\cdot  [(0,\c^0, 0)_{s_0}+e_{\id^{-1}} \cdot (  \c^0\iota^{-1}, 0,0) _{1}]   \\
 & = (0,\c^0, 0)_{s_0}\cdot (0,\c^0, 0)_{s_0} + \\
 & \qquad\qquad (0,\c^0, 0)_{s_0}\cdot  (  \c^0\iota^{-1}, 0,0) _{1}\cdot e_{\id}  +   e_{\id^{-1}} \cdot (  \c^0\iota^{-1}, 0,0) _{1}\cdot  (0,\c^0, 0)_{s_0}  \qquad\text{(using \eqref{f:condeven})}  \\
 & = (0,\c^0, 0)_{s_0}\cdot (0,\c^0, 0)_{s_0} - \sum_{u\in \mathbb F_p^\times}\big(u^{-1}[(0,\c^0, 0)_{s_0}\cdot\tau_{\omega_u} \cup \tau_{s_0}\cdot (  \c^0\iota^{-1}, 0,0) _{\omega_u}]   \\
 & \qquad\qquad\qquad\qquad\qquad\qquad\qquad\qquad\qquad - u^{-1}[(  \c^0\iota^{-1}, 0,0) _{\omega_u}\cdot  \tau_{s_0}\cup \tau_{\omega_u} (0,\c^0, 0)_{s_0}]\big)  \\
 & = (0,\c^0, 0)_{s_0}\cdot (0,\c^0, 0)_{s_0} + \sum_{u\in \mathbb F_p^\times}u^{-1}[(0,\c^0, 0)_{s_0\omega_u} \cup  (  0,0,\c^0\iota^{-1}) _{s_0\omega_u}]   \\
 & \qquad\qquad\qquad\qquad\qquad\qquad\qquad\qquad\qquad  - \sum_{u\in \mathbb F_p^\times}u[(  \c^0\iota^{-1}, 0,0) _{s_0\omega_u}\cup  (0,\c^0, 0)_{s_0 \omega_u}]\ .
\end{align*}
But by \eqref{f:prod-badlength}, there exists $\gamma_{s_0^2}\in H^2(I, \X(s_0^2))$ such that (see Lemma \ref{lemma:righteasy}-ii)
\begin{align*}
(0,\c^0, 0)_{s_0}\cdot (0,\c^0, 0)_{s_0}
& = [(0,\c^0, 0)_{s_0}\cdot \tau_{s_0}\cup \tau_{s_0}\cdot (0,\c^0, 0)_{s_0}]+\gamma_{s_0^2}\cr
& = [(-e_1\cdot (0, \c^0, 0)_{s_0}-e_{\id^{-1}}\cdot (\c^0\iota^{-1}, 0,0)_{s_0}))   \\
&  \qquad\qquad\qquad\qquad\qquad    \cup ((0,-\c^0, 0)_{s_0}\cdot e_1 +  (0,0, \c^0\iota^{-1})_{s_0}\cdot e_{\id}]+\gamma_{s_0^2} \\
& = - [(e_1\cdot (0, \c^0, 0)_{s_0})\cup ((0,0, \c^0\iota^{-1})_{s_0}\cdot e_{\id})]   \\
& \qquad\qquad\qquad + [(e_{\id^{-1}}\cdot (\c^0\iota^{-1}, 0,0)_{s_0}) \cup ((0,\c^0, 0)_{s_0}\cdot e_1)]  \\
& \qquad\qquad\qquad\qquad - [(e_{\id^{-1}}\cdot (\c^0\iota^{-1}, 0,0)_{s_0})) \cup ((0,0, \c^0\iota^{-1})_{s_0}\cdot e_{\id})] + \gamma_{s_0^2}  \\
& = - [(\sum_{u\in \mathbb F_p^\times} (0, \c^0, 0)_{\omega_u s_0})\cup (\sum_{v\in \mathbb F_p^\times}  v^{-1}(0,0, \c^0\iota^{-1})_{ s_0\omega_{v}})]  \\
& \qquad\qquad\qquad + [(\sum_{u\in \mathbb F_p^\times} u^{-1}(\c^0\iota^{-1}, 0,0)_{\omega_u s_0}) \cup
                 (\sum_{v\in \mathbb F_p^\times}  (0,\c^0, 0)_{s_0 \omega_v})]   \\
& \qquad\qquad\qquad\qquad - [(e_{\id^{-1}}\cdot (\c^0\iota^{-1}, 0,0)_{s_0})) \cup ((0,0, \c^0\iota^{-1})_{s_0}\cdot e_{\id})] + \gamma_{s_0^2}  \\
& = - \sum_{u\in \mathbb F_p^\times} u^{-1} (0, \c^0, 0)_{s_0 \omega_u}\cup (0,0, \c^0\iota^{-1})_{ s_0\omega_{u}}   \\
& \qquad\qquad\qquad + \sum_{u\in \mathbb F_p^\times} u(\c^0\iota^{-1}, 0,0)_{s_0\omega_u}\cup   (0,\c^0, 0)_{s_0 \omega_u}   \\
& \qquad\qquad\qquad\qquad - [(e_{\id^{-1}}\cdot (\c^0\iota^{-1}, 0,0)_{s_0})) \cup ((0,0, \c^0\iota^{-1})_{s_0}\cdot e_{\id})]
+\gamma_{s_0^2} \ .
\end{align*}
So
\begin{equation*}
  {\bf x}_0\cdot {\bf x}_0 =-[(e_{\id^{-1}}\cdot (\c^0\iota^{-1}, 0,0)_{s_0}) \cup ((0,0, \c^0\iota^{-1})_{s_0}\cdot e_{\id})]
+\gamma_{s_0^2}\ .
\end{equation*}
Compute that
\begin{gather*}
(e_{\id^{-1}}\cdot (\c^0\iota^{-1}, 0,0)_{s_0})) \cup ((0,0, \c^0\iota^{-1})_{s_0}\cdot e_{\id})  \qquad\qquad\qquad\qquad\qquad\qquad\qquad\qquad\qquad\qquad\qquad  \\
\begin{split}
  & = (\sum_{u\in \mathbb F_p^\times} u^{-1} (\c^0\iota^{-1}, 0,0)_{\omega_u s_0})\cup (\sum_{v\in \mathbb F_p^\times} v^{-1} (0,0, \c^0\iota^{-1})_{s_0\omega_v})  \\
  & = (\sum_{u\in \mathbb F_p^\times} u^{-1} (\c^0\iota^{-1}, 0,0)_{\omega_u s_0})\cup (\sum_{v\in \mathbb F_p^\times} v (0,0, \c^0\iota^{-1})_{\omega_vs_0}) \\
  & = \sum_{u\in \mathbb F_p^\times}  (\c^0\iota^{-1}, 0,0)_{\omega_u s_0} \cup (0,0, \c^0\iota^{-1})_{\omega_us_0} = - \sum_{u\in \mathbb F_p^\times} (0, \upalpha ^0, 0)_{\omega_us_0} \qquad\text{by \eqref{f:pairing}} \\
  & = e_1\cdot (0, \upalpha ^0, 0)_{s_0} = - e_1\cdot  \upalpha^ \star _{s_0}  \qquad\text{by \eqref{f:explicitalphastar}}\ .
\end{split}
\end{gather*}
Since ${\bf x} _{0}$  and $ \upalpha ^\star _{s_0}$  both lie in the kernel of the left action of $(\tau_{s_0}+e_1)$ (Remark \ref{rema:actionstar0}) we obtain directly, using   the formulas of Prop. \ref{prop:theformulas'}, that $\gamma_{s_0^2}=0$.
So as expected
$\mathbf{x}_0\cdot \mathbf{x}_{0}=e_1\cdot \upalpha_{s_{0}}^\star$
The same result is valid with $s_1$ instead of $s_0$ by conjugation by $\varpi$ (Remark \ref{rema:conjkerg2} and proof of Proposition \ref{F1part}-v which says that $\Gamma_\omega({\bf x}_0)={\bf x}_1$).
\end{itemize}
\end{proof}

\subsection{The products $(\cz{i}, \cz{{3-i}})\rightarrow \cz{3}$ for $i=1,2$}

For $\tau\in F^1H$, we have the homomorphisms of left, resp. right,  $H$-modules
$$
L_\tau:{}^\anti H^\anti \rightarrow {}^\anti (F^1H)^\anti, \quad  h\mapsto h\cdot \tau=\anti(\tau) h\quad\text{and}\quad R_\tau:{}^\anti H^\anti \rightarrow {}^\anti (F^1H)^\anti, \quad  h\mapsto  \tau\cdot h=h\anti(\tau)
$$
which by pullback give homomorphisms of right, resp. left,  $H$-modules
$$
L_\tau^*:{}^{\anti}((F^1H)^{\vee})^\anti \rightarrow {}^{\anti}(H^{\vee})^\anti, \quad  \alpha\mapsto \alpha\circ L_\tau
\quad\text{and}\quad
R_\tau^*:{}^{\anti}((F^1H)^{\vee})^\anti \rightarrow {}^{\anti}(H^{\vee})^\anti, \quad  \alpha\mapsto \alpha\circ R_\tau
$$
such that $L_{x\tau y}^*(\alpha)=x\cdot (L_\tau^*( y\, \cdot\, \alpha))$ and
$R_{x\tau y}^*(\alpha)=(R_\tau^*(  \alpha\cdot x))\cdot y$
for $x,y\in H$ and $\alpha\in {}^{\anti}((F^1H)^{\vee})^\anti$. We therefore have  natural homomorphisms of $H$-bimodules
$$
\begin{array}{ccc}
F^1H\otimes_H {}^{\anti}((F^1H)^{\vee})^\anti&\longrightarrow &{}^{\anti}(H^{\vee})^\anti\cr \tau\otimes \alpha&\longmapsto& -L_\tau^*(\alpha) =-\alpha(\anti(\tau)_-)
\end{array}
$$
$$
\begin{array}{ccc}
{}^{\anti}((F^1H)^{\vee})^\anti\otimes F^1H&\longrightarrow &{}^{\anti}(H^{\vee})^\anti\cr \alpha\otimes \tau&\longmapsto& -R_\tau^*(\alpha) =-\alpha(_-\anti(\tau))
\end{array}
$$
which respectively induce homomorphisms of $H$-bimodules
\begin{equation}\label{f:prod12}\begin{array}{ccc} F^1H\otimes_H  {}^{\anti}((F^1H)^{\vee,f})^\anti&\longrightarrow &{}^{\anti}(H^{\vee,f})^\anti
\end{array}\end{equation}
\begin{equation}\label{f:prod21}\begin{array}{ccc}   {}^{\anti}((F^1H)^{\vee,f})^\anti\otimes_H F^1H&\longrightarrow &{}^{\anti}(H^{\vee,f})^\anti \ .
\end{array}
\end{equation}


\begin{proposition}\label{prop:1221} Assume that $G={\rm SL}_2(\mathbb Q_p)$, $p\neq 2,3$ and  $\pi=p$.
We have  commutative diagrams of $H$-bimodules
\begin{equation}
  \xymatrix{
    \cz{1}\otimes_H\cz{2} \ar[d]^{\cong}_{\eqref{f:cz1}\otimes \eqref{f:cz2}} \ar[rrr]^{\quad\quad\text{Yoneda product}} && &\cz{3}= E^3\ar[d]^{\Delta^3 \text{(see \eqref{f:dual})}}_{\cong} \\
F^1H\otimes _H {}^{\anti}((F^1H)^{\vee,f})^\anti  \ar[rrr]^{\eqref{f:prod12}} & && {}^{\anti}(H^{\vee,f})^\anti }
\label{f:prodC12}
\end{equation}
\begin{equation}
  \xymatrix{
    \cz{2}\otimes_H\cz{1} \ar[d]^{\cong}_{\eqref{f:cz2}\otimes \eqref{f:cz1}} \ar[rrr]^{\quad\quad\text{Yoneda product}} && &\cz{3}= E^3\ar[d]^{\Delta^3 \text{(see \eqref{f:dual})}}_{\cong} \\
 {}^{\anti}((F^1H)^{\vee,f})^\anti \otimes _H F^1H\ar[rrr]^{\eqref{f:prod21}} & && {}^{\anti}(H^{\vee,f})^\anti }
\label{f:prodC21}
\end{equation}
Both these Yoneda product maps have image $\ker(\trace^3)$ namely the space of $\zeta$-torsion in $E^3$.
\end{proposition}
\begin{proof}
Preliminary observations:
\begin{itemize}
  \item[A)] For $s\in\{s_0, s_1\}$ and $w\in \widetilde W$, $\ell(w)\geq 1$,  the map \eqref{f:prod12} sends $\tau_{s}\otimes \tau_w^\vee\vert_{F^1H}$ to
$-\tau_{s}\cdot   \tau_w^\vee\in {}^{\anti}(H^{\vee,f})^\anti \ .$
and  \eqref{f:prod21} sends $ \tau_w^\vee\vert_{F^1H}\otimes \tau_{s}$ to
$  - \tau_w^\vee\cdot \tau_s\in {}^{\anti}(H^{\vee,f})^\anti \ .$
  \item[B)] By Remark \ref{rema:calcufg}-iii, we have $\ker(g_1)\cdot \ker(f_2)\subseteq \ker(f_3)$ and likewise
$\ker(f_2)\cdot \ker(g_1)\subseteq \ker(f_3)$.
 But $\ker(f_3)$ is a one dimensional vector space with basis  $e_1\cdot \phi_1$  and supporting the character   $\chi_{triv}$ of $H$ (Lemma \ref{lemma:cap}).  Therefore,    $e_\lambda \cdot \ker(g_1)\cdot \ker(f_2)=\{0\}$
 and   $e_\lambda \cdot \ker(f_2)\cdot \ker(g_1)=\{0\}$
 for any $\lambda\neq 1$.
\end{itemize}
We now turn to the proof of the commutativity of the diagrams. Since the left $H$-module $\cz{1}$ is generated by ${\bf x_0}$ and ${\bf x}_1$ and
using observation A) and \eqref{f:leftrightHd} above, it is enough to prove,  for $\epsilon\in\{0, 1\}$ and $w\in \widetilde W$, $\ell(w)\geq 1$:
$$
{\bf x}_{\epsilon}\cdot \upalpha^\star_ w =-\tau_{s_\epsilon}\cdot \phi_w\ =
\begin{cases}
-\phi_{s_\epsilon w}+e_1\cdot \phi_w&\text{if $w\in\widetilde W^{1-\epsilon}$}\cr
0&\text{if $w\in\widetilde W^{\epsilon}$}
\end{cases}
$$
$$
\upalpha^\star_ w\cdot{\bf x}_{\epsilon} = -\phi_w\cdot \tau_{s_\epsilon} =
\begin{cases}
-\phi_{ ws_\epsilon}+e_1\cdot \phi_w&\text{if $w^{-1}\in\widetilde W^{1-\epsilon}$}\cr
0&\text{if $w^{-1}\in\widetilde W^{\epsilon}$ \ .}
\end{cases}
$$
Using Remark \ref{rema:psiw}, these identities show that the Yoneda product maps have image $\ker(\trace^3)$.

By the proof of Proposition \ref{F1part}-iv, we know that $\anti({\bf x}_\epsilon)=-\tau_{s_\epsilon^2}\cdot {\bf x}_\epsilon$
and this is equal to $- {\bf x}_\epsilon\cdot \tau_{s_\epsilon^2}$ (since $f_{( {\bf x}_0, {\bf x}_1)})$ is an homomorphism of $H$-bimodules).
By Remark \ref{rema:conjkerg2}-i we have, that  $\anti(\upalpha^\star_w)=-\upalpha^\star_{w^{-1}}$. Lastly,
$\anti(\phi_w)=\phi_{{w}^{-1}}$ by \cite{Ext}  (8.2).
Since $\anti$ is an anti-involution of the graded algebra $E^*$,  it is therefore enough
to prove the first identity above namely we focus on the commutativity of \eqref{f:prodC12}.
 \begin{itemize}
\item
Suppose $w\in \widetilde W^\epsilon$ with $\ell(w)\geq 1$. Then by Remark \ref{rema:actionstar0}
we have  $\upalpha^\star_ w=(\tau_{s_\epsilon}+e_1)\cdot  \upalpha_{s_{\epsilon}^{-1}w}^\star$
But ${\bf x}_{\epsilon}\cdot
 (\tau_{s_\epsilon}+e_1)=0$. Therefore ${\bf x}_{\epsilon}\cdot \upalpha^\star_ w=0$.
\item Suppose $w\in \widetilde W^{1-\epsilon}$ with $\ell(w)\geq 1$.
We know from  \eqref{f:explicitalphastar}  
 that
$$ \begin{cases}\upalpha^\star_ w\in-(0,\upalpha^0, 0)_w+e_{\gamma_0}\cdot \ker(f_2)&\text{if $\epsilon=0$}\cr
\upalpha^\star_ w\in(0,\upalpha^0, 0)_w+e_{\gamma_0}\cdot \ker(f_2)&\text{if $\epsilon=1$}
\end{cases}$$  so by observation B) above, we have
 $${{\bf x}_\epsilon}\cdot \upalpha^\star_ w=
 \begin{cases}-{{\bf x}_\epsilon}\cdot (0,\upalpha^0, 0)_w&\text{if $\epsilon=0$}\cr
{{\bf x}_\epsilon}\cdot (0,\upalpha^0, 0)_w&\text{if $\epsilon=1 .$}
\end{cases}$$
Therefore, when $\epsilon=0$ we compute, using Proposition \ref{prop:yo} and  Lemma \ref{lemma:righteasy}-i,
\begin{align*}{{\bf x}_0}\cdot\upalpha^\star_ w=&
((0, \c^0, 0)_{s_0}+e_{\id^{-1}} \cdot (   \c^0\iota^{-1}, 0,0) _{1})\cdot (0, \upalpha^0, 0)_{w}\cr=&
(0, \c^0, 0)_{s_0}\cdot (0, \upalpha^0, 0)_{w}+e_{\id^{-1}} \cdot [(   \c^0\iota^{-1}, 0,0) _{1}\cdot \tau_w\cup  (0, \upalpha^0, 0)_{w}]\cr=&
(0, \c^0, 0)_{s_0}\cdot (0, \upalpha^0, 0)_{w}+e_{\id^{-1}} \cdot [(   \c^0\iota^{-1}, 0,0) _{w}\cup  (0, \upalpha^0, 0)_{w}]\cr=&(0, \c^0, 0)_{s_0}\cdot (0, \upalpha^0, 0)_{w}\cr=&
[(0, \c^0, 0)_{s_0}\cdot \tau_w\cup \tau_{s_0}\cdot (0, \upalpha^0, 0)_{w}] +\mu_{w}\phi_{s_0w}\text{ where $\mu_w\in k$}.
\end{align*}

Now using Lemma \ref{lemma:righteasy}-ii,  Proposition \ref{prop:theformulas'}, and \eqref{f:leftomegaE2}, we compute
\begin{gather*}
   (0, \c^0, 0)_{s_0}\cdot \tau_w\cup \tau_{s_0}\cdot (0, \upalpha^0, 0)_{w}  \qquad\qquad\qquad\qquad\qquad\qquad\qquad\qquad\qquad\qquad\qquad\qquad\qquad\qquad  \\
\begin{split}
& = (e_1\cdot (0, \c^0, 0)_{w})\cup(e_1\cdot  (0, \alpha^0, 0)_w)-
(e_{\id^{-1}}\cdot(\c^0\iota^{-1}, 0,0)_{w})\cup (e_{\id}\cdot (2\iota(\alpha^0),0,0)_w)  \\
& = [\sum_{u, v\in \mathbb F_p^\times}(0,\c^0, 0)_{\omega_u w}\cup  (0, \upalpha^0, 0)_{\omega_vw} ]-[\sum_{u, v\in \mathbb F_p^\times}u ^{-1} (\c^0\iota^{-1}, 0,0)_{\omega_u w}\cup v (2\iota(\alpha^0),0,0)_{\omega_v w}]   \\
& = [\sum_{u\in \mathbb F_p^\times}(0,\c^0, 0)_{\omega_u w}\cup  (0, \upalpha^0, 0)_{\omega_uw} ]-[\sum_{u\in \mathbb F_p^\times}(\c^0\iota^{-1}, 0,0)_{\omega_u w}\cup  (2\iota(\alpha^0),0,0)_{\omega_u w}]   \\
& = [\sum_{u\in \mathbb F_p^\times}\phi_{\omega_uw} ] -2[\sum_{u\in \mathbb F_p^\times}\phi_{\omega_u w}]= e_1\cdot \phi_w \text{ by \eqref{f:pairing'}.}
\end{split}
\end{gather*}
So
${{\bf x}_0}\cdot\upalpha^\star_ w=e_1\cdot \phi_w+\mu_{w}\phi_{s_0w}.$
But $(\tau_{s_0}+e_1)\cdot(e_1\cdot \phi_w+\mu_{w}\phi_{s_0w})=e_1\cdot \phi_{s_0w}+\mu_{w}e_1\cdot \phi_{s_0w}$ (see \eqref{f:leftrightHd})  and ${{\bf x}_0}$ being in the kernel of $\tau_{s_0}+e_1$, we obtain
 $\mu_w=-1$. Therefore,  as expected,
${{\bf x}_0}\cdot\upalpha^\star_ w=e_1\cdot \phi_w-\phi_{s_0w}=-\tau_{s_0}\cdot \phi_{w}$.
The case when $\epsilon=1$ may then be obtained by conjugation by $\varpi$ (\eqref{f:conjphi}, the proof of Proposition \ref{F1part}-v which says that $\Gamma_\omega({\bf x}_0)={\bf x}_1$),  and \ref{rema:conjkerg2}-i).

\end{itemize}
\end{proof}

\begin{remark}
For $w\in\widetilde W$ with length $1$ and $\epsilon\in\{0,1\}$ the map \eqref{f:prod12} satisfies:
\begin{equation*}
\tau_{s_\epsilon}\otimes \tau_w^\vee\vert_{F^1H}\longmapsto
\begin{cases}
0 & \text{if $w\in \widetilde W^\epsilon$}\cr
-\psi_w & \text{if $w\in \widetilde W^{1-\epsilon}$}
\end{cases}
\end{equation*}
where $\psi_w$ was defined in Remark \ref{rema:psiw}.

Together with Remark \ref{rema:prodF1} and using Propositions \ref{prop:11} and \ref{prop:1221}, this completely describes the triple Yoneda product $\cz{1} \otimes_H \cz{1} \otimes_H \cz{1} \rightarrow \cz{3} = E^3$ with image the subspace $k e_1 \cdot \psi_{s_0} \oplus k e_1 \cdot \psi_{s_1} \subseteq \ker(\trace^3)$.

\end{remark}

\section{Appendix}

\subsection{\label{proofbadlength}Proof of Proposition \ref{prop:yo}} This  proposition is written in the general context of $G := \mathbf{G}(\mathfrak{F})$ being the group of $\mathfrak{F}$-rational points of a connected reductive group $\mathbf{G}$  over $\mathfrak{F}$ which we  assume to be $\mathfrak{F}$-split. The first point was proved in \cite{Ext} Cor.\ 5.5.
To prove the second point, we recall some notations of \cite{Ext}. The affine Coxeter system $(W_{aff}, S_{aff})$ attached to $G$ was introduced in \S 2.1.3 \emph{loc. cit.} Recall that $W_{aff}$ is a subgroup of $W=N_G(T)/T^0$ and that $\widetilde W= N_G(T)/T^0$ (see \S\ref{sec:not}).

The action of $\tau_\omega$ where $\omega\in \widetilde W $ has length zero is given in  \cite{Ext} Prop.\ 5.6 (it is the same formula as \eqref{f:omega}). Using  this formula together with  \eqref{f:support}, we  see that  it is enough to prove the second point  of Proposition \ref{prop:yo} in the case when $v$ is a lift in $N(T)/T^1$ of $s\in S_{aff}$. For $s\in S_{aff}$, we pick the element $n_s\in N(T)$ as defined in \S2.1.6 \emph{loc. cit.}  and let $v:= n_sT^1$. Recall that each $s\in S_{aff}$ corresponds to an affine simple root of the form $(\alpha, \mathfrak h)$.  As in (2.13)  \emph{loc. cit.}, the corresponding cocharacter $\check\alpha$ carves out the finite subgroup $\ima=\{\check\alpha([z]), z\in \mathbb F_q^\times\}$ of $T^0$, where $[_-]: \mathbb F_q^\times\rightarrow \mathfrak O^\times$ denotes the  multiplicative Teichm\"uller lift.
By (2.18) \emph{loc. cit.}, we have
\begin{equation*}
   n_s I n_s^{-1} I=I \;\dot\cup \; \dot\bigcup_{z\in\mathbb F_q^\times} x_{\alpha}(\pi^{\mathfrak h}[z]) \check\alpha([z]) n_s^{-1} I \subset  I  \;\dot\cup \;\bigcup_{z\in\mathbb F_q^\times}I  \check\alpha([z]) n_s^{-1} I=I  \;\dot\cup \;\dot\bigcup_{\omega\in \ima }I \omega n_s^{-1} I
\end{equation*}
where $x_{\alpha}(\pi^{\mathfrak h}[z])\in I$ is defined in \emph{loc. cit.} (2.14). We choose  a lift $\dot w\in N(T)$ of $w\in \widetilde W$. Because of the condition on the length (namely  $\ell(vw)=\ell(w)-1$), we know that $I \dot w I=I  n_s^{-1}  I n_s \dot w I$ and therefore
\begin{align}\label{f:quadraticw2}
 n_s I \dot w I =I n_s \dot w I \;\dot\cup \;\bigcup_{z\in\mathbb F_q^\times} x_{\alpha}(\pi^{\mathfrak h}[z]) \check\alpha([z]) n_s^{-1} I n_s \dot w I \subset  I n_s \dot w I  \;\dot\cup \;\bigcup_{z\in\mathbb F_q^\times}I  \check\alpha([z]) \dot w I=I n_s \dot w I \;\dot\cup \;\dot\bigcup_{\omega\in \ima }I \omega \dot w I
\end{align}
This shows  a result which is more precise  than the one announced  in Proposition \ref{prop:yo}. Namely, when $v= n_sT^1$, we have
$$
\alpha\cdot \beta\in H^{i+j}(I,\mathbf{X}(v w))\oplus \bigoplus_{\omega\in \ima} H^{i+j}(I,\mathbf{X}( \omega \dot w)) \ .
$$

Let $\omega \in \ima$ and  $u_\omega:= \omega \dot w$.  We study the component  $\gamma_{u_\omega}$ of  $\alpha\cdot \beta$ in
$H^{i+j}(I,\mathbf{X}( u_\omega)) $.
We have $n_s^{-1} Iu_\omega I\cap I\dot wI=n_ s^{-1}(I \omega \dot w  I\cap n_ s I \dot w I).$ From \eqref{f:quadraticw2}
we obtain that
$$n_s^{-1} Iu_\omega I\cap I\dot wI=\bigcup_{z\in\mathbb F_q^\times, \check\alpha([z])=\omega} n_s^{-1} x_{\alpha}(\pi^{\mathfrak h}[z]) \check\alpha([z]) n_s^{-1} I n_s \dot w I=\bigcup_{z\in\mathbb F_q^\times, \check\alpha([z])=\omega} I_{n_s} n_s^{-1} x_{\alpha}(\pi^{\mathfrak h}[z]) u_\omega  I$$
The second identity comes from the fact that $I_{n_s}= I_{n_s^{-1}}$ is normalized by $J$ by Cor. 2.5-iii. and from (2.7) in Lemma 2.2 (still in \cite{Ext}).
 Now suppose that $\mathbf G$ is semisimple and simply connected, then  by the proof of Lemma 2.8  \emph{loc. cit.}, the map $\check\alpha$ is injective. Therefore there is a unique
 $z\in\mathbb F_q^\times$ such that  $\check\alpha([z])=\omega$ and
 $$n_s^{-1} Iu_\omega I\cap I\dot wI=I_{n_s} n_s^{-1} x_{\alpha}(\pi^{\mathfrak h}[z]) u_\omega   I\ .$$
To apply the formula of Prop.\ 5.3 of \cite{Ext}, we need to study the double cosets
$ I_{n_s} \backslash (n_s^{-1} Iu_\omega  \cap I\dot wI)/I_{u_\omega^{-1}}$. But from Lemma 5.2 \emph{loc. cit.} and the above identity,  we obtain immediately:
$$n_s^{-1} Iu_\omega  \cap I\dot wI = I_{n_s} n_s^{-1} x_{\alpha}(\pi^{\mathfrak h}[z]) u_\omega I_{u_\omega^{-1}}$$
Let $h:=n_s^{-1} x_{\alpha}(\pi^{\mathfrak h}[z]) u_\omega$.
We have $ u_\omega  h ^{-1} I h  u_\omega^{-1}=x_{\alpha}(\pi^{\mathfrak h}[z]) ^{-1}n_ s I n_ s^{-1}x_{\alpha}(\pi^{\mathfrak h}[z]) $. Since
$x_{\alpha}(\pi^{\mathfrak h}[z]) \in I$ normalizes $I_{n_s}$ and since $I_w\subset I_s$ (Lemma 2.2 \emph{loc. cit.}), we obtain:
\begin{align*}I_{u_\omega}\cap u_\omega  h  ^{-1} I h   u_\omega^{-1}&= I\cap w I w^{-1}\cap
\big(x_{\alpha}(\pi^{\mathfrak h}[z]) ^{-1}n_ s I n_ s^{-1}x_{\alpha}(\pi^{\mathfrak h}[z])\big)\cr&=x_{\alpha}(\pi^{\mathfrak h}[z]) ^{-1} I_{n_s} x_{\alpha}(\pi^{\mathfrak h}[z]) \cap w I w^{-1}= I_{n_s}\cap w I w^{-1}=I_s\cap I_w= I_w= I_{u_\omega} \end{align*}
By Remark 5.4 \emph{loc. cit.}, it implies that  the component of $\alpha\cdot \beta - \alpha\cdot \tau_w\cup \tau_{n_s}\cdot \beta$ in
$H^{i+j}(I, \X(u_\omega))$ is zero. So
$$\alpha\cdot \beta - \alpha\cdot \tau_w\cup \tau_{n_s}\cdot \beta\in H^{i+j}(I, \X(n_sw)).$$
This concludes the proof. We add the computation of  this element. Using Lemma 2.2 and Lemma  5.2-i \emph{loc. cit.}, we obtain the following.
Let $u:=n_s \dot w $.
We have $n_ s ^{-1}I n_ s \dot w\subset I_{n_s} \dot w I$ therefore
 $n_ s ^{-1}I n_ s \dot w\ I\cap I \dot w I= I_{n_s} \dot w I$ and $I_{n_s} \backslash (n_s^{-1} I u \cap I\dot  wI)/I_{u^{-1}}$ is made of  only one  double coset $I_{n_s} \dot  w  I_{u^{-1}}$.
We have $I_u= I_{n_s  \dot w}$ and $I_u\cap u \dot w^{-1} I \dot  w u=n_s I_w n_s^{-1}$ while
$I \cap u \dot w^{-1} I  \dot wu^{-1}= I_s$ and
$ uI u^{-1} \cap u \dot w^{-1} I  \dot wu^{-1}=n_ s I_w n_ s^{-1}.$
So, by Prop. 5.3,  the component $\gamma_{n_s \dot w}$ in  $H^{i+j}(I, \X(n_ s \dot w))$ of $\alpha\cdot \beta$  is given by
\begin{equation*} \Sh_{n_ s \dot w}(\gamma_{n_s \dot w}) =  \mathrm{cores}^{n_s I_w n_s^{-1}}_{I_{n_ s w}}
   \big( \res^{I_{n_s}}_{n_ s I_w n_s^{-1}} \big( \Sh_{n_ s}(\alpha)
   \big) \cup  \big( {n_ s}_*\Sh_w(\beta)\big) \big).\\
\end{equation*} In particular if $\mathbf G$ is semisimple and simply connected, then the image by $ \Sh_{n_ s\dot  w}$ of the element
$$
\alpha\cdot \beta- \alpha\cdot \tau_w\cup \tau_{n_s}\cdot \beta \ ,
$$
which  lies in $H^{i+j}(I, \X(n_ s \dot w))$, is
 \begin{multline*}
 \mathrm{cores}^{n_s I_w n_s^{-1}}_{I_{n_ s w}}
   \big( \res^{I_{n_s}}_{n_ s I_w n_s^{-1}} \big( \Sh_{n_ s}(\alpha)
   \big) \cup  \big( {n_ s}_*\Sh_w(\beta)\big) \big)  \\
   - \mathrm{cores}^{n_s I_w n_s^{-1}}_{I_{n_s w}}
   \big( \res^{I_{n_s}}_{n_ s I_w n_ s^{-1}} \big( \Sh_{n_s}(\alpha)
   \big)  \big)\cup  \mathrm{cores}^{n_s I_w n_s^{-1}}_{I_{n_s w}}
   \big(    {n_s}_*\Sh_w(\beta) \big).
\end{multline*}

\subsection{Computation of  some transfer maps}

We use notations introduced in \S\ref{rootdatum} and  \S\ref{triples}, see in particular Remark \ref{rema:simpliQp}.
\begin{lemma}\label{lemma:conjtrans}Suppose $p\neq 2$ and $G={\rm SL}_2(\mathfrak F)$.
Let  $w\in \widetilde W$ with  length $m:=\ell(w)$ which we suppose $\geq 1$. Let $s\in \{s_0, s_1\}$ be the unique element such that $\ell(sw)=\ell(w)-1$.

\begin{itemize}
\item[i.] Suppose $\mathfrak F\neq \mathbb Q_p$. If $m \geq 2$ or $m = 1$ and $q \neq 3$, then
the transfer map  $(I_{sw})_\Phi\rightarrow (s I_w s^{-1})_\Phi$   is the zero map.

\item[ii.]  Suppose that $\mathfrak F= \mathbb Q_p$. If $m \geq 2$ or $m = 1$ and $p \neq 3$ then the transfer map
  $(I_{sw})_\Phi\rightarrow (s I_w s^{-1})_\Phi$   is

\begin{align*}\begin{psmallmat} 1+\pi x& y\cr \pi^m z& 1+\pi t\end{psmallmat}\mapsto \begin{psmallmat} 1& p y\cr0& 1\end{psmallmat} \bmod  \left(\begin{smallmatrix}
   1+\pi^{2} \mathbb Z_p&   \pi^{2}\mathbb Z_p\\
\pi^{m+1}\mathbb Z_p  & 1+\pi^2\mathbb Z_p
   \end{smallmatrix}\right) &\quad \quad\textrm{ if $s=s_0$}\\
\begin{psmallmat} 1+\pi x& \pi^{m-1}y\cr \pi z& 1+\pi t\end{psmallmat}\mapsto \begin{psmallmat} 1& 0\cr p \pi z& 1\end{psmallmat} \bmod  \left(\begin{smallmatrix}
   1+\pi^2\mathbb Z_p&\pi^{m}\mathbb Z_p\\
\pi^{3}\mathbb Z_p  & 1+\pi^2\mathbb Z_p
   \end{smallmatrix}\right) &\quad \quad \textrm{ if $s=s_1$.}\end{align*}


\end{itemize}

\end{lemma}
\begin{proof}Compare with \cite[Prop. 3.65]{embed}. We let $
m:=\ell(w)$.
By conjugation by $\varpi$,  it is enough  to treat the case of the transfer map
$(I_{m-1}^+)_\Phi\rightarrow (s_0 I_m^- s_0^{-1})_\Phi$ in both the proofs of i. and ii. We denote this map by $\rm{tr}$.
Recall that when $s=s_0$, then
$I_w= I_m^-$  and $I_{sw}= I_{m-1}^+$ where
and
$$I_{m-1}^+ := \left(
\begin{smallmatrix}
1+\mathfrak{M} & \mathfrak{O} \\ \mathfrak{M}^{m} & 1+\mathfrak{M}
\end{smallmatrix}\right), \quad s_0 I_m^-s_0^{-1} = \left(
\begin{smallmatrix}
1+\mathfrak{M} & \mathfrak{M} \\ \mathfrak{M}^m & 1+\mathfrak{M}
\end{smallmatrix}
\right),$$

\noindent By the Iwahori factorization of $I^+_{m-1}$, it suffices to compute the transfer of elements of the form
$\left(\begin{smallmatrix}
       1 & 0 \\
       \pi ^m v & 1
       \end{smallmatrix}\right)
       $,
$\left(\begin{smallmatrix}
       t & 0 \\
       0 & t^{-1}
       \end{smallmatrix}\right)$,
and
$\left(\begin{smallmatrix}
       1 &  u \\
       0 & 1
       \end{smallmatrix}\right)
$ of  $I_{m-1}^+$.
Let $S \subseteq \mathfrak{O}$  be a set of representatives for the cosets in $\mathfrak{O}/\mathfrak{M}$. Then the matrices
$\left(\begin{smallmatrix}
  1 & b \\
  0 & 1
\end{smallmatrix}\right)$, for $b \in S$, form a set of representatives in the right cosets $s_0 I_m^- s_0^{-1}\backslash I_{m-1}^+$.
\begin{itemize}
\item[-]
Since
$\left(\begin{smallmatrix}
       1 & 0 \\
       \pi^m v & 1
       \end{smallmatrix}\right)
        \in s_0 I^-_{m} s_0^{-1}$, which is normal in $I_{m-1}^+$, we have
\begin{equation*}
  \mathrm{tr}(
\left(\begin{smallmatrix}
1 & 0 \\
\pi^m v & 1
\end{smallmatrix}\right) ) \equiv
\prod_{b \in S}
\left(\begin{smallmatrix}
1 & b \\
0 & 1
\end{smallmatrix}\right)
\left(\begin{smallmatrix}
1 & 0 \\
\pi^m v & 1
\end{smallmatrix}\right)
\left(\begin{smallmatrix}
1 & - b \\
0 & 1
\end{smallmatrix}\right)  \equiv \prod_{b \in S}
\left(\begin{smallmatrix}
1+b\pi^m v & -b^2  \pi^{m}v \\
\pi^m v & 1-b\pi^{m} v
\end{smallmatrix}\right) \mod  \Phi(s_0I^-_{m}s_0^{-1})\end{equation*} where
$\Phi(s_0I^-_{m}s_0^{-1})$ denotes the Frattini subgroup of $s_0I^-_{m}s_0^{-1}$.
From \cite[Prop. 3.62]{embed} we get
  $[s_0I^-_m s_0^{-1},s_0I^-_m s_0^{-1}] = s_0[I^-_m, I^-_m] s_0^{-1}=
   \left(\begin{smallmatrix}
   1+\mathfrak{M}^{m+1} &   \mathfrak{M}^{2}\\
 \mathfrak{M}^{m+1}  & 1+\mathfrak{M}^{m+1}
   \end{smallmatrix}\right)$
so $$(s_0I^-_{m}s_0^{-1})_\Phi \cong \mathfrak{M}^m/\mathfrak{M}^{m+1} \times (1+\mathfrak{M}/((1+\mathfrak{M}^{m+1})(1+\mathfrak{M})^p) \times  \mathfrak{M}/\mathfrak{M}^{2}.$$ In this isomorphism the above element corresponds to \begin{align*}
  & (q\pi^m v \bmod \mathfrak M^{m+1},\: \prod_b (1+ b \pi^{m} v) \bmod (1 + \mathfrak{M}^{m+1})(1+\mathfrak{M})^p, \:- \pi^{m}v \sum_b  b^2 \bmod \mathfrak{M}^{2})  \\
  & = (0,\: 1+ \pi^{m} v \sum_b b \bmod (1 + \mathfrak{M}^{m+1})(1+\mathfrak{M})^p, \:- \pi^{m}v \sum_b  b^2 \bmod \mathfrak{M}^{2}) \ .
\end{align*}
The  zero coordinate comes from the fact that for any choice of $\mathfrak F$ we have $q \mathfrak M^m\subseteq   \mathfrak M^{m+1}$.

View $b \longmapsto b$ and $b \longmapsto b^2$ as  $\mathbb{F}_q$-valued characters of the group $\mathbb{F}_q^\times$ of order prime to $p$. By the orthogonality relation for characters the sum $\sum_{b \in \mathbb{F}_q^\times} b$, resp.\  $\sum_{b \in \mathbb{F}_q^\times} b^2$,  vanishes if and only if the respective character is nontrivial if and only if $q \neq 2$, resp.\ $q \neq 2,3$. Since we assume $p \neq 2$ the second component is zero whereas the last component is zero if either $m\geq 2$, or $m= 1$ and $q \neq 3$.

\item[-] For $t\in 1+\mathfrak{M}$, the element
$\left(\begin{smallmatrix}
       t & 0 \\
       0 & t^{-1}
\end{smallmatrix}\right)$ again lies in $s_0I^-_ms_0^{-1}$ so  we have
\begin{equation*}
  \mathrm{tr}( \left(\begin{smallmatrix}
t & 0 \\
0 & t^{-1}
\end{smallmatrix}\right) ) \equiv
\prod_{b \in S}
\left(\begin{smallmatrix}
1 & b \\
0 & 1
\end{smallmatrix}\right)
\left(\begin{smallmatrix}
t & 0 \\
0 & t^{-1}
\end{smallmatrix}\right)
\left(\begin{smallmatrix}
1 & -b \\
0 & 1
\end{smallmatrix}\right)  \equiv \prod_{b \in S}
\left(\begin{smallmatrix}
t & b(t^{-1}-t) \\
0 & t^{-1}
\end{smallmatrix}\right) \mod  \Phi(s_0I^-_{m}s_0^{-1}).
\end{equation*}
The above element seen in $(s_0I^-_{m}s_0^{-1})_{\Phi}$ corresponds to
\begin{equation*}
  (0,\: t^q \bmod(1 + \mathfrak{M}^{m+1})(1+\mathfrak{M})^p , \:(t^{-1} - t) \sum_b  b \bmod \mathfrak{M}^{2}) \ .
\end{equation*}
Since $t^q$  is a $p$th power, the second component is zero. The last component is zero since $q \neq 2$.

\item[-] To compute
$\mathrm{tr}( \left(\begin{smallmatrix}
1 & u \\
0 & 1
\end{smallmatrix}\right) )$, where $u\in \mathfrak O$, we follow  the  argument of  the proof of \cite[Lemma 3.40.i.a)]{embed}.
Let $U(\mathfrak{M}) := \left(
\begin{smallmatrix}
1 & \mathfrak{M} \\
 0& 1
\end{smallmatrix}
\right)$ and $U(\mathfrak{O}) := \left(
\begin{smallmatrix}
1 & \mathfrak{O} \\
 0& 1
\end{smallmatrix}
\right)$. Since $I^+_{m-1} = U(\mathfrak{O}) s_0 I^-_{m}s_0^{-1}$ we  obtain the commutative diagram (\cite{NSW} Cor.\ 1.5.8)
\begin{equation*}
  \xymatrix{
    H^1(s_0 I^-_{m}s_0^{-1} ,k) \ar[d]^{\mathrm{res}} \ar[r]^-{\mathrm{cores}} & H^1(I^+_{m-1},k) \ar[d]^{\mathrm{res}} \\
    H^1(U(\mathfrak{M}),k) \ar[r]^-{\mathrm{cores}} & H^1(U(\mathfrak{O}),k).   }\text{ or dually }   \xymatrix{
   U(\mathfrak{O})_\Phi \ar[d]^{} \ar[r]^-{} &  U(\mathfrak{M})_\Phi   \ar[d] \\
    (I_{m-1}^+)_\Phi \ar[r]^-{\mathrm{tr}} & (s_0 I^-_{m}s_0^{-1})_\Phi.   }
\end{equation*}
The upper right horizontal arrow  is the transfer map $U(\mathfrak{O})_\Phi\rightarrow U(\mathfrak{M})_\Phi$ and it coincides with the $q$th power map $g \longmapsto g^q$ (\cite{Hup} Lemma IV.2.1). So we study the image of $u\in \mathfrak O$ under the map
 $ \mathfrak{O} \longrightarrow \mathfrak{M}, x\mapsto qx$. If $\mathfrak{F} \neq \mathbb{Q}_p$, then we have $q\mathfrak{O} \subseteq \mathfrak{M}^{2}$. Therefore
$\mathrm{tr}(
\left(\begin{smallmatrix}
1 &  u \\
0 & 1
\end{smallmatrix}\right) ) \equiv 0 \bmod \Phi(s_0 I^-_{m}s_0^{-1})$.
If $\mathfrak{F} = \mathbb{Q}_p$, then
$\mathrm{tr}(
\left(\begin{smallmatrix}
1 &  u \\
0 & 1
\end{smallmatrix}\right) ) \equiv \left(\begin{smallmatrix}
1 &  pu \\
0 & 1
\end{smallmatrix}\right) \bmod \Phi(s_0 I^-_{m}s_0^{-1})$.
\end{itemize}

\noindent Under the hypotheses   $p\neq 2$, and $m\geq 2$ or $m= 1$ and $q \neq 3$ we have proved:
 if $\mathfrak F\neq \mathbb Q_p$  then
the transfer map  $(I_{m-1}^+)_\Phi\rightarrow (s_0 I_m^- s_0^{-1})_\Phi$   is trivial;
if $\mathfrak F= \mathbb Q_p$, then
the image of
$$\left(\begin{smallmatrix}
1+\pi x&  y\\
\pi ^m z & 1+\pi t
\end{smallmatrix}\right)=\left(\begin{smallmatrix}
1&  0\\
\frac{\pi ^m z}{1+\pi x} & 1
\end{smallmatrix}\right)\left(\begin{smallmatrix}
1+\pi x&  0\\
0 & (1+\pi x)^{-1}
\end{smallmatrix}\right)\left(\begin{smallmatrix}
1&  \frac{y}{1+\pi x}\\
0& 1
\end{smallmatrix}\right)\in I_{m-1}^+$$
by  the transfer map  $(I_{m-1}^+)_\Phi\rightarrow (s_0 I_m^- s_0^{-1})_\Phi$
 is
$\left(\begin{smallmatrix}
1&  py\\
0& 1
\end{smallmatrix}\right) \bmod \Phi(s_0 I^-_{m}s_0^{-1})$.
\end{proof}
\subsection{Proof of Proposition \ref{prop:theformulas}\label{subsec:proofformu}}

Here $G={\rm SL}_2(\mathbb Q_p)$ with $p\neq 2,3$ and  $\pi = p$.
Let $w\in \widetilde W$ with length $m:=\ell(w)$.
For $s\in\{s_0, s_1\}$ we compute the action of $\tau_{s}$ on an element  $c\in H^1(I,\X(w))$  seen as a triple $(c^-, c^0, c^+)_w$.
Using Lemma \ref{lemma:conjtrip} and knowing that the map \eqref{gammapi} of conjugation by $\varpi$ is compatible with the Yoneda product hence the action of $H$, it is enough to prove the formulas for $s= s_0$.
We recall the following result from \cite{Ext} Prop.\ 5.6. There we worked with $n_{s_i}$ (instead of the matrices $s_i$ of the current article) where $n_{s_0}= s_0$ (but $n_{s_1}=s_1^{-1}$). Recall $s_0 = \left(
\begin{smallmatrix}
0 & 1 \\ -1 & 0
\end{smallmatrix}
\right)$.
We have either $\ell(s_0 w)=\ell(w)+1$
and $\tau_{ s_0} \cdot c\in h^1(s_0w)$  with
\begin{equation}\label{f:leftsgood}
  \Sh_{ s_0 w}(\tau_{s_0}	\cdot c)=\res^{s_0 I_w s_0^{-1}}_{I_{s_0w}} \big({s_0}_* \Sh_w(c)\big)\ ,
\end{equation}
or  $\ell(s_0 w)=\ell(w)-1$ and
\begin{equation}\label{f:leftsbad-support}
  \tau_{{s_0}}	\cdot c =\gamma_{ {s_0} w}+\sum_{\omega \in \Omega }\gamma_{\omega w}\in h^1({s_0}w)\oplus \bigoplus_{\omega \in \Omega } h^1(\omega w)
\end{equation}
with
\begin{equation}\label{f:leftsbad1}
  \Sh_{s_0 w}(\gamma_{{s_0} w})=\cores_{I_{s_0w}}^{s_0 I_w s_0^{-1}}\big({s_0}_* \Sh_w(c)\big) \textrm{ and }
   \end{equation}
\begin{equation}\label{f:leftsbad2}
\Sh_{\omega_u w}(\gamma_{\omega_u w})= \big({s_0} \omega_u^{-1} \left(\begin{smallmatrix}1&[u]^{-1}\cr 0&1\end{smallmatrix}\right) {s_0}^{-1}\big)_*\Sh_w(c)\ .\end{equation}

\noindent \textbf{A) Case when $\ell(s_0w)=\ell(w)+1$.}
It means that $w\in \widetilde W^0$,  $I_w= I_{m}^+$ and $I_{s_0w}=I_{m+1}^-$. We compute the composite map  $H^1(I_m^+,k) \overset{{s_0}_*}{\longrightarrow}  H^1(s_0I_m^+s_0^{-1},k) \overset{\res}{\longrightarrow}  H^1(I_{m+1}^-,k)$. Let $X=\left(\begin{smallmatrix}
1+p x& p^{m+1} y \\ p z & 1+p t
\end{smallmatrix}\right)\in  I_{m+1}^-$. Then $s_0^{-1} X s_0=\left(\begin{smallmatrix}
1+p t& -p z \\ -p^{m+1}y  & 1+p x
\end{smallmatrix}\right)$. Its image in $(I_m^+)_\Phi$ (see \eqref{f:I+ab}) corresponds to
$$(-y, 1-p x, 0)= (-y, 1+p t, 0)\in  \mathbb Z_p/p\mathbb Z_p \times (1+p\mathbb Z_p) \big/ (1+p^2\mathbb Z_p) \times  \mathbb Z_p/p\mathbb Z_p \ .$$
This proves that   $\Sh_{ s_0 w}(\tau_{s_0}	\cdot c)$ is given by
$(y,1+px, z)\mapsto -c^-(y)-c^0(1+px)$ namely $$\tau_{s_0}	\cdot c=(0, -c^0, -c^-)_{s_0w} \textrm{ if $m\geq 1$}$$ and if $m=0$ then  $\tau_{s_0}	\cdot c=(0, 0, -c^-)_{s_0w}.$
%


\noindent\textbf{B) Now suppose $\ell(s_0w)=\ell(w)-1$.} Then $\tau_{s_0}\cdot c$ has a component
$\gamma_{s_0w}\in h^1(s_0w)$ and a component  $\sum_{u\in \mathbb F_p^\times }\gamma_{\omega_u w}\in \oplus \bigoplus_{\omega \in \Omega } h^1(\omega w)$. Recall that $\omega_u$ was defined in \eqref{omegaz}.

\noindent\textbf{1)} We compute  $\sum_{u\in \mathbb F_p^\times }\gamma_{\omega_u w}\in \oplus \bigoplus_{\omega \in \Omega } h^1(\omega w)$.\\
In fact, for  all $u\in\mathbb F_p^\times $,  we compute the elements $\varepsilon_u\in H^1(I_w,k)$ defined by  $$e_1\cdot c+\sum_{u\in \mathbb F_p^\times }\gamma_{\omega_u w} = \sum_{{u\in \mathbb F_p^\times }} \Sh_{\omega_u w}^{-1}(\varepsilon_u)\in  \bigoplus_{{u\in \mathbb F_p^\times }} h^1(\omega_u w)$$ namely
$\varepsilon_u=\Sh_{\omega_u w}(\gamma_{\omega_u w})-(\omega_u)_*\Sh_w(c)$.\\
Recall $I_{\omega w}=I_m^-=\left(\begin{smallmatrix}
1+{p\mathbb Z_p} & p^m\mathbb Z_p \\ p\mathbb Z_p & 1+p\mathbb Z_p
\end{smallmatrix}\right)$ for any $\omega\in \Omega $. Compute
 $s_0 \omega_u^{-1} \left(\begin{smallmatrix}1&[u]^{-1}\cr 0&1\end{smallmatrix}\right) s_0^{-1}=\begin{psmallmat} 1&0\cr-[u]& 1\cr\end{psmallmat} \omega_u$. Therefore, by  \eqref{f:leftsbad2}
$$ \Sh_{\omega_u w}(\gamma_{\omega_u w})-(\omega_u)_*\Sh_w(c):  X\mapsto
(\omega_u)_*\Sh_w(c)\big(\begin{psmallmat} 1&0\cr-[u]& 1\cr\end{psmallmat} ^{-1} X\begin{psmallmat} 1&0\cr-[u]& 1\cr\end{psmallmat} X^{-1}\big)$$
 for any
$X:=\left(\begin{smallmatrix}
1+p x& p^m y \\ p z & 1+p t
\end{smallmatrix}\right)\in I_w$.
We have
$$\begin{psmallmat} 1&0\cr-[u]& 1\cr\end{psmallmat}^{-1} X\begin{psmallmat} 1&0\cr-[u]& 1\cr\end{psmallmat} X^{-1}=\left(\begin{smallmatrix}
1+p x- p^m y [u]& p^m y \\  p z+p(x-t)[u]-p^m y[u^2] & 1+p t+p m y[u]
\end{smallmatrix}\right) X^{-1}.$$ Via \eqref{f:I-abQp}  this image of this element  in $(I_m^-)_{\Phi}$ corresponds to
$$(2 x[u]-p^{m-1} y [u]^2, 1-p^{m}y[u], 0)\in  \mathbb Z_p/p\mathbb Z_p \times (1+p\mathbb Z_p) \big/ (1+p^2\mathbb Z_p) \times  \mathbb Z_p/p\mathbb Z_p.$$
So for $u\in \mathbb F_p^\times$, we just computed that
$\Sh_{\omega_u w}(\gamma_{\omega_u w})-(\omega_u)_*\Sh_w(c)$ is the element $\varepsilon_u$ in  $\Hom( I_w, k)$ sending $X\in I_w$ to
\begin{align*}(\omega_u)_*\Sh_w(c)( (2 x[u]-p^{m-1} y [u]^2, 1-p^my[u], 0))&=\Sh_w(c) ((2 x[u]^{-1}-p^{m-1} y, 1-p^my[u], 0))\cr&=c^-(2 x[u]^{-1}-p^{m-1} y)+ c^0(1-p^my[u])
\end{align*}
If $m=1$,  then
$\varepsilon_u$ sends $X$ onto (see notation \eqref{f:iota}):
$$[u]^{-1}2c^-\iota(1+p x)\:- [u]^{-2} c^-( [u]^2y)\:- [u]^{-1}c^0\iota^{-1}(y[u]^2)).$$ Using \eqref{f:uplambdac} we see that its preimage by $\Sh_{\omega_u w}$ is the component in $h^1(\omega_uw)$ of
$e_{\id}\cdot (0, -2c^-\iota, 0)_w+ e_{\id^2}\cdot (0, 0, c^-)_w+e_{\id}\cdot (0, 0, c^0\iota^{-1})_w$
so when $m=1$, we have
$$\sum_{u\in \mathbb F_p^\times }\gamma_{\omega_u w}=- e_1\cdot (c^-, c^0, c^+)_w+e_{\id}\cdot (0, -2c^-\iota, c^0\iota^{-1})_w+ e_{\id^2}(0, 0, c^-)_w.$$
If $m\geq 2$, then the only remaining component of $\varepsilon_u$ is $X\mapsto [u]^{-1}2c^-\iota(1+p x)$ so we obtain
$$\sum_{u\in \mathbb F_p^\times }\gamma_{\omega_u w}=- e_1\cdot (c^-, c^0, c^+)_w+e_{\id}\cdot (0, -2c^-\iota, 0)_w.$$

\noindent \textbf{2)} We compute $\gamma_{s_0w}\in h^1(s_0w)$.  By \eqref{f:leftsbad1} we have
  $\Sh_{s_0 w}(\gamma_{ s_0 w})=\cores_{I_{s_0w}}^{ s_0 I_w  s_0^{-1}}\big( {s_0}_* \Sh_w(c)\big)$.  By  Lemma \ref{lemma:conjtrans},  the composite map $(I_{s_0w})_\Phi\xrightarrow{{\rm tr}} (s_0 I_w s_0^{-1})_\Phi\xrightarrow{{s_0^{-1}}_- {s_0}} (I_w)_\Phi$   is
\begin{align*}
(z, 1+px,y) \mapsto (-y,0,0) \in \mathbb Z_p/p\mathbb Z_p \times (1+p\mathbb Z_p) \big/ (1+p^2\mathbb Z_p) \times  \mathbb Z_p/p\mathbb Z_p
\end{align*}
This shows that $\gamma_{s_0 w}=(0,0,-c^-)_{s_0w}$.
%
%
%

\subsection{\label{proof:theformulas'}Proof of Proposition \ref{prop:theformulas'}} Let $w\in {\widetilde W}$ and $\alpha=(\alpha^-, \alpha^0, \alpha^+)_w\in h^1(w)^\vee\subset {}^\anti((E^{1})^{\vee, f})^\anti$. We suppose that $s=s_0$, the case $s=s_1$ following by conjugation by $\varpi$ (by the map \eqref{gammapi} which is compatible with the Yoneda product).
\begin{itemize}

\item Suppose that $\ell({s_0}w)=\ell(w)+1$. By  \eqref{f:support} we know that  $\tau_{s_0^{-1}} \cdot \alpha=\alpha(\tau_{{s_0}}\,\cdot\, _ -)$ has support in $h^1(s_0^{-1}w)$.
Let $c=(c^-, c^0, c^+)_{s_0^{-1}w}\in h^1(s_0^{-1}w)$.
We compute $(\tau_{s_0^{-1}} \cdot  \alpha) (c)=\alpha(\tau_{{s_0}}\cdot c)$.
By Proposition \ref{prop:theformulas}, the component in  $h^1(w)$ of $\tau_{s_0}\cdot   c$ is $(0,0,-c^-)_{w}$.  Therefore
  $( \tau_{s_0^{-1}} \cdot  \alpha)(c)=\alpha((0,0,-c^-)_{w})=-c^-(\alpha^+)$ and   $\tau_{s_0^{-1}}\cdot   \alpha= (-\alpha^+,0,0)_{s_0^{-1}w}$.
  Using \eqref{f:lefts2'}, it gives $\tau_{s_0}  \cdot \alpha= (-\alpha^+,0,0)_{s_0w}$.


\item Suppose that $\ell({s_0}w)=\ell(w)-1$. By  Proposition \ref{prop:yo} (or \eqref{f:support}) we know that
 $\tau_{s_0^{-1}}\cdot   \alpha=\alpha(\tau_{{s_0}} \cdot \,_ -)$
 has support in
$
 h^1( s_0^{-1}w) \oplus \bigoplus_{\omega\in \Omega }  h^1( \omega w)$.
\begin{itemize}
\item Compute its component in $(h^1(s_0^{-1} w))^\vee$:  \\
We compute $(\tau_{s_0^{-1}}\cdot   \alpha) (c)=\alpha(\tau_{s_0} \cdot   c)$
for $c=(c^-, c^0, c^+)_{s_0^{-1}w}\in h^1(s_0^{-1}w)$ with $c^0=0$ if $\ell(w)=1$. By Proposition \ref{prop:theformulas}, the element $\tau_{s_0} \cdot   c$ lies in $h^1(w)$ and is equal to
$(0,  {-}c^0, -c^-)_{w}$. Therefore
  $(\tau_{s_0^{-1}}\cdot  \alpha)(c)= {-} c^0(\alpha_0)-c^-(\alpha^+)$, and the component in $(h^1(s_0^{-1}w))^\vee$ of $\tau_{s_0^{-1}}\cdot  \alpha$ is $(-\alpha^+,  {-} \alpha_0,0)_{s_0^{-1}w}$ if $\ell(w)\geq 2$
 and
$ (-\alpha^+,  0,0)_{s_0^{-1}w}$ if $\ell(w)= 1$.   Using \eqref{f:lefts2'},
 the component in $(h^1(s_0w))^\vee$ of $\tau_{s_0}\cdot  \alpha$ is $(-\alpha^+,  {-} \alpha_0,0)_{s_0w}$ if $\ell(w)\geq 2$   and
$ (-\alpha^+,  0,0)_{s_0w}$ if $\ell(w)= 1$.
\item
Compute the component $\sum_{u\in \mathbb F_p^\times}\beta_{\omega_uw}$ in $\oplus_{u\in \mathbb F_p^\times}(h^1( \omega_u w))^\vee$ of $\tau_{s_0^{-1}} \cdot  \alpha$: \\
The component in
$(h^1(w))^\vee$ of $(\tau_{{\omega_u}^{-1}}\tau_{s_0^{-1}} \cdot  \alpha)$ is $\tau_{{\omega_u}^{-1}}\cdot \beta_{\omega_uw}$. We therefore compute
$(\tau_{\omega_u}^{-1}\cdot \beta_{\omega_uw})(c)=\alpha(\tau_{s_0} \tau_{\omega_u}\cdot  c) $  for $c =(c^-, c^0, c^+)_{w} \in  h^1(w)$.
By  Proposition \ref{prop:theformulas} and the definition of the idempotents \eqref{defel} (see also \eqref{tll}), the component in $h^1(w)$  of  $\tau_{s_0} \tau_{\omega_u}\cdot  c= \tau_{\omega_u^{-1}}\tau_{s_0}\cdot  c$ is
\begin{align*}
\begin{cases}    (c^-, c^0 , c^+)_{ w}+
  u^{-1} (0, 2c^-\iota , 0)_{ w} & \text{ if   $\ell(w) \geq 2$, }\cr  (c^-, c^0 , c^+)_{ w}
+  u^{-1} (0, 2c^-\iota , -c^0\iota^{-1})_{ w} +u^{-2} (0,0 , -c^-)_{ w}  & \text{ if  $\ell(w) =1$ }.
\end{cases}
\end{align*}
Therefore
\begin{multline*}
\alpha(\tau_{s_0} \tau_{\omega_u}\cdot  c) =   \\
\begin{cases}
  c^-(\alpha^-)+c^0(\alpha^0)+c^+(\alpha^+)+u^{-1}2c^-\iota (\alpha^0)
   & \text{ if   $\ell(w) \geq 2$, }\cr
   c^-(\alpha^-)+c^0(\alpha^0)+c^+(\alpha^+)+u^{-1}2c^-\iota (\alpha^0)
     -u^{-1}c^0\iota^{-1}(\alpha^+)-u^{-2}c^-(\alpha^+)
      & \text{ if   $\ell(w) =1$ }.
\end{cases}
\end{multline*}
So
\begin{multline*}
\beta_{\omega_uw} =   \\
\begin{cases}
 \tau_{\omega_u}\cdot  (\alpha^-, \alpha^0, \alpha^+)_w+u^{-1}\tau_{\omega_u}\cdot(2\iota(\alpha^0),0,0)_w      & \text{ if   $\ell(w) \geq 2$, }\cr
 \tau_{\omega_u}\cdot  (\alpha^-, \alpha^0, \alpha^+)_w+u^{-1}\tau_{\omega_u}\cdot( 2\iota(\alpha^0),-\iota^{-1}(\alpha^+),0)_w-u^{-2}\tau_{\omega_u}\cdot ( \alpha^+, 0,0)_w
      & \text{ if   $\ell(w) =1$ }
\end{cases}
\end{multline*}
and
\begin{multline*}
\sum_{u\in \mathbb F_p^\times }  \beta_{\omega_uw} =   \\
\begin{cases}
-e_1 \cdot (\alpha^-, \alpha^0, \alpha^+)_w- e_{\id}\cdot (2\iota(\alpha^0),0,0)_w      & \text{ if   $\ell(w) \geq 2$, }\cr
-e_1 \cdot (\alpha^-, \alpha^0, \alpha^+)_w-e_{\id}\cdot ( 2\iota(\alpha^0),-\iota^{-1}(\alpha^+),0)_w+e_{\id^2}\cdot ( \alpha^+, 0,0)_w
      & \text{ if   $\ell(w) =1$ }.
\end{cases}
\end{multline*}
The component in $\oplus_{u\in \mathbb F_p^\times}(h^1( \omega_u w))^\vee$ of $\tau_{s_0} \cdot  \alpha$ is
\begin{multline*}
\tau_{s_0^2}\cdot \sum_{u\in \mathbb F_p^\times }  \beta_{\omega_uw} =   \\
\begin{cases}
-e_1 \cdot (\alpha^-, \alpha^0, \alpha^+)_w+e_{\id}\cdot (2\iota(\alpha^0),0,0)_w      & \text{ if   $\ell(w) \geq 2$, }\cr
-e_1 \cdot   (\alpha^-, \alpha^0, \alpha^+)_w+e_{\id}\cdot ( 2\iota(\alpha^0),-\iota^{-1}(\alpha^+),0)_w+e_{\id^2} \cdot  ( \alpha^+, 0,0)_w
      & \text{ if   $\ell(w) =1$ }.
\end{cases}
\end{multline*}

 \end{itemize}

  \end{itemize}

\noindent Rachel Ollivier\\
Mathematics Department, 
University of British Columbia\\
1984 Mathematics Road,
Vancouver BC V6T 1Z2, Canada\\
ollivier@math.ubc.ca\\
http://www.math.ubc.ca/~ollivier \\

\noindent Peter Schneider\\
Mathematisches Institut, 
Westf\"alische Wilhelms-Universit\"at M\"unster\\
Einsteinstr.\ 62,
48149 M\"unster, Germany\\
pschnei@uni-muenster.de\\
https://www.uni-muenster.de/Arithm/schneider\\



%
%
%
%
%
%
%
%
%
%
%

\addcontentsline{toc}{section}{References}

\begin{thebibliography}{B-GAL}


\bibitem[HV]{HV}
Henniart G., Vigneras M.-V.: \emph{Representations of a $p$-adic group in characteristic $p$}. Representations of Reductive Groups, Conf.\ in honor of J.\ Bernstein (Eds.\ Aizenbud, Gourevitch, Kazhdan, Lapid), Proc.\ Symp.\ Pure Math.\ 101, 171 - 210 (2019)

\bibitem[Hup]{Hup}
Huppert B.: \emph{Endliche Gruppen I}. Springer 1967

\bibitem[Jen]{Jen}
Jensen C.U.: \emph{Les Foncteurs D\'eriv\'es de $\varprojlim$ et leurs Applications en Th\'eorie des Modules}. Springer Lect.\ Notes Math.\ 254 (1972)

\bibitem[KS]{KS} Klopsch B., Snopce I.: \emph{A characterization of uniform pro-p groups}. Q. J. Math. 65, no. 4, 1277--1291.
(2014)

\bibitem[Koz1]{Koz1} Koziol K.: \emph{Pro-p-Iwahori invariants for SL2 and L-packets of Hecke modules}
International Mathematics Research Notices, no. 4, 1090-1125 (2016)

\bibitem[Koz2]{Koziol} Koziol K.: \emph{Hecke module structure on first and top pro-$p$-Iwahori Cohomology}.  Acta Arithmetica 186  349-376 (2018)

\bibitem[Laz]{Laz}
Lazard M.: \emph{Groupes analytiques $p$-adiques}.
Publ.\ Math.\ IHES 26, 389--603 (1965)


\bibitem[MCR]{MCR}
McConnell J.C., Robson J.C.: \emph{Noncommutative Noetherian Rings}. AMS 2001

\bibitem[NSW]{NSW}
Neukirch J., Schmidt A., Wingberg K.: \emph{Cohomology of Number Fields}. Springer 2000

\bibitem[Oll1]{Ollequiv}
Ollivier R.: \emph{Le foncteur des invariants sous l'action du pro-$p$-Iwahori de $GL_2(F)$}. J.\ reine angew.\ Math.\ 635, 149--185 (2009)

\bibitem[Oll2]{Ollcompa}
Ollivier R.: \emph{Compatibility between Satake and Bernstein  isomorphisms in characteristic $p$}. ANT 8-5, 1071--1111 (2014)


\bibitem[OS1]{OS1}
Ollivier R., Schneider P.: \emph{Pro-$p$ Iwahori Hecke algebras are Gorenstein}. J.\ Inst.\ Math.\ Jussieu 13, 753--809 (2014)

\bibitem[OS2]{embed}
Ollivier R., Schneider P.: \emph{A canonical torsion theory for pro-$p$ Iwahori-Hecke modules}.
Advances in Math.\ 327, 52 - 127 (2018)

\bibitem[OS3]{Ext}
Ollivier R., Schneider P.: \emph{The modular pro-$p$ Iwahori-Hecke $\Ext$-algebra}. Representations of Reductive Groups, Conf.\ in honor of J.\ Bernstein (Eds.\ Aizenbud, Gourevitch, Kazhdan, Lapid), Proc.\ Symp.\ Pure Math.\ 101, 255 - 308 (2019)

\bibitem[OV]{OV}Ollivier R., Vign\'eras M.-F.,
\emph{Parabolic induction in characteristic p.} . Selecta Math., Vol. 24, Issue 5,  3973-4039 (2018)




\bibitem[Pas]{Pas}
Pa$\check{\textrm{s}}$k$\bar{\textrm{u}}$nas, V.:\emph{The image of Colmez's Montr\'eal functor}. Publ.\ Math.\ IHES 118, 1--191 (2013)

\bibitem[Sch-DGA]{SDGA}
Schneider P.: \emph{Smooth representations and Hecke modules in characteristic $p$}. Pacific J. Math. 279, 447--464 (2015) \\


\end{thebibliography}
 \end{document}